\documentclass[a4paper,11pt]{article}
\usepackage[utf8]{inputenc}
\usepackage[margin=2cm]{geometry}

\usepackage{xcolor}
\definecolor{myblue}{rgb}{0.1 0.1 0.6}
\usepackage{hyperref}
\hypersetup{
   colorlinks=true,
   linkcolor=myblue,
   citecolor=myblue,
   urlcolor=myblue
}

\usepackage[T1]{fontenc}
\usepackage{amsmath,amssymb,amsthm}
\usepackage{newpxtext}
\usepackage{eulerpx}

\usepackage{url}
\usepackage[shortlabels]{enumitem}

\usepackage{tikz}
\usetikzlibrary{calc}
\newenvironment{sd}{\begin{array}{c} \begin{tikzpicture}}{\end{tikzpicture} \end{array}}
\usetikzlibrary{positioning}
\def\hsep{1cm}

\theoremstyle{plain}
\newtheorem{proposition}{Proposition}[section]
\newtheorem{theorem}[proposition]{Theorem}
\newtheorem{corollary}[proposition]{Corollary}
\newtheorem{lemma}[proposition]{Lemma}

\newtheorem{definition}[proposition]{Definition}
\newtheorem{example}[proposition]{Example}
\newtheorem{question}[proposition]{Question}

\newtheorem{namedtheoremx}{Theorem}
\newenvironment{namedtheorem}[1]
 {\renewcommand\thenamedtheoremx{#1}\namedtheoremx}
 {\endnamedtheoremx}

\theoremstyle{remark}
\newtheorem{remarkx}[proposition]{Remark}

\newenvironment{remark}
  {\pushQED{\qed}\remarkx}
  {\popQED\endremarkx}

\newcommand{\mc}[1]{\mathcal{#1}}
\newcommand{\mf}[1]{\mathfrak{#1}}

\newcommand{\ip}[2]{{\langle {#1} , {#2} \rangle}}
\newcommand{\lin}{\operatorname{lin}}
\newcommand{\id}{\operatorname{id}}

\newcommand{\vnten}{\bar\otimes}
\newcommand{\proten}{\widehat\otimes}

\newcommand{\wh}{\widehat}

\newcommand{\op}{{\operatorname{op}}}
\newcommand{\Tr}{{\operatorname{Tr}}}
\newcommand{\im}{{\operatorname{Im}}}
\newcommand{\KMS}{\textrm{KMS}}
\newcommand{\clos}{\overline{\phantom{d}}}
\newcommand{\closhs}{\clos^{\|\cdot\|_{HS}}}
\newcommand{\floatinghat}{\hat{\phantom{x}}}
\newcommand{\rankone}[2]{|{#1}\rangle \langle{#2}|}
\newcommand{\rd}{\,\mathrm{d}}

\begin{document}

\title{Quantum graphs in infinite-dimensions: Hilbert--Schmidts and Hilbert modules}
\author{Matthew Daws}
\maketitle

\begin{abstract}
We develop two approaches to Quantum (or Non-commutative) Graphs based on arbitrary von Neumann algebras $M\subseteq\mc B(H)$: one looking at operator bimodules of Hilbert--Schmidt (instead of bounded) operators, and the second looking at Quantum Adjacency Operators.  Hilbert--Schmidt Quantum Graphs relate to Weaver's picture of Quantum Graphs in a complex way: by defining certain hull operations, we find a bijection between certain subsets of both objects.  Given a nfs weight $\varphi$ on $M$ the operator-valued weight $\varphi^{-1}$ can be defined, as considered by Wasilewski for direct sums of matrix algebras.  We show how to build a natural self-dual Hilbert $C^*$-module from this, which mediates a bijection between HS Quantum Relations and projections $e\in M\vnten M^\op$.  When $e$ is integrable for the slice-map $\id\otimes\varphi^\op$ there is a related normal CP map $A\colon M\to M$: this is a Quantum Adjacency Operator, which has a Kraus operator representation built from the HS Quantum Relation.  When $e$ and its tensor swap map are both integrable, we find certain symmetries of $A$.  We illustrate our theory by a careful consideration of certain examples, including detailed links with the finite-dimensional setting.
\end{abstract}

\tableofcontents

\section{Introduction}

A \emph{Quantum (or non-commutative) graph} is a non-commutative topology generalisation of a graph: that is, an object modelled on a C$^*$-algebra or von Neumann algebra which for a commutative algebra reduces to the notion of a graph.  The concept arose from work of Weaver \cite{Weaver_QuantumRelations} on Quantum Relations (a simple graph with a loop at each vertex can be regarded as a reflexive and symmetric relation on the vertex set), namely given a von Neumann algebra $M\subseteq \mc B(H)$, a quantum relation is a weak$^*$-closed $\mc S\subseteq\mc B(H)$ which is a bimodule over the commutant $M'$, meaning that $a,b\in M', x\in \mc S \implies axy\in\mc S$.  Essentially the same concept also arose in the study of quantum communication channels \cite{DSW_ZeroError}, where here $M = \mathbb M_n$ is a finite matrix algebra, and in \cite{MRV_Compositional} from a categorical approach to finite-dimensional $C^*$-algebras and quantum information theory.  This latter reference also introduced the idea of a \emph{Quantum adjacency operator} which directly generalises the notion of the adjacency matrix of a graph.  Initially the link between subspaces $\mc S$ and adjacency operators $A$ was mediated by use of a trace on the $C^*$-algebra in question (say $B$, here always finite-dimensional and so the sum of matrix algebras).

Subsequent work, for example \cite{BCEHPSW_Bigalois, daws_quantum_graphs, matsuda_class_m2}, explored how the theory changes if a general (faithful) state on $B$ is used instead of a trace.  In particular, there are a number of axioms on $A$ which in the commutative case correspond to the graph being undirected ($A$ is self-adjoint, or ``symmetric'', or ``real'', see \cite{matsuda_class_m2} for a careful discussion).  Taken together, these imply a hidden symmetry: $A$ and/or $\mc S$ must commute with the modular automorphism group of the state (see \cite[Section~5]{daws_quantum_graphs}) which could be considered to be a rather strong condition.  The paper \cite{CW_RandomQGraphs} instead considered the axiom that $A$ is completely positive (equivalent to $A$ being ``real'', \cite{matsuda_class_m2}) which was further explored in \cite[Section~5.4]{daws_quantum_graphs} and \cite{Wasilewski_Quantum_Cayley}.  We believe that this is perhaps the most natural axiom: see Section~\ref{sec:fd_summary} for a self-contained summary of this approach.

The approach via adjacency operators is finite-dimensional, while Weaver's approach allows for a completely general von Neumann algebra.  A step towards considering infinite-dimensional adjacency operators was taken in \cite{Wasilewski_Quantum_Cayley}, by looking at the case when $M$ is the possibly infinite direct sum of matrix algebras.  Another generalisation is \cite{BH_QuantumGraphs_Subfactors} where 2-$C^*$-categories are used, but in many ways this is still a finite-dimensional setting, due to the use of $Q$-systems which have a notion of a basis.  This paper continues this study of (possibly) infinite-dimensional quantum graphs, where we look at arbitrary von Neumann algebras equipped with a nfs weight.

It is our contention that finite-dimensional quantum graphs are a fundamentally \emph{Hilbert space} theory.  In particular, our approach in \cite{daws_quantum_graphs} used in an essential way when $H$ is a finite-dimensional Hilbert space, the bounded operators $\mc B(H)$ coincide with the Hilbert--Schmidt operators $HS(H)$, which in turn may be identified with the Hilbert space tensor product $H \otimes \overline H$.  See Section~\ref{sec:fd_summary} for a refinement of this idea.  When $M \subseteq \mc B(H)$ is a von Neumann algebra (of course, here just a finite direct sum of matrix algebras) we can naturally represent $M \otimes M^\op$ on $H\otimes\overline H$, and quickly obtain a bijection between projections $e \in M \otimes M^\op$ and $M' \otimes (M')^\op$-invariant subspaces (just the image of $e$) of $H\otimes\overline H$, equivalently, $M'$-bimodules of $HS(H)$.  By ``bimodule'' we again mean $V \subseteq HS(H)$ a subspace such that $x\in V, a,b\in M' \implies axb\in V$.

All of this works verbatim for an infinite-dimensional Hilbert space, excepting that $HS(H)$ is of course no longer equal to $\mc B(H)$.  This motivated us to define in Section~\ref{sec:HS} the notion of a \emph{Hilbert--Schmidt (HS) quantum relation}, namely a subspace $V\subseteq HS(H)$ which is closed for the Hilbert--Schdmit norm, and which is an $M'$-bimodule, where here $M\subseteq\mc B(H)$ is a von Neumann algebra.  Here we are considering general graphs (so relations which might be neither reflexive nor symmetric) and so hew to Weaver's terminology of ``relation''.  These objects are seen to biject with projections $e\in M\vnten M^\op$.  Letting $\mc S$ be the weak$^*$-closure of $V$, we obtain a quantum relation in Weaver's sense; conversely, starting with $\mc S$ we can let $V = \mc S \cap HS(H)$ (which is closed, Lemma~\ref{lem:hs_facts}) to obtain a Hilbert--Schmidt quantum relation.  These operations are not inverses: it can easily be that $\mc S \cap HS(H) = \{0\}$ with $\mc S$ non-trivial, for example.  However, we show that there are natural ``closure'' or ``hull'' operations, and this gives a bijection between a subclass of quantum relation and a subclass of HS quantum relation, Theorem~\ref{thm:nice_subclass}.  In Weaver's theory, it is important that when $M$ is represented on two Hilbert spaces $H_1, H_2$ then quantum relations over $H_1$ biject with those over $H_2$.  We show that the same holds for HS quantum relations, and that our natural operations (of weak$^*$-closure, taking hulls, and so forth) respect this bijection, see Section~\ref{sec:inv_reps}.  We make some brief comments about the ``pullback'' construction in Section~\ref{sec:pullbacks}.  We give many examples to show some of the complications with this idea; however, if $M$ has a sufficiently rich finite-dimensional structure, in particular, if $M$ is the (possibly infinite) direct sum of matrix algebras, then there is in fact a bijection between HS quantum relations, and Weaver quantum relations.

Our aim in the rest of this paper is to explore what can be said about the quantum adjacency operator picture, in infinite-dimensions.  Here we are strongly influenced by Wasilewski's work in \cite{Wasilewski_Quantum_Cayley}, as well as Matsuda's notion of a ``real'' adjacency operator \cite{matsuda_class_m2}, and Yamashita's lecture notes \cite{Yamashita_QG_Notes}.  In Section~\ref{sec:fd_summary} we give a summary of the finite-dimensional picture, through the lens of these references.  In particular, compared to our paper \cite{daws_quantum_graphs}, we make some comments about the diagrammatic calculus, and follow the idea of \cite{Yamashita_QG_Notes} to introduce a non-trivial inner-product on $\mc B(L^2(B))$.  Here $B$ is our finite-dimensional $C^*$-algebra with a state $\varphi$; that we no longer use the canonical trace on $\mc B(L^2(B))$ means that the association between $V\subseteq H\otimes \overline H$ and $S\subseteq\mc B(L^2(B))$ becomes ``twisted'' by the modular operator.  Rather than working with many possible links between adjacency operators and projections (the family of maps $\Psi_{s,t}$ and $\Psi'_{s,t}$ which we defined and explored in \cite{daws_quantum_graphs}) we instead follow \cite[Section~5.4]{daws_quantum_graphs} and take as an axiom that quantum adjacency operators are \emph{completely positive} (equivalently, \emph{real}, \cite{matsuda_class_m2}).  This suggests using the bijection $\Psi'_{0,1/2}$, which we show agrees with the work in \cite{Wasilewski_Quantum_Cayley} which is motivated by KMS-inner-products.  With these conventions in place, we obtain a very pleasing collection of results linking the different pictures of quantum graphs, Section~\ref{sec:adj_proj}.

One key difference with \cite{Wasilewski_Quantum_Cayley} compared to other work is the use of KMS-inner-products, Section~\ref{sec:KMS_IPs}, which motivates introducing what we call a ``twist'': not looking at the subspace $S\subseteq\mc B(L^2(B))$ but rather $S_{i/4}$, Section~\ref{sec:diff}.  We provide some different motivation for this in Section~\ref{sec:op_sys_A} by considering the operation of the tensor swap map on $B\otimes B^\op$, which maps $e$ to a (possibly) different projection $\tau(e)$.  When $\varphi$ is a trace, this corresponds to taking the adjoint on $S$ (which in Weaver's picture is the natural operation corresponding to ``reversing the direction of edges'' in the graph).  This no longer holds for non-tracial $\varphi$, but it does work for $S_{i/4}$.  Section~\ref{sec:fd_summary} is party a survey, but it forms key motivation for the technical infinite-dimensional results to follow.  We are fairly convinced now that taking a quantum adjacency matrix $A$ to be completely positive is the ``correct'' axiom (instead of assuming $A$ is self-adjoint and/or ``symmetric''); we believe that it will be an aid to the reader to have a short, and self-contained, summary of what this entails, while also making links with ideas such as the diagrammatic calculus, and KMS inner-products.

There is an issue here which may puzzle the reader, as it did the author.  For any state $\varphi$ on $B$ we have established bijections between projections $e\in B\otimes B^\op$ (or equivalently, in our language, HS quantum relations on $L^2(B)$), and Weaver quantum relations $S$ (or, $S_{i/4}$ which seems the ``better'' choice).  Neither of these objects make any use of $\varphi$: only the definition of the bijection does.  The adjacency operator $A$ does ``see'' $\varphi$ (in the definition of what it means for $A$ to be Schur--idempotent).  However, we now have the following procedure for reducing the study of any quantum graph to that of a tracial quantum graph: start with $A$ (or $S_{i/4}$) and then form $e$ using the bijection from $\varphi$.  Now use the bijection coming from a trace on $B$ to move back to a tracial adjacency operator, say $A_{\Tr}$, or a new Weaver quantum graph.  We explore this in Section~\ref{sec:move_tracial_case}.  Have we reduced all quantum graphs to the tracial case?  In one sense yes, but we look at an example-- the notion of ``connectivity'' from \cite{courtney2025connectivityquantumgraphsquantum}-- to show that natural properties of $A$ are \emph{not} transported to natural properties of $A_{\Tr}$.  This is as it should be, for the richness of the theory.  Nevertheless, we suspect that this idea of moving the tracial case will find applications in the future.

Our task in the main body of the paper is to marry these finite-dimensional ideas with that of HS quantum relations, a task we start in Section~\ref{sec:hilb_mods}.  Fix a von Neumann algebra $M$ and a nfs weight $\varphi$: we make no assumptions about $\varphi$ at all.  Due to the technical nature of many of the ideas in these later sections, and our desire to appeal to a readership perhaps coming from the finite-dimensional picture of quantum graphs, we have collected both technical results, and background reading, into some appendices.  For weights, see Section~\ref{sec:notation} below, and Appendix~\ref{sec:weights}.  Using $\varphi$ we form the GNS space, and almost always regard $M$ as acting on $H = L^2(M)$.

We are motivated by \cite{Wasilewski_Quantum_Cayley} to consider the operator-valued weight (see Appendix~\ref{sec:op_valued_weights} and Appendix~\ref{sec:inv_weight}) $\varphi^{-1}$ from $\mc B(H)$ to $M' \subseteq\mc B(H)$.  As $H=L^2(M)$ there is a canonical isomorphism $M' \cong M^\op$, and a canonical choice of weight $\varphi'$ on $M'$ coming from $\varphi$ and/or $\varphi^\op$.  We define $\tilde\varphi = \varphi' \circ \varphi^{-1}$ which by general theory is a nfs weight on $\mc B(H)$.  Indeed, in the finite-dimensional situation, this is \emph{exactly} the state which was introduced in Section~\ref{sec:ip}: this is the first of many ways we find ``shadows'' of the finite-dimensional situation occurring in the more technical infinite-dimensional picture.  Any weight on $\mc B(H)$ is given by a ``density'', Appendix~\ref{sec:weights_bh}, and we show carefully in Appendix~\ref{sec:inv_weight} that the density of $\tilde\varphi$ is $\nabla^{-1}$ the inverse of the modular operator.  As hinted at in, for example, \cite{BDH_IndexCondExp}, the definition ideal $\mf n_{\varphi^{-1}}$ of $\varphi^{-1}$ is a pre-Hilbert $C^*$-module over $M'$.  We explore this idea more fully in Section~\ref{sec:hilb_mods}.  Any Hilbert $C^*$-module over a von Neumann algebra has a ``self-dual completion'': this work, going back to \cite{paschke_inner_prod_mods, Rieffel_mortia_cstar_wstar}, is summarised in Appendix~\ref{sec:selfdual_mods}.  We apply this to $\mf n_{\varphi^{-1}}$, and use the various concrete isomorphisms we have developed to show that the self-dual completion is $E_{\varphi^{-1}} = \mc B_{M'}(H, HS(H))$.  This is the (concrete) Hilbert $C^*$-module of maps $H \to HS(H)$ which are right module maps for the $M'$ action.  In particular, we see the natural occurrence of $HS(H)$ here.  As $\varphi^{-1}$ is actually a bimodule map, $\mf n_{\varphi^{-1}}$ also has a natural left $M'$-action, which extends to $E_{\varphi^{-1}}$.  Indeed, the algebra of adjointable operators on $E_{\varphi^{-1}}$ is isomorphic to $M \vnten M^\op$, and we obtain exactly one half of our picture from before, and notice the natural appearance of HS quantum relations:

\begin{namedtheorem}{{\ref{thm:comp_submods}}}
There is a bijection between: (1) projections $e\in M\vnten M^\op$; (2) complemented $M'$-sub-bimodules of $E_{\varphi^{-1}}$; (3) $M'$-bimodules in $HS(H)$.
\end{namedtheorem}

In Section~\ref{sec:CB_maps} we turn our attention to adjacency operators: this means a completely bounded normal map $A$ on $M$, with additional properties.  As it is hard to see how to formulate directly the idea that $A$ is Schur--idempotent, when $M$ is infinite-dimensional, we instead look for a bijection between $A$ and $e$.  We re-formulate the finite-dimensional bijection from \cite[Theorem~5.36]{daws_quantum_graphs} (also Section~\ref{sec:adj_proj}) using slice maps, and show how ``integrable'' elements $f,g \in N \vnten M$ give rise to normal CB maps $A = A_{f,g} \colon M \to N$.  Here we think that actually the abstract theory leads to a clearer picture than the concrete formulae we wrote down in \cite{daws_quantum_graphs}.  In particular, integrable projections $e$ (these correspond to the ``bounded degree'' terminology of \cite{Wasilewski_Quantum_Cayley}) give rise to normal CP maps $A \colon M \to M$, which seem like a natural candidate for quantum adjacency operators.  We then proceed to give a very close link between these ideas and the Hilbert module constructions in Section~\ref{sec:hilb_mods}.  In particular, any self-dual Hilbert $C^*$-module has an ``orthogonal basis'', Theorem~\ref{thm:selfdual_is_weak_direct_sum}, and when $e$ is integrable, this basis always arises from $\mf n_{\varphi^{-1}}$ and not its completion, in such a way that a ``Kraus operator'' representation of $A$ is obtained:

\begin{namedtheorem}{{\ref{thm:Kraus_rep}}}
Let $e\in M\vnten M^\op$ be an integrable projection, and form $A = A_{e,e}$ the quantum adjacency operator.  Let $(t_i)$ be an orthogonal basis for $\mc B_{M'}(L^2(M'), \im(e))$.  Then each $t_i = \hat\alpha_i$ for some $\alpha_i \in \mf n_{\varphi^{-1}}$, and $A(x) = \sum_i \alpha_i x \alpha_i^*$ for $x\in M$.
\end{namedtheorem}

Here for $\alpha\in \mf n_{\varphi^{-1}}$ we write $\hat\alpha$ for the image in the self-dual completion $E_{\varphi^{-1}}$.
Letting $V=\im(e)$, we find that $V\subseteq HS(H)$ is an HS quantum relation, and the image of $e$ acting on $E_{\varphi^{-1}}$ is $\mc B_{M'}(L^2(M'), V)$, as seen in the theorem.  It is natural to consider $S \subseteq \mc B(H)$, the intersection of this with $\mf n_{\varphi^{-1}}$, which in the finite-dimensional case is exactly the $\mc S$ from Section~\ref{sec:ip}.  In this general case, $S$ can unfortunately be a lot ``smaller'' than $V$.  In \cite{Wasilewski_Quantum_Cayley}, it is always assumed that $\varphi^{-1}$ is bounded (in fact, normalised to be a conditional expectation) and when this occurs, $S$ is weak$^*$-closed, Proposition~\ref{prop:S_in_case_varphiinv_bdd}; furthermore, when $M$ is the direct sum of matrix algebras, $S$ does ``remember'' all the information about $V$, see Section~\ref{sec:eg_matrix_algs}.  Indeed, in \cite{Wasilewski_Quantum_Cayley}, $S$ itself is sometimes considered as a Hilbert module, and a link is shown between $S$ being self-dual, and $e$ being of bounded degree.  We find the same relation is general:

\begin{namedtheorem}{{\ref{thm:S_bdd_below_e_int}}}
Let $V\subseteq HS(H)$ be an $M'$-bimodule, and let $S = \{ \alpha\in\mf n_{\varphi^{-1}} : \im\hat\alpha \subseteq V \}$.  Then $e$ is integrable if and only if the map $S \to \mc B_{M'}(H,V); \alpha \mapsto \hat\alpha$ is surjective.
\end{namedtheorem}

Here it is important that $M$ acts on $H=L^2(M)$.  In Section~\ref{sec:change_space} we briefly look at more general $H$, and in particular, Example~\ref{eg:matrix_sd_notint} shows that when $H$ is allowed to be ``small'', we can have $S$ self-dual, but $e$ not integrable, hence answering a question from \cite{Wasilewski_Quantum_Cayley}.

In Section~\ref{sec:hilb_ops}, we consider when $A\colon M\to M$ induces a bounded operator $L^2(M)\to L^2(M)$; this occurs naturally when both $e$ and its tensor swap are integrable.  Hence the tensor swap map forms a symmetric in this case, and we show how the KMS inner product arises (in a disguised way, see Proposition~\ref{prop:int_and_coint} and Lemma~\ref{lem:real_J_ad_KMS_ad}).  When we have this extra symmetry, Section~\ref{sec:considerations_of_S} shows that the space $S$ is large enough, and that we can form something like $S_{i/4}$.  This is achieved by showing a direct analogue of Proposition~\ref{prop:S_is_bimod_A}:

\begin{namedtheorem}{\ref{thm:e_int_co_int_A0_generates_S}}
Let $e\in M\vnten M^\op$ be a projection in $\mf n_{\id\otimes\varphi^\op} \cap \mf n_{\varphi\otimes\id}$.  Form $A = A_{e,e}\colon M\to M$ and the associated $A_0\in\mc B(L^2(M))$, and let $V \subseteq HS(H)$ be the image of $e$.  Then $A_0$ is a member of $S = \{ \alpha\in\mf n_{\varphi^{-1}} : \im\hat\alpha \subseteq V \}$.  Furthermore, $\{ \im\hat\alpha : x,y\in M', \alpha = x A_0 y \}$ is dense in $V$, so $\lin \{ \hat\alpha : x,y\in M', \alpha = x A_0 y \}$ is weak$^*$-dense in $e E_{\varphi^{-1}} = \mc B_{M'}(H,V)$.
\end{namedtheorem}

We have said little about Weaver quantum graphs/relations: namely, taking the weak$^*$-closure of $S$ (or perhaps $S_{i/4}$).  This is discussed in Section~\ref{sec:weakstar_closures}.  Later in examples we show that all of the behaviour seen for HS quantum relations occurs in this general situation: it remains for future work to find a fully satisfactory way of passing to weak$^*$-closures.  In Section~\ref{sec:cb_projs} we consider the possibility of the map $\theta_A$ from Section~\ref{sec:adj_proj} being generalised to the infinite-dimensional case: our conclusion is that this seems rather unlikely.

In Section~\ref{sec:egs} we develop a number of examples.  Throughout the development of the theory we have motivated our constructions from the finite-dimensional situation.  In Section~\ref{sec:ellinfty} we quickly look at the commutative atomic case: as expected, we recover the usual notion of a relation / simple graph on a possibly infinite set.  In Section~\ref{sec:eg_matrix_algs} we look at (infinite) direct sums of matrix algebras, making extensive links with \cite{Wasilewski_Quantum_Cayley}, and as already discussed, in Section~\ref{sec:change_space} we consider algebras $M$ acting on Hilbert spaces which are not $L^2(M)$.  In Section~\ref{sec:example_BH} we look at $M=\mc B(K)$ endowed with a general weight.  When the weight is the canonical trace, we show how we exactly recover the ideas of HS quantum relations from Section~\ref{sec:HS}, thus coming a full circle in our narrative.  For general weights, we show how the general theory is rather similar to HS quantum relations, Example~\ref{eg:BH_get_HS_QRs}.

\subsection{Background and notation}\label{sec:notation}

Our inner-products will be linear in the 2nd variable, and written as $(\xi|\eta)$.  Sometimes we consider a Banach space $E$ and its dual $E^*$, and write $\ip{\mu}{x} = \mu(x)$ for the dual pairing between $\mu\in E^*$ and $x\in E$.  We find it helpful to write rank-one operators on a Hilbert space using bra-ket notation: so for $\xi,\eta\in H$ we write $|\xi\rangle\langle\eta|$ for the rank-one operator $\alpha \mapsto (\eta|\alpha) \xi$; in \cite{daws_quantum_graphs} we denoted this by $\theta_{\eta,\xi}$.

We frequently use the conjugate Hilbert space $\overline H = \{ \overline\xi : \xi\in H \}$.  For $x\in\mc B(H)$ we define $x^\top\in\mc B(\overline H)$ as $x^\top\overline\xi = \overline{x^*\xi}$.  For an algebra $A$ let $A^\op$ be the $A$ with the opposite multiplication.  Writing elements of $A^\op$ as $a^\op$ for $a\in A$, we have $a^\op b^\op = (ba)^\op$.  Then $\mc B(H)^\op \cong \mc B(\overline H)$ via $x^\op \mapsto x^\top$.
We will always unitarily identify the Hilbert--Schmidt operators $HS(H)$ with $H\otimes\overline{H}$ for $\rankone{\xi}{\eta} \leftrightarrow \xi\otimes\overline\eta$.  

The early parts of the paper do not have a heavy technical requirement, although as in \cite[Section~5]{daws_quantum_graphs}, we use aspects of modular theory, though only in finite dimensions, where one can easily concretely write down all the maps.  In the later parts of the paper, we make extensive use of modular theory, and Connes spatial theory of von Neumann algebras, for which the books \cite{TakesakiII, StratilaZsido2nd, Stratila_ModTheoryBook2} are our references.  The appendices give summaries of useful results from the literature, as well as giving (sketch) proofs of some results we have not found explicit references for, along with some new results which seem too technical to include in the main body of the paper.  At points, we use the theory of unbounded operators on Hilbert spaces, for which \cite{Schmudgen_UnboundedBook}, for example, is a useful reference.

\subsection{Acknowledgments}

Initial parts of this paper were written at the Isaac Newton Institute for Mathematical Sciences, Cambridge, during the programme Quantum information, quantum groups and operator algebras; this programme was supported by EPSRC grant EP/Z000580/1.  Parts of this work were also supported by EPSRC grant EP/T030992/1.
For the purpose of open access, the author has applied a CC BY public copyright licence to any Author Accepted Manuscript version arising.
No data were created or analysed in this study.

I thank Mateusz Wasilewski and Makoto Yamashita for helpful correspondence.  The TikZ code for the diagrams was adapted from source-code for the paper \cite{matsuda_class_m2}.

\section{Finite-dimensional quantum graphs}\label{sec:fd_summary}

We give a quick summary of the finite-dimensional theory, picking out points which provide motivation for ideas we consider in the infinite-dimensional case.  We follow our paper \cite{daws_quantum_graphs}, but make explicit links with ideas from \cite{Wasilewski_Quantum_Cayley}.  We also make a little use of the diagrammatic calculus, see for example \cite{matsuda_class_m2}, here being inspired in part by the presentation of Yamashita in his notes \cite{Yamashita_QG_Notes}.\footnote{Some of this section was previously informally available on the author's website \url{https://github.com/MatthewDaws/Mathematics/tree/master/Quantum-Graphs}.}

Let $B$ be a finite-dimensional $C^*$-algebra and $\varphi$ a faithful state on $B$.  We do not assume that $\varphi$ is a $\delta$-form.  Choosing a trace on $B$, there is an invertible $Q\in B$ with $\varphi(a) = \Tr(Qa)$ for $a\in B$.  Then the modular data associated to $\varphi$ can be expressed using $Q$, see \cite[Section~5]{daws_quantum_graphs}.

Using $\varphi$ we form the GNS space $L^2(B)$ with GNS map $\Lambda\colon B \to L^2(B)$, which is a linear isomorphism of vector spaces, as $B$ is finite-dimensional.  This allows us to view the multiplication map $m\colon B\otimes B \to B$ as a map $L^2(B)\otimes L^2(B) \to L^2(B)$, and hence to form the Hilbert space adjoint map $m^*$.  Similarly the unit map $\eta \colon \mathbb C \to L^2(B); \lambda \mapsto \lambda\Lambda(1)$ which has adjoint $\eta^* \colon L^2(B) \to \mathbb C; \Lambda(a) \mapsto \varphi(a)$.

\subsection{An inner product on operators}\label{sec:ip}

Using the diagrammatical calculus, introduce an inner-product on $\mc B(L^2(B))$ by
\[ (T_1|T_2) = \begin{sd}
  \node[circle,draw] (T2) {$T_2$}; 
  \node[circle,draw,above=0.6 of T2.center] (T1) {$T_1^*$};
  \draw (T1)--(T2);
  \draw (T1) to[out=90,in=90] coordinate[pos=1/2] (m) ([xshift=-\hsep]T1.north);
  \draw (T2) to[out=-90,in=-90] coordinate[pos=1/2] (m*) ([xshift=-\hsep]T2.south);
  \draw (m)--++(0,\hsep/3) arc(-90:270:0.1);
  \draw (m*)--++(0,-\hsep/3) arc(90:450:0.1);
  \draw ([xshift=-\hsep]T1.north) -- ([xshift=-\hsep]T2.south);
  \end{sd}
  = \eta^* m (1 \otimes T_1^* T_2) m^* \eta
  = \varphi(m(1 \otimes T_1^* T_2)m^*(1))
  \qquad (T_1, T_2 \in \mc B(L^2(B))). \]
We compute this more explicitly.  To avoid notational clutter, we identify $B$ and $L^2(B)$ and suppress the GNS map $\Lambda$.  Let $m^*(1) = \sum_i e_i \otimes f_i$, so that $\varphi(b^*a^*) = (ab|1) = (a\otimes b|m^*(1)) = \sum_i (a|e_i)(b|f_i) = \sum_i \varphi(a^*e_i) \varphi(b^*f_i)$ for each $a,b\in B$.  Let $T_j = |t_j\rangle\langle s_j|$ for $j=1,2$ be rank-one operators.  Then
\begin{align*}
(T_1|T_2) &= \varphi\Big( \sum_i e_i T_1^*T_2(f_i) \Big)
= \sum_i \varphi(e_i s_1)  (t_1|t_2) (s_2|f_i).
\end{align*}
With $\varphi$ having density $Q$, we see that $\varphi(e_is_1) = \Tr(Qe_is_1) = \Tr(QQ^{-1}s_1Qe_i) = (Qs_1^*Q^{-1}|e_i)$.  Thus
\begin{align*}
(T_1|T_2) &= (t_1|t_2) \sum_i (Qs_1^*Q^{-1}|e_i) (s_2|f_i)
= (t_1|t_2) (Qs_1^*Q^{-1}s_2|1)
= (t_1|t_2) \varphi(s_2^* Q^{-1} s_1 Q) \\
&= (t_1|t_2) \varphi(s_2^* \sigma_i(s_1))
= (t_1|t_2) (s_2 | \sigma_i(s_1)),
\end{align*}
where $(\sigma_t)$ is the modular automorphism group.
We briefly stop suppressing the GNS map $\Lambda$, and recall that the modular operator is given by $\nabla \Lambda(a) = \Lambda(\sigma_{-i}(a))$.  Thus
\[ (T_1|T_2) = (\Lambda(t_1)|\Lambda(t_2)) (\Lambda(s_2)|\nabla^{-1}\Lambda(s_1))
= (\Lambda(t_1)|\Lambda(t_2)) (\nabla^{-1/2}\Lambda(s_2)|\nabla^{-1/2}\Lambda(s_1)). \]
Consequently, with $T_j = |\xi_j\rangle\langle\eta_j|$ for $j=1,2$, we find that
\[ (T_1|T_2) = (\xi_1|\xi_2) (\nabla^{-1/2}\eta_2|\nabla^{-1/2}\eta_1)
= \big( \xi_1 \otimes \overline{\nabla^{-1/2}\eta_1} \big| \xi_2 \otimes \overline{\nabla^{-1/2}\eta_2} \big), \]
the final inner-product being the natural one on $L^2(B) \otimes \overline{L^2(B)}$.  Thus the map
\begin{equation}
\mc B(L^2(B)) \to L^2(B) \otimes \overline{L^2(B)}, \quad
|\xi\rangle\langle\eta| \mapsto \xi\otimes\overline{\nabla^{-1/2}\eta}
\label{eq:twistedGNS}
\end{equation}
extends linearly to a unitary.

In \cite{daws_quantum_graphs}, we similarly identified $\mc B(L^2(B))$ with $L^2(B) \otimes \overline{L^2(B)}$, but for the ``untwisted map'' $|\xi\rangle\langle\eta| \mapsto \xi\otimes\overline{\eta}$, not the unitary \eqref{eq:twistedGNS}.  We continue to let $B\otimes B^\op$ act on $\mc B(L^2(B))$ as $(a\otimes b) \cdot T = aTb$, for $a\in B, b\in B^\op, T\in\mc B(L^2(B))$; similarly for $B'\otimes (B')^\op$.  We now see what this action becomes on $L^2(B) \otimes \overline{L^2(B)}$, under the isomorphism \eqref{eq:twistedGNS}.

For $a,b\in B$, when $T = |\xi\rangle\langle\eta|$ we have that $aTb = |a\xi\rangle\langle b^*\eta|$, and so
\[ (a\otimes b) \cdot (\xi\otimes \overline{\nabla^{-1/2}\eta})
 \cong aTb = |a\xi\rangle\langle b^*\eta|
 \cong a\xi\otimes \overline{\nabla^{-1/2}b^*\eta}. \]
Equivalently,
\begin{equation} \notag
(a\otimes b) \cdot (\xi\otimes \overline{\eta})
= a\xi\otimes \overline{\nabla^{-1/2}b^*\nabla^{1/2}\eta}
= a\xi \otimes \sigma_{-i/2}(b)^\top \overline{ \eta }
\qquad (a\otimes b\in B\otimes B^\op, \xi,\eta\in L^2(B)).
\end{equation}
This uses that $\sigma_{-i/2}(b) = \nabla^{1/2} b \nabla^{-1/2}$.  Exactly the same calculations shows that 
\[ (a\otimes b) \cdot (\xi\otimes\overline\eta) = a\xi \otimes (\nabla^{1/2} b \nabla^{-1/2})^\top\overline\eta \qquad (a,b\in B'). \]
Hence the natural actions of $B\otimes B^\op$, and of $B'\otimes (B')^\op$, on $L^2(B) \otimes \overline{L^2(B)}$ are ``twisted''.  The following lemma shows that in both cases, we can think of this as a twist on the algebra ($B\otimes B^\op$ or $B'\otimes (B')^\op$ respectively) followed by the ``natural'' action on $L^2(B) \otimes \overline{L^2(B)}$.

\begin{lemma}\label{lem:mod_aut_bcomm}
For any $b\in B' \subseteq \mc B(L^2(B))$ and $z\in\mathbb C$, we have that $\nabla^z b \nabla^{-z} \in B'$, and so $b\mapsto \nabla^z b \nabla^{-z}$ defines an automorphism of $B'$ (in general not $*$-preserving).
\end{lemma}
\begin{proof}
This is related to identifying $B'$ and $B^\op$, compare \cite[Lemma~5.32]{daws_quantum_graphs}.  However, a simple calculation suffices.  As $\nabla^z a \nabla^{-z} = \sigma_{-iz}(a)\in B$ for each $a\in B$, for $a\in B, b\in B'$ we see that
\[ \nabla^z b \nabla^{-z} a = \nabla^z b \nabla^{-z} a \nabla^z \nabla^{-z}
= \nabla^z b \sigma_{iz}(a) \nabla^{-z}
= \nabla^z \sigma_{iz}(a) b \nabla^{-z}
= \nabla^z \nabla^{-z} a \nabla^z b \nabla^{-z}
= a \nabla^z b \nabla^{-z}, \]
and so $\nabla^z b \nabla^{-z}$ commutes with $a$, for each $a\in B$, and hence is in $B'$.
\end{proof}

While it might seem odd to twist the action, notice that by doing so we maintain the bijections we used in \cite{daws_quantum_graphs}, namely between:
\begin{enumerate}[(1)]
  \item\label{im:1} $B'$-bimodules $\mc S \subseteq \mc B(L^2(B))$;
  \item\label{im:2} $B'\otimes (B')^\op$-invariant subspaces $V \subseteq L^2(B) \otimes \overline{L^2(B)}$;
  \item\label{im:3} projections $e\in B\otimes B^\op$.
\end{enumerate}

To be explicit, here we use the unitary from equation \eqref{eq:twistedGNS} to link $\mc S$ and $V$.  The twisted action of $B'\otimes (B')^\op$ on $L^2(B) \otimes \overline{L^2(B)}$ is not a $*$-homomorphism, but that does not matter when we are showing that \ref{im:1} and \ref{im:2} biject; all that matters is the compatibility of the $B'$ actions.

When showing that \ref{im:2} and \ref{im:3} biject, we do need a $*$-homomorphism, and so here we consider the usual action of $B\otimes B^\op$ on $L^2(B) \otimes \overline{L^2(B)}$.  We now have two actions of $B'\otimes (B')^\op$ on $L^2(B) \otimes \overline{L^2(B)}$, but they have the same invariant subspaces.  Henceforth, we shall only consider the ``natural'' action.

\subsection{Linking adjacency operators and projections}\label{sec:adj_proj}

We follow \cite[Section~5.4]{daws_quantum_graphs}.  Let $A\in\mc B(L^2(B))$ be a quantum adjacency operator, which here means that:
\begin{itemize}
  \item the associated linear map $B \to B$ is completely positive (equivalently, is ``real'', see \cite[Proposition~2.23]{matsuda_class_m2});
  \item $A$ is idempotent for the Schur product: $m^*(A\otimes A)m = A$.
\end{itemize}
As in \cite[Theorem~5.36]{daws_quantum_graphs}, it is then natural to use the linear bijection
\[ \Psi' = \Psi'_{0,1/2} \colon \mc B(L^2(B)) \to B \otimes B^\op; 
|b\rangle \langle a| \mapsto b \otimes \sigma_{i/2}(a)^*. \]

\begin{theorem}
The map $\Psi'$ gives a bijection between quantum adjacency operators $A$ and (self-adjoint) projections $e = \Psi'(A) \in B\otimes B^\op$.
\end{theorem}
\begin{proof}
This follows from \cite[Theorem~5.37]{daws_quantum_graphs} and its proof.  In reading \cite[Theorem~5.37]{daws_quantum_graphs} note that there we denote $\Psi'(A)$ by ``$f$'', and use ``$e$'' for a related element.
\end{proof}

Henceforth, we continue with the notation that $A$ is linked with $e = \Psi'_{0,1/2}(A)$.
Letting $B\otimes B^\op$ act ``naturally'' on $L^2(B) \otimes \overline{L^2(B)}$, we obtain the operator $b \otimes (\sigma_{i/2}(a)^*)^\top$.  Considering the unitary given by formula \eqref{eq:twistedGNS}, we have
\begin{align*}
\mc B(L^2(B)) \ni  |\xi\rangle \langle\eta| \mapsto \xi \otimes \overline{\nabla^{-1/2}\eta}
\ \xrightarrow{e} \ 
b\xi \otimes \overline{\sigma_{i/2}(a) \nabla^{-1/2} \eta}
= b\xi \otimes \overline{\nabla^{-1/2} a \eta}
\mapsto |b\xi\rangle \langle a\eta|
= b |\xi\rangle \langle\eta| a^*.
\end{align*}
Thus the associated action on $\mc B(L^2(B))$ is simply $T\mapsto bTa^*$; compare with Proposition~\ref{prop:thetaA} below.

We now find the associated subspaces $V$ and $\mc S$ associated to $e$, in terms of $A$.
Consider $A = \sum_{j=1}^k | b_j \rangle \langle a_j |$ assumed to be Schur idempotent and completely positive, so that $e = \Psi'_{0,1/2}(A) = \sum_j b_j \otimes \sigma_{i/2}(a_j)^*$ is a (self-adjoint) projection.  Hence
\[ V = \im (e) = \lin \Big\{ \sum_j b_j\xi \otimes (\sigma_{i/2}(a_j)^*)^\top\overline\eta : \xi,\eta\in L^2(B) \Big\} \subseteq L^2(B) \otimes \overline{L^2(B)}, \]
and given the bijections just established, we have that
\begin{equation}
\mc S = \Big\{ \sum_j b_j T a_j^* : T\in\mc B(L^2(B)) \Big\}.
\label{eq:S_from_A}
\end{equation}

\begin{remark}
One way to think here is that the ``twist'' introduced by $\Psi'_{0,1/2}$ is cancelled out by the unitary given by \eqref{eq:twistedGNS}, and so the relation between $A$, and for the formula for $\mc S$, looks ``untwisted''.
\end{remark}

We can nicely write much of this using diagrams.  Consider the map
\begin{equation} \label{eq:defn_thetaA}
\theta_A \colon \mc B(L^2(B)) \to \mc B(L^2(B)); \qquad
T \mapsto 
  \begin{sd}
  \node[circle, draw] (A) {$A$};
  \node[circle, draw, right=\hsep*0.7 of A.center] (T) {$T$};
  \draw (A) to [out=90, in=90] coordinate[pos=1/2] (m) (T);
  \draw (A) to [out=-90, in=-90] coordinate[pos=1/2] (m*) (T);
  \draw (m) -- ++(0,\hsep/3);
  \draw (m*) -- ++(0,-\hsep/3);
  \end{sd}
  = A \bullet T.
\end{equation}
This defines an idempotent, which can be seen diagrammatically, using that the multiplication on $B$ is associative and that $A$ is Schur idempotent:
\[ \theta_A(\theta_A(T)) = 
  \begin{sd}
  \node[circle, draw] (A) {$A$};
  \node[circle, draw, right=\hsep*0.7 of A.center] (T) {$T$};
  \draw (A) to [out=90, in=90] coordinate[pos=1/2] (m) (T);
  \draw (A) to [out=-90, in=-90] coordinate[pos=1/2] (m*) (T);
  \draw (m) -- ++(0,\hsep/3) coordinate (mt);
  \draw (m*) -- ++(0,-\hsep/3) coordinate (m*t);

  \node[circle, draw, left=\hsep*0.7 of A.center] (A1) {$A$};
  \draw (mt) to [out=90, in=90] coordinate[pos=1/2] (mm) (A1);
  \draw (m*t) to [out=-90, in=-90] coordinate[pos=1/2] (mm*) (A1);
  \draw (mm) -- ++(0,\hsep/3);
  \draw (mm*) -- ++(0,-\hsep/3);
  \end{sd}
  =
  \begin{sd}
    \node[circle, draw] (A) {$A$};
    \node[circle, draw, right=\hsep*0.7 of A.center] (T) {$T$};
    \node[circle, draw, left=\hsep*0.7 of A.center] (A1) {$A$};
    \draw (A) to [out=90, in=90] coordinate[pos=1/2] (m) (A1);
    \draw (m) to [out=90, in=90] coordinate[pos=1/2] (mm) (T);
    \draw (mm) -- ++(0,\hsep/3);
    \draw (A) to [out=-90, in=-90] coordinate[pos=1/2] (mb) (A1);
    \draw (mb) to [out=-90, in=-90] coordinate[pos=1/2] (mmb) (T);
    \draw (mmb) -- ++(0,-\hsep/3);
  \end{sd} 
  =
  \begin{sd}
  \node[circle, draw] (A) {$A$};
  \node[circle, draw, right=\hsep*0.7 of A.center] (T) {$T$};
  \draw (A) to [out=90, in=90] coordinate[pos=1/2] (m) (T);
  \draw (A) to [out=-90, in=-90] coordinate[pos=1/2] (m*) (T);
  \draw (m) -- ++(0,\hsep/3);
  \draw (m*) -- ++(0,-\hsep/3);
  \end{sd}
  = \theta_A(T). \]
  
\begin{proposition}\label{prop:thetaA}
With $A = \sum_{j=1}^k | b_j \rangle \langle a_j |$, we have that $\theta_A(T) = \sum_j b_j T a_j^*$ for $T\in\mc B(L^2(B))$.  Hence, with $\mc S$ given by \eqref{eq:S_from_A} using $A$, we have that $\mc S$ is the image of the idempotent $\theta_A$.
\end{proposition}
\begin{proof}
Let $a\in B$ and let $m^*(a) = \sum_k c_k \otimes d_k$, so $\sum_k (a_1|c_k) (a_2|d_k) = (a_1a_2|a)$ for each $a_1,a_2\in B$.  Then, for $b\in A$, and with $T = |c\rangle\langle d|$, we have
\begin{align*}
(b|\theta_A(T)a)
&= (m^*(b)|(A\otimes T)m^*(a))
= \sum_{j,k} (m^*(b)|b_j\otimes T(d_k)) (a_j|c_k)
= \sum_{j,k} (m^*(b)|b_j\otimes c) (a_j|c_k) (d|d_k) \\
&= \sum_{j,k} (b|b_j c) (a_j|c_k) (d|d_k)
= \sum_j (b|b_j c) (a_jd|a),
\end{align*}
from which it follows that
\[ \theta_A(T) = \sum_j |b_jc\rangle\langle a_jd|
= \sum_j b_j T a_j^*. \]
By linearity, this holds for all $T$, as claimed.  It is now clear that $\mc S$ is the image of $\theta_A$.
\end{proof}

Diagrammatically, we see that
\[ \begin{sd}
  \coordinate (v1);
  \coordinate[right=\hsep of v1.center] (v2);
  \draw (v1) to [in=90,out=90] coordinate[pos=1/2] (m) (v2) arc(90:450:0.1);
  \draw (m) -- ++(0,\hsep/3);
  \end{sd}
  =
  \begin{sd}
  \coordinate (v1);
  \draw (v1) -- ++(0,\hsep);
  \end{sd}
  \qquad\implies\qquad
  \theta_A(|1\rangle\langle1|) =
  \begin{sd}
  \node[circle, draw] (A) {$A$};
  \coordinate[right=\hsep of A.north] (v1);
  \draw (A) to [in=90,out=90] coordinate[pos=1/2] (m) (v1) arc(90:450:0.1);
  \draw (m) -- ++(0,\hsep/3);
  \coordinate[right=\hsep of A.south] (v2);
  \draw (A) to [in=-90,out=-90] coordinate[pos=1/2] (mb) (v2) arc(-90:270:0.1);
  \draw (mb) -- ++(0,-\hsep/3);
  \end{sd}
  =
  \begin{sd}
    \node[circle, draw] (A) {$A$};
    \draw (A) -- ++(0,\hsep);
    \draw (A) -- ++(0,-\hsep);
  \end{sd}    
  =  A
  \]
Thus we can recover $A$ from $\theta_A$.

\begin{proposition}\label{prop:S_is_bimod_A}
We have that $\mc S = \lin B' A B'$, the $B'$-bimodule generated by the operator $A$.
\end{proposition}
\begin{proof}
The previous proposition shows that $\mc S = \{ \sum_j b_j T a_j^* : T\in\mc B(L^2(B)) \}$ where $A = \sum_{j=1}^k | b_j \rangle \langle a_j |$.  Let $J$ be the modular conjugation, which satisfies $JaJ \Lambda(b) = \Lambda(b \sigma_{-i/2}(a^*))$, for $a,b\in B$, and we have $JBJ=B'$.
Thus
\begin{align*}
(Ja^*J) A (Jb^*J)
&= \sum_j | Ja^*J\Lambda(b_j) \rangle \langle JbJ \Lambda(a_j) |
= \sum_j | \Lambda(b_j \sigma_{-i/2}(a)) \rangle \langle \Lambda(a_j \sigma_{-i/2}(b^*)) | \\
&= \sum_j b_j |\sigma_{-i/2}(a)\rangle\langle \sigma_{-i/2}(b^*)| a_j^*
= \theta_A(|\sigma_{-i/2}(a)\rangle\langle \sigma_{-i/2}(b^*)|).
\end{align*}
Taking the linear span as $a,b$ vary, it follows that $B' A B' = \theta_A(\mc B(L^2(B))) = \mc S$, as claimed.
\end{proof}

In particular, $\mc S$ is a $B'$-bimodule, and so a ``quantum relation'' in the sense of Weaver, \cite{Weaver_QuantumRelations}.  We shall see shortly what it means for $\mc S$ to be a quantum graph in Weaver's sense.

\subsection{KMS Inner products}\label{sec:KMS_IPs}

We now make links with \cite{Wasilewski_Quantum_Cayley}.  Here Wasilewski works with the ``KMS inner-product'', which we shall denote by
\begin{equation}
(a|b)_K = \varphi(a^* \sigma_{-i/2}(b)) \qquad (a,b\in B).
\label{eq:KMS_defn}
\end{equation}
A simple calculation shows that $(a|b)_K = (\sigma_{-i/4}(a)|\sigma_{-i/4}(b))$.  We continue to identify $B$ with $L^2(B)$, and so the rank-one operator $|a\rangle\langle b|$ can be considered as the map $B\to B; c \mapsto (b|c) a = \varphi(b^*c) a$.  Similarly, we define $|a\rangle\langle b|_K$ using the KMS inner-product, so
\[ |a\rangle\langle b|_K \colon B \to B; \quad c \mapsto (b|c)_K a =  \varphi(b^* \sigma_{-i/2}(c)) a. \]
Hence $|a\rangle\langle b|_K = |a\rangle\langle b| \circ \sigma_{-i/2}$.  As $\varphi(b^* \sigma_{-i/2}(c)) = \varphi(\sigma_{-i/2}(b)^* c)$, also $|a\rangle\langle b|_K = |a\rangle\langle\sigma_{-i/2}(b)|$.

Then \cite[Lemma~3.3]{Wasilewski_Quantum_Cayley} defines a map $\Psi^{\KMS} \colon \mc B(L^2(B)) \to B\otimes B^\op$ by, in our notation,
\[ \Psi^{\KMS} \colon  |a\rangle\langle b|_K  \mapsto  a \otimes b^*. \]
Consequently $\Psi^{\KMS}(|a\rangle\langle b|) = \Psi^{\KMS}(|a\rangle\langle \sigma_{i/2}(b)|_K)= a \otimes \sigma_{i/2}(b)^* = \Psi'_{0, 1/2}(|a\rangle\langle b|)$.  Thus $\Psi^{\KMS} = \Psi'_{0,1/2}$ in our notation, the map we considered in \cite[Section~5.4]{daws_quantum_graphs}.

Note that \cite[Proposition~3.7]{Wasilewski_Quantum_Cayley} nicely separates out the correspondence between properties of $\Psi^{\KMS}(A)$ to those of $A$.

In Section~\ref{sec:adj_proj}, we established bijections between $A\in\mc B(L^2(B))$ with $e\in B\otimes B^\op$, and with $\theta_A \in \mc B(\mc B(L^2(B)))$.  These are
\begin{align*}
A &= |b\rangle \langle a| 
&\leftrightarrow&&
e &= b \otimes \sigma_{i/2}(a)^*
&\leftrightarrow&&
\theta_A &\colon T \mapsto bTa^*.
\end{align*}
We can write this alternatively as
\begin{align}  \label{eq:correspondence}
A &= |b\rangle \langle \sigma_{-i/2}(c^*) | 
&\leftrightarrow&&
e &= b \otimes c
&\leftrightarrow&&
\theta_A &\colon T \mapsto bT \sigma_{i/2}(c).
\end{align}
This does not agree verbatim with \cite[Proposition~3.14]{Wasilewski_Quantum_Cayley}, but private communication with Wasilewski verifies that this is a typo (compare with the paragraph before \cite[Proposition~3.14]{Wasilewski_Quantum_Cayley}, where $\sigma_{-i/2}^{\varphi^{-1}}$ is used: when restricted to $B$, as is occurring in this context, this map agrees with $\sigma_{i/2}$).  Thus \cite[Proposition~3.14]{Wasilewski_Quantum_Cayley} uses the same map $\theta_A$.

\subsection{Twisted subspaces}\label{sec:diff}

An important difference between the identifications used so far and \cite{Wasilewski_Quantum_Cayley} is that the relation between $\mc S$ and $A$ (and/or $e$) is different, compare \cite[Theorem~3.15]{Wasilewski_Quantum_Cayley}.  (Again, there is a sign convention here, whether one uses $i/4$ or $-i/4$ below, and we have chosen the convention that seems to work with our other conventions, and which correctly recreate key formulae from \cite{Wasilewski_Quantum_Cayley}.)

For $z\in\mathbb C$ and a $B'$-bimodule $\mc S\subseteq\mc B(L^2(B))$, define $\mc S_z = Q^{-iz} \mc S Q^{iz}$ where $Q\in B$ is the density of $\varphi$ as in Section~\ref{sec:ip}.

\begin{lemma}
We have that $\mc S_z$ is a $B'$-bimodule.
\end{lemma}
\begin{proof}
This is similar to the proof of Lemma~\ref{lem:mod_aut_bcomm}.  As $\sigma_z(a) = Q^{iz} a Q^{-iz}$ for each $a\in B$, exactly the same argument as before shows that $Q^{iz} x Q^{-iz} \in B'$ for each $x\in B'$.  Hence, given $T\in\mc S$, so that $Q^{-iz} T Q^{iz} \in \mc S_z$, and given $x\in B'$, we see that
\[ Q^{-iz} T Q^{iz} x = Q^{-iz} T (Q^{iz} x Q^{-iz}) Q^{iz}
\in Q^{-iz} \mc S Q^{iz} = \mc S_z, \]
as $Q^{iz} x Q^{-iz} \in B'$ and using that $\mc S$ is a $B'$-bimodule.  Similarly $B' \mc S_z \subseteq \mc S_z$.
\end{proof}

Given $A$ and hence $\theta_A$ as before, we now define $\mc T = \mc S_{i/4}$ where $\mc S$ is the image of the idempotent $\theta_A$.  From Proposition~\ref{prop:thetaA}, we see that
\begin{align}
T\in\mc T = \mc S_{i/4}
&\quad\Leftrightarrow\quad
T \in Q^{1/4} \mc S Q^{-1/4}
\quad\Leftrightarrow\quad
Q^{-1/4} T Q^{1/4} \in \mc S   \notag \\
&\quad\Leftrightarrow\quad
\sum_j b_j Q^{-1/4} T Q^{1/4} a_j^* = Q^{-1/4} T Q^{1/4}  \notag \\
&\quad\Leftrightarrow\quad
\sum_j Q^{1/4} b_j Q^{-1/4} T Q^{1/4} a_j^*Q^{-1/4} = T    \notag \\
&\quad\Leftrightarrow\quad
\sum_j \sigma_{-i/4}(b_j) T \sigma_{i/4}(a_j)^* = T
\label{eq:twistedT}
\end{align}
This may look strange, but we'll see below in Section~\ref{sec:op_sys_A} that in some senses it is quite natural.

\begin{remark}\label{rem:Q_or_nabla_to_twist}
We could also use $\nabla$ in place of $Q$, when defining $\mc S_z$.  Remember that, with $\Lambda \colon B \to L^2(B)$ the GNS map, we have that $\nabla^z\Lambda(a) = \Lambda(Q^zaQ^{-z})$ for $a\in B$.  Hence $\nabla^z Q^{-z} \Lambda(a) = \Lambda(a Q^{-z}) = Q^{-z} \nabla^z \Lambda(a)$ for $a\in B$, and hence $\nabla^z Q^{-z} = Q^{-z} \nabla^z \in B'$.  So $\nabla^z Q^{-z} \mc S Q^z \nabla^{-z} \subseteq B' \mc S B' \subseteq \mc S$ for all $z$, and hence $\mc S_z = Q^{-iz} \mc S Q^{-iz} \subseteq \nabla^{-iz} \mc S \nabla^{iz}$.  Reversing the roles of $Q$ and $\nabla$ shows the other inclusion, and so we have equality.
\end{remark}

Using Lemma~\ref{lem:mod_aut_bcomm} (and the above remark) and Proposition~\ref{prop:S_is_bimod_A}, the following is an easy calculation.

\begin{proposition}
For $z\in\mathbb C$ we have that $\mc S_z = B' (Q^{-iz}AQ^{iz}) B' = B' (\nabla^{-iz}A\nabla^{iz}) B'$.
\end{proposition}

\subsection{Operator systems and Weaver's quantum graphs}\label{sec:op_sys_A}

In \cite{Weaver_QuantumRelations}, Weaver terms a ``quantum graph'' to be an operator system $\mc S$ which is also a $B'$-bimodule.  That $\mc S$ is self-adjoint corresponds to the ``graph'' being undirected, and that $1\in\mc S$ corresponds to having a loop at every ``vertex''.
Here we consider these extra properties.
As in \cite[Definition~5.11]{daws_quantum_graphs}, we write
\[ J_0 \colon L^2(B) \otimes \overline{L^2(B)} \to L^2(B) \otimes \overline{L^2(B)}; \xi \otimes \overline\eta \mapsto \eta \otimes \overline\xi \]
for the anti-linear unitary ``tensor swap map''.

When is $\mc S$ self-adjoint?  Due to the relation given by formula \eqref{eq:twistedGNS}, if $\mc S$ and $V$ are related, then as $|\xi\rangle\langle\eta| \mapsto \xi\otimes\overline{\nabla^{-1/2}\eta}$, we see that
\[ |\eta\rangle\langle\xi| \mapsto \eta\otimes\overline{\nabla^{-1/2}\xi}
= (\nabla^{1/2} \otimes (\nabla^{-1/2})^\top)(\nabla^{-1/2}\eta \otimes \overline\xi)
= (\nabla^{1/2} \otimes (\nabla^{-1/2})^\top) J_0 (\xi\otimes\overline{\nabla^{-1/2}\eta}). \]
Hence $\mc S^*$ corresponds to $(\nabla^{1/2} \otimes (\nabla^{-1/2})^\top) J_0(V) = V_a$ (where ``a'' is chosen for ``adjoint'').  Let $e$ be the orthogonal projection onto $V$, and $e_a$ the orthogonal projection onto $V_a$.  It seems hard to write down $e_a$ in terms of $e$, as $V_a^\perp = J_0 (\nabla^{1/2} \otimes (\nabla^{-1/2})^\top) (V^\perp)$.

We shall not pursue this more, but instead make links with the idea explored in Section~\ref{sec:diff}.  Set $\mc T = \mc S_{i/4}$, and to avoid notational clashes, we write $\tau \colon B\otimes B^\op \to B\otimes B^\op$ for the tensor swap map, an anti-$*$-homomorphism.

The expression $A^*_K = \nabla^{-1/2} A^* \nabla^{1/2}$ occurring in the following is the KMS adjoint, that is, satisfies
\[ (a|A^*_K(b))_K = (A(a)|b)_K  \qquad (a,b\in B), \]
a fact easily verified from the definition of the KMS inner-product \eqref{eq:KMS_defn}.
For the following, compare \cite[Theorem~A]{Wasilewski_Quantum_Cayley}.

\begin{proposition}\label{prop:swap_e_KMS_ad_A}
Let $\mc T$ correspond to $e$ and $A$.  Then $\mc T^*$ corresponds with $\tau(e)$ and $A^*_K = \nabla^{-1/2} A^* \nabla^{1/2}$.
\end{proposition}
\begin{proof}
Recall \eqref{eq:twistedT}, so with $A = \sum_{j=1}^k | b_j \rangle \langle a_j |$ we have that $T\in\mc T$ if and only if $\sum_j \sigma_{-i/4}(b_j) T \sigma_{i/4}(a_j)^* = T$.  Hence $T\in\mc T^*$ if and only if $T^* = \sum_j \sigma_{-i/4}(b_j) T^* \sigma_{i/4}(a_j)^*$, equivalently,
\begin{align*}
T = \sum_j \sigma_{i/4}(a_j) T \sigma_{-i/4}(b_j)^*
= \sum_j \sigma_{-i/4}(\sigma_{i/2}(a_j)) T \sigma_{i/4}(\sigma_{-i/2}(b_j))^*  .
\end{align*}
Thus $\mc T^*$ is associated to the operator
\[ \sum_{j=1}^k | \sigma_{i/2}(a_j) \rangle \langle \sigma_{-i/2}(b_j) |
= \sum_{j=1}^k \nabla^{-1/2} |a_j\rangle \langle b_j| \nabla^{1/2} = \nabla^{-1/2} A^* \nabla^{1/2} = A^*_K, \]
as claimed.  In turn, this corresponds to
\[ \Psi'_{0,1/2}( A^*_K ) = \sum_j \sigma_{i/2}(a_j) \otimes \sigma_{i/2}(\sigma_{-i/2}(b_j))^* =  \sum_j \sigma_{i/2}(a_j) \otimes b_j^*
= \tau \Big( \sum_j b_j \otimes \sigma_{i/2}(a_j)^* \Big)^*
= \tau(e)^*, \]
which of course equals $\tau(e)$, as $e=e^*$.
\end{proof}

In particular, $\mc T$ is self-adjoint if and only if $e = \tau(e)$.  The second condition to be an operator system, that $1\in\mc S$, is easier to study.  We have that $1\in\mc S$ if and only if $\theta_A(1)=1$ if and only if $m(A\otimes 1)m^* = 1$, which is the usual axiom; compare \cite[Definition~2.4]{daws_quantum_graphs}.  Notice that $1\in\mc T$ if and only if $1\in\mc S$.

We haven't mentioned the other possible axioms for $A$, the ``undirected'' axiom, or that $A$ is self-adjoint.  As shown in \cite[Section~5]{daws_quantum_graphs}, either one of these (together with $A$ being real) implies that $V$ (and/or $\mc S$) commute with the modular operator, which is a strong condition, outside of the case when $\varphi$ is a trace.

\section{Moving to the tracial case}\label{sec:move_tracial_case}

In the previous section, we showed how to obtain a bijection between quantum adjacency operators $A$, projections $e\in B\otimes B^\op$, and $B'$-invariant subspaces $V$ and $\mc S$ (or $\mc T$).  The choice of state $\varphi$ affects the computation of $m^*$, and hence the definition of $A$, and also the bijection between $V$ and $\mc S$, and between $\mc S$ and $\mc T$.  However, projections $e$ (and the correspondence to $V$) are independent of the choice of $\varphi$, and so quantum graphs do not seem to actually depend on the state $\varphi$.  Here we make this statement more precise, but also show that it does not reduce the non-tracial case to a triviality.

We fix $e\in B \otimes B^\op$ a projection, and see what $A$ this corresponds to as $\varphi$ varies.  In particular, we consider our fixed reference trace $\Tr$, and the state $\varphi =\Tr(Q\cdot)$.  Then we have the bijection, as in \eqref{eq:correspondence},
\[ e = \sum_j b_j \otimes c_j
\quad\leftrightarrow\quad
A_\varphi = \sum_j |b_j\rangle \langle \sigma_{-i/2}(c_j^*)|
\quad\leftrightarrow\quad
A_\varphi = \sum_j \Tr(Q^{1/2}c_jQ^{1/2}(\cdot)) b_j. \]
Of course $A_\Tr$ has the same form, with $Q$ removed, and hence $A_\varphi = A_\Tr(Q^{1/2}(\cdot)Q^{1/2})$, and it's now readily apparent that $A_\varphi$ is (completely) positive if and only if $A_\Tr$ is.

The GNS map $\Lambda \colon B \to L^2(B)$ gives a linear bijection between $B$ and $L^2(B)$, and so a linear bijection between $\mc B(L^2(B))$ and $\hom(B) = \{ \text{linear maps }B\to B\}$.  It also gives rise to a bijection
\[ B' \cong \hom_B(B) = \{ T\colon B\to B : T(ab) = aT(b) \ (a,b\in B) \}. \]
This allows us to regard a $B'$-bimodule $\mc S \subseteq \mc B(L^2(B))$ as a $\hom_B(B)$-bimodule in $\hom(B)$, so removing reference to $L^2(B)$.

As in \eqref{eq:correspondence}, $e$ also corresponds to a map $\theta \in \hom(B)$.  For $a,b\in B$ let us write $\theta_{a,b}\in \hom(B)$ for the map $c \mapsto \Tr(bc)a$.  Then $e$ corresponds to
\[ \theta_{A_\varphi} : \hom(B)\to\hom(B); T \mapsto  \sum_j b_j T \sigma_{i/2}(c_j)
= \sum_j b_j T Q^{-1/2}c_jQ^{1/2}. \]
As $\mc S = \mc S_\varphi$ is the image of $\theta_{A_\varphi}$, at least in principle we see how to reconstruct $\mc S_\varphi$ from $e$.  Notice that $\mc S_{\Tr} = \{ \sum_j b_j T c_j : T\in\hom(B) \}$, and that then
\[ \mc S_\varphi
= \Big\{ \sum_j b_j TQ^{-1/2} c_j Q^{1/2} : T\in\hom(B) \Big\}
= \Big\{ \sum_j b_j T c_j : T\in\hom(B) \Big\} Q^{1/2}. \]
Hence $\mc S_\varphi = \mc S_\Tr Q^{1/2}$ for each $\varphi = \Tr(Q\,\cdot\,)$.

Also consider $\mc T = \mc S_{i/4} = Q^{1/4} \mc S Q^{-1/4}$, so that
\[ \mc T_\varphi
= \Big\{ \sum_j Q^{1/4}b_j TQ^{-1/2} c_j Q^{1/4} : T\in\hom(B) \Big\}
= Q^{1/4} \Big\{ \sum_j b_j T c_j  : T\in\hom(B) \Big\} Q^{1/4}. \]
We have that $\mc S_{\Tr} = \mc T_{\Tr}$ and that $\mc T_\varphi = Q^{1/4} \mc T_{\Tr} Q^{1/4}$.

Suppose we start with $\mc S \subseteq \hom(B)$ a $\hom_B(B)$-bimodule.  Given our state $\varphi$, which $e$ corresponds to $\mc S$?  We want $\mc S = \mc S_{\varphi}$ so it is natural to set $\mc S_{\Tr} = \mc S Q^{-1/2}$.  Then $\mc S_{\Tr}$ directly corresponds to $V\subseteq L^2(B,\Tr) \otimes \overline{L^2(B,\Tr)}$, and $e$ is the orthogonal projection onto $V$.  The following results are simple calculations.

\begin{proposition}
Under these bijections, we have that:
\begin{enumerate}[(1)]
  \item $1\in \mc S_{\varphi}$ if and only if $1\in\mc T_{\varphi}$ if and only if $Q^{-1/2} \in \mc S_{\Tr}$;
  \item $\mc S_\varphi = \mc S_\varphi^*$ if and only if $\mc S_\Tr Q^{1/2} = Q^{1/2} \mc S_\Tr^*$ if and only if $\mc S_\Tr^* = Q^{-1/2} \mc S_\Tr Q^{1/2}$;
  \item $\mc T_\varphi^* = \mc T_\varphi$ if and only if $Q^{1/4} \mc S_{\Tr}^* Q^{1/4} = Q^{1/4} \mc S_{\Tr} Q^{1/4}$ if and only if $\mc S_{\Tr}^* = \mc S_{\Tr}$.
\end{enumerate}
\end{proposition}

We now consider a more non-trivial property which a quantum graph may have, looking at the proposed definition of ``connectivity'' from \cite{courtney2025connectivityquantumgraphsquantum}.  \cite[Lemma~3.5]{courtney2025connectivityquantumgraphsquantum} shows that $\mc T_\varphi$ generates $\mc B(L^2(B))$ as an algebra if and only if there is not a non-trivial projection $p$ with $(1-p)\mc T_\varphi p = 0$.  Here we can take $p$ in $\mc B(L^2(B))$, or equivalently, in $B$ (see the referenced lemma for a proof).
\cite[Propositions~3.7, 4.8]{courtney2025connectivityquantumgraphsquantum} show this is equivalent to $A_\varphi$ being \emph{irreducible}, meaning that there is not a non-trivial $p\in B$ with $A_\varphi(p) \leq M p$ for some $M>0$.  As \cite[Definition~2.1]{courtney2025connectivityquantumgraphsquantum} and the comments after show, here we only assume $A_\varphi$ is CP.  In the \emph{undirected case} (meaning $A_\varphi$ is KMS-self-adjoint, if and only if $\mc T_\varphi$ is self-adjoint) this is the definition of the quantum graph being \emph{connected}.

\begin{example}
Let $B = \mathbb M_2$ and set
\[ S = \mc S_\Tr = \lin\Big\{
  \begin{pmatrix} 1 & 0 \\ 0 & 2 \end{pmatrix}, 
  \begin{pmatrix} 0 & 1 \\ 1 & 1 \end{pmatrix}    \Big\}
= \Big\{ \begin{pmatrix} a & b \\ b & 2a+b \end{pmatrix} : a,b\in\mathbb C \Big\}. \]
Then $\mc S_\Tr$ is self-adjoint, so $\mc T_\varphi$ is self-adjoint for any $\varphi$.  We shall show that the tracial quantum graph is connected, but there is a choice of $\varphi$ so that the quantum graph represented by $A_{\varphi}$ is not connected.

Consider the condition $(1-p)Sp=0$.  As $B = \mathbb M_2$ the only non-trivial projections are the rank-one projections, say $p$ is the projection onto the span of $\xi\not=0$.  Then $(1-p)Sp=0$ if and only if $x\xi \in \mathbb C\xi$ for each $x\in S$.  If $\xi = (\xi_1, \xi_2)$ is such a common eigenvector, then
\[ \begin{pmatrix} 1 & 0 \\ 0 & 2 \end{pmatrix}
\begin{pmatrix} \xi_1 \\ \xi_2 \end{pmatrix}
= \lambda \begin{pmatrix} \xi_1 \\ \xi_2 \end{pmatrix}\quad\implies\quad
\begin{pmatrix} \xi_1 \\ \xi_2 \end{pmatrix}\in\mathbb C\begin{pmatrix} 1 \\ 0 \end{pmatrix}
\text{ or }
\begin{pmatrix} \xi_1 \\ \xi_2 \end{pmatrix}\in\mathbb C\begin{pmatrix} 0 \\ 1 \end{pmatrix}. \]
However, neither of these standard basis vectors are eigenvectors for the other generator of $S$.  Hence $S$ is irreducible, and so $A_\Tr$ is irreducible, that is, the tracial quantum graph is connected.  (Notice that direct calculation shows that the algebra generated by $S$ is indeed all of $\mathbb M_2$).

Now define
\[ Q^{-1/2} = \begin{pmatrix} 1 & 1 \\ 1 & 3 \end{pmatrix}. \]
This is positive and invertible (the eigenvalues are $2\pm\sqrt 2$) and so $Q$ exists.  Use this $Q$ to define $\varphi$, so that $\mc T_\varphi = Q^{1/4} S Q^{1/4}$.  Set
\[ \xi' = Q^{-1/4} \begin{pmatrix} 1 \\ 1 \end{pmatrix}
\quad \Leftrightarrow \quad \begin{pmatrix} 1 \\ 1 \end{pmatrix} = Q^{1/4}\xi'. \]
We claim that for every $x\in S$ we have $Q^{1/4} x Q^{1/4} \xi' \in\mathbb C \xi'$, showing that $\mc T_\varphi$ has a common eigenvector, and so $A_\varphi$ is reducible, and hence the quantum graph is not connected.

To show the claim, notice that this is equivalent to $x Q^{1/4}\xi' \in \mathbb C Q^{-1/4}\xi'$ for every $x\in S$.  Notice that
\[ Q^{-1/4}\xi' = Q^{-1/2} Q^{1/4}\xi' = Q^{-1/2}  \begin{pmatrix} 1 \\ 1 \end{pmatrix} 
=  \begin{pmatrix} 2 \\ 4 \end{pmatrix}. \]
We test for $x$ either of the two generators:
\[ \begin{pmatrix} 1 & 0 \\ 0 & 2 \end{pmatrix} \begin{pmatrix} 1 \\ 1 \end{pmatrix} 
= \begin{pmatrix} 1 \\ 2 \end{pmatrix}  = \tfrac12 \begin{pmatrix} 2 \\ 4 \end{pmatrix},
\quad
\begin{pmatrix} 0 & 1 \\ 1 & 1 \end{pmatrix} \begin{pmatrix} 1 \\ 1 \end{pmatrix}
= \begin{pmatrix} 1 \\ 2 \end{pmatrix}. \]
\end{example}

In conclusion, while this correspondence between tracial quantum graphs and general quantum graphs allows trivial enumeration of all quantum graphs, deciding if these graphs have non-trivial properties will still require (possibly careful) consideration of the state $\varphi$.

\section{Hilbert--Schmidt quantum graphs}\label{sec:HS}

We consider $M \subseteq \mc B(H)$ to be a general von Neumann algebra, also with $H$ general.  Weaver's definition,  \cite{Weaver_QuantumRelations}, of a quantum graph, is that of an operator system $\mc S\subseteq\mc B(H)$ which is an operator bimodule over $M'$, meaning that $x\in\mc S, a,b\in M' \implies axb\in \mc S$.  This notion is essentially independent of the choice of $H$.  For the moment, we consider just the bimodule property, which defines the notion of a \emph{quantum relation}.

A key idea behind the bijections in Section~\ref{sec:fd_summary} is that, when $H$ is finite-dimensional, we may identify $\mc B(H)$ with $H\otimes\overline H$, and hence identify $M'$-bimodules with $M'\otimes (M')^\op$-invariant subspaces $V\subseteq H\otimes\overline H$.  (That, in the non-tracial situation, this gets ``twisted'' as in \eqref{eq:twistedGNS} is a matter we come back to in the rest of the paper.)  Of course, $H\otimes\overline H$ can be identified, as a Hilbert space, with the Hilbert--Schmidt operators.  This notion of course makes sense for general $H$:

\begin{definition}
With $M\subseteq\mc B(H)$, we define a \emph{Hilbert--Schmidt (HS) quantum relation} to be a subspace $V\subseteq HS(H)$ with $M' V M' \subseteq V$, and such that $V$ is closed for the Hilbert--Schmidt norm.
\end{definition}

As $HS(H)$ is an operator ideal, we have $M' HS(H) M' \subseteq HS(H)$.  There is a canonical identification of $HS(H)$ with $H\otimes\overline H$ which associates a rank-one operator $\rankone{\eta}{\xi}$ with the tensor $\eta\otimes\overline\xi$.  Under this, the $\mc B(H)$-bimodule structure is mapped to $\mc B(H) \times \mc B(H)^\op$, namely the map $x \mapsto axb$ becomes $a \otimes b^\top$.  Thus an HS quantum relation is a closed subspace $V\subseteq H\otimes\overline H$ which is invariant for the action of $M' \otimes (M')^\op$.  As $(M'\otimes (M')^\op)' = M \vnten M^\op$, we see that such subspaces biject with projections $e \in M \vnten M^\op$.
In later sections we shall see how to bring an operator $A$ into this picture.

Given a quantum relation $\mc S$, as $HS(H)$ is an operator ideal, $\mc S \cap HS(H)$ will be a $M'$-bimodule, and so $\mc S_{HS} = (\mc S \cap HS(H))\closhs$ will be an HS quantum relation.  Similarly, given $V\subseteq HS(H)$ an HS quantum relation, taking the weak$^*$-closure in $\mc B(H)$, say $V_{w^*} = V\clos^{w^*}$, gives a quantum relation.

Recall from \cite{Weaver_QuantumRelations} that $\mc S$ is:
\begin{enumerate}
  \item \emph{symmetric} when $\{x^*:x\in \mc S\} = \mc S$.  We make the analogous definition for HS quantum relations, which makes sense as the adjoint is an isometry on $HS(H)$, 
  and note that $\mc S\mapsto \mc S_{HS}$ and $V\mapsto V_{w^*}$ will preserve this property.
  \item \emph{reflexive} when $M' \subseteq \mc S$, equivalently, $1\in \mc S$.  It is not so clear what this means for HS quantum relations.
  \item \emph{transitive} when $\mc S^2=\{xy:x,y\in \mc S\} \subseteq \mc S$.  This definition makes sense for HS quantum relations, and again $\mc S\mapsto \mc S_{HS}$ and $V\mapsto V_{w^*}$ preserve this property.
\end{enumerate}

This final claim is maybe not obvious.  As we work with Hilbert spaces and Banach spaces in this section, given a Banach space $E$ (typically the trace-class operators $\mc B(H)_*$) and a subset $V\subseteq E^*$ define
\[ {}^\circ V = \{ x\in E : \ip{\mu}{x}=0 \ (\mu\in V) \}, \]
the pre-annihilator of $V$, which is a norm-closed subspace of $E$.  Similarly, for a subset $X\subseteq E$ define the annihilator $X^\circ = \{ \mu\in E^* : \ip{\mu}{x}=0 \ (x\in X) \}$ a weak$^*$-closed subspace of $E^*$.  By Hahn--Banach, we have that $({}^\circ V)^\circ = (\lin V) \clos^{w^*}$ and ${}^\circ (X^\circ) = (\lin X)\clos^{\|\cdot \|}$.  We use the perp symbol $\perp$ just for Hilbert spaces.

\begin{lemma}
When $\mc S$ is a transitive quantum relation, also $\mc S_{HS}$ is transitive.
When $V$ is a transitive HS quantum relation, also $V_{w^*}$ is transitive.
\end{lemma}
\begin{proof}
When $\mc S$ is transitive, for $x,y\in \mc S \cap HS(H)$ we have $xy\in HS(H)$ and $xy\in\mc S$, so $xy \in \mc S \cap HS(H)$.  Taking the closure shows that also $(\mc S_{HS})^2 \subseteq \mc S_{HS}$ by continuity of the product for the HS norm.

Now let $V$ be transitive.  Given $\omega \in {}^\circ V$ and $x,y\in V$, as $xy\in V$ we have $0 = \ip{xy}{\omega} = \ip{y}{\omega x}$, and so $\omega x \in {}^\circ V$.  Hence for $s\in V_{w^*}$ we have $0 = \ip{s}{\omega x} = \ip{x}{s \omega}$ for all $x\in V$, so $s \omega \in {}^\circ V$.  Finally, for $s,t \in V_{w^*}$ we hence have $\ip{ts}{\omega} = \ip{t}{s \omega} = 0$, for any $\omega \in {}^\circ V$, so $ts\in V_{w^*}$ and $V_{w^*}$ is hence transitive.  Notice that this argument is essentially using that the product on $\mc B(H)$ is separately weak$^*$-continuous.
\end{proof}

\begin{example}\label{eg:0}
Let $H$ be infinite-dimensional, and set $\mc S = \mathbb C 1$ thought of as a quantum relation over $\mc B(H)$ (so that $\mc B(H)' = \mathbb C1$).  Then $\mc S \cap HS(H) = \{0\}$ so $\mc S_{HS} = \{0\}$ even though $\mc S$ is non-zero.
\end{example}

\begin{example}\label{eg:1}
There can be multiple HS quantum relations giving rise to the same $\mc S$.  Again we work with $M=\mc B(H)$.

Let $H = \ell^2$, let $(e_{ij})$ be the matrix units in $\mc B(H)$.  Then we can view elements of $HS(H)$ as matrices $x=(x_{ij}) = \sum_{i,j} x_{ij} e_{ij}$ with $\|x\|_{HS}^2 = \sum_{i,j} |x_{ij}|^2 < \infty$.  Set
\[ V_2 = \Big\{ \sum_n x_n e_{nn} : \sum_n x_n / n = 0, \sum_n |x_n|^2<\infty \Big\} \subseteq HS(H). \]
Let $z = (z_n) = (1/n) \in \ell^2$, so $x=\sum_n x_n e_{nn} \in V_2$ exactly when $(z|x)=0$, and hence for any $x=(x_n)\in\ell^2$ we have that $x' = x - \|z\|_2^{-2} (z|x) z$ satisfies that $\sum_n x'_n e_{nn} \in V_2$.  This also shows that $V_2$ is closed in $HS(H)$.

Given $\xi,\eta\in H$ and a diagonal $x\in\mc B(H)$,
\[ \ip{x}{\omega_{\xi,\eta}} = (\xi|x\eta) = \sum_n x_{nn} \overline{\xi_n} \eta_n = \sum_n x_{nn} a_n, \]
say, where $a = (a_n) = (\overline{\xi_n} \eta_n) \in \ell^1$.  Thus on the diagonal matrices, the predual $\mc B(H)_*$ gives $\ell^\infty$-$\ell^1$ duality.  For $a=(a_n)\in \ell^1$ if $\ip{x}{a}=0$ for all $x\in V_2$, then regarding $\overline a$ as an element of $\ell^2$, we have that $(\overline a|x)=0$ for all $x\in V_2$.  By the above,
\[ 0 = (\overline a|x - \|z\|_2^{-2} (z|x) z)
=(\overline a|x) - \|z\|_2^{-2} (z|x) (\overline a|z)
= (\overline a - \|z\|_2^{-2} (z|\overline a)z |x)
\qquad(x\in\ell^2), \]
which means that $\overline a = \|z\|_2^{-2} (z|\overline a)z$ but $a\in\ell^1$ and $z\not\in\ell^1$, so $a=0$ is the only possibility.  So ${}^\circ V_2 \cap \ell^1 = \{0\}$, and we conclude that $V_2$ is weak$^*$-dense in the space of all diagonal matrices, $\ell^\infty \subseteq \mc B(H)$.

If we now set $V_1$ to be all diagonal matrices in $HS(H)$ then also $V_1$ is weak$^*$-dense in $\ell^\infty$.  As $V_1, V_2$ are $\|\cdot\|_{HS}$-closed, we conclude that different HS quantum relations can have the same weak$^*$-closure.
\end{example}

The underlying von Neumann algebra makes a big difference to these considerations: we give two cases where considering HS quantum relations is the same as considering (usual) quantum relations.

\begin{example}\label{eg_nice_vn_1}
Again let $H=\ell^2$, and now let $M=\ell^\infty$.  Any quantum relation $\mc S\subseteq\mc B(H)$ is the weak$^*$-closed linear span of the matrix units it contains, and so $\mc S$ corresponds to a relation on $\mathbb N$.  In the same way, any HS quantum relation $V\subseteq HS(H)$ is the norm-closed linear span of the matrix units it contains.  So the maps $\mc S \mapsto \mc S_{HS}$ and $V \mapsto V_{w^*}$ are mutual inverses.
\end{example}

\begin{example}\label{eg_nice_vn_2}
Let $M = \prod \mathbb M_{n(i)}$ be the product of matrix algebras, each with dimension $n(i)$.  (Recall that this is the setting of \cite{Wasilewski_Quantum_Cayley}.)  If $M$ acts on $H$ then $H$ is of the form $H = \bigoplus H_i$ where $H_i = \mathbb C^{n(i)} \otimes K_i$ for each $i$, with $K_i$ some Hilbert space.  Each matrix block acts as $x\otimes 1$ on $H_i$, for $x\in\mathbb M_{n(i)}$.  Then $M' = \prod \mc B(K_i)$ acting on each $H_i$ in the natural way.

A quantum relation $\mc S\subseteq\mc B(H)$, being an $M'$-bimodule, is the weak$^*$-closed linear span of the subspaces $\mc S_{i,j} = \mc S \cap \mc B(H_j,H_i)$.  Furthermore, $\mc S_{i,j} = \mc S'_{i,j} \otimes \mc B(K_j, K_i)$ for some subspace $\mc S'_{i,j} \subseteq \mc B(\mathbb C^{n(j)}, \mathbb C^{n(i)})$.  There are no conditions on which subspaces $\mc S'_{i,j}$ arise, for different quantum relations.

The same analysis applies to a HS quantum relation $V$, which must be the norm closed linear span of the $V_{i,j} = V \cap HS(H_j, H_i)$, and further $V_{i,j} = V'_{i,j} \otimes HS(K_j, K_i)$.  Again, any $V'_{i,j}$ can arise.

As $\mc B(\mathbb C^{n(j)}, \mathbb C^{n(i)})$ is finite-dimensional, given $\mc S$, we can set $V'_{i,j}$ to be $\mc S'_{i,j}$ thought of as a subspace of $HS(\mathbb C^{n(j)}, \mathbb C^{n(i)})$.  Let $(e_1,\cdots,e_m)$ be an orthonormal basis for $V'_{i,j}$, so any $x \in \mc S_{i,j}$ is of the form $x = \sum_{i=1}^m e_i \otimes x_i$ for some $x_i\in\mc B(K_j, K_i)$.  Thus $\|x\|_{HS}^2 = \sum_i \|x_i\|_{HS}^2$, and so $x\in \mc S \cap HS(H)$ if and only if $x\in V'_{i,j} \otimes HS(K_j,K_i)$.  Hence $V_{i,j} = \mc S_{i,j} \cap HS(H_j, H_i)$ is weak$^*$-dense in $\mc S_{i,j}$, and the map $\mc S \mapsto V; \mc S_{i,j} \mapsto V_{i,j}$ is a bijection.

In conclusion, quantum relations over $M$ and HS quantum relations over $M$ canonically biject.
\end{example}

We regard $\mc B(H)_*$ as the trace-class operators, so we have natural norm-decreasing inclusions $\mc B(H)_* \subseteq HS(H) \subseteq \mc B(H)$.  The pairings are also compatible: for $a\in\mc B(H)_*$ and $b\in HS(H)$ we have
\[ (a|b)_{HS} = \Tr(a^*b) = \ip{b}{a^*}_{\mc B(H), \mc B(H)_*}, \qquad (b|a)_{HS} = \ip{b^*}{a}_{\mc B(H), \mc B(H)_*}. \]
Thus, for $U\subseteq\mc B(H)_* \subseteq HS(H)$ we have $U^\circ = \{ x\in\mc B(H) : \Tr(xa)=0 \ (a\in U) \}$ and so, in $HS(H)$ we have
\[ U^\perp = \{ b\in HS(H) : \Tr(b^*a)=0 \ (a\in U) \} = \{ b\in HS(H) : b^*\in U^\circ \}, \]
and similarly, for $V \subseteq HS(H)$ we have ${}^\circ V = \{ a\in \mc B(H)_* : a^* \in V^\perp \}$.  Furthermore, we see that $\{ a^*: a\in V \}^\perp = \{ b\in HS(H) : 0 = (b|a^*) = \Tr(b^*a^*) \ (a\in V) \} = \{ b \in HS(H) : 0 = \Tr(ba) = (b^*|a) \ (a\in V) \} = \{ b^* : b \in V^\perp \}$.  Given the remarks above about (pre-)annihilators we see that the collection of weak$^*$-closed subspaces of $\mc B(H)$ injects with the collection of norm-closed subspaces of $\mc B(H)_*$, say $\mc S \leftrightarrow {}^\circ \mc S = X$ with inverse $X \leftrightarrow X^\circ = \mc S$.  This restricts to a bijection between $M'$-bimodules.

\begin{lemma}\label{lem:hs_facts}
For any weak$^*$-closed subspace $\mc S\subseteq\mc B(H)$, we have that $\mc S_{HS} = (\mc S \cap HS(H))\closhs$ is equal to $\mc S \cap HS(H)$, which is also equal to $\{ a^* : a\in {}^\circ \mc S \}^\perp$.
\end{lemma}
\begin{proof}
Set $X = \{ a^* : a\in{}^\circ\mc S \}$ so by the above, $X^\perp = \{ b\in HS(H) : b^* \in X^\circ \}$, and note that $x\in X^\circ$ if and only if $0 = \ip{x}{a^*} = \Tr(xa^*) = \overline{\Tr(x^*a)}$ for $a\in {}^\circ\mc S$, which is equivalent to $x^* \in ({}^\circ\mc S)^\circ = \mc S$.  So $X^\perp = \{ b\in HS(H) : b \in \mc S \} = \mc S \cap HS(H)$.  In particular, $\mc S \cap HS(H)$ is closed, and the claim follows.
\end{proof}

\begin{definition}
Given $\mc S \subseteq\mc B(H)$ weak$^*$-closed, set $\mc S^{hsh} = (\mc S_{HS})_{w^*}$, the \emph{Hilbert--Schmidt hull} of $\mc S$.
\end{definition}

Clearly $\mc S^{hsh}$ is an $M'$-bimodule when $\mc S$ is.
We next show that this is indeed a ``hull (or closure) operation'' (although for the reversed inclusion partial ordering).

\begin{proposition}\label{prop:hsh_props}
Let $\mc S$ be weak$^*$-closed.  Then $\mc S^{hsh} \subseteq \mc S$, $(\mc S^{hsh})_{HS} = \mc S_{HS}$ and so $(\mc S^{hsh})^{hsh} = \mc S^{hsh}$.  If $\mc S \subseteq \mc T$ then $\mc S^{hsh} \subseteq \mc T^{hsh}$.  We have that $\mc S = \mc S^{hsh}$ if and only if ${}^\circ \mc S = \mc B(H)_* \cap ({}^\circ\mc S)\closhs$.
\end{proposition}
\begin{proof}
When $\mc S \subseteq \mc T$ we have $\mc S_{HS} \subseteq \mc T_{HS}$ and so $\mc S^{hsh} = (\mc S_{HS})_{w^*} \subseteq (\mc T_{HS})_{w^*} = \mc T^{hsh}$.  By definition,
\[ \mc S^{hsh} = (\mc S_{HS})\clos^{w^*} = ({}^\circ \mc S_{HS})^\circ
\quad\implies\quad
{}^\circ (\mc S^{hsh})
= {}^\circ ( ({}^\circ \mc S_{HS})^\circ )
= {}^\circ \mc S_{HS}. \]
As $\mc S_{HS} \subseteq \mc S$ we have ${}^\circ \mc S \subseteq {}^\circ \mc S_{HS}$, and so $\mc S^{hsh} \subseteq \mc S$ by taking the annihilators.

By the lemma, $\mc S_{HS} = \{ a^* : a\in {}^\circ\mc S \}^\perp$ and so $(\mc S_{HS})^\perp = \{ a^* : a\in {}^\circ \mc S \}\closhs$.  As the adjoint is an isometry on $HS(H)$, we see that $\{ a^* : a\in (\mc S_{HS})^\perp \}$ is the closure of ${}^\circ \mc S$ in $HS(H)$.  Thus, from above,
\begin{align}
{}^\circ (\mc S^{hsh})
&= {}^\circ \mc S_{HS}
= \{ a\in \mc B(H)_* : a^* \in (\mc S_{HS})^\perp \}
= \{ a\in \mc B(H)_* : a \in ({}^\circ \mc S)\closhs \}.
\label{eq:perp_S_hsh}
\end{align}
As $\mc B(H)_* \cap ({}^\circ \mc S)\closhs \subseteq ({}^\circ \mc S)\closhs$ we also have $(\mc B(H)_* \cap ({}^\circ \mc S)\closhs)\closhs \subseteq ({}^\circ \mc S)\closhs$, but similarly ${}^\circ\mc S \subseteq \mc B(H)_* \cap ({}^\circ \mc S)\closhs$ so also $({}^\circ\mc S)\closhs \subseteq (\mc B(H)_* \cap ({}^\circ \mc S)\closhs)\closhs$ and hence we have equality.  As the closure is the double perp, we equivalently have that
\[ ({}^\circ\mc S)^{\perp} = (\mc B(H)_* \cap ({}^\circ \mc S)\closhs)^\perp. \]
Applying the lemma to $\mc S^{hsh}$ gives
\begin{align*}
(\mc S^{hsh})_{HS}
&= \{ a^* : a\in {}^\circ(\mc S^{hsh}) \}^\perp
= \{ a^* : a \in \mc B(H)_* \cap ({}^\circ \mc S)\closhs \}^\perp  \\
&= \{ a^* : a\in (\mc B(H)_* \cap ({}^\circ \mc S)\closhs)^\perp \}
= \{ a^*: a \in ({}^\circ\mc S)^{\perp} \}
= \{ a^* : a \in {}^\circ\mc S \}^\perp
= \mc S_{HS}.
\end{align*}
Taking weak$^*$-closures gives $(\mc S^{hsh})^{hsh} = \mc S^{hsh}$.  We see that $\mc S = \mc S^{hsh}$ if and only if ${}^\circ \mc S = {}^\circ(\mc S^{hsh}) = ({}^\circ \mc S)\closhs \cap \mc B(H)_*$ from \eqref{eq:perp_S_hsh}.
\end{proof}

We make an analogous definition for Hilbert--Schmidt quantum relations.

\begin{definition}
Let $V \subseteq HS(H)$ be a HS quantum relation.  Define $V^{hsc} = (V_{w^*})_{HS}$, the \emph{Hilbert--Schmidt closure of $V$}.
\end{definition}

Again, this is a ``closure'' operation, this time for the usual ordering (and so we feel more confident in the ``closure'' terminology).

\begin{proposition}\label{prop:hsc_props}
Let $V$ a HS quantum relation.  Then $V \subseteq V^{hsc}$ and $(V^{hsc})_{w^*} = V_{w^*}$ so that $(V^{hsc})^{hsc} = V^{hsc}$.  If $V \subseteq U$ then $V^{hsc} \subseteq U^{hsc}$.  We have that $V = V^{hsc}$ if and only if $V = (\mc B(H)_* \cap V^\perp)^\perp$.
\end{proposition}
\begin{proof}
When $V\subseteq U$ clearly $V^{hsc} \subseteq U^{hsc}$.  As $V \subseteq V_{w^*}$ and $V\subseteq HS(H)$, obviously $V \subseteq V^{hsc}$, and so also $V_{w^*} \subseteq (V^{hsc})_{w^*}$.  We aim to show the other inclusion, $(V^{hsc})_{w^*} \subseteq V_{w^*}$, which is equivalent to ${}^\circ V \subseteq {}^\circ(V^{hsc})$.  By Lemma~\ref{lem:hs_facts} we have that $V^{hsc} = V_{w^*} \cap HS(H)$, and so
\begin{align}
V^{hsc} &= \{ a\in HS(H) : \Tr(ax)=0 \ (x\in {}^\circ V) \}   \notag \\
&= \{ a\in HS(H) :  \Tr(ax^*)=0 \ (x\in \mc B(H)_* \cap V^\perp) \}
= (\mc B(H)_* \cap V^\perp)^\perp.
\label{eq:V_hsc}
\end{align}
We just used that ${}^\circ V = \{ x \in \mc B(H)_* : x^* \in V^\perp \}$ and so ${}^\circ(V^{hsc}) = \{ x\in\mc B(H)_* : x^* \in (\mc B(H)_* \cap V^\perp)\closhs \}$.
As the adjoint is an isometry on $\mc B(H)_*$, given $x\in {}^\circ V$ we have that $x^*\in V^\perp$ and $x\in\mc B(H)_*$, so clearly $x^* \in (\mc B(H)_* \cap V^\perp)\closhs$, so $x\in {}^\circ (V^{hsc})$.  Hence ${}^\circ V \subseteq {}^\circ(V^{hsc})$ as required.
From \eqref{eq:V_hsc} it follows that $V = V^{hsc}$ if and only if $V = (\mc B(H)_* \cap V^\perp)^\perp$.
\end{proof}

\begin{example}
We might wonder if the condition $V = (\mc B(H)_* \cap V^\perp)^\perp$ is the same as $V = (\mc B(H)_* \cap V)\closhs$, but this is not the case, as we now show.

Let $V = V_2$ as in Example~\ref{eg:1}.  Then any $x\in HS(H)$ supported off the diagonal is in $V^\perp$, and so $y\in V^\perp$ exactly when the diagonal part of $y$ is a scalar multiple of the vector $z = (1/n)$, and the off-diagonal part is arbitrary.  Suppose $a \in \mc B(H)_* \cap V^\perp$.  As the idempotent map of projecting onto the diagonal is contractive, the diagonal part of $a$ is again a multiple of $(1/n)$ which is not in $\ell^1$; so the diagonal part of $a$ is $0$.  We conclude that $\mc B(H)_* \cap V^\perp$ is the space of all off-diagonal matrices in $\mc B(H)_*$.  In particular, this contains all finitely-supported off-diagonals, and so $(\mc B(H)_* \cap V^\perp)^\perp$ is all diagonal matrices in $HS(H)$, namely all of $\ell^2$.  So $V \not= V^{hsc}$.

However, $\mc B(H)_* \cap V$ consists of all $a\in\ell^1$ with $\sum_n a_n/n = 0$.  This is the annihilator of $\mathbb C(1/n) \subseteq c_0$.  Let $x\in\ell^2$ with $(a|x)=0$ for all such $a$, equivalently, $(\overline a|x)=0$.  Thinking of $x$ as a $c_0$ vector, Hahn--Banach shows that $x\in \mathbb C(1/n)$.  So $(\mc B(H)_* \cap V)^\perp$ consists of those $x\in HS(H)$ whose off-diagonal is arbitrary, and whose diagonal is a multiple of $z$.  That is, $(\mc B(H)_* \cap V)\closhs = (\mc B(H)_* \cap V)^{\perp\perp} = V$.
\end{example}

These operations commute: for a quantum relation $\mc S$ we have
\[ (\mc S^{hsh})_{HS} = \mc S^{hsh} \cap HS(H) = ((\mc S_{HS})_{w^*})_{HS}
= (\mc S_{HS})^{hsc}, \]
and for any HS quantum relation $V$, similarly
\[ (V^{hsc})_{w^*} = ((V_{w^*})_{HS})_{w^*} = (V_{w^*})^{hsh}. \]

\begin{theorem}\label{thm:nice_subclass}
The maps $\mc S \mapsto (\mc S)_{HS}$ and $V \mapsto V_{w^*}$ restrict to give mutually inverse bijections between the sets
\[ \{ \mc S_{HS} : \mc S \text{ a quantum relation} \} = \{ V : V = V^{hsc} \},  \quad
\{ V_{w^*} : V \text{ a HS quantum relation} \} = \{ \mc S : \mc S=\mc S^{hsh} \}. \]
\end{theorem}
\begin{proof}
Using Proposition~\ref{prop:hsh_props} twice we see that
\[ \{ \mc S_{HS} : \mc S \text{ a quantum relation} \} = \{ (\mc S^{hsh})_{HS} : \mc S \} = \{ \mc S_{HS} : \mc S = \mc S^{hsh} \}. \]
Given $\mc S = \mc S^{hsh}$ set $V = \mc S_{HS}$, so $V_{w^*} = \mc S^{hsh} = \mc S$, and $V^{hsc} = (\mc S^{hsh})_{HS} = V$.  Conversely, when $V = V^{hsc}$, if we set $\mc S = V_{w^*}$ then $\mc S_{HS} = V^{hsc} = V$ and $\mc S^{hsh} = (V^{hsc})_{w^*} = \mc S$.  So we have shown the first equality.  Using Proposition~\ref{prop:hsc_props} twice shows that
\[ \{ V_{w^*} : V \text{ a HS quantum relation} \} = \{ (V^{hsc})_{w^*} : V \} = \{ V_{w^*} : V = V^{hsc} \} \]
and we've now shown that this set equals $\{ \mc S : \mc S = \mc S^{hsh} \}$, giving the second equality.  Our argument also shows that we have bijections as claimed.
\end{proof}

\begin{question}
We know that HS quantum relations $V$ over $M$ biject with projections $e\in M\vnten M^\op$.  Thus we obtain an operation on such projections: given $e$, form the associated $V$, then form $V^{hsc}$, then form the associated projection, say $e^{hsc}$.  Theorem~\ref{thm:hs_indep_rep} below shows that the operation $V \mapsto V^{hsc}$ does not depend upon the embedding $M\subseteq\mc B(H)$, and so the map $e\mapsto e^{hsc}$ is well-defined, independent of the choice of $H$.  We have $e \leq e^{hsc}$.  How can we describe this map at the level of the von Neumann algebra $M\vnten M^\op$?
\end{question}

\subsection{Invariance under representation}\label{sec:inv_reps}

Weaver shows in \cite[Theorem~2.7]{Weaver_QuantumRelations}, which we explored further in \cite[Section~7]{daws_quantum_graphs}, that given a von Neumann algebra $M$ with $M\subseteq\mc B(H)$ and $M\subseteq\mc B(H_1)$, there is an order-preserving bijection between quantum relations $\mc S\subseteq\mc B(H)$ and quantum relations $\mc S_1\subseteq\mc B(H_1)$.

Firstly, given $M\subseteq\mc B(H)$, for a Hilbert space $K$ we have that $M \cong M\otimes 1 \subseteq \mc B(H\otimes K)$, with $(M\otimes 1)' = M' \vnten \mc B(K)$.  Then quantum relations over $M\otimes 1$ are all of the form $\mc S \vnten \mc B(K)$, and this gives the bijection in this case.  We will show the same for Hilbert--Schmidt quantum relations (this alone, of course, also follows by considering the associated projection $e\in M\vnten M^\op$), and show that all the operations considered in the previous section are preserved.

In the following, we use the operator space project tensor product, \cite[Chapter~7]{EffrosRuanBook} for example, so that $(\mc B(H) \vnten \mc B(K))_* = \mc B(H\otimes K)_* = \mc B(H\otimes K)_* = \mc B(H)_* \proten \mc B(K)_*$.  We will also use slice maps: for $\omega\in\mc B(K)_*$ there is a (completely) bounded map $\id\otimes\omega \colon \mc B(H\otimes K) \to \mc B(H)$ satisfying $(\id\otimes\omega)(x\otimes y) = \ip{y}{\omega} x$ for $x\in\mc B(H), y\in\mc B(K)$.  For $\omega\in\mc B(H)_*$ and $\tau\in\mc B(K)_*$ we have that $\ip{x}{\omega\otimes\tau} = \ip{(\id\otimes\tau)(x)}{\omega}$ for $x\in\mc B(H\otimes K)$.  See for example \cite{Kraus_SliceMap} where in particular (see Theorem~1.9 and Remark~1.5) it is shown that when $\mc S\subseteq\mc B(H)$ is a weak$^*$-closed subspace, we have that
\begin{equation}
\mc S \vnten \mc B(K) = \{ x\in\mc B(H\otimes K) : (\id\otimes\omega)(x) \in \mc S \ (\omega\in\mc B(K)_*) \}.
\label{eq:Slice_against_BK}
\end{equation}
Here the left-hand-side is defined as the weak$^*$-closure of the algebraic tensor product $\mc S \odot \mc B(K)$.  It follows that
\begin{align*}
({}^\circ \mc S \proten \mc B(K)_*)^\circ
&= \{ x\in\mc B(H\otimes K) : \ip{x}{\omega\otimes\tau}=0 \ (\omega\in {}^\circ \mc S, \tau\in\mc B(K)_*) \} \\
&= \{ x\in\mc B(H\otimes K) : \ip{(\id\otimes\tau)(x)}{\omega}=0 \ (\omega\in {}^\circ \mc S, \tau\in\mc B(K)_*) \} \\
&= \{ x\in\mc B(H\otimes K) : (\id\otimes\tau)(x) \in ({}^\circ \mc S)^\circ = \mc S \ (\tau\in\mc B(K)_*) \}
= \mc S\vnten\mc B(K).
\end{align*}

\begin{proposition}\label{prop:HS_dilation}
The map $V \mapsto V \otimes HS(K)$ gives a bijection between HS quantum relations over $M$ and over $M\otimes 1$.  For a quantum relation $\mc S$ we have that $(\mc S\vnten\mc B(K))_{HS} = \mc S_{HS} \otimes HS(K)$, and for HS quantum relations $V$ we have that $(V\otimes HS(K))_{w^*} = V_{w^*} \vnten \mc B(K)$.  As such, $(\mc S \vnten \mc B(K))^{hsh} = \mc S^{hsh} \vnten \mc B(K)$ and $(V\otimes HS(K))^{hsc} = V^{hsc} \otimes HS(K)$.
\end{proposition}
\begin{proof}
As $(M\otimes 1)' = M' \vnten \mc B(K)$ we see that $V\otimes HS(K)$ is a $(M\otimes 1)'$-bimodule.  Let $U \subseteq HS(H\otimes K) = HS(H) \otimes HS(K)$ be a bimodule over $M' \vnten \mc B(K)$.  By shuffling tensor factors, equivalently, we have a subspace $X_U \subseteq H \otimes \overline H \otimes K \otimes \overline K$ which is invariant for operators of the form $a\otimes b^\top \otimes x \otimes y^\top$ where $a,b\in M', x,y\in\mc B(K)$.  Let $\xi\in X_U$, so for $\eta_0\in K\otimes\overline K$, there exists $\hat\xi\in H\otimes\overline H$ with $(\alpha|\hat\xi) = (\alpha\otimes\eta_0|\xi)$ for each $\alpha\in H\otimes\overline H$.  Letting $x,y$ range over all rank-one operators shows that $\hat\xi \otimes \eta \in X_U$, for any $\eta\in K\otimes\overline K$, and any such $\hat\xi$.  Let $X_V$ be the collection of such $\hat\xi$, which is seen to be an $M'\otimes(M')^\top$-invariant subspace of $H\otimes\overline H$, and is hence equivalent to a HS quantum relation $V$.  Any element $\xi\in X_U$ can be expanded using an orthonormal basis for $K\otimes\overline K$ and so is a member of $X_V \otimes K\otimes\overline K$.  We have hence shown that $U = V \otimes HS(K)$, and obviously any $V\otimes HS(K)$ is a bimodule over $M'\vnten\mc B(K)$.

As above, $({}^\circ \mc S\proten \mc B(K)_*)^\circ = \mc S\vnten\mc B(K)$.
Let $x\in (\mc S\vnten\mc B(K))_{HS}$ so $x\in ({}^\circ \mc S \proten \mc B(K)_*)^\circ \cap (HS(H)\otimes HS(K))$ which is thus equivalent to
\[ 0 = \ip{x}{\omega\otimes\tau} = (x^*|\omega\otimes\tau) \qquad (\omega\in {}^\circ \mc S, \tau\in\mc B(K)_*). \]
By continuity, this also holds for each $\tau\in HS(K)$, and so, taking an orthonormal basis in $HS(K)$ for example, we see that equivalently $x^* \in ({}^\circ \mc S \odot HS(K))^\perp = ({}^\circ \mc S)^\perp \otimes HS(K)$.  By Lemma~\ref{lem:hs_facts}, $y^*\in ({}^\circ \mc S)^\perp$ if and only if $y\in \mc S_{HS}$, and so we conclude that $(\mc S\vnten\mc B(K))_{HS} = \mc S_{HS} \otimes HS(K)$.

Now let $V\subseteq HS(H)$ be a HS quantum relation.  Let $K' \subseteq K$ be a finite-dimensional subspace, and $p \colon K\to K'$ the orthogonal projection.  Then
\[ (1\otimes p) (V_{w^*} \vnten \mc B(K)) (1\otimes p) = V_{w^*} \vnten \mc B(K') = V_{w^*} \odot \mc B(K'), \]
which is obviously equal to $(V\otimes HS(K'))_{w^*}$, which in turn is the weak$^*$ closure of $(1\otimes p)(V\otimes HS(K))(1\otimes p)$.  Let $(p_\alpha)$ be a net of such projections, with $p_\alpha\to 1$ strongly (simply order such $K'$ by inclusion).  Then $(1\otimes p_\alpha)x(1\otimes p_\alpha) \to x$ weak$^*$ for any $x\in\mc B(H\otimes K)$, and so we conclude that
\[ V_{w^*} \vnten \mc B(K) \subseteq (V\otimes HS(K))_{w^*}. \]
The reverse inclusion is clear, because $V\otimes HS(K) \subseteq V_{w^*} \vnten \mc B(K)$, and hence we have equality.
\end{proof}

We now follow Weaver's argument in \cite[Theorem~2.7]{Weaver_QuantumRelations}.  Given an injective normal unital $*$-homomorphism $\pi\colon M \to \mc B(H_1)$, if $K$ is a sufficiently large Hilbert space, there is a unitary $u\colon H\otimes K \to H_1\otimes K$ with $u(x\otimes 1)u^* = \pi(x)\otimes 1$ for $x\in M$.  Then a quantum relation $\mc S$ over $M$ bijects with a quantum relation $\mc S_1$ over $M_1 = \pi(M)$ exactly when $u(\mc S\vnten\mc B(K))u^* = \mc S_1\vnten\mc B(K)$.  (Of course, every quantum relation over $M\otimes 1$, respectively, $M_1\otimes 1$, is of the form $\mc S\vnten\mc B(K)$, respectively, $\mc S_1\vnten\mc B(K)$.)  We show the same for HS quantum relations.

\begin{theorem}\label{thm:hs_indep_rep}
Let $\pi\colon M \to \mc B(H_1)$ as above, with $M_1=\pi(M)$.  There is an order preserving bijection between HS quantum relations over $M$ and over $M_1$, say $V \leftrightarrow V_1$.  Similarly denote the bijection between quantum relations as $\mc S \leftrightarrow \mc S_1$.  This satisfies that $(V_{w^*})_1 = (V_1)_{w^*}$, that $(\mc S_1)_{HS} = (\mc S_{HS})_1$, and hence that $V^{hsc}$ bijects with $V_1^{hsc}$, and that $\mc S^{hsh}$ bijects with $\mc S_1^{hsh}$.
\end{theorem}
\begin{proof}
The same argument works for HS quantum relations: with $K,u$ as above, we define $V_1\otimes HS(K) = u(V\otimes HS(K))u^*$, which makes sense by Proposition~\ref{prop:HS_dilation}.  Notice that for $a\in M'$ we have that $a' = u(a\otimes 1)u^* \in (M_1\otimes 1)'$ as $u$ intertwines $M\otimes 1$ and $M_1\otimes 1$, and from this it follows that $u(V\otimes HS(K))u^*$ will indeed be an $(M_1\otimes 1)'$-bimodule.

As $u$ is unitary, we have that $u(HS(H_1\otimes K))u^* = HS(H \otimes K)$, and from this it follows that $\mc S_1\vnten\mc B(K) \cap HS(H_1\otimes K) = u(\mc S \vnten \mc B(K) \cap HS(H\otimes K)) u^* = u(\mc S_{HS} \otimes HS(K))u^* = (\mc S_{HS})_1 \otimes HS(K)$ so $(\mc S_1)_{HS} = (\mc S_{HS})_1$.  The other claims are similar.
\end{proof}

\subsection{Pullbacks}\label{sec:pullbacks}

There is a more general construction here, that of the \emph{pullback} of a quantum relation along a UCP map, \cite{Weaver_QuantumGraphs} and \cite[Section~7]{daws_quantum_graphs}.  Any normal UCP map $\theta \colon M \to N\subseteq \mc B(H_N)$ has a Stinespring dilation: a normal unital $*$-homomorphism $\pi\colon M \to \mc B(K)$, and an isometry $u\colon H_N \to K$ with $\theta(x) = u^*\pi(x)u$ for $x\in M$.  That $\theta$ maps into $N$ is equivalent to there being a normal unital $*$-homomorphism $\rho \colon N' \to \pi(M)'$ with $u y = \rho(y) u$ for $y\in N'$, see \cite[Lemma~7.2]{daws_quantum_graphs} for example.  We can always arrange that $K = H\otimes K'$ with $\pi(x) = x\otimes 1$ for $x\in M$.  We can then take \cite[Theorem~7.5]{daws_quantum_graphs} as a definition: for a quantum relation $\mc S$ the pullback is $\overleftarrow{\mc S}$, defined to be the weak$^*$-closure of $u^*(\mc S \vnten\mc B(K')) u$.  This is a quantum relation over $N$, and is independent of the choice of Stinespring dilation.

\begin{definition}
Let $V\subseteq HS(H_M)$ be a HS quantum relation over $M$, and let $\theta \colon M \to N$ be a normal UCP map.  Define $\overleftarrow{V}$ to be the HS-norm closure of $u^*(V\otimes HS(K')) u$ where $u,K$ are as above.
\end{definition}

We immediately show that this is well-defined.

\begin{proposition}
We have that $\overleftarrow{V}$ is a HS quantum relation, with the definition independent of the choice of Stinespring dilation.
\end{proposition}
\begin{proof}
This follows much as the proof of \cite[Theorem~7.5]{daws_quantum_graphs}.  There is a unital normal $*$-homomorphism $\rho\colon N' \to \pi(M)' = M'\vnten\mc B(K')$ with $uy = \rho(y)u$ for each $y\in N'$.  Thus for $x\in V\otimes HS(K')$ we have that $x\rho(y) \in V\otimes HS(K')$ for any $y\in N'$, as $V$ is bimodule over $M'$, and so $u^*xuy = u^*x\rho(y)u \in u^*(V\otimes HS(K'))u$.  Similarly for $y u^*xu$, showing that $u^*(V\otimes HS(K'))u$ is an $N'$-bimodule, and so its closure is an HS quantum relation.

As in the proof of \cite[Theorem~7.5]{daws_quantum_graphs} we pick a minimal Stinespring dilation, say $\pi_0,K_0,u_0$, so $K_0$ is the closed linear span of $\{ \pi_0(x)u_0\xi : \xi\in H_N, x\in M \}$.  (Note that $K_0$ need not be of the form $H\otimes K'$.)  There is always an isometry $w_0 \colon K_0 \to H\otimes K'$ satisfying $w_0 \pi_0(x)u_0\xi = \pi(x)u\xi$ for $x\in M, \xi\in H_N$, which follows as we have two dilations of the same UCP map.  In particular, $w_0 \pi_0(x) = (x\otimes 1) w_0$ for each $x\in M$.

Suppose we also have an isometry $w_1 \colon K_0 \to H\otimes K''$, for some $K''$, with $w_1 \pi_0(x) = (x\otimes 1)w_1$.  Then $w_1 w_0^*(x\otimes 1) = w_1\pi(x)w_0^* = (x\otimes 1)w_1w_0^*$ so $w_1w_0^* \in M' \vnten \mc B(K', K'')$.  Compare here with the proof of \cite[Lemma~7.4]{daws_quantum_graphs} where a $2\times 2$ matrix argument is used to make this a little more precise.  As $V$ is an $M'$ bimodule, we conclude that $w_1w_0^* (V \otimes HS(K')) w_0 w_1^* \in V \otimes HS(K'')$, and so $w_0^* (V \otimes HS(K')) w_0 \subseteq w_1^* (V \otimes HS(K'')) w_1$ as $w_1$ is an isometry.  By symmetry we have the other inclusion, and so we have equality.  We may hence define
\[ V_0 = u_0^* w_0^* (V\otimes HS(K')) w_0 u_0 \subseteq HS(H_N). \]
this definition being independent of the choice of $w_0, K'$.  However, of course $w_0u_0 = u$ and so $V_0 = u^*(V\otimes HS(K'))u$.  This shows that $u^*(V\otimes HS(K'))u$, and hence also $\overleftarrow{V}$, can be defined just using the minimal dilation, and so in particular that it is independent of the choice of Stinespring dilation chosen.
\end{proof}

\begin{remark}\label{rem:any_dilation_closed}
The proof shows slightly more, that $u^*(V\otimes HS(K'))u$ is an $N'$-bimodule, independent of the choice of Stinespring dilation.  The proof of \cite[Theorem~7.5]{daws_quantum_graphs} shows the same for quantum relations.  In Example~\ref{eg:pullback_not_weakstar_closed} we give an example where 
$u^*(\mc S\vnten\mc B(K')) u$ is not weak$^*$-closed, and an example of $V$ with $u^* (V\otimes HS(K')) u$ not HS-norm closed.  This is independent of the dilation chosen, and so we do need to take the relevant closure when defining the pullback.
\end{remark}

One might now hope that $(\overleftarrow{\mc S})_{HS} = \overleftarrow{\mc S_{HS}}$ and so forth, but this turns out not to be the case.  We do have the following.

\begin{lemma}\label{lem:HS_wstar_pullback}
We have $\overleftarrow{\mc S_{HS}} \subseteq (\overleftarrow{\mc S})_{HS}$ and 
$(V_{w^*})^{\leftarrow} = (\overleftarrow{V})_{w^*}$.  Also $\overleftarrow{\mc S^{hsh}} \subseteq (\overleftarrow{\mc S})^{hsh}$ and $\overleftarrow{V^{hsc}} \subseteq (\overleftarrow{V})^{hsc}$.
\end{lemma}
\begin{proof}
Let our UCP map $\theta$ have dilation $(u,\pi,K)$ where $\pi\colon M \to \mc B(H\otimes K)$ is of the form $\pi(x)=x\otimes 1$.  Then $\overleftarrow{\mc S_{HS}} = (u^*(\mc S_{HS}\otimes HS(K))u)\closhs = (u^* (\mc S\vnten\mc B(K))_{HS} u)\closhs$ using Proposition~\ref{prop:HS_dilation}.  As $u^* HS(H\otimes K)u \subseteq HS(H_N)$, it follows that $u^* (\mc S\vnten\mc B(K))_{HS} u \subseteq \overleftarrow{\mc S} \cap HS(H_N) = (\overleftarrow{\mc S})_{HS}$ and so $\overleftarrow{\mc S_{HS}} \subseteq (\overleftarrow{\mc S})_{HS}$.

Let $\omega \in {}^\circ(u^*(V\otimes HS(K))u)$ which means that $0 = \ip{u^*yu}{\omega} = \ip{y}{u\omega u^*}$ for each $y\in V \otimes HS(K)$, that is, $u\omega u^* \in {}^\circ(V\otimes HS(K))$.  Given $x\in u^* (V\otimes HS(K))_{w^*} u$, say $x = u^*zu$ for some $z\in (V\otimes HS(K))_{w^*} = V_{w^*} \vnten \mc B(K)$, again by Proposition~\ref{prop:HS_dilation}, we have $\ip{z}{u\omega u^*} = 0$ and so $\ip{x}{\omega}=0$.  So $u^* (V\otimes HS(K))_{w^*} u \subseteq (u^*(V\otimes HS(K))u)_{w^*}$ and hence $(V_{w^*})^\leftarrow \subseteq (\overleftarrow{V})_{w^*}$.

However, also $u^*(V\otimes HS(K))u \subseteq u^*(V\otimes HS(K))_{w^*} u = u^*(V_{w^*}\vnten\mc B(K))u \subseteq \overleftarrow{V_{w^*}}$ which is weak$^*$-closed by definition, so $(\overleftarrow{V})_{w^*} \subseteq (V_{w^*})^\leftarrow$ and so we have equality.
The remaining claims now follow.
\end{proof}

We now present a series of examples to show that these conclusions are optimal.

\begin{example}\label{eg:pullback_not_weakstar_closed}
We give an example of a UCP map $\theta \colon M \to N$ so that the pullback of the trivial quantum graph $M'$ is not equal to $u^*(M'\vnten\mc B(K))u$ for some (hence any, see Remark~\ref{rem:any_dilation_closed}) dilation of $\theta$.  In particular, the weak$^*$-closure needs to be taken.  Notice also that $\overleftarrow{M'}$ is a ``quantum confusability graph'', namely a quantum graph arising from a quantum channel, see \cite[Section~6.2]{daws_quantum_graphs} for example.

Working in $\mathbb C^3$ we set
\[ e_1 = \begin{pmatrix} 1/\sqrt2 \\ 1/\sqrt 2 \\ 0 \end{pmatrix}, \quad
e_2 = \begin{pmatrix} a \\ -a \\ b \end{pmatrix}, \]
where $a,b$ are both non-zero, with $2|a|^2 + |b|^2=1$.  So $\{e_1,e_2\}$ are orthonormal, and hence there is an isometry $u\colon \mathbb C^2 \to \mathbb C^3$ sending the standard basis of $\mathbb C^2$ to $e_1,e_2$.  Let $x,y,z\in\mathbb C$ and let $T\in\ell^\infty_3$ be the diagonal matrix with entries $(x,y,z)$.  Then
\begin{align*} (e_1|Te_1) &= \frac12 (x+y), \quad
(e_1|Te_2) = \frac{a}{\sqrt 2}(x-y), \\
(e_2|Te_1) &= \frac{\overline a}{\sqrt 2}(x-y), \quad
(e_2|Te_2) = |a|^2(x+y) + |b|^2z.
\end{align*}
These give the entries of $u^*Tu \in\mathbb M_2$.  Let $V = \{ u^*Tu : T\in\ell^\infty_3 \}$ so we see that $\ell^\infty_3 \to V, T\mapsto u^*Tu$ is injective, and of course contractive, as $u$ is an isometry.  We see that $e_{22} \in V$ is mapped to by $(0,0,1/|b|^2) \in \ell^\infty_3$.

Let $(b_n)$ be a sequence in $(0,1)$ converging to $0$, let $H = \mathbb C^2 \otimes \ell^2 \cong \ell^2(\mathbb C^2)$ and $K = \mathbb C^3 \otimes \ell^2$ with $M = \ell^\infty_3 \vnten \ell^\infty \cong \ell^\infty(\ell^\infty_3)$ acting on $K$.  Let $u = (u_{b_n})$ with each $u_{b_n} \colon \mathbb C^2 \to \mathbb C^3$ as above (so $a_n$ is to be chosen to satisfy the condition imposed by $b_n$).  Then for $T = (T_n) \in M = M'$ we see that $u^*Tu = (u_{b_n}^* T_n u_{b_n}) \in \ell^\infty(\mathbb M_2) \subseteq \mc B(H)$ and so $T\mapsto u^*Tu$ is injective.  If $T_n = (0,0,z_n)$ for each $n$ then $u^*Tu = (|b_n|^2 z_n e_{22})$, and so a sequence $(c_n e_{22})$ is in $u^*Mu$ if and only if $\sup_n |c_n| |b_n|^{-2} < \infty$.

The map $\ell^\infty \ni (z_n) \mapsto (|b_n|^2z_n) \in \ell^\infty$ is injective but not bounded below, so by the Open Mapping Theorem, does not have closed range.  Let $(c_n)$ be in the closure but not the range.  Thus $(c_n e_{22}) \not\in u^*Mu$ but for each $\epsilon>0$ there is $(z_n)\in\ell^\infty$ with $\| (|b_n|^2z_n) - (c_n) \|_\infty < \epsilon$, and hence $(c_n e_{22})$ is in the norm closure of $u^*Mu$.

Let $\theta$ be the UCP map $\ell^\infty \cong \ell^\infty(\ell^\infty_3) \to \ell^\infty(\mathbb M_2); T \mapsto u^*Tu$.  Then $u^*(M')u = u^*Mu$ is not norm closed, so certainly not weak$^*$-closed.

We can adapt this example to HS quantum relations.  Let $V = \ell^2(\ell^2_3) \subseteq M$, so each element of $V$ is a Hilbert--Schmidt operator on $K$.  Clearly $V$ is a bimodule over $M'=M$ and so $V$ is a HS quantum relation.  Again, the map $\ell^2 \ni (z_n) \mapsto (|b_n|^2z_n) \in \ell^2$ is injective but not bounded below, and so the same argument shows that $u^* V u$ is not closed.  Hence $u^* V u$ is only dense in $\overleftarrow{V}$.
\end{example}

\begin{example}
We use the idea of Example~\ref{eg:1} to find an $V$ with $\overleftarrow{V^{hsc}} \subsetneq (\overleftarrow{V})^{hsc}$.

Let $X\subseteq\ell^2$ be the subspace of those $(x_n)$ with $\sum_n \frac1n x_n = 0$ and $\sum_n n^2|x_n|^2 < \infty$.  Define $T\colon X\to\ell^2$ by $T((x_n)) = (n x_n)$ so $T$ is linear.  We claim that
\[ T(X) = \big\{ y = (y_n) \in \ell^2 : \sum_n n^{-2} y_n=0 \big\}. \]
Indeed, if $y=T(x)$ for some $x\in X$ then $\sum_n n^{-2} y_n = \sum_n n^{-2} n x_n = 0$, and conversely, if $y\in\ell^2$ with $\sum_n n^{-2}y_n=0$ then set $x_n = \frac1n y_n$ for each $n$, so $\sum_n n^2|x_n|^2 = \sum_n |y_n|^2 < \infty$ and hence in particular $(x_n)\in\ell^2$, and $\sum_n \frac1n x_n = \sum_n n^{-2}t_n = 0$, so $(x_n)\in X$.

Let $V = \{ (x,T(x)) : x\in X \} \subseteq \ell^2 \oplus \ell^2$ be the graph of $T$.  We shall compute the weak$^*$-closure of $V$ in $\ell^\infty \oplus \ell^\infty$.

Let $(x_\alpha)$ be a net in $X$ with $x_\alpha\to x$ and $T(x_\alpha) \to y$, both weak$^*$ in $\ell^\infty$.  For each $n$, let $e_n\in\ell^1$ be the $n$th standard unit vector basis element, so $x_n = \ip{x}{e_n} = \lim_\alpha x^{(\alpha)}_n$ and $y_n = \lim_\alpha (T(x_\alpha))_n = \lim_\alpha nx^{(\alpha)}_n = nx_n$.
As $(n^{-2}) \in \ell^1$ we have that
\[ \sum_n \frac1n x_n = \sum_n \frac{1}{n^2} y_n = \lim_\alpha \ip{(n^{-2})}{y_\alpha}
= \lim_\alpha \ip{(n^{-2})}{T(x_\alpha)} = 0, \]
in the final step using the characterisation of $T(X)$ given above.

Let $a=(a_i) \in \ell^1$ be such that $\ip{T(x)}{a} = 0$ for each $x\in X$.  So $\sum_k a_k y_k = 0$ for all $y\in\ell^2$ with $\sum_k k^{-2} y_k = 0$.  In particular, $a_1 - k^2a_k = 0$ for each $k\geq 2$, so $a = a_1(1,2^{-2},3^{-3},\cdots)$.  Clearly $(1/k^2)$ does annihilate all of $T(X)$, and so ${}^\circ T(X) = \mathbb C(1/k^2) \subseteq \ell^1$, showing that $\overline{T(X)}^{w^*} = \{ y=(y_n) \in\ell^\infty : \sum_k k^{-2} y_k = 0 \}$.

Given any $y=(y_n) \in\ell^\infty$ with $\sum_k k^{-2} y_k = 0$, pick a net $(y_\alpha) = (Tx_\alpha)$ in $T(X)$ with $y_\alpha\to y$ weak$^*$ in $\ell^\infty$.  Then for $a\in\ell^1$ we have
\[ \ip{x_\alpha}{a} = \sum_n a_n n^{-1} y^{(\alpha)}_n \to \sum_n a_n n^{-1} y_n, \]
as $(a_n/n)\in\ell^1$.  So $x_\alpha \to x$ weak$^*$, where $x = (n^{-1}y_n)$.  We conclude that
\[ \overline{V}^{w^*} = \Big\{ (x,y) \in \ell^\infty\oplus\ell^\infty : y = (nx_n), \sum_n n^{-2} y_n = 0 \Big\}. \]

Now let $(x,y) \in \overline{V}^{w^*} \cap (\ell^2\oplus\ell^2)$.  Then $x\in\ell^2$ and $\sum_n n^{-1}x_n = \sum_n n^{-2}y_n = 0$, so also $x\in X$, because $\sum_n n^2|x_n|^2 = \sum_n |y_n|^2 < \infty$.  Thus $x\in X$ and $y=T(x)$, so $\overline{V}^{w^*} \cap (\ell^2\oplus\ell^2) = V$.  This also shows that $V$ is norm closed in $\ell^2\oplus\ell^2$.

Finally, notice that $X \subseteq (\frac1n)^\perp$ and so the norm closure of $X$ in $\ell^2$ is also contained in $(\frac1n)^\perp$, and so is a proper subspace.  As in Example~\ref{eg:1}, we have however that $X$ is weak$^*$-dense in $\ell^\infty$.

As in Example~\ref{eg:1}, by using the diagonal, we can convert this example to find $X\subseteq HS(H)$ and a linear map $T\colon X\to HS(H)$ so that, with
\[ V = \Big\{ \begin{pmatrix} x & 0 \\ 0 & T(x) \end{pmatrix} : x\in X \Big\} \subseteq HS(K)  \cong M_2(HS(H)), \]
we have that $V$ is norm closed in $HS(K)$, and $V^{hsc} = \overline{V}^{w^*} \cap HS(K) = V$.  Let $u\colon H \to K$ be inclusion onto the first coordinate, so $u^*Vu = X$ and hence $\overleftarrow{V^{hsc}} = \overleftarrow{V} = (u^*Vu)\closhs = X\closhs \subseteq (\frac1n)^\perp$.  However, as $\overleftarrow{V} = X$ and $\overline{X}^{w^*} = \ell^\infty$ thought of as diagonal matrices in $\mc B(H)$, we hence see that $(\overleftarrow{V})^{hsc} = \ell^2$ and we have $\overleftarrow{V}^{hsc} \subsetneq (\overleftarrow{V})^{hsc}$.
\end{example}

As a brief conclusion, we see that HS quantum relations are reasonably well-behaved, especially those with $V=V^{hsc}$.  However, interactions with constructions like the pullback are a little complicated.  We have not discussed examples coming from UCP maps (aka ``quantum channels'') but Example~\ref{eg:pullback_not_weakstar_closed} suggests that this is likely to be complicated.  Of course, as Examples~\ref{eg_nice_vn_1} and~\ref{eg_nice_vn_2} show, for some von Neumann algebras there is essentially no difference between HS quantum relations and (usual) quantum relations.

\section{Hilbert modules}\label{sec:hilb_mods}

We now turn our attention to the situation of a general von Neumann algebra $M$ with a (normal, semifinite, faithful) weight $\varphi$.  We are motivated by some of the ideas of \cite{Wasilewski_Quantum_Cayley}, and will make explicit links to Section~\ref{sec:fd_summary}, and eventually to Section~\ref{sec:HS} (where HS quantum relations will arise as a special case of ideas explored here).

Our first step is to consider the operator-valued weight $\varphi^{-1}$.  We follow \cite[Chapter~IX.4]{TakesakiII}.  We recall the \emph{extended positive cone} of $\wh M_+$, and hence the notion of an operator-valued weight $T \colon M_+ \to \wh N_+$ where $N\subseteq M$ is a sub-von Neumann algebra.  We also find \cite{Haagerup_OpValuedWeightsI,Haagerup_OpValuedWeightsII} to be very readable, although the techniques are different.  By \cite[Corollary~IX.4.25]{TakesakiII} there is a bijection between nfs weights $\varphi$ on $M$ and operator-valued weights $\varphi^{-1} \colon \mc B(H) \to \wh{M'}$ (denoted $T_\varphi$ in \cite{TakesakiII}); see also \cite[Corollary~16]{Connes_SpatialVN}.

We shall always regard $M$ as acting on the GNS space for $\varphi$, which means $M$ is in \emph{standard position}, \cite[Chapter~IX.1]{TakesakiII}.  As all standard forms are isomorphic, we write $L^2(M)$ for $L^2(M,\varphi)$.  We want to find a fairly concrete way to express $\varphi^{-1}$, for which we use the theory of Hilbert $C^*$-module, which we now explore.  We use \cite{Lance_HilbModsBook} as a general reference.  A \emph{Hilbert $C^*$-module} $E$ is a right $M$-module with a compatible $M$-valued inner-product.

\begin{example}
Let $H$ be a Hilbert space which is a right $M$-module: this means there is a normal unital $*$-homomorphism $\rho \colon M^\op \to \mc B(H)$; we write $\xi\cdot x =\pi(x^\op)\xi$ for $\xi\in H, x\in M$.  One such example is $L^2(M)$ where, using the modular conjugation, we define $\rho(x^\op) = Jx^*J \in M'$ for $x\in M$; in this special case, $\rho$ is an isomorphism.  Given a general $\rho$, let $E = \mc B_M(L^2(M), H)$ be the space of bounded linear operators $\alpha \colon L^2(M) \to H$ which are right module homomorphisms: $\alpha(\xi\cdot x) = \alpha(\xi) \cdot x$ for $x\in M$, equivalently, $\alpha Jx^*J = \rho(x^\op) \alpha$ for $x\in M$.

We defined an operator-valued inner-product on $E$ by $(\alpha|\beta) = \alpha^*\beta$.  Then $\alpha^*\beta Jx^*J = \alpha^* \rho(x^\op) \beta = Jx^*J \alpha^*\beta$ for each $x\in M$ so $\alpha^*\beta \in M'' = M$.  Define $\alpha\cdot x = \alpha x$ for $x\in M$, and note then that $\alpha x Jy^*J = \alpha Jy^*J x = \rho(y^\op) \alpha x$ for each $y\in M$, so $\alpha x\in E$ as well.  One easily checks the axioms that $E$ is a Hilbert $C^*$-module over $M$.
\end{example}

In fact, we will essentially only be interested in modules which arise in this way.  Recall that the \emph{adjointable operators} $T\in\mc L(E)$ are those linear maps $T\colon E\to E$ for which there is $T^*\colon E\to E$ with $(T^*x|y) = (x|Ty)$ for each $x,y\in E$.  Then $T$ is bounded and an $M$-module map, and $\mc L(E)$ becomes a $C^*$-algebra.  A Hilbert $C^*$-module $E$ is \emph{self-dual} when every module map $T\colon E\to M$ is of the form $T(x) = (x_0|x)$ for some $x_0\in E$.

\begin{example}
We continue the example.  
Given $t\in\rho(M^\op)'$, for $\alpha\in\mc B_M(L^2(M), H)$ and $x\in M$ we see that $t\alpha Jx^*J = t \rho(x) \alpha = \rho(x) t\alpha$ so also $t\alpha\in \mc B_M(L^2(M), H)$.  This map $T(\alpha) = t\circ\alpha$ is adjointable, with $T^*(\alpha) = t^*\circ\alpha$, and so $T\in\mc L(E)$.  We shall shortly see that every member of $\mc L(E)$ arises in this way.
\end{example}

We now relegate some results to the appendix: this avoids breaking the flow of the text with results which are either minor extensions of results in the literature, or which are technical and away from our main interest.  The following is a summary of results from Appendix~\ref{sec:selfdual_mods}.

\begin{theorem}
Modules of the form $E=\mc B_M(L^2(M), H)$ are self-dual, and every self-dual module over a $M$ is of this form.  Indeed, $H = E \otimes_M L^2(M)$, the interior tensor product.  There is an isomorphism $\mc L(E) \cong \rho(M^\op)'$.
\end{theorem}

Given an operator-valued weight $T \colon M_+ \to \wh N_+$, following the notation of \cite[Chapter~IX.4]{TakesakiII}, we define
\[ \mf n_T=\{ x\in M : \|T(x^*x)\|<\infty \}, \quad \mf m_T = \lin \mf n_T^*\mf n_T. \]
As for weights, $\mf m_T =\lin \mf m_T^+$ where $\mf m_T^+ = \mf m_T \cap M_+ = \{ x\in M_+ : \|T(x)\|<\infty \}$, we have that $\mf m_T$ and $\mf n_T$ are bimodules over $N$, and that $\mf n_T$ is a left ideal in $M$.  We extend $T$ to $\mf m_T$, so $T$ becomes an $N$-bimodule map.  From this we can construct a pre-Hilbert module: this is defined in \cite[Proposition~2.8]{BDH_IndexCondExp}, but no theory (except when $T$ is finite) is expanded.  We shall take the study a little further.

\begin{definition}\label{defn:module_from_opvalweight}
Let $T$ be a (semifinite, faithful, normal) operator-valued weight and let $E^0_T = \mf n_T$ with the $N$-valued inner-product $(\alpha|\beta) = T(\alpha^*\beta)$.  Let $E_T \cong \mc B_N(L^2(N), E^0_T \otimes_N L^2(N))$ be the self-dual completion.  Set $H_T = E^0_T \otimes_N L^2(N)$.
\end{definition}

See Theorem~\ref{thm:self_dual_completion} for more on self-dual completions.  As shown in \cite[Proposition~2.8]{BDH_IndexCondExp}, each $x\in M$ gives $\pi(x)\in\mc L(E_T^0)$ by $\pi(x)\alpha = x\alpha$ for $\alpha\in E_T^0=\mf n_T$.  Then $\pi$ is a normal $*$-homomorphism, and as we assume $T$ is faithful, $\pi$ is injective.  As in Proposition~\ref{prop:adjointables_extend_sdc}, we extend $\pi(x)$ to $E_T=\mc B_N(L^2(N), H_T)$ and so regard $\pi(x)\in \rho(N^\op)' \subseteq \mc B(H)$ where $\rho\colon N^\op\to\mc B(H_T)$ is the right action.  For each $\alpha\in \mf n_T$ let $\hat\alpha \in \mc B_N(L^2(N), H_T)$ be the operators $\xi \mapsto \alpha\otimes\xi$; compare after Theorem~\ref{thm:self_dual_completion}.

We seek to understand $E_T$, which means understanding $H_T$ and $\hat\alpha$ for $\alpha\in\mf n_T$.  We develop the technical theory in Appendix~\ref{sec:op_valued_weights}.  Recall that we can compose operator-valued weights, \cite[Proposition~IX.4.16]{TakesakiII} and so in particular can compose an operator-valued weight and a weight: this procedure is used extensively in the theory, compare \cite[Lemma~IX.4.21]{TakesakiII} for one example.

\begin{theorem}[{Proposition~\ref{prop:HT_is_L2_tildephi}} and discussion after]
Let $\phi$ be a (nfs) weight on $N$, and set $\tilde\phi = \phi\circ T$ a (nfs) weight on $M$.
The Hilbert space $H_T = E^0_T \otimes_N L^2(N)$ is isomorphic to $L^2(M, \tilde\phi)$ for the map $\alpha\otimes\Lambda_\phi(x) \mapsto \Lambda_{\tilde\phi}(\alpha x)$, for $\alpha\in \mf n_T, x\in\mf n_{\phi}$.

Under this isomorphism, for $\alpha\in \mf n_T$ we have that $\hat\alpha \Lambda_\phi(x) = \Lambda_{\tilde\phi}(\alpha x)$, and the homomorphism $\pi \colon M \to \mc L(E_T) \subseteq\mc B(H_T)$ becomes the natural left action of $M$ on $L^2(M,\tilde\phi)\cong H_T$.
The right action of $N$ on $H_T = L^2(M,\tilde\phi)$ is the restriction of the natural right action of $M$ on $L^2(M, \tilde\phi)$.
\end{theorem}

We can also characterise $\mf n_T \subseteq M$ using the Hilbert module structure.

\begin{theorem}[{Theorem~\ref{thm:nT_vs_module}}]
For $\alpha\in M$ the following are equivalent:
\begin{enumerate}[(a)]
  \item $\alpha\in\mf n_T$;
  \item $\alpha x \in \mf n_{\tilde\phi}$ for each $x\in \mf n_{\phi}$, and the resulting map $\Lambda_\phi(x) \mapsto \Lambda_{\tilde\phi}(\alpha x)$ is bounded.
\end{enumerate}
The map $\mf n_T \to E_T; \alpha \mapsto \hat\alpha$ is closed for the $\sigma$-strong topology on $\mf n_T$ and the SOT on $E_T$.
\end{theorem}

Our aim now is to understand $\varphi^{-1}$ using this construction: the outcome is that we can very concretely describe $E_{\varphi^{-1}}$ without ever really having to understand $\varphi^{-1}$.  Recall that we fix $M$ acting on $H = L^2(M,\varphi)$; we write $\Lambda = \Lambda_\varphi$.  We choose a weight on $M'$, namely $\varphi'$ defined by $\varphi'(Jx^*J) = \varphi(x)$ for $x\in M_+$.

\begin{lemma}\label{lem:com_weight}
We have that $\mf n_{\varphi'} = \{ JxJ : x\in\mf n_\varphi\}$.  We identify $L^2(M', \varphi')$ with $L^2(M)$ via the unitary $u \colon \Lambda'(JxJ) = J\Lambda(x)$ for $x\in\mf n_{\varphi}$.  Then $\pi' \colon M' \to L^2(M')$ satisfies that $u \pi(\cdot)u^*$ is the identity representation of $M'$ on $L^2(M)$.  The modular operators satisfy $u\nabla'u^* = \nabla^{-1}, uJ'u^*=J, uS'u^*=F, uF'u^*=S$.
\end{lemma}
\begin{proof}
This is a calculation.\footnote{See for example \url{https://github.com/MatthewDaws/Mathematics/tree/master/Weights}}
\end{proof}

\begin{theorem}[{Theorem~\ref{thm:varphiinv}}]
The operator-valued weight $\varphi^{-1} \colon \mc B(H)_+ \to M'$ satisfies
\[ \varphi^{-1}( \rankone{J\Lambda(x)}{J\Lambda(y)}) = Jxy^*J \qquad (x,y\in \mf n_{\varphi}). \]
The weight $\tilde\varphi = \varphi' \circ \varphi^{-1}$ on $\mc B(H)$ has density $\nabla^{-1}$.
\end{theorem}

Informally, that $\tilde\varphi$ has density $\nabla^{-1}$ means that $\tilde\varphi(x) = \Tr(\nabla^{-1}x)$; see Appendix~\ref{sec:weights_bh} for how to interpret this rigourously, and for technical results about the weight $\tilde\varphi$.

\begin{proposition}[{Propositions~\ref{prop:Trh} and~\ref{prop:mod_Trh}}]
\label{prop:L2BS_is_HSH}
We identify $L^2(\tilde\varphi)$ with $HS(H)$, where $\mf n_{\tilde\varphi} = \{x\in\mc B(H) : x\nabla^{-1/2} \in HS(H)\}$ and $\Lambda_{\tilde\varphi}(x) = x\nabla^{-1/2}$.  The resulting modular operator and conjugation are $\tilde J(x)=x^*$ and $\tilde\nabla(x) = \nabla^{-1} x \nabla$ for $x\in HS(H)$.  The left and right actions of $\mc B(H)$ on $L^2(\tilde\varphi) = HS(H)$ are left/right multiplication.
\end{proposition}

\begin{remark}
We make links with the finite-dimensional situation, Section~\ref{sec:fd_summary}.  For $x,y\in\mf n_\varphi$ we have
\[ \tilde\varphi\big( \rankone{J\Lambda(x)}{J\Lambda(y)} \big)
= \varphi'(Jxy*J) = \varphi(yx^*) = (\Lambda(y^*)|\Lambda(x^*)). \]
Writing $\tilde\Lambda$ for the GNS map of $\tilde\varphi$, given $y_1, y_2 \in D(\sigma_{i/2})$ and $\xi_1,\xi_2\in H$, we have
\begin{align*}
& \big( \tilde\Lambda\rankone{\xi_1}{\Lambda(y_1)} \big| \tilde\Lambda\rankone{\xi_2}{\Lambda(y_2)} \big)
= \tilde\varphi\big( \rankone{\Lambda(y_1)}{\xi_1} \rankone{\xi_2}{\Lambda(y_2)} \big) \\
&= (\xi_1|\xi_2) \tilde\varphi\big( \rankone{J\Lambda(\sigma_{-i/2}(y_1^*))}{J\Lambda(\sigma_{-i/2}(y_2^*))} \big)
= (\xi_1|\xi_2) (\Lambda(\sigma_{i/2}(y_2))|\Lambda(\sigma_{i/2}(y_1))) \\
&= (\xi_1 \otimes \overline{\nabla^{1/2}\Lambda(y_1)} | \xi_2 \otimes \overline{\nabla^{1/2}\Lambda(y_2)} ).
\end{align*}
In the finite-dimensional case, this is exactly the inner-product we defined on $\mc B(H)$.  Our identification of $L^2(\tilde\varphi)$ with $HS(H)$ is exactly the infinite-dimensional analogue of \eqref{eq:twistedGNS}.
\end{remark}

We can now form $E_{\varphi^{-1}} = \mc B_{M'}(L^2(M'), L^2(\tilde\varphi))$, using that $H_{\varphi^{-1}} = L^2(\tilde\varphi)$.  The right action of $M'$ on $L^2(\tilde\varphi) = HS(H)$ is the restriction of the right action of $\mc B(H)$ on $HS(H)$, namely operator composition.  The left action of $\mc B(H)$ on $L^2(\tilde\varphi) = HS(H)$ is left multiplication of operators.  The embedding of $\alpha\in\mf n_{\varphi^{-1}}$ into $\mc B_{M'}(L^2(M'), L^2(\tilde\varphi))$ is $\hat\alpha \colon \Lambda'(x) \mapsto \Lambda_{\tilde\varphi}(\alpha x) = \alpha x \nabla^{-1/2} \in HS(H)$ (see Proposition~\ref{prop:Trh} for what precisely this means).  Furthermore, we know that $\alpha\in\mf n_{\varphi^{-1}}$ if and only if this operator is bounded.

\begin{remark}\label{rem:L2M'_vs_L2M}
As $M \to M'; x\mapsto Jx^*J$ is an anti-$*$-isomorphism, and we have identified $L^2(M)$ with $L^2(M')$, one can write the condition for an operator $\alpha \colon L^2(M) \to HS(H)$ to be in $\mc B_{M'}(L^2(M'), L^2(\tilde\varphi))$ purely in terms of $M$.  Indeed, the condition is that $\alpha(\xi\cdot a) = \alpha(\xi) \cdot a$ for $a\in M', \xi\in L^2(M')$.  Equivalently, $\alpha J' a^*J' = (1\otimes a^\top)\alpha$ here identifying $HS(H)$ with $H\otimes\overline H$.  In turn, this is equivalent to $\alpha b = (1\otimes (Jb^*J)^\top) \alpha$ for each $b\in M$.

Similarly, for $\alpha\in\mf n_{\varphi^{-1}}$ we have $\hat \alpha \Lambda'(a) = \alpha a\nabla^{-1/2} \in HS(H)$ for $a\in\mf n_{\varphi'}$.  By Lemma~\ref{lem:com_weight} this becomes $\hat\alpha J\Lambda(b) = \alpha JbJ \nabla^{-1/2}$ for $b\in\mf n_\varphi$.

For simplicity, we shall mostly continue to work with $M'$ and $L^2(M')$, working with $\Lambda', \mf n_{\varphi'}$ and so forth.  However, do remember that $H = L^2(M)$ is identified with $L^2(M')$ throughout.
\end{remark}

In this way, we have identified $E_{\varphi^{-1}}$ in a very concrete way, without having to explicitly work with the operator-valued weight $\varphi^{-1}$.  We think of this as an analogue of the finite-dimensional calculations in Section~\ref{sec:fd_summary}.  Let us now see how the algebra $M\vnten M^\op$ arises naturally.

\begin{proposition}\label{prop:MtenMop_as_adj_ops}
Regarding $HS(H)$ as $H\otimes\overline{H}$, we have that $\mc L(E_{\varphi^{-1}}) = \mc B(H) \vnten M^\top$.  Denote by ${}_{M'} \mc L(E_{\varphi^{-1}})$ the adjointable operators which commute with the (left) action of $M'$ on $L^2(\tilde\varphi)$.  Then ${}_{M'} \mc L(E_{\varphi^{-1}}) = M \vnten M^\top$.
\end{proposition}
\begin{proof}
We have that $\mc L(E_{\varphi^{-1}}) = \rho((M')^\op)'$ where $\rho\colon (M')^\op \to \mc B(HS(H))$ is the right action.  As $\rho$ is just multiplication, we have
$\mc L(E_{\varphi^{-1}}) = \{ x\in\mc B(HS(H)) : x(ay) =x(a)y \ (a\in HS(H), y\in M') \}$.
That is, $x(\rankone{\xi}{y^*\eta}) = x(\rankone{\xi}{\eta} y) = x(\rankone{\xi}{\eta}) y$ for $y\in M'$.  Using $HS(H) = H\otimes\overline H$ this becomes $x (1\otimes y^\top) = (1\otimes y^\top) x$ for each $y\in M'$, that is, $x\in (1\otimes (M')^\top)' = \mc B(H) \vnten M^\top$, as claimed.

The left action of $M'$ on $L^2(\tilde\varphi)$ is just the restriction of the left action of $\mc B(H)$ on $HS(H)$.  So $x$ additionally satisfies $x(y\otimes 1) = (y\otimes 1)x$ for $y\in M'$, that is, $x\in (M'\otimes 1)'$ as well.  So ${}_{M'} \mc L(E_{\varphi^{-1}}) = M\vnten\mc B(\overline H) \cap \mc B(H) \vnten M^\top = M \vnten M^\top$.
\end{proof}

As for Hilbert spaces, given a Hilbert C$^*$-module $E$ and a closed submodule $F$, we define $F^\perp = \{ \alpha\in E : (\alpha|\beta)=0 \ (\beta \in F) \}$ which is also a closed submodule, orthogonal to $F$ by definition.  However, as \cite[page~7]{Lance_HilbModsBook} shows, we need not have that $F \oplus F^\perp = E$; when this does happen we say that $F$ is \emph{complemented}.

\begin{theorem}\label{thm:comp_submods}
There is a bijection between:
\begin{enumerate}[(1)]
  \item\label{prop:comp_submods:1}
  projections $e\in M\vnten M^\op$;
  \item\label{prop:comp_submods:2}
  complemented $M'$-sub-bimodules of $E_{\varphi^{-1}}$;
  \item\label{prop:comp_submods:3} HS quantum relations, that is, $M'$-bimodules in $HS(H)$.
\end{enumerate}
\end{theorem}
\begin{proof}
The discussion at the start of \cite[Chapter~3]{Lance_HilbModsBook} shows a bijection between complemented submodules $F$ and projections $e\in\mc L(E_{\varphi^{-1}})$ where $F^\perp$ is the image of the projection $1-e$.  Suppose $e\in {}_{M'}\mc L(E_{\varphi^{-1}})$ so for $x\in M', \alpha\in F$ we have $x\cdot\alpha = x\cdot e(\alpha) = e(x\cdot \alpha)$, hence showing that $F$ is an $M'$-bimodule.  Conversely, if $F$ is an $M'$-bimodule, then one checks that also $F^\perp$ is an $M'$-bimodule.  For $\alpha\in E$ we have $\alpha = e(\alpha) + (1-e)(\alpha) \in F \oplus F^\perp$, so for $x\in M'$ also $x\cdot\alpha = x\cdot e(\alpha) + x\cdot (1-e)(\alpha) \in F \oplus F^\perp$, hence $x\cdot e(\alpha) = e(x\cdot \alpha)$.  It follows that $e$ commutes with the left $M'$-action, and so the equivalence of \ref{prop:comp_submods:1} and \ref{prop:comp_submods:2} follows from Proposition~\ref{prop:MtenMop_as_adj_ops}.

The action of $M\vnten M^\op \subseteq \mc B(HS(H))$ on $E_{\varphi^{-1}} = \mc B_{M'}(L^2(M'), HS(H))$ is just post-composition.  Hence for a projection $e$,
\[ e(E_{\varphi^{-1}}) = \{ e\circ\alpha : \alpha \in \mc B_{M'}(L^2(M'), HS(H)) \}
= \mc B_{M'}(L^2(M'), e(HS(H))), \]
as $e$ is a projection.  We see that $K = e(HS(H))$ is an $M'$-bimodule, as $e\in M\vnten M^\op$, and so \ref{prop:comp_submods:1} and \ref{prop:comp_submods:3} are equivalent.
\end{proof}

The proposition is the analogue of the bijection between projections $e$ and subspaces $V$ described at the end of Section~\ref{sec:ip}.  Again notice that $\varphi$ really plays no role here.  What about subspaces $\mc S$?
To answer this, we need to look at $\mf n_{\varphi^{-1}} \to E_{\varphi^{-1}}; \alpha \mapsto \hat\alpha$.  In the following, we remind the reader that we consider $M'\subseteq \mc B(H)$ when forming the composition $a\alpha b$, and also $a\in\mc B(H) \subseteq \mc B(HS(H))$ and $b\in\mc B(L^2(M'))$, when forming the composition $a \hat\alpha b$.

\begin{lemma}\label{lem:alpha_hat_bimod}
The subspace $\mf n_{\varphi^{-1}}$ is an $M'$-bimodule (maybe not closed) in $\mc B(H)$.  Indeed, for $\alpha\in \mf n_{\varphi^{-1}}$ and $a,b\in M'$ we have that $(a\alpha b)\floatinghat = a \hat\alpha b$.
\end{lemma}
\begin{proof}
Let $\alpha\in \mf n_{\varphi^{-1}}$ and let $a,b\in M'$.  Set $\beta = a\alpha b\in\mc B(H)$.  We know that $\hat\alpha \Lambda'(x) = \alpha x \nabla^{-1/2} \in HS(H)$ for each $x\in \mf n_{\varphi'}$.  Thus for such $x$ we get $\beta x \nabla^{-1/2} = a\alpha bx\nabla^{-1/2} = a \hat\alpha \Lambda'(bx) = a \hat\alpha b \Lambda'(x) \in HS(H)$.  Hence $\hat\beta = a\circ \hat\alpha\circ b$ exists and is bounded, so $\beta\in\mf n_{\varphi^{-1}}$ (by Theorem~\ref{thm:nT_vs_module}).
\end{proof}

\begin{proposition}\label{prop:V_to_S}
Given $V \subseteq HS(H)$ an $M'$-bimodule, the image of a projection $e$, set
\[ S = \{ \alpha \in \mf n_{\varphi^{-1}} : \im(\hat\alpha) \subseteq V \}
= \{ \alpha \in \mf n_{\varphi^{-1}} : e\hat\alpha =\hat\alpha \}. \]
Then $S$ is a (not necessarily closed) $M'$-bimodule in $\mc B(H)$.
\end{proposition}
\begin{proof}
The equivalence of the two definitions of $S$ is clear.  By the lemma, when $\alpha\in S$ and $x,y\in M'$ we have $e (x\alpha y)\floatinghat = e x \hat\alpha y = x e \hat\alpha y = x\hat\alpha y = (x\alpha y)\floatinghat$, as $e$ commutes with the left $M'$-action.  So also $x\alpha y\in S$.
\end{proof}

Of course, in infinite-dimensions, we might then define $\mc S$ to be the weak$^*$-closure of $S$, to obtain a quantum graph in the sense of Weaver.  Furthermore, considerations of an analogue of $\mc S_{i/4}$, section~\ref{sec:diff}, seem even more complicated.  Let us just remark here that it is perfectly possible for $S=\{0\}$ even though $V$ is non-zero, Example~\ref{eg:V_can_give_zero_S}.  We study these ideas further in Section~\ref{sec:considerations_of_S}.

Suppose that $\varphi^{-1}$ is bounded, equivalently, $1 \in \mf n_{\varphi^{-1}}$.  In this case, we see that $\mf n_{\varphi^{-1}} = \mc B(H)$ and for $\alpha\in\mc B(H)$ we have $\|\hat\alpha\|^2 = \|\hat\alpha^*\hat\alpha\| = \|(\alpha|\alpha)\| = \|\varphi^{-1}(\alpha^*\alpha)\| \leq \|\varphi^{-1}\| \|\alpha\|^2$.  Example~\ref{eg:V_can_give_zero_S_varphiinv_bdd} shows that we can still have $S=\{0\}$ with $V$ non-zero in this situation.

\begin{proposition}\label{prop:S_in_case_varphiinv_bdd}
Let $\varphi^{-1}$ be bounded, let $V\subseteq HS(H)$ be an $M'$-bimodule, and let $S$ be as in Proposition~\ref{prop:V_to_S}.  Then $S$ is weak$^*$-closed as a subspace of $\mc B(H)$, and the map $S \to \mc B_{M'}(H,V)$ is weak$^*$-weak$^*$-continuous.
\end{proposition}
\begin{proof}
By the Krein--Smulian Theorem, \cite[Theorem~12.1]{Conway_FunctionalAnalysisBook} for example, it suffices to show that if $(\alpha_i)$ is a bounded net in $S$ with $\alpha_i\to\alpha$ weak$^*$ in $\mc B(H)$, then $\alpha\in S$.  For $b\in\mf n_\varphi$, as $1\in\mf n_{\varphi^{-1}}$ we know from Remark~\ref{rem:L2M'_vs_L2M} that $\hat 1 J\Lambda(b) = JbJ \nabla^{-1/2} \in HS(H)$.  For any $u\in HS(H)$ the map $\mc B(H) \ni \beta \mapsto (u|\beta(JbJ\nabla^{-1/2}))$ is a normal functional, and so
\[ \lim_i (u|\hat\alpha_i J\Lambda(b))_{HS} = \lim_i (u|\alpha_i (JbJ \nabla^{-1/2}))_{HS}
= (u|\alpha (JbJ \nabla^{-1/2}))_{HS} = (u|\hat\alpha J\Lambda(b)). \]
As $(\hat\alpha_i)$ is a bounded net in $\mc B_{M'}(H, HS(H))$, this shows that $\hat\alpha_i \to \hat\alpha$ weak$^*$.
As $\im(\hat\alpha_i)\subseteq V$ for all $i$, also $\im(\hat\alpha) \subseteq V$, and so $\alpha\in S$, as required.

To verify that $S \to \mc B_{M'}(H,V)$ is weak$^*$-weak$^*$-continuous it suffices to show that if $(\alpha_i)$ is a bounded net in $\mc B(H)$ converging weak$^*$ to $\alpha$, then $\hat\alpha_i \to \hat\alpha$ weak$^*$, see \cite[Lemma~10.1]{Daws_multipliers} for example.  The above argument already shows this.
\end{proof}

We explore the relation between $S$ and its image in $\mc B_{M'}(H,V)$ more in Sections~\ref{sec:self-dual_S} and~\ref{sec:considerations_of_S}.

\begin{example}\label{eg:fd1}
We see how $M$ being finite-dimensional fits into this framework.  We have $E_{\varphi^{-1}} = E_{\varphi^{-1}}^0 = \mc B(H)$.  Indeed, given $t\in E_{\varphi^{-1}} = \mc B_{M'}(L^2(M'), HS(H))$, by Remark~\ref{rem:L2M'_vs_L2M}, $tb = (1\otimes (Jb^*J)^\top)t$ for each $b\in M$.  Set $x = t\Lambda(1) \in HS(H) = \mc B(H)$, so
\begin{align*}
t J \Lambda(b)
&= t \Lambda(\sigma_{-i/2}(b^*))
= t \sigma_{-i/2}(b^*) \Lambda(1)
= (1\otimes (J\sigma_{i/2}(b)J)^\top) t\Lambda(1)
= (1\otimes (J\sigma_{i/2}(b)J)^\top) x \\
&= x \circ J\sigma_{i/2}(b)J
=  x J \nabla^{-1/2} b \nabla^{1/2} J
= x \nabla^{1/2} JbJ \nabla^{-1/2}.
\end{align*}
So $t = \hat\alpha$ where $\alpha = x\nabla^{1/2} \in HS(H)$.

The map $\mc B(H) = \mf n_{\varphi^{-1}} \ni \alpha \mapsto \hat\alpha\Lambda(1) \in HS(H) = H\otimes\overline H$ is $\rankone{\xi}{\eta} \mapsto \rankone{\xi}{\eta} \nabla^{-1/2} = \xi \otimes \overline{\nabla^{-1/2}\eta}$, that is, exactly the unitary in \eqref{eq:twistedGNS}.  Consequently, Proposition~\ref{prop:V_to_S} does extend the bijection between $V, e$ and $S=\mc S$ given in Section~\ref{sec:adj_proj}.
\end{example}

\section{Completely bounded maps}\label{sec:CB_maps}

We have not so far considered quantum adjacency operators $A$.  At a minimum, we want $A \colon M \to M$ to be a normal completely positive map.  To describe when $A$ is ``Schur idempotent'', at this level of generality, seems rather hard, so we proceed instead to find direct links with projections $e\in M \vnten M^\op$.

We will consider an infinite-dimensional version of the construction in \cite[Section~5.4]{daws_quantum_graphs}.  Again, it perhaps aids the presentation to consider two von Neumann algebras $M,N$ and a normal CP map $A \colon M \to N$.  To be consistent, we work with $N \vnten M^\op$, but this means that we consider, for example, $L^2(N) \otimes L^2(M^\op)$ as a right $N$-module, the action of $N$ on the \emph{left} factor $L^2(N)$.  Hopefully this will not cause undue confusion.

In fact, it is more notationally convenient to look at maps $M^\op \to N$ defined by members of $N\vnten M$.  We can interpret $\id\otimes\varphi$ as an operator valued weight from $N \vnten M$ to $N$, see Section~\ref{sec:slicing_with_weights} for the details.  We form $\mf m_{\id\otimes\varphi}$ which becomes an $N$-bimodule, with self-dual completion $\mc B_N(L^2(N), L^2(N) \otimes L^2(M))$.  Here we choose some weight $\psi$ on $N$, with GNS map $\Lambda_\psi$.  Each $x\in \mf m_{\id\otimes\varphi}$ induces $\hat x = \Lambda_{\id\otimes\varphi}(x) \colon L^2(N) \to L^2(N) \otimes L^2(M); \Lambda_\psi(c) \mapsto (\Lambda_\psi\otimes\Lambda)(x(c\otimes 1))$.

For $f \in \mf n_{\id\otimes\varphi}$, from Lemma~\ref{lem:slice_alt} we have 
\begin{equation}
\big( \Lambda_{\id\otimes\varphi}(f) \xi \big| \eta\otimes\Lambda(b) \big)
= \varphi( (\omega_{\xi,\eta}\otimes\id)(f^*) b)
\qquad (b\in\mf n_{\varphi}, \xi,\eta\in L^2(C)).
\label{eq:slicing_restatement}
\end{equation}
Define $\pi^\op \colon M^\op \to \mc B(L^2(M)); b^\op \mapsto Jb^*J$, a unital injection normal $*$-homomorphism.  (This might seem arbitrary; compare with Remark~\ref{rem:other_Mop_GNS} below for some rationale.)  Let $f,g\in \mf n_{\id\otimes\varphi}$ and define
\[ A_{f,g}\colon M^\op \to N; \quad
x \mapsto
\Lambda_{\id\otimes\varphi}(f)^* (1\otimes\pi^\op(x)) \Lambda_{\id\otimes\varphi}(g) \in \mc B_N(L^2(N)) \cong N. \]
By definition, $A_{f,g}$ is a normal completely bounded map.

\begin{example}\label{eg:fd3}
Let $M$ be finite-dimensional, so $\mf n_{\id\otimes\varphi} = N\otimes M$ and for $f = \sum_i a_i \otimes b_i \in N\otimes M$, we have that $\Lambda_{\id\otimes\varphi}(f) = \sum_i a_i \otimes \Lambda(b_i)$, in the sense that $\Lambda_{\id\otimes\varphi}(f) \colon \xi \mapsto \sum_i a_i\xi \otimes \Lambda(b_i)$ for $\xi\in L^2(N)$.  Also $\Lambda_{\id\otimes\varphi}^*(\eta\otimes\Lambda(c)) = \sum_i a_i^*\eta \varphi(b_i^*c)$.  Let $g = \sum_i c_i \otimes d_i$, so that
\begin{align*}
A_{f,g}(x^\op)
&= \sum_{i,j} a_i^*c_j (\Lambda(b_i)|Jx^*J\Lambda(d_j))
= \sum_{i,j} a_i^*c_j (\Lambda(b_i)|\Lambda(d_j\sigma_{-i/2}(x))) \\
&= \sum_k \varphi(y_k \sigma_{-i/2}(x)) x_k
= \sum_k \varphi(\sigma_{i/2}(y_k) x) x_k,
\end{align*}
if $f^*g = \sum_k x_k \otimes y_k$ say.  Thus $A_{f,g}$ depends only on $f^*g$ (compare with Corollary~\ref{corr:Afg_only_fstarg} below) and the map $N\otimes M \to \mc B(M^\op, N); f^*g \mapsto A_{f,g}$ is exactly (the inverse of) the map $\Psi'$ used in \cite[Theorem~5.36]{daws_quantum_graphs} (see also Section~\ref{sec:adj_proj}), so we have generalised the finite-dimensional situation.  Let us also remark that using slice maps and Hilbert modules, while more technical, does seem to lead to a conceptually easier to understand framework than the calculations in \cite{daws_quantum_graphs}.
\end{example}

We now work towards finding a more direct link between $A_{f,g}$ and $f,g$ (in fact, as we'll see, $f^*g$ is all that matters).  In the completely positive case, we also show some sort of converse.

\begin{lemma}\label{lem:slice_f_moved_out}
Let $f\in \mf n_{\id\otimes\varphi}$ and let $b^*\in\mf n_{\varphi}$.  Then
\begin{equation}
(1\otimes\pi^\op(b^\op)) \Lambda_{\id\otimes\varphi}(f) \xi
= f (\xi \otimes J \Lambda(b^*) ) \qquad (\xi\in L^2(N)).
\label{eq:swap_hatb_slice_better}
\end{equation}
\end{lemma}
\begin{proof}
We first show \eqref{eq:swap_hatb_slice_better} for $b\in D(\sigma_{i/2}) \cap D(\sigma_{-i/2})$ with $b^*, \sigma_{-i/2}(b) \in \mf n_\varphi$.  For such $b$, by Lemma~\ref{lem:form_J} we have $J\Lambda(b^*) = \Lambda(\sigma_{-i/2}(b))$.
For $a\in \mf n_\varphi$, by Lemma~\ref{lem:action_J}, we have $JbJ\Lambda(a) = \Lambda(a \sigma_{-i/2}(b^*))$ and hence $\pi^\op(b^{*\op}) \Lambda(a) = JbJ\Lambda(a)
= \Lambda(a \sigma_{-i/2}(b^*))$.
Thus, for $c\in\mf n_{\psi}$, from \eqref{eq:slicing_restatement} we have
\begin{align*}
\big( (1\otimes\pi^\op(b^\op))\Lambda_{\id\otimes\varphi^\op}(f) \xi \big| \eta\otimes\Lambda(a) \big)
&= \varphi^\op\big( (\omega_{\xi,\eta}\otimes\id)(f^*) a \sigma_{-i/2}(b^*) \big),
\end{align*}
while also
\begin{align*}
\big( f (\xi\otimes\Lambda(\sigma_{-i/2}(b))) \big| \eta\otimes\Lambda(a) \big)
&= \big( \Lambda(\sigma_{-i/2}(b)) \big| (\omega_{\xi,\eta}\otimes\id)(f^*) \Lambda(a) \big) \\
&= \varphi\big( \sigma_{i/2}(b^*) (\omega_{\xi,\eta}\otimes\id)(f^*) a \big).
\end{align*}
From Lemma~\ref{lem:slice_alt}, we have that $(\omega_{\Lambda_C(c),\eta}\otimes\id)(f^*) \in \mf n_{\varphi}^*$ and so $x = (\omega_{\Lambda_C(c),\eta}\otimes\id)(f^*) a \in \mf m_{\varphi}$.  As $b \in D(\sigma_{i/2}) \cap D(\sigma_{-i/2})$ also $b^* \in D(\sigma_{i/2}) \cap D(\sigma_{-i/2})$ and so $\sigma_{i/2}(b^*) \in D(\sigma_{-i})$.  By Lemma~\ref{lem:sigmaiswap}, we conclude that $\varphi(x \sigma_{-i/2}(b^*)) = \varphi(\sigma_{i/2}(b^*) x)$.  So the above expressions are equal, and as vectors of the form $\eta \otimes \Lambda^\op(a^\op)$ are dense, we have
\[ (1\otimes\pi^\op(b^\op))\Lambda_{\id\otimes\varphi}(f) \xi
= f (\xi\otimes\Lambda(\sigma_{-i/2}(b))). \]
This establishes \eqref{eq:swap_hatb_slice_better} in this case.

Now let $b\in\mf n_\varphi^*$.  Smearing $b$ (see Appendix~\ref{sec:weights}) gives a sequence $(b_n)$ analytic with $\pi^\op(b_n^\op) \to \pi^\op(b^\op)$ $\sigma$-strongly.  As smearing interacts as we hope with the GNS map (the inverse argument is before Corollary~2 of \cite[Section~10.21]{StratilaZsido2nd}) we have that $\Lambda(b_n^*) \to \Lambda(b^*)$.  The result follows.
\end{proof}

\begin{corollary}\label{corr:Afg_only_fstarg}
For $f,g \in \mf n_{\id\otimes\varphi}$, the map $A_{f,g} \colon M^\op \to N$ depends only on $f^*g \in N \vnten M$.
\end{corollary}
\begin{proof}
Let $b_0, b_1 \in \mf n_\varphi^*$ and let $\xi_1,\xi_2 \in L^2(N)$.  Then
\begin{align*}
\big( \xi_1 \big| A_{f,g}(b_1^{*\op} b_0^{\op})\xi_2 \big)
&= \big( \Lambda_{\id\otimes\varphi}(f) \xi_1 \big| (1\otimes\pi^\op(b_1^{*\op} b_0^{\op}) \Lambda_{\id\otimes\varphi}(g) \xi_2 \big) \\
&= \big( f(\xi_1 \otimes J\Lambda(b_1^*)) \big| g(\xi_2\otimes J\Lambda(b_0^*)) \big)
= \big( \xi_1 \big| (\id\otimes\omega_{J\Lambda(b_1^*), J\Lambda(b_0^*)})(f^*g)  \xi_2 \big).
\end{align*}
As $b_0 b_1^* \in \mf n_\varphi^* \mf n_\varphi$, the linear span of such elements is $\mf m_\varphi$ which is weak$^*$-dense in $M$.  This shows that $A_{f,g}$ only depends upon $f^*g \in N \vnten M$, as elements $b_1^{*\op} b_0^{\op} = (b_0b_1^*)^\op$ have dense span in $M^\op$.
\end{proof}

We now reverse this relation between $A$ and $f,g$, at least in the completely positive case when $f=g$.

\begin{proposition}\label{prop:Abdd_implies_f_int}
Let $A\colon M^\op \to N$ be a bounded weak$^*$-continuous map and let $f\in N\vnten M$ with
\[ A(b_1^{*\op}b_0^\op) = (\id\otimes\omega_{J\Lambda(b_1^*), J\Lambda(b_0^*)})(f^*f) 
\qquad (b_0,b_1\in\mf n_\varphi^*). \]
Then $f\in \mf n_{\id\otimes\varphi}$ and $A = A_{f,f}$.
\end{proposition}
\begin{proof}
Let $(a_i)$ be the net of analytic elements given by Proposition~\ref{prop:tomita_nice_bai}.  Then $a_i \to 1$ strongly, and so $a_i^*a_i \to 1$ weak$^*$ and hence $A((a_i^*a_i)^\op) \to A(1^\op)$ weak$^*$.

For each $i$ let $b_i = \sigma_{-i/2}(a_i^*) \in \mf n_\varphi$ so $b_i \to 1$ $\sigma$-strong$^*$, and $\sigma_{i/2}(b_i)^* = a_i \to 1$ as well.  Hence for $x\in M_+$ we have
\[ ( J\Lambda(a_i) | x J\Lambda(a_i) )
= ( \Lambda(\sigma_{-i/2}(a_i^*)) | x \Lambda(\sigma_{-i/2}(a_i^*)) )
= \varphi( b_i^* x b_i ) = \varphi(x), \]
by Proposition~\ref{prop:motivation_approx_weight}.

Hence for $\xi\in L^2(C)$,
\[ (\xi|A(1^\op)\xi)
= \lim_i (\xi|A((a_i^*a_i)^\op)\xi)
= \lim_i ( f(\xi\otimes J\Lambda(a_i)) | f(\xi\otimes J\Lambda(a_i)) )
= \varphi( (\omega_{\xi,\xi}\otimes\id)(f^*f) ). \]
This exactly shows that $f^*f \in \mf m^+_{\id\otimes\varphi^\op}$.  We can hence form $A_{f,f}$, and then Corollary~\ref{corr:Afg_only_fstarg} shows that $A = A_{f,f}$.
\end{proof}

The following is a generalisation of \cite[Definition~3.22, Proposition~3.23]{Wasilewski_Quantum_Cayley}, which uses the ``bounded degree'' terminology; compare Section~\ref{sec:ellinfty} below.

\begin{definition}\label{defn:integrable_e_to_cp_A}
Let $e\in M\vnten M^\op$ be a projection.  We say that $e$ is \emph{integrable} or of \emph{bounded degree} if $e \in \mf n_{\id\otimes\varphi^\op}$, equivalently, $(\id\otimes\varphi^\op)(e) < \infty$.  When this holds, the associated quantum adjacency operator is $A = A_{e,e} \colon M \to M$.
\end{definition}

As $e = e^* = e^2$, we have the equivalence given in the definition.  Proposition~\ref{prop:Abdd_implies_f_int} gives a criteria for a CP map $A$ to arise in this way.

Notice that now we are working with $M^\op$, and so we should choose a natural weight and GNS construction for $M^\op$.  We use that $M^\op$ is isomorphic to $M^\top \subseteq \mc B(\overline H) = \mc B(\overline{L^2(M)})$ for $x^\op \mapsto x^\top$.  This leads to the following, which is simply a calculation, compare \cite[Lemma~5.32]{daws_quantum_graphs}.

\begin{lemma}\label{lem:L2Mop_L2M}
Define a weight $\varphi^\op$ on $M^\op$ by $\varphi^\op(x^\op) = \varphi(x)$.  This has natural GNS map $\Lambda^\op(x^\op) = \overline{\Lambda(x^*)}$ for $x\in \mf n_{\varphi^\op} = \{ x^{\op*} : x\in\mf n_\varphi\}$.  Under this isomorphism $L^2(M^\op) \cong \overline{L^2(M)}$ we have that $x^\op$ acts as $x^\top$, and that $J^\op = \overline J, \nabla^\op = \nabla^\top$.
\end{lemma}

\begin{remark}\label{rem:other_Mop_GNS}
We could instead have used that $M'$ and $M^\op$ are isomorphic, via $Jx^*J \mapsto x^\op$ for $x\in M$, and then used Lemma~\ref{lem:com_weight} to let $\varphi'$ define a weight on $M^\op$.  We would obtain the same weight, but GNS map $\Lambda^\op(x^\op) = J\Lambda(x^*)$.  Of course, this is not surprising, as $J$ establishes a linear isomorphism from $L^2(M)$ to $\overline{L^2(M)}$.
\end{remark}
  
In particular, we have $A_{e,e}(x) = \Lambda_{\id\otimes\varphi^\op}(e)^* (1\otimes \hat\pi(x)) \Lambda_{\id\otimes\varphi^\op}(e)^*$ where here $\hat\pi$ is ``$(\pi^\op)^\op$'', that is, $\hat\pi(x) = J^\op x^{*\top} J^\op = (Jx^*J)^\top \in \mc B(\overline H)$, for $x\in M$.

We now wish to make links with the ideas of Section~\ref{sec:hilb_mods}.  Let $\alpha \in E_{\varphi^{-1}} = \mc B_{M'}(L^2(M'), HS(H))$, so $\alpha \alpha^*$ is identified with the ``rank one'' operator $\rankone{\alpha}{\alpha}$ on the Hilbert $C^*$-module $E_{\varphi^{-1}}$, namely $\beta \mapsto \alpha \cdot (\alpha|\beta) = \alpha \alpha^* \beta$.  By Proposition~\ref{prop:MtenMop_as_adj_ops}, $\alpha \alpha^* \in \mc L(E_{\varphi^{-1}}) = \mc B(H) \vnten M^\top \subseteq \mc B(HS(H))$.  We remark that one can also show this directly by considering that $\alpha$ intertwines the $M'$-actions.  Indeed, in this section we shall identify $L^2(M')$ with $L^2(M)$, and then by Remark~\ref{rem:L2M'_vs_L2M}, a bounded map $\alpha \colon H \to H\otimes\overline H$ is in $E_{\varphi^{-1}}$ exactly when $\alpha x = (1\otimes\hat\pi(x))\alpha$ for each $x\in M$.  Then $\alpha \alpha^* \in (1\otimes (M')^\top)' = \mc B(H) \vnten M^\top$.

We are hence motivated to extend our theory from $M\vnten M^\op$ to $\mc B(H) \vnten M^\op$.  Given $f \in \mf n_{\id\otimes\varphi^\op} \subseteq \mc B(H) \vnten M^\op$ we form $\Lambda_{\id\otimes\varphi^\op}(f) \in \mc B_{\mc B(H)}(HS(H), HS(H) \otimes \overline H)$.  Here, and henceforth, we identify $L^2(M^\op)$ with $\overline H$.  However, we don't really want to work with $HS(H)$ here, so instead we use (the complex conjugate of) \eqref{eq:slicing_restatement} directly to define $\Lambda^H_{\id\otimes\varphi^\op} \colon \mf n_{\id\otimes\varphi^\op} \to \mc B(H, H\otimes\overline H)$, say, satisfying
\[ (\xi\otimes\overline{\Lambda(b^*)} | \Lambda^H_{\id\otimes\varphi^\op}(f) \eta)
= \varphi^\op\big( b^{*\op} (\omega_{\xi,\eta}\otimes\id)(f) \big)
\qquad (\xi,\eta\in H=L^2(M), b^\op\in\mf n_{\varphi^\op}). \]
Indeed, as the action of $\mc B(H)$ on $HS(H)$ only ``sees the left-hand side'' of $HS(H) = H\otimes\overline H$, moving to $\Lambda^H_{\id\otimes\varphi^\op}$ hasn't lost any information.  Then the argument of Lemma~\ref{lem:slice_f_moved_out} shows that
\begin{equation}
(1\otimes\hat\pi(b)) \Lambda^H_{\id\otimes\varphi^\op}(f) \xi
= f (\xi \otimes \overline{J\Lambda(b)})
\qquad (b\in\mf n_\varphi, \xi\in H).
\label{eq:swap_hatb_slice_betterH}\tag{\ref*{eq:swap_hatb_slice_better}'}
\end{equation}
As in Corollary~\ref{corr:Afg_only_fstarg} (and Proposition~\ref{prop:Abdd_implies_f_int}) we then see that for $a,b\in\mf n_\varphi, f,g\in\mf n_{\id\otimes\varphi^\op}$,
\begin{equation}
A_{f,g}(a^*b)
= \Lambda^H_{\id\otimes\varphi^\op}(f)^* (1\otimes\hat\pi(a^*b)) \Lambda^H_{\id\otimes\varphi^\op}(g)
= (\id\otimes\omega_{\overline{J\Lambda(a)}, \overline{J\Lambda(b)}})(f^*g).
\label{eq:Afg_slice_prop}
\end{equation}

In the following, to avoid notational confusion, we use $\alpha$ for an element of $\mf n_{\varphi^{-1}}$ and $t$ for elements of $E_{\varphi^{-1}}$.
We shall now show a very close link between $\mf n_{\varphi^{-1}}$ and elements $t\in E_{\varphi^{-1}}$ with $t^*$ integrable.

\begin{proposition}\label{prop:bounded_to_nvarphiinv}
Given $t\in E_{\varphi^{-1}}$ and $\alpha\in\mc B(H)$ we have that $t^*(\xi \otimes \overline{J\Lambda(b)}) = b\alpha^*\xi$ for each $\xi\in H, b\in\mf n_{\varphi}$ if and only if $\alpha \in \mf n_{\varphi^{-1}}$ and $t = \hat \alpha$.
\end{proposition}
\begin{proof}
If $t=\hat \alpha$ for some $\alpha\in\mf n_{\varphi^{-1}}$ then, by Remark~\ref{rem:L2M'_vs_L2M} again, $tJ\Lambda(a) = \alpha JaJ\nabla^{-1/2}\in HS(H)$ for $a\in\mf n_{\varphi}$.
Now let $a\in \mf n_\varphi \cap D(\sigma_{i/2})$ with $\sigma_{i/2}(a)^*\in\mf n_{\varphi}$, and let $b\in\mf n_\varphi \cap \mf n_{\varphi}^*$.  Using the inner-product on $HS(H) = H\otimes\overline{H}$, we have
\begin{align*}
&(J\Lambda(a) | t^*(\xi\otimes\overline{J\Lambda(b)}))
= ( \alpha JaJ\nabla^{-1/2} | \xi\otimes\overline{J\Lambda(b)} ) \\
&= \Tr\big( (\alpha JaJ\nabla^{-1/2})^* \rankone{\xi}{J\Lambda(b)} \big)
= ( \alpha JaJ\nabla^{-1/2} J\Lambda(b)) | \xi ).
\end{align*}
As $\nabla^{-1/2}J\Lambda(b) = \Lambda(b^*)$, using that $b\in\mf n_\varphi \cap \mf n_{\varphi}^*$, we continue the calculation as
\begin{align*}
= ( \alpha JaJ\Lambda(b^*) | \xi )
= ( b^*\Lambda(\sigma_{-i/2}(a^*)) | \alpha ^*\xi )
= ( J\Lambda(a) | b\alpha ^*\xi ).
\end{align*}
Here we used Lemmas~\ref{lem:action_J} and~\ref{lem:form_J}.  By density of such $a$, we conclude that $t^*(\xi \otimes \overline{J\Lambda(b)}) = b\alpha^*\xi$, and then by density of such $b$, this holds for all $b\in\mf n_\varphi$.

Conversely, suppose that for some $t,\alpha $ this relation holds.  Under the same assumptions on $a,b$, we reverse the above calculation to see that
\[ ( tJ\Lambda(a) | \xi\otimes\overline{J\Lambda(b)} )
= ( \alpha JaJ\nabla^{-1/2} | \xi\otimes\overline{J\Lambda(b)} ). \]
By density, it follows that $t J\Lambda(a) = \alpha  JaJ \nabla^{-1/2}$ for each $a\in\mf n_\varphi$.  By Theorem~\ref{thm:nT_vs_module}, it follows that $\alpha \in\mf n_{\varphi^{-1}}$ and $t = \hat \alpha $.
\end{proof}

\begin{proposition}\label{prop:int_t_are_fin}
For $t\in E_{\varphi^{-1}}$, we have that $t = \hat\alpha$ for some $\alpha\in\mf n_{\varphi^{-1}}$ if and only if $tt^* \in \mf m_{\id\otimes\varphi^\op}^+$.  In this case, $(\id\otimes\varphi^\op)(tt^*) = \alpha\alpha^*$.  Furthermore, set $f = |t^*| = (tt^*)^{1/2}$, and let $t^* = uf$ be the polar decomposition.  Then $u^*\in E_{\varphi^{-1}}$ and $f\in \mf n_{\id\otimes\varphi^\op}$.  We have that $A_{f,f}(x) = \alpha x \alpha^*$ for $x\in M$.
\end{proposition}
\begin{proof}
Let $\alpha\in\mf n_{\varphi^{-1}}$.  Let $(a_i)$ be a net as in Proposition~\ref{prop:tomita_nice_bai}.  By Proposition~\ref{prop:bounded_to_nvarphiinv} we see that for $a,b\in\mf n_\varphi$ we have
\[ (\id\otimes\omega_{\overline{J\Lambda(a)}, \overline{J\Lambda(b)}})(\hat\alpha\hat\alpha^*)
= \alpha a^*b \alpha^*. \]
Given $\omega\in \mc B(H)_*^+$ let $x = (\omega\otimes\id)(\hat\alpha\hat\alpha^*) \in M^\top \cong M^\op$.  Then $x\geq 0$ and
\begin{align*}
\ip{\alpha a_i a_i^*\alpha^*}{\omega}
&= (\overline{J\Lambda(a_i^*)}|x\overline{J\Lambda(a_i^*)})
= (J\Lambda(a_i^*) | x^{\top} J\Lambda(a_i^*))
= (\Lambda(\sigma_{-i/2}(a_i)) | x^\top \Lambda(\sigma_{-i/2}(a_i))) \\
&= \varphi( \sigma_{-i/2}(a_i)^* x^\top \sigma_{-i/2}(a_i))
\xrightarrow[i\to\infty]{} \varphi(x^\top)
= \varphi^\op(x),
\end{align*}
As $a_i \to 1$ $\sigma$-strong$^*$ we see that $a_ia_i^* \to 1$ $\sigma$-weakly, and so
\[ \ip{\alpha\alpha^*}{\omega} = \varphi^\op((\omega\otimes\id)(\hat\alpha\hat\alpha^*))
\qquad (\omega\in\mc B(H)^+_*). \]
That is, $\hat\alpha\hat\alpha^* \in \mf m^+_{\id\otimes\varphi^\op}$ with $(\id\otimes\varphi^\op)(\hat\alpha\hat\alpha^*) = \alpha\alpha^*$.

Conversely, suppose that $tt^* \in \mf m_{\id\otimes\varphi^\op}^+$.  Let $t^* = u|t^*| = uf$ be the polar decomposition.  As $\mc B_{M'}(L^2(M'), HS(H))$ is defined by a commutation relationship, it is easy to see that $t\in E_{\varphi^{-1}}$ means that also $u^*\in E_{\varphi^{-1}}$, compare \cite[Lemma~1.2]{Rieffel_mortia_cstar_wstar}.  As $f^*f = tt^* \in \mf m^+_{\id\otimes\varphi^\op}$ we have that $f\in \mf n_{\id\otimes\varphi^\op}$.  From \eqref{eq:swap_hatb_slice_betterH} we see that
\[ t^*(\xi\otimes \overline{J\Lambda(b)})
= uf(\xi\otimes \overline{J\Lambda(b)})
= u (1\otimes\hat\pi(b))\Lambda^H_{\id\otimes\varphi^\op}(f)\xi
= b u \Lambda^H_{\id\otimes\varphi^\op}(f)\xi, \]
as $u^* \in E_{\varphi^{-1}}$.  So with $\alpha = (u \Lambda^H_{\id\otimes\varphi^\op}(f))^*$ we have verified the condition of Proposition~\ref{prop:bounded_to_nvarphiinv}, and hence $\alpha \in \mf n_{\varphi^{-1}}$ with $t = \hat\alpha$.  Finally, for $x,y\in\mf n_\varphi$, by \eqref{eq:Afg_slice_prop}, we have
\[ A_{f,f}(x^*y)
= (\id\otimes\omega_{\overline{J\Lambda(x)},\overline{J\Lambda(y)}})(f^*f)
= (\id\otimes\omega_{\overline{J\Lambda(x)},\overline{J\Lambda(y)}})(tt^*)
= \alpha x^*y\alpha^*, \] 
so as $A_{f,f}$ is weak$^*$-continuous, by density, $A_{f,f}(x) = \alpha x\alpha^*$ for all $x\in M$.
\end{proof}

Let $e\in M\vnten M^\op$ be a projection, and let $V \subseteq H\otimes\overline H$ be the image of $e$.  We also regard $e$ as acting on $E_{\varphi^{-1}}$, with image $\mc B_{M'}(L^2(M'), V)$, itself a self-dual Hilbert C$^*$-module.  As in Theorem~\ref{thm:selfdual_is_weak_direct_sum} we can find a family of partial isometries $(t_i)$ which forms an orthogonal basis for $\mc B_{M'}(L^2(M'), V)$.  This means that $t_i^* t_j=0$ for $i\not=j$, that $t_i^*t_i = p_i$ is a projection for each $i$, and that $t = \sum_i t_i t_i^*t$ for each $t\in E_{\varphi^{-1}}$ with $et=t$.  This last point is equivalent $\lin \{ \im(t_i) \}$ being dense in $V$, equivalently, $\sum_i t_i t_i^* = e$, noting that $t_it_i^*$ is a projection for each $i$.

\begin{theorem}\label{thm:Kraus_rep}
Let $e\in M\vnten M^\op$ be an integrable projection, and form $A = A_{e,e}$ the quantum adjacency operator.   With $V$ the image of $e$, let $(t_i)$ be an orthogonal basis for $\mc B_{M'}(L^2(M'), V)$.  Then each $t_i = \hat\alpha_i$ for some $\alpha_i \in \mf n_{\varphi^{-1}}$, and $A(x) = \sum_i \alpha_i x \alpha_i^*$ for $x\in M$.
\end{theorem}
\begin{proof}
As observed, for each $i$, we have that $t_i t_i^* \leq e$.  Hence $(\id\otimes\varphi^\op)(t_it_i^*) < \infty$, and so Proposition~\ref{prop:int_t_are_fin} shows that $t_i = \hat\alpha_i$ for some $\alpha_i\in \mf n_{\varphi^{-1}}$, and furthermore, $A_{|t_i^*|, |t_i^*|}(x) = \alpha_i x \alpha_i^*$ for $x\in M$, equivalently,
\[ (\id\otimes\omega_{\overline{J\Lambda(x)},\overline{J\Lambda(y)}})(t_it_i^*)
= \alpha_i x^*y\alpha_i^*
\qquad (x,y\in\mf n_\varphi). \]
As slice maps are weak$^*$-continuous, and $e = \sum_i t_i t_i^*$, we obtain
\[ A(x^*y) = (\id\otimes\omega_{\overline{J\Lambda(x)},\overline{J\Lambda(y)}})(e)
= \sum_i (\id\otimes\omega_{\overline{J\Lambda(x)},\overline{J\Lambda(y)}})(t_it_i^*)
= \sum_i \alpha_i x^*y\alpha_i^*, \]
for $x,y\in\mf n_\varphi$, and hence by density, the result follows.
\end{proof}

In the integrable case, we have hence written the quantum adjacency operator $A$ using an ``orthogonal basis for $V$''.  Below in Remark~\ref{rem:same_as_W} we expand upon this in the case of matrix algebras, showing how this generalises \cite[Proposition~3.30]{Wasilewski_Quantum_Cayley}.

\begin{example}\label{eg:fd2}
Again let $M$ be finite-dimensional, continuing Example~\ref{eg:fd1}.  We give some explicit formulae for some of the ideas in this section.

Again with reference to Remark~\ref{rem:L2M'_vs_L2M}, for $\alpha\in \mc B(H)$ and $a,b\in M$ we have $\hat\alpha J\Lambda(a) \in HS(H)$ and so regard it as an operator $H\to H$, and then
\begin{align*}
(\hat \alpha J\Lambda(a))\Lambda(b)
&= \alpha JaJ \nabla^{-1/2} \Lambda(b)
= \alpha JaJ \Lambda(\sigma_{i/2}(b))
= \alpha\Lambda(\sigma_{i/2}(b)\sigma_{-i/2}(a^*))
= \alpha \sigma_{i/2}(b) J \Lambda(a).
\end{align*}
Let $(d_k) \subseteq M$ be an algebraic basis such that $(\Lambda(d_k))$ forms an orthonormal basis for $L^2(M)$.  Then
\[ \hat\alpha J\Lambda(a)
= \alpha JaJ \nabla^{-1/2}
= \sum_k \alpha JaJ \nabla^{-1/2} \rankone{\Lambda(d_k)}{\Lambda(d_k)}
= \sum_k \alpha \sigma_{i/2}(d_k) \rankone{J\Lambda(a)}{\Lambda(d_k)}. \]
So for $\xi,\eta_1,\eta_2\in L^2(M)$,
\begin{align*}
\big( \xi \big| \hat\alpha^*(\rankone{\eta_1}{\eta_2}) \big)
&= \sum_k \big( \alpha \sigma_{i/2}(d_k) \rankone{\xi}{\Lambda(d_k)} \big| \rankone{\eta_1}{\eta_2} \big)_{HS(H)}
= \sum_k (\eta_2|\Lambda(d_k)) (\alpha \sigma_{i/2}(d_k)\xi|\eta_1) \\
\implies\quad &
\hat\alpha^*(\rankone{\eta_1}{\eta_2}) = \sum_k (\eta_2|\Lambda(d_k)) \sigma_{i/2}(d_k)^* \alpha^* \eta_1, \\
\text{and}\quad &
\hat\alpha \hat\alpha^*(\rankone{\eta_1}{\Lambda(a)})
   = \sum_{k,l} (\Lambda(a)|\Lambda(d_k)) \alpha \sigma_{i/2}(d_l) \rankone{\sigma_{i/2}(d_k)^* \alpha^* \eta_1}{\Lambda(d_l)}  \\
& \quad
= \sum_{l} \alpha \sigma_{i/2}(d_l) \sigma_{i/2}(a)^* \alpha^* \rankone{\eta_1}{\Lambda(d_l)}.
\end{align*}
So finally, for $\alpha,\beta\in\mc B(H)$,
\begin{align*}
\hat \beta^* \hat \alpha \xi
&= \hat\beta^* \sum_k \alpha \sigma_{i/2}(d_k) \rankone{\xi}{\Lambda(d_k)}
= \sum_{k,l} (\Lambda(d_k)|\Lambda(d_l)) \sigma_{i/2}(d_l)^* \beta^* \alpha \sigma_{i/2}(d_k) \xi \\
\implies \quad &
\hat\beta^* \hat\alpha = \varphi^{-1}(\beta^*\alpha) = \sum_{k} \sigma_{i/2}(d_k)^* \beta^* \alpha \sigma_{i/2}(d_k).
\end{align*}
We recognise this final formula as some sort of ``averaging'' operation.
One could now continue the calculations to verify the various propositions in this section, but this does not seem very enlightening.

When is $(\hat\alpha_i)$ an orthogonal basis for $\mc B_{M'}(H, V) \cong V$?  We need that:
\begin{enumerate}[(1)]
\item $\hat\alpha_i^* \hat\alpha_j = \delta_{i,j} p_i$ for some projections $p_i\in M'$;
\item $\lin\{ \hat\alpha_i(\xi) : i\in I, \xi\in H \} = V$.
\end{enumerate}
As $\hat\alpha b = (1\otimes (Jb^*J)^\top)\hat\alpha$, condition (2) is equivalent to $\{ \hat\alpha_i\Lambda(1) = \alpha_i\nabla^{-1/2} \in HS(H) : i\in I\}$ generating $V$ as a right $M'$-module.  It seems hard to say more, but again see Remark~\ref{rem:same_as_W} below for some explicit formulae.
\end{example}

\subsection{Self-dual modules}\label{sec:self-dual_S}

Again consider the subspace $V\subseteq HS(H)$, the image of the projection $e$, and form $S \subseteq \mf n_{\varphi^{-1}}$ as in Proposition~\ref{prop:V_to_S}.  By Lemma~\ref{lem:alpha_hat_bimod} we see that $S$ is a (possibly not closed) submodule of $\mf n_{\varphi^{-1}}$, and so we can ask when $S$ is self-dual.  By Proposition~\ref{prop:all_self_dual_mods} and Theorem~\ref{thm:self_dual_completion}, this is equivalent to $\{ \hat\alpha : \alpha\in S \} = \mc B_{M'}(H, S \otimes_{M'} H)$.  The inner-product on $S \otimes_{M'} H$ is
\[ (\alpha\otimes\xi | \beta\otimes\eta) = (\xi|(\alpha|\beta)\eta) = (\xi|\hat\alpha^*\hat\beta\eta) = (\hat\alpha\xi | \hat\beta\eta)
\qquad (\alpha,\beta\in S, \xi,\eta\in H). \]
Hence $S \otimes_{M'} H \to V; \alpha\otimes\xi\mapsto\hat\alpha\xi$ is an isometry.  Example~\ref{eg:V_can_give_zero_S_varphiinv_bdd} below shows that we can have $S=\{0\}$ for a non-zero $V$, even when $\varphi^{-1}$ is bounded.  However, the next proof shows that this cannot occur with $e$ integrable.

Notice that in the following we make no assumptions about $\varphi^{-1}$ (such as it being bounded, as in \cite{Wasilewski_Quantum_Cayley}).  Compare the following to \cite[Propositions~3.26]{Wasilewski_Quantum_Cayley}.

\begin{theorem}\label{thm:int_implies_selfdual}
Let $e\in M\vnten M^\op$ be the projection associated to $V$.  When $e$ is integrable, the associated $S \subseteq \mf n_{\varphi^{-1}}$ is self-dual.
\end{theorem}
\begin{proof}
By definition, $S = \{ \alpha\in\mf n_{\varphi^{-1}} : e\hat\alpha = \hat\alpha \}$.  Let $(\alpha_i)$ be as in Theorem~\ref{thm:Kraus_rep} so that $(\hat\alpha_i)$ is an orthogonal basis for $e E_{\varphi^{-1}} = \mc B_{M'}(H, V)$, where of course $V = e(H\otimes\overline H)$.  Then $e = \sum_i \hat\alpha_i \hat\alpha_i^*$, and for each $i$ we have $\alpha_i \in S$.  So for $v\in V$ we have $v = ev = \sum_i \hat\alpha_i \hat\alpha_i^* v$, the sum converging in norm.  Hence $v \in \overline\lin\{ \im \hat\alpha_i \}$, therefore $V = \overline\lin\{ \im\hat\alpha : \alpha\in S\}$, and so by the above discussion, $S\otimes_{M'}H = V$.

Let $t\in \mc B_{M'}(H, V) \subseteq \mc B_{M'}(H, HS(H))$, so $tt^* \leq \|t\|^2 e$ as $e$ is the projection onto $V$.  Hence, much as in the proof of Theorem~\ref{thm:Kraus_rep}, we see that $tt^*\in \mf m^+_{\id\otimes\varphi^\op}$ and so Proposition~\ref{prop:int_t_are_fin} shows that $t = \hat\alpha$ for some $\alpha\in\mf n_{\varphi^{-1}}$.  As then $\im(\hat\alpha) = \im(t) \subseteq V$, we have $\alpha\in S$.  Hence $S$ is self-dual.
\end{proof}

As we can have $S=\{0\}$ with $V\not=\{0\}$, we might impose the non-degeneracy condition that $S \otimes_{M'} H \cong \overline\lin\{ \im\hat\alpha : \alpha \in S \} = V$, equivalently, that $\hat S = \{\hat\alpha:\alpha\in S\}$ is weak$^*$-dense in $E_{\varphi^{-1}}$, see Proposition~\ref{prop:wstar_dense_nondeg}.  We now work towards a converse of the previous theorem.

\begin{lemma}\label{lem:sum_alpha_finite_e_int}
Let $S \otimes_{M'} H=V$, and let $(\hat\alpha_i)$ be an orthogonal basis for $\mc B_{M'}(H,V)$, for some family $(\alpha_i)$ in $\mf n_{\varphi^{-1}}$.  If $\sum_i \alpha_i \alpha_i^* < \infty$ then $e$ is integrable.
\end{lemma}
\begin{proof}
We can define $A(x) = \sum_i \alpha_i x \alpha_i^*$ for $x\in M$, the sum converging weak$^*$ by our hypothesis.  Letting $\hat\alpha_i^* = u_i f_i$ be the polar decomposition, as in the proof of Proposition~\ref{prop:int_t_are_fin} we find that $u_i \in \mf n_{\varphi^{-1}}$.  By Proposition~\ref{prop:bounded_to_nvarphiinv}, we have that $u_i f_i(\xi\otimes\overline{J\Lambda(b)}) = b \alpha_i^* \xi$ for each $\xi\in H, b\in\mf n_\varphi$.  Thus for $a,b\in\mf n_\varphi$ we have
\[ \alpha_i a^*b \alpha_i^*
= (\id\otimes\omega_{\overline{J\Lambda(a)}, \overline{J\Lambda(b)}})(f_i^*f_i)
= (\id\otimes\omega_{\overline{J\Lambda(a)}, \overline{J\Lambda(b)}})(\hat\alpha_i \hat\alpha_i^*). \]
Summing over $i$ gives $A(a^*b) = (\id\otimes\omega_{\overline{J\Lambda(a)}, \overline{J\Lambda(b)}})(e)$ which also shows that $A$ maps into $M$.  Now apply Proposition~\ref{prop:Abdd_implies_f_int} to conclude that $e$ is integrable, and $A = A_{e,e}$.
\end{proof}

In the following, to be clear, when we write $\|\alpha\|$ for some $\alpha\in\mf n_{\varphi^{-1}}$ (e.g. for $\alpha\in S$) we mean the norm of $\alpha$ as a member of $\mc B(H)$.  Similarly, we regard $M_n(S)$ as a subspace of $M_n(\mc B(H))$, and $M_n(\mc B_{M'}(H,V)) \subseteq \mc B(H^n, V^n)$, when computing norms.

\begin{proposition}\label{prop:inv_to_cb}
Suppose that there is a constant $K>0$ such that for each $t\in\mc B_{M'}(H,V)$ there is $\alpha\in S$ with $\hat\alpha=t$ and $\|\alpha\| \leq K \|t\|$.  For any $n$ and any matrix $(\alpha_{i,j}) \in M_n(S)$ we have $\|(\alpha_{i,j})\|\leq K \|(\hat\alpha_{i,j})\|$.
\end{proposition}
\begin{proof}
We of course always have a map $S \to \mc B_{M'}(H,V); \alpha\mapsto\hat\alpha$, and so the hypothesis is that this map is a bijection, and bounded below.  We shall show that the inverse, say $T\colon\hat\alpha \mapsto \alpha$, is even completely bounded, by adapting the proof of \cite[Theorem~2.1]{Smith_CB_ModMaps} to this situation.

By \cite[Lemma~2.3]{Smith_CB_ModMaps} it follows that for every finite-dimensional subspace $H_0\subseteq H$ there is $\xi\in H$ with $H_0 \subseteq \overline{{M'}\xi}$, because $H=L^2(M)$.  Let $t = (\hat\alpha_{i,j}) \in M_n(\mc B_{M'}(H,V))$ be of norm $1$ and suppose towards a contradiction that $\alpha = (\alpha_{ij}) \in M_n(\mc B(H))$ has norm greater than $\|T\|$.  It follows that we can find vectors $\xi=(\xi_k)$ and $\eta=(\eta_k)$ in $H\otimes\mathbb C^n$ with $\|\xi\|<1, \|\eta\|<1$ and $|\sum_{i,j} (\eta_i|\alpha_{ij}\xi_j)| = |(\eta|\alpha \xi)| > \|T\|$.  By the observation just made, for $\epsilon>0$ we can find $\mu,\nu \in H$ and $a_i, b_i$ in $M'$ with $\sum_j \|\xi_j - a_j\mu\|^2 < \epsilon$ and $\sum_i \|\eta_i - b_i\nu\|^2 < \epsilon$.  Thus
\begin{equation}
\Big| \sum_{i,j} (\nu|b_i^*\alpha_{ij}a_j \mu) \Big|
= \big|\big( (b_i\nu) \big| \alpha (a_j\mu) \big) \big|
> \|T\|, \label{eq:Tnorm}
\end{equation}
if $\epsilon$ is sufficiently small.  Furthermore, setting $\alpha_{i,n+1} = 0 = \alpha_{n+1,j}$ for each $i,j$ so as to obtain a matrix in $M_{n+1}(\mc B(H))$, and setting $a_{n+1} = b_{n+1} = \epsilon 1$, we still have \eqref{eq:Tnorm}.  Then $a = \sum_{i=1}^{n+1} a_i^*a_i = \epsilon^2 + \sum_{i=1}^{n} a_i^*a_i$ is positive and invertible, and similarly for $b = \sum_{i=1}^{n+1} b_i^*b_i$.  Set $\mu' = a^{1/2}\mu, a_i' = a_i a^{-1/2}$ and similarly $\nu' = b^{1/2}\nu, b_j'=b_jb^{-1/2}$, so \eqref{eq:Tnorm} becomes
\[ \Big| \sum_{i,j} (\nu'|{b'_i}^*\alpha_{ij}a'_j \mu') \Big| > \|T\|. \]
We now see that
\[ \|\mu'\|^2 = (a^{1/2}\mu|a^{1/2}\mu) = \sum_{i=1}^{n+1} (\mu|a_i^*a_i\mu)
= \sum_i \|a_i\mu\|^2 < 1, \]
again if $\epsilon>0$ is sufficiently small.  Similarly $\|\nu'\|<1$.

Set $\beta = \sum_{i,j} {b'_i}^*\alpha_{ij}a'_j \in S$ so by Lemma~\ref{lem:alpha_hat_bimod} we have that $\hat\beta = \sum_{i,j} {b'_i}^* \hat\alpha_{ij} a'_j$, so denoting this by $s\in\mc B_{M'}(H,V)$ we see that $T(s) = \beta$ and hence $\|\beta\| \leq \|T\| \|s\|$.  As $|(\nu'|\beta\mu')| > \|T\|$ we have $\|\beta\| > \|T\|$ and so $\|s\|>1$.  However, the matrix $(\hat\alpha_{i,j})$ has norm $1$, and so
\[ \|s\| \leq \Big\| \sum_j {b_j'}^*{b_j'} \Big\|^{1/2}\Big\| \sum_i {a_i'}^*{a_i'} \Big\|^{1/2}
= \Big\| \sum_j b^{-1/2} b_j^* b_j b^{-1/2}\Big\|^{1/2} \Big\| \sum_i a^{-1/2} a_i^* a_i a^{-1/2} \Big\|^{1/2} = 1, \]
contradiction, as required.
\end{proof}

The following is related to \cite[Proposition~3.27]{Wasilewski_Quantum_Cayley}.

\begin{theorem}\label{thm:S_bdd_below_e_int}
Let $V\subseteq HS(H)$ be an $M'$-bimodule, and let $S = \{ \alpha\in\mf n_{\varphi^{-1}} : \im\hat\alpha \subseteq V \}$.  Then $e$ is integrable if and only if the map $S \to \mc B_{M'}(H,V); \alpha \mapsto \hat\alpha$ is surjective.
\end{theorem}
\begin{proof}
We have shown the ``only if'' case in Theorem~\ref{thm:int_implies_selfdual}.  Suppose that $S \to \mc B_{M'}(H,V); \alpha \mapsto \hat\alpha$ is surjective, hence bijective.  We claim that the inverse $\hat\alpha\mapsto\alpha$ is bounded, which we shall prove using the Closed Graph Theorem.  Let $(\alpha_n)$ be a sequence in $S$ with $\hat\alpha_n \to \hat\alpha$ and $\alpha_n\to\beta$, both in the respective norms.  By Proposition~\ref{prop:bounded_to_nvarphiinv}, for $\xi\in H$ and $b\in\mf n_\varphi$, we have
\[ b\alpha^*\xi = \hat\alpha^*(\xi\otimes\overline{J\Lambda(b)}) = \lim_n \hat\alpha_n^*(\xi\otimes\overline{J\Lambda(b)})
= \lim_n b \alpha_n^* \xi = b \beta^*\xi. \]
It follows that $\alpha = \beta$, as required.  By Proposition~\ref{prop:inv_to_cb} , the inverse map is even completely bounded, say with cb-norm $K$. 

Given $\alpha_1,\cdots,\alpha_n \in S$, considering the matrix $\alpha$ in $M_n(S) \subseteq M_n(\mc B(H))$ whose first row is $(\alpha_i)$, and zero otherwise, we see that $\alpha^*\alpha = (\alpha_i^*\alpha_j) \in M_n(\mc B(H))$.  Then the matrix in $M_n(\mc B_{M'}(H,V))$ with first row $(\hat\alpha_i)$, and zero otherwise, has norm at least $K^{-1} \|\alpha\|$, and so
\[ \| (\hat\alpha_i^* \hat\alpha_j) \|_{M_n(M')} = \| (\hat\alpha_j) \|^2_{M_{1,n}(\mc B_{M'}(H,V))} \geq K^{-2}\|\alpha\|^2 = K^{-2} \| (\alpha_i^*\alpha_j) \|. \]
Let $(a_i)$ in $\mc B(H)$ with $\sum_i a_i^*a_i \leq 1$.  For any unit vector $\xi\in H$ we have that $\sum_i \|a_i\xi\|^2 = \sum_i (\xi|a_i^*a_i\xi) \leq \|\xi\|^2 = 1$, and similarly for a unit vector $\eta$, and hence $|( \eta | \sum_{i,j} a_i^* \alpha_i^* \alpha_j a_j \xi )| \leq \|(\alpha_i^*\alpha_j)\|$.  Thus $\|\sum_{i,j} a_i^* \alpha_i^* \alpha_j a_j \| \leq \|(\alpha_i^*\alpha_j)\|$ and so
\[ \Big\| \sum_{i,j} a_i^*\alpha_i^*\alpha_ja_j \Big\| \leq K^2 \|(\hat\alpha_i^*\hat\alpha_j)\|. \]

Let $(\hat\alpha_i)_{i\in I}$ be an orthogonal basis for $\mc B_{M'}(H,V)$.  For a finite $I_0\subseteq I$ let $\beta = \sum_{i\in I_0} \alpha_i \alpha_i^* \leq \|\beta\| 1$.  Hence if we set $a_i = \|\beta\|^{-1/2} \alpha_i^*$ for $i\in I_0$, then $\sum_{i\in I_0} a_i^*a_i = \|\beta\|^{-1} \sum_{i\in I_0} \alpha_i\alpha_i^* \leq 1$.  Thus 
\[ K^2 \|(\hat\alpha_i^*\hat\alpha_j)\|
\geq \Big\| \sum_{i,j\in I_0} a_i^* \alpha_i^* \alpha_j a_j \Big\|
= \|\beta\|^{-1} \Big\| \sum_{i,j\in I_0} \alpha_i \alpha_i^* \alpha_j \alpha_j^* \Big\|
= \|\beta\|^{-1} \| \beta^*\beta \| = \|\beta\|. \]
However, of course $\hat\alpha_i^*\hat\alpha_j = \delta_{i,j} p_i$ for some projections $p_i\in M'$, and so $\|(\hat\alpha_i^*\hat\alpha_j)\| = 1$.  Thus $\|\beta\| \leq K^2$, and so $\sum_{i\in I} \alpha_i \alpha_i^*$ is bounded above and hence converges.  By Lemma~\ref{lem:sum_alpha_finite_e_int}, $e$ is integrable.
\end{proof}

The previous result works because $M$ acts on $L^2(M)$, which was used in an essential way in the proof of Proposition~\ref{prop:inv_to_cb}.  As in Weaver's work, we could of course think of $M$ acting on other spaces: for example, in \cite{Wasilewski_Quantum_Cayley}, $M = \prod_i \mathbb M_{n(i)}$ acts on $H_0 = \bigoplus \mathbb C^{n(i)}$, which seems natural, but is of course not $L^2(M)$.  \cite[Remark~3.28]{Wasilewski_Quantum_Cayley} essentially asks if an analogue of Theorem~\ref{thm:S_bdd_below_e_int} could hold in this case.  In Example~\ref{eg:self-dual_changes} we show that for this choice of $M$ acting on $H_0$, we can have $S$ being self-dual but $e$ not integrable.  However, this example does not have $\varphi^{-1}$ being bounded, which is the setting of \cite{Wasilewski_Quantum_Cayley}.  A more complicated example, Example~\ref{eg:BH_selfdual_notint}, shows that for $M = \mc B(K)$ acting on $K$ (again, not $L^2(M)$) we can have $\varphi^{-1}$ bounded and a self-dual $S$ with the associated projection not integrable.  We then combine these examples to give a matrix algebra example, with bounded $\varphi^{-1}$, in Example~\ref{eg:matrix_sd_notint}.

\subsection{Hilbert space operators}\label{sec:hilb_ops}

In the finite-dimensional case, we identify $M$ with $L^2(M)$, but so far we have not said much about operators on $L^2(M)$.  We continue with the setting of Section~\ref{sec:CB_maps}, considering two algebras.

\begin{proposition}\label{prop:when_A_bdd_L2}
Let $f,g\in \mf n_{\id\otimes\varphi} \subseteq N\vnten M$ and consider $A = A_{f,g} \colon M^\op \to N$.  Suppose that $(\psi\otimes\id)(g^*g) < \infty$.  For each $x\in \mf n_{\varphi^\op}$ we have that $A(x) \in \mf n_\psi$, and the resulting map $A_0 \colon \Lambda^\op(x) \mapsto \Lambda_\psi(A(x))$ is in $\mc B(L^2(M^\op), L^2(N))$.
\end{proposition}
\begin{proof}
Let $(c_i) \subseteq N$ be the net given by Proposition~\ref{prop:tomita_nice_bai}, applied to $\psi$.  
For $b\in\mf n_{\varphi}^*$ we have $A(b^\op) = \Lambda_{\id\otimes\varphi}(f)^* (1\otimes\pi^\op(b^\op)) \Lambda_{\id\otimes\varphi}(g)$, and so, setting $K^2 = \|\Lambda_{\id\otimes\varphi}(f)\|^2$, we have
\begin{align*}
\psi(A(b^\op)^* A(b^\op))
&= \psi\big( \Lambda_{\id\otimes\varphi}(g)^* (1\otimes\pi^\op(b^{*\op})) \Lambda_{\id\otimes\varphi}(f) \Lambda_{\id\otimes\varphi}(f)^* (1\otimes\pi^\op(b^\op)) \Lambda_{\id\otimes\varphi}(g) \big) \\
&\leq K^2 \psi( \Lambda_{\id\otimes\varphi}(g)^*(1\otimes\pi^\op(b^{*\op}b^\op)) \Lambda_{\id\otimes\varphi}(g) ) \\
&= K^2 \lim_i \psi( c_i^*\Lambda_{\id\otimes\varphi}(g)^*(1\otimes\pi^\op(b^{*\op}b^\op)) \Lambda_{\id\otimes\varphi}(g) c_i ) \\
&= K^2 \lim_i \| (1\otimes\pi^\op(b^\op))\Lambda_{\id\otimes\varphi}(g)\Lambda_\psi(c_i) \|^2
\end{align*}
By Lemma~\ref{lem:slice_f_moved_out}, this is equal to
\begin{align*}
&= K^2 \lim_i \| g (\Lambda_\psi(c_i) \otimes J \Lambda(b^{*}) ) \|^2
= K^2 \lim_i \psi(c_i^* (\id\otimes\omega)(g^*g) c_i),
\end{align*}
where we set $\omega = \omega_{J\Lambda(b^*)}$.  By Proposition~\ref{prop:motivation_approx_weight}, this in turn equals
\[ K^2 \psi((\id\otimes\omega)(g^*g))
= K^2 (\psi\otimes\id)(g^*g)(\omega)
\leq K^2 \|(\psi\otimes\id)(g^*g)\| \|J\Lambda(b^*)\|^2. \]
Using Remark~\ref{rem:other_Mop_GNS}, we identify $J\Lambda(b^*)$ with $\Lambda^\op(b^\op)$, and so this shows that $A(b^\op) \in \mf n_{\psi}$ and that the map $\Lambda^\op(b^\op) \mapsto \Lambda_\psi(A(b^\op))$ is bounded.  The existence of $A_0$ follows.
\end{proof}

The argument is nearly reversible, except for the first inequality.  We don't know if there is some sort of converse.

Our aim is ultimately a statement generalising Proposition~\ref{prop:S_is_bimod_A}, though for the moment we continue working in the more general situation of two algebras.  We begin with some technical results; recall the discussion about Tomita algebras in Appendix~\ref{sec:weights}.

\begin{proposition}\label{prop:int_coint_HS}
Let $f,g$ as above, and form $A=A_{f,g}$ and $A_0$.  For $a\in \mf n_{\varphi^\op}$ we have that $u = A_0 J^\op a^\op J^\op \nabla^{\op -1/2}$ is a Hilbert--Schmidt operator with $\|u\|_{HS} \leq K \|\Lambda^\op(a^\op)\|$ for some constant $K$.
\end{proposition}
\begin{proof}
Let $a^\op,b^\op$ be such that $\Lambda^\op(a^\op), \Lambda^\op(b^\op)$ are in the Tomita algebra of $\varphi^\op$ (which is equivalent to $\Lambda(a),\Lambda(b)$ being in the Tomita algebra for $\varphi$).  Then
\begin{align*}
u \Lambda^\op(b^\op)
&= A_0 \Lambda^\op( \sigma_{i/2}^\op(b^\op) \sigma_{-i/2}^\op(a^{\op*}) )
= \Lambda_\psi\big( A_{f,g}( \sigma_{-i/2}(b)^\op \sigma_{i/2}(a^{*})^\op ) \big) \\
&= \Lambda_\psi\big( (\id\otimes\omega_{ J\Lambda(\sigma_{-i/2}(b)) , J\Lambda(\sigma_{-i/2}(a)) })(f^*g) \big)   \\
&= \Lambda_\psi\big( (\id\otimes\omega_{ \Lambda(b^*) , \Lambda(a^*) })(f^*g) \big),
\end{align*}
again using the calculation in Corollary~\ref{corr:Afg_only_fstarg}, and Lemma~\ref{lem:form_J}.  With Lemma~\ref{lem:L2Mop_L2M} in mind, we now see that
\begin{align*}
\|u \Lambda^\op(b^\op)  &  \|^2
= \psi\big( (\id\otimes\omega_{ \Lambda(a^*) , \Lambda(b^*) })(g^*f) (\id\otimes\omega_{ \Lambda(b^*) , \Lambda(a^*) })(f^*g) \big) \\
&\leq \|\Lambda(b^*)\|^2 \|f\|^2 \psi\big( (\id\otimes\omega_{ \Lambda(a^*) , \Lambda(a^*) })(g^*g) \big)
= \|\Lambda^\op(b^\op)\|^2 \|f\|^2 \|(\psi\otimes\id)(g^*g)\| \|\Lambda^\op(a^\op)\|^2.
\end{align*}
So $u$ is a bounded map $L^2(M^\op) \to L^2(N)$.  We now use Lemma~\ref{lem:L2Mop_L2M} to identify $L^2(M^\op)$ with $\overline{L^2(M)}$.  Then continuity of $\alpha$ shows that
\[ u \overline\xi = \Lambda_\psi\big( (\id\otimes\omega_{ \xi , \Lambda(a^*) })(f^*g) \big) \qquad (\xi\in L^2(M)). \]
Let $(e_i)$ be an orthonormal basis for $L^2(M)$, so adapting the previous calculation gives
\begin{align*}
\sum_i \|u(e_i)\|^2
&= \sum_i \psi\big( (\id\otimes\omega_{ \Lambda(a^*) , \Lambda(a^*) })(g^*f(1\otimes\rankone{e_i}{e_i})f^*g) \big)
= \psi\big( (\id\otimes\omega_{ \Lambda(a^*) , \Lambda(a^*) })(g^*ff^*g) \big) \\
&\leq \|\Lambda(a^*)\|^2 \|f\|^2 \|(\psi\otimes\id)(g^*g)\|
\end{align*}
using here that $\psi$ is normal.  So $u$ is a Hilbert--Schmidt operator with $\|u\|_{HS} \leq K \|\Lambda^\op(a^\op)\|$ for some constant $K$ depending on $f,g$.

Now let $a\in\mf n_{\varphi}^*$ (that is, $a^\op\in\mf n_{\varphi^\op}$), so by \cite[Theorem~VI.1.26]{TakesakiII} there is a sequence $(a_n)$ with $\Lambda(a_n)$ in the Tomita-algebra and $\Lambda^\op(a_n) \to \Lambda^\op(a)$ and $a_n^\op \to a^\op$ strongly.  Form $u_n$ from $a_n$, and $u$ from $a$.  Then $\|u_n\|_{HS} \leq K \|\Lambda^\op(a_n^\op)\| \to K \|\Lambda^\op(a^\op)\|$ so $(u_n)$ is a bounded sequence in the Hilbert space $HS(L^2(M^\op), L^2(N))$.  With $b$ as before,
\[ \|(u_n - u) \Lambda^\op(b^\op)\| = \|\Lambda_\psi\big( (\id\otimes\omega_{ \Lambda(b^*) , \Lambda(a_n^*-a^*) })(f^*g) \big)\|
\leq K \|\Lambda^\op(b^\op)\| \|\Lambda^\op(a_n^\op-a^\op)\| \to 0. \]
So $u_n \to u$ strongly, as $(u_n)$ is also bounded in operator norm.  There is a subset $u_{n(i)}$ say such that $u_{n(i)} \to v$ weakly in the Hilbert space $HS(L^2(M^\op), L^2(N))$ for some Hilbert--Schmidt operator $v$.  Then
\[ (\xi|v(\eta)) = ( \rankone{\xi}{\eta} | v )_{HS} = \lim_i ( \rankone{\xi}{\eta} | u_{n(i)} )_{HS}
= \lim_i (\xi | u_{n(i)}(\eta) ) = \lim_n (\xi | u_n(\eta)) = (\xi | u(\eta) ), \]
for each $\xi,\eta$.  Hence $u = v$, showing that $u$ is Hilbert--Schmidt, with $\|u\|_{HS} =\|v\|_{HS}\leq K \|\Lambda^\op(a^\op)\|$.
\end{proof}

\begin{lemma}\label{lem:action_gives_HS_image}
We continue with the same hypotheses.  Consider $u$ as a member of $L^2(N) \otimes \overline{L^2(M^\op)} \cong L^2(N) \otimes L^2(M)$.  For $x\in D(\sigma^\psi_{i/2}) \cap D(\sigma^\psi_{-i/2}) \cap \mf n_\psi^*$ with $\sigma^\psi_{-i/2}(x)\in\mf n_\psi$, we have $(J_\psi x^*J_\psi\otimes 1)u = f^*g(J_\psi\Lambda_\psi(x^*)\otimes\Lambda(a^*)) \in L^2(N) \otimes L^2(M)$.
\end{lemma}
\begin{proof}
Let $v = (J_\psi x^*J_\psi\otimes 1)u$ so as an operator $L^2(M^\op)\to L^2(N)$ we have $v = J_\psi x^*J_\psi A_0 J^\op a^\op J^\op \nabla^{\op -1/2}$ which is also Hilbert--Schmidt.  For $c\in \mf n_{\psi}$ and $b\in \mf n_\varphi^*$,
\begin{align*}
\big( \Lambda_\psi(c) \otimes \overline{\Lambda^\op(b^\op)} \big| v \big)_{HS}
&= (\Lambda_\psi(c) | J_\psi x^*J_\psi u \Lambda^\op(b^\op) )
= ( \Lambda_\psi(c\sigma^\psi_{-i/2}(x^*)) | u \Lambda^\op(b^\op) ) \\
&= \psi\big( \sigma^\psi_{i/2}(x) c^* (\id\otimes\omega_{ \Lambda(b^*) , \Lambda(a^*) })(f^*g) \big)
= \psi\big( c^* (\id\otimes\omega_{ \Lambda(b^*) , \Lambda(a^*) })(f^*g) \sigma^\psi_{-i/2}(x) \big) \\
&= \big( \Lambda_\psi(c) \otimes \Lambda(b^*) \big| f^*g(J_\psi\Lambda_\psi(x^*)\otimes\Lambda(a^*)) \big).
\end{align*}
Here we used Lemma~\ref{lem:action_J}, Lemma~\ref{lem:sigmaiswap}, and Lemma~\ref{lem:form_J}.  Once again identifying $L^2(M^\op) \cong \overline{L^2(M)}$, the result follows.
\end{proof}

Suppose that $f,g \in \mf n_{\id\otimes\varphi} \cap \mf n_{\psi\otimes\id}$ so we can form $A_{f,g} \colon M^\op \to N$ and also $A_{\tau(g), \tau(f)} \colon N^\op \to M$, where here $\tau$ is the tensor swap map.  By Proposition~\ref{prop:when_A_bdd_L2}, we have $A_0 \colon L^2(M^\op) \to L^2(N)$ associated to $A_{f,g}$ as well as $B_0 \colon L^2(N^\op) \to L^2(M)$ associated to $A_{\tau(g), \tau(f)}$.

\begin{proposition}\label{prop:swap_Hilbert_op}
Identify $L^2(M)$ with $L^2(M^\op)$, see Remark~\ref{rem:other_Mop_GNS}, and similarly $L^2(N) \cong L^2(N^\op)$.  Then $B_0 = A_0^*$.
\end{proposition}
\begin{proof}
We shall write $(\sigma_t)$ for both the modular automorphism groups of $\varphi$ and $\psi$, it being clear which one from context.  Let $b_0 \in \mf n_\psi$ with $\Lambda_\psi(b_0)$ in the Tomita algebra for $\psi$, and similarly $b_1$, and similarly $a_0, a_1$ for $\varphi$.  Using Lemma~\ref{lem:form_J}, and as in Corollary~\ref{corr:Afg_only_fstarg}, we have
\begin{align*}
& \big(\Lambda(a_1) \big| A_{\tau(g), \tau(f)}(b_1^{*\op} b_0^\op) \Lambda(a_0)\big)
= \big( \tau(g) (\Lambda(a_1) \otimes J\Lambda_\psi(b_1^*)) \big| 
   \tau(f) (\Lambda(a_0) \otimes J\Lambda_\psi(b_0^*)) \big) \\
&= \big( f(\Lambda_\psi(\sigma_{-i/2}(b_0)) \otimes J\Lambda(\sigma_{-i/2}(a_0^*))) \big| g(\Lambda_\psi(\sigma_{-i/2}(b_1)) \otimes J\Lambda(\sigma_{-i/2}(a_1^*))) \big) \overline{\phantom{d}} \\
&= \big( \Lambda_\psi(\sigma_{-i/2}(b_0)) \big| A_{f,g}( \sigma_{-i/2}(a_0^*)^{\op} \sigma_{-i/2}(a_1^*)^{*\op} ) \Lambda_\psi(\sigma_{-i/2}(b_1)) \big) \overline{\phantom{d}} \\
&= \big( \Lambda_\psi(\sigma_{-i/2}(b_1)) \big| A_{f,g}( \sigma_{-i/2}(a_0^*)^{\op} \sigma_{i/2}(a_1)^{\op} )^* \Lambda_\psi(\sigma_{-i/2}(b_0)) \big).
\end{align*}
Set $a = \sigma_{i/2}(a_1) \sigma_{-i/2}(a_0^*)$ and $b=b_0 b_1^*$, so using Lemma~\ref{lem:mod_aut_usage}, we continue the calculation as
\begin{align*}
= \psi\big( \sigma_{-i/2}(b_1)^* A_{f,g}(a^\op)^* \sigma_{-i/2}(b_0) \big)
&= \psi\big( A_{f,g}(a^\op)^* \sigma_{-i/2}(b_0) \sigma_{-i/2}(b_1^*) \big)
= \big( A_0 \Lambda^\op(a^\op) \big| \Lambda_\psi(\sigma_{-i/2}(b)) \big) \\
&= \big( \Lambda^\op(a^\op) \big| A_0^*\Lambda_\psi(\sigma_{-i/2}(b)) \big).
\end{align*}
As in Remark~\ref{rem:other_Mop_GNS}, we use that $\Lambda^\op(x^\op) = J\Lambda(x^*)$ for $x\in\mf n_\varphi^*$, and so $\Lambda_\psi(\sigma_{-i/2}(b)) = \nabla_\psi^{1/2} \Lambda_\psi(b) = J\Lambda_\psi(b^*) = \Lambda_\psi^\op(b^\op)$.  Similarly, $\Lambda^\op(a^\op) = \Lambda(\sigma_{-i/2}(a)) = \Lambda(a_1 \sigma_{-i}(a_0^*))$.  By the definition of $B_0$, we also have
\begin{align*}
& \big(\Lambda(a_1) \big| A_{\tau(g), \tau(f)}(b_1^{*\op} b_0^\op) \Lambda(a_0)\big)
= \varphi\big( a_1^* A_{\tau(g), \tau(f)}(b^\op) a_0 \big)
= \varphi\big( \sigma_i(a_0) a_1^* A_{\tau(g), \tau(f)}(b^\op) \big) \\
&= \big( \Lambda(a_1 \sigma_i(a_0)^*) \big| B_0 \Lambda_\psi^\op(b^\op) \big)
= \big( \Lambda(a_1 \sigma_i(a_0)^*) \big| A_0^*\Lambda_\psi^\op(b^\op) \big),
\end{align*}
the final equality coming from the calculations above.  Notice that we can apply Lemma~\ref{lem:mod_aut_usage} as $\Lambda(a_0)$ is in the Tomita algebra, and $a_1^* A_{\tau(g), \tau(f)}(b^\op) \in \mf n_\varphi \cap \mf n_\varphi^*$, because $a_1\in\mf n_\varphi$ and $A_{\tau(g), \tau(f)}(b^\op)\in \mf n_\varphi$, by Proposition~\ref{prop:when_A_bdd_L2}.  By density of such $a,b$, we conclude that $A_0^* = B_0$, as claimed.
\end{proof}

We now apply these results to a projection $e\in M\vnten M^\op$.  If we have both $e \in \mf n_{\id\otimes\varphi^\op}$ and $e\in\mf n_{\varphi\otimes\id}$ then we can form the associated $A\in\mc{CB}^\sigma(M)$, and then $A$ induces a bounded $A_0 \in \mc B(L^2(M))$.
Proposition~\ref{prop:swap_Hilbert_op} now seems strange in light of Proposition~\ref{prop:swap_e_KMS_ad_A}, as we would expect the KMS adjoint to appear, not the usual adjoint.  However, we need to be careful about the use of $M$ vs $M^\op$, and so forth.  Indeed, let $e \in M \vnten M^\op$ be a projection with both $e \in \mf n_{\id\otimes\varphi^\op}$ and $e\in\mf n_{\varphi\otimes\id}$.  Set $A = A_{e,e} \colon M \to M$ with associated $A_0\in\mc B(L^2(M))$.  Then $\tau(e) \in M^\op\vnten M$ and so $A_{\tau(e), \tau(e)} \colon M^\op \to M^\op$ with associated $B_0 \in \mc B(L^2(M^\op))$.  That is, $A_0$ and $B_0$ actually act on different spaces.  The resolution to this problem is to notice that in Section~\ref{sec:op_sys_A}, and indeed always in the finite-dimensional setting, we regard $\tau$ as an anti-homomorphism, and so think of $\tau(e) \in M \vnten M^\op$.  It would be better to denote this by $\tau(e)^\op$, which we now do.

\begin{proposition}\label{prop:int_and_coint}
Let $e$ be a projection in $\mf n_{\id\otimes\varphi^\op} \cap \mf n_{\varphi\otimes\id}$.  Let $\tau(e)^\op \in M\vnten M^\op$ give rise to the CP map $A_\tau = A_{\tau(e)^\op, \tau(e)^\op} \colon M\to M$.  The associated operator on $L^2(M)$ is $JA_0^*J$.
\end{proposition}
\begin{proof}
By Proposition~\ref{prop:Abdd_implies_f_int}, for $b_0, b_1\in\mf n_\varphi^*$ we have
\[ A_{\tau(e),\tau(e)}(b_1^{*\op} b_0^\op) = (\omega_{J\Lambda(b_1^*), J\Lambda(b_0^*)}\otimes \id)(e) \in M^\op. \]
We again use Remark~\ref{rem:other_Mop_GNS} to identify $L^2(M^\op)$ with $L^2(M)$ via $\Lambda^\op(a^\op) = J\Lambda(a^*)$ for $a\in\mf n_\varphi^*$, which identifies $x^\op \in \mc B(L^2(M^\op))$ with $Jx^*J \in M' \subseteq \mc B(L^2(M))$.  So for $b_0, b_1\in \mf n_\varphi$, and $f = x\otimes y^\op\in M\vnten M^\op$, we have
\begin{align*}
& (\id\otimes\omega_{J^\op\Lambda^\op(b_1^{\op*}), J^\op\Lambda^\op(b_0^{\op*})})(\tau(f)^\op)
= \ip{x^\op}{\omega_{J^\op\Lambda^\op(b_1^{\op*}), J^\op\Lambda^\op(b_0^{\op*})}} y
= \ip{Jx^*J}{\omega_{\Lambda(b_1), \Lambda(b_0)}} y \\
&= (\Lambda(b_1)|Jx^*J\Lambda(b_0)) y
= (J\Lambda(b_0)|x J\Lambda(b_1)) y
= (\omega_{J\Lambda(b_0), J\Lambda(b_1)}\otimes\id)(f)^\op.
\end{align*}
It hence follows that
\[ A_\tau(b_1^*b_0) = (\id\otimes\omega_{J^\op\Lambda^\op(b_1^{\op*}), J^\op\Lambda^\op(b_0^{\op*})})(\tau(e)^\op) = (\omega_{J\Lambda(b_0), J\Lambda(b_1)}\otimes\id)(e)^\op
= A_{\tau(e),\tau(e)}(b_0^\op b_1^{*\op})^\op. \]
Consequently,
\begin{align*}
\Lambda\big( A_\tau(b_1^*b_0) \big)
&= \Lambda\big( A_{\tau(e),\tau(e)}((b_1^*b_0)^\op)^\op \big)
= J^\op\Lambda^\op\big( A_{\tau(e),\tau(e)}((b_1^*b_0)^\op)^* \big) \\
&= J^\op\Lambda^\op\big( A_{\tau(e),\tau(e)}((b_0^*b_1)^\op) \big)
= J^\op B_0 \Lambda^\op( (b_0^*b_1)^\op )
= J B_0 J \Lambda(b_1^* b_0)
\end{align*}
using that $J^\op=J$ under our identifications, and that our maps are CP, hence $*$-maps.  By the previous proposition, $B_0 = A_0^*$, and so the bounded operator associated with $A_\tau$ is $JA_0^*J$.
\end{proof}

Once we remember that $A$ is ``real'', this resolves the apparent conflict with Proposition~\ref{prop:swap_e_KMS_ad_A}.

\begin{lemma}\label{lem:real_J_ad_KMS_ad}
Let $A\colon M\to M$ be a ``real'' map, $A(x^*) = A(x)^*$ for $x\in M$, and suppose that $A(\mf n_\varphi) \subseteq \mf n_\varphi$ with the resulting map $A_0 \colon \Lambda(x) \mapsto \Lambda(A(x))$ being bounded.  Then $JA_0^*J = \nabla^{-1/2} A_0^* \nabla^{1/2}$ in the sense that $D(\nabla^{-1/2} A_0^* \nabla^{1/2}) = D(\nabla^{1/2})$ and $\nabla^{-1/2} A_0^* \nabla^{1/2}$ agrees with $JA_0^*J$ on $D(\nabla^{1/2})$.
\end{lemma}
\begin{proof}
Let $F$ be the conjugate adjoint to the map $S$, as in \cite[Lemma~VI1.5]{TakesakiII}.  Then $\eta\in D(F)$ if and only if $D(S)\ni\xi\mapsto (\eta|S\xi)$ is bounded, and in this case, $(F\eta|\xi) = (S\xi|\eta)$ for each $\xi\in D(S)$.  Let $a\in\mf n_\varphi \cap \mf n_\varphi^*$ so $\Lambda(a)\in D(S)$, and such elements form a core for $S$.  Then $A_0S\Lambda(a) = A_0\Lambda(a^*) = \Lambda(A(a^*)) = \Lambda(A(a)^*) = S \Lambda(A(a)) = SA_0\Lambda(a)$.  For $\xi\in D(S)$ there is a sequence $(a_n)$ with $\Lambda(a_n)\to\xi, \Lambda(a_n^*)\to S\xi$, so $A_0S\xi = \lim_n A_0S\Lambda(a_n) = \lim_n SA_0\Lambda(a_n)$ and $A_0\Lambda(a_n) \to A_0\xi$, showing that $A_0\xi\in D(S)$ with $SA_0\xi = A_0S\xi$, as $S$ is closed.

For $\eta\in D(F)$ we claim that $A_0^*\eta\in D(F)$.  Indeed, for $\xi\in D(S)$ we have $(A_0^*\eta|S\xi) = (\eta|SA_0\xi) = (A_0\xi|F\eta)$ and so $|(A_0^*\eta|S\xi)| \leq \|F\eta\| \|A_0\| \|\xi\|$.  So $A_0^*\eta\in D(F)$ and $(FA_0^*\eta|\xi) = (S\xi|A_0^*\eta) = (SA_0\xi|\eta) = (F\eta|A_0\xi) = (A_0^*F\eta|\xi)$ for each $\xi\in D(S)$, showing that $FA_0^*\eta = A_0^*F\eta$.  As $S^2=1$ also $F^2=1$ and so $A_0^*F\eta = FA_0^*\eta\in D(F)$ with $FA_0^*F\eta = A_0^*\eta$.

Finally, using that $JF = \nabla^{-1/2}, FJ = \nabla^{1/2}$ and $D(\nabla^{-1/2}) = D(F)$, we see that for $\xi\in D(\nabla^{1/2})$ we have $J\xi\in D(F)$ and hence $A_0^*J\xi = F A_0^* F J\xi = FA_0^*\nabla^{1/2}\xi$.  So also $JA_0^*J\xi = \nabla^{-1/2}A_0^*\nabla^{1/2}\xi$, showing also that $A_0^*\nabla^{1/2}\xi \in D(\nabla^{-1/2})$.
\end{proof}

\subsection{From Hilbert space operators to subspaces}\label{sec:considerations_of_S}

In this section we explore generalisations of Proposition~\ref{prop:S_is_bimod_A} and the ideas of Section~\ref{sec:op_sys_A}, mostly in the setting of when $e \in \mf n_{\id\otimes\varphi^\op} \cap \mf n_{\varphi\otimes\id}$ so that there is a Hilbert space form of $A = A_{e,e}$, say $A_0$.  The following shows that $A_0$ generates $V$ in a certain sense.

\begin{theorem}\label{thm:e_int_co_int_A0_generates_S}
Let $e\in M\vnten M^\op$ be a projection in $\mf n_{\id\otimes\varphi^\op} \cap \mf n_{\varphi\otimes\id}$.  Form $A = A_{e,e}\colon M\to M$ and the associated $A_0\in\mc B(L^2(M))$, and let $V \subseteq HS(H)$ be the image of $e$.  Then $A_0$ is a member of $S = \{ \alpha\in\mf n_{\varphi^{-1}} : \im\hat\alpha \subseteq V \}$.  Furthermore, $\{ \im\hat\alpha : x,y\in M', \alpha = x A_0 y \}$ is dense in $V$, so $\lin \{ \hat\alpha : x,y\in M', \alpha = x A_0 y \}$ is weak$^*$-dense in $e E_{\varphi^{-1}} = \mc B_{M'}(H,V)$.
\end{theorem}
\begin{proof}
Let $\alpha = A_0$.  By Proposition~\ref{prop:int_coint_HS}, for $a\in\mf n_{\varphi}$ we have $\hat\alpha J\Lambda(a) = A_0 JaJ \nabla^{-1/2} \in HS(H)$ with $\|\hat\alpha J\Lambda(a)\|_{HS} \leq K \|\Lambda(a)\|$ for some constant $K$ (depending on $e$).  So $\alpha\in\mf n_{\varphi^{-1}}$.  By Lemma~\ref{lem:action_gives_HS_image}, for $x\in\mf n_\varphi$ with $\Lambda(x)$ in the Tomita algebra, say, we have $(Jx^*J\otimes 1) \hat\alpha J\Lambda(a) = e(J\Lambda(x^*)\otimes\overline{\Lambda(a)}) \in V$.  This applies in particular when $\Lambda(x)$ is in the Tomita algebra, so we can let $x$ run through an approximate identity (compare with the proof of Proposition~\ref{prop:tomita_nice_bai} for example) to see that $\im\hat\alpha \subseteq V$.  So $\alpha\in S$.

By Proposition~\ref{prop:V_to_S}, $S$ is an $M'$-bimodule, indeed, by Lemma~\ref{lem:alpha_hat_bimod}, for $x,y\in M'$ we have that $(x A_0 y)\floatinghat = x \wh{A_0} y$, and so if $\alpha = x A_0 y$, certainly $\im\hat\alpha \subseteq V$.  Again, for $\Lambda(c)$ in the Tomita algebra, set $x = Jc^*J \in M'$, so with $a\in\mf n_\varphi$, we have
\begin{equation}
x \wh{A_0} y J\Lambda(a) = (Jc^*J\otimes 1) A_0 y JaJ \nabla^{-1/2} = e (J\Lambda(c^*) \otimes \overline{\Lambda(JyJa)}),
\label{eq:action_xA_0y}
\end{equation}
noting that $JyJ\in M$.  Taking $y=1$ and letting $c$ vary shows that we obtain a dense subspace of $V$, so already $\{ \im\hat\alpha : x\in M', \alpha = x A_0 \}$ is dense in $V$.  The claim about weak$^*$-density is Proposition~\ref{prop:wstar_dense_nondeg}.
\end{proof}

We now explore an analogue of Section~\ref{sec:op_sys_A}.  Recall that we denote by $J_0$ the anti-linear tensor swap map on $H\otimes\overline H$; when this is identified with $HS(H)$ we have that $J_0(x)=x^*$.  By Proposition~\ref{prop:mod_Trh}, the modular automorphism group of $\tilde\varphi = \Tr_{\nabla^{-1}}$ is $\tilde\sigma_t(x) = \nabla^{-it} x \nabla^{it}$ for $x\in\mc B(H)$.  As discussed before Lemma~\ref{lem:sigma_nabla_commutation}, for $t\in\mathbb R$, $x \in D(\tilde\sigma_{it})$ if and only if $D(\nabla^t x \nabla^{-t}) = D(\nabla^{-t})$ and $\nabla^t x \nabla^{-t}$ is a bounded operator.

\begin{proposition}\label{prop:flip_to_twisted_adj}
Let $e\in M\vnten M^\op$ be a projection in $\mf n_{\id\otimes\varphi^\op} \cap \mf n_{\varphi\otimes\id}$.  Let $V$ be the image of $e$, and $V_\tau = J_0(V)$ the image of $\tau(e)^\op$.  Let $S, S_\tau \subseteq \mf n_{\varphi^{-1}}$ be associated to $V$ and $V_\tau$ respectively.  Then
\[ \{ \nabla^{-1/2} \alpha^* \nabla^{1/2} : \alpha \in S \cap D(\tilde\sigma_{i/2})\}  \cap \mf n_{\varphi^{-1}} \]
is a subset of $S_{\tau}$ which is weak$^*$-dense in $\mc B_{M'}(H, V_\tau)$.
\end{proposition}
\begin{proof}
Let $\alpha \in S \cap D(\tilde\sigma_{i/2})$ so $\alpha^* \in \tilde\sigma_{-i/2}$, and hence $\nabla^{-1/2} \alpha^* \nabla^{1/2}$ is a bounded operator.  Let $\beta = \nabla^{-1/2} \alpha^* \nabla^{1/2}$, so for $a,b\in\mf n_{\varphi} \cap D(\sigma_{i/2})$ with $\sigma_{i/2}(b)^*\in\mf n_\varphi$, we have
\begin{align*}
JbJ \beta JaJ \nabla^{-1/2}
&= JbJ \nabla^{-1/2} \alpha^* \nabla^{1/2} JaJ \nabla^{-1/2}
= Jb \nabla^{1/2}J \alpha^* J \nabla^{-1/2} a\nabla^{1/2} J 
= \nabla^{-1/2} J \sigma_{i/2}(b) J \alpha^* J \sigma_{i/2}(a) J \\
&= \big( \hat\alpha J\Lambda(\sigma_{i/2}(b)^*) \big)^* J \sigma_{i/2}(a) J.
\end{align*}
Here we used Lemma~\ref{lem:domain_S_Tr_h} to see that $\nabla^{-1/2} x = (x^*\nabla^{-1/2})^*$ for any $x$.
As $\hat\alpha$ maps into $V$, we have $( \hat\alpha J\Lambda(\sigma_{i/2}(b)^*) )^* \in V_\tau$, which is an $M'$-bimodule, and so $JbJ \beta JaJ \nabla^{-1/2} \in V_\tau$.  Letting $b$ run through an approximate identity, compare Proposition~\ref{prop:tomita_nice_bai}, we conclude that $\beta JaJ \nabla^{-1/2} \in V_\tau$.  Assuming that $\beta\in\mf n_{\varphi^{-1}}$, so $\hat\beta$ is a bounded map, by density of such $a$, it follows that $\beta \in S_\tau$.

Let $x,y\in M'$ and $\alpha = x A_0 y \in S$, as in Theorem~\ref{thm:e_int_co_int_A0_generates_S}.  Let $\xi\in D(\nabla^{1/2})$, and suppose that actually $x,y\in D(\sigma'_{i/2})$, so by Proposition~\ref{prop:int_and_coint} and Lemma~\ref{lem:real_J_ad_KMS_ad}, we have
\[ \nabla^{-1/2} \alpha^* \nabla^{1/2}\xi
= \nabla^{-1/2} y^* A_0^* x^* \nabla^{1/2} \sigma'_{-i/2}(y^*) \nabla^{-1/2} A_0^* \nabla^{1/2} \sigma'_{-i/2}(x^*) \xi
= \sigma'_{-i/2}(y^*) A_{\tau,0} \sigma'_{-i/2}(x^*)\xi \]
here also using Lemma~\ref{lem:com_weight} to see that $\nabla^{-1} = \nabla'$, and Lemma~\ref{lem:sigma_nabla_commutation}.  Hence $\nabla^{-1/2} \alpha^* \nabla^{1/2}$ is bounded, and a member of $S_\tau$.  The proof of Theorem~\ref{thm:e_int_co_int_A0_generates_S} shows that such elements of $S_\tau$ will be weak$^*$ dense in $\mc B_{M'}(H, V_\tau)$.
\end{proof}

Example~\ref{eg:BH_int_coint} shows that when $e$ is only integrable, we can have $S_\tau = \{0\}$ with $V$ (and hence $S$) non-zero.  Example~\ref{eg:S_to_V_not_nice_but_int} shows that in general, we do need to intersect with $\mf n_{\varphi^{-1}}$ in the above statement.

We now come to an analogue of Section~\ref{sec:diff}.  Define
\[ S_{i/4} = \{ \tilde\sigma_{i/4}(\alpha) : \alpha \in S \cap D(\tilde\sigma_{i/4}) \}
= \{ \nabla^{1/4} \alpha \nabla^{-1/4} : \alpha \in S \cap D(\tilde\sigma_{i/4}) \}. \]
In general, this can be $\{0\}$, see Example~\ref{eg:Si4}.  Even in the setting of the following result, $S_{i/4}$ might not have non-zero intersection with $\mf n_{\varphi^{-1}}$ (also Example~\ref{eg:Si4}) and so we think of $S_{i/4}$ as simply a subset of $\mc B(H)$.

\begin{proposition}\label{prop:Si4_when_int_coint}
When $e$, the projection onto $V$, is in $\mf n_{\id\otimes\varphi^\op} \cap \mf n_{\varphi\otimes\id}$, we have that $S_{i/4}$ is sufficiently large in the sense that
\[ \{ \nabla^{-1/4} \beta \nabla^{1/4} : \beta\in S_{i/4} \cap D(\tilde\sigma_{-i/4}) \} = S \cap D(\tilde\sigma_{i/4}) \]
is weak$^*$ dense in $\mc B_{M'}(H,V)$.
\end{proposition}
\begin{proof}
Again form $A = A_{e,e}$ and $A_0\in\mc B(L^2(M))$, so that $M' A_0 M' \subseteq S$ is weak$^*$-dense in $e E_{\varphi^{-1}} = \mc B_{M'}(H,V)$ by Theorem~\ref{thm:e_int_co_int_A0_generates_S}.  Furthermore, by Lemma~\ref{lem:real_J_ad_KMS_ad} we see that $A_0^* \in D(\tilde\sigma_{-i/2})$ so $A_0 \in D(\tilde\sigma_{i/2}) \subseteq D(\tilde\sigma_{i/4})$.  As in the proof of Proposition~\ref{prop:flip_to_twisted_adj}, for $x,y\in D(\sigma'_{i/4})$ we have
\[ \nabla^{1/4} xA_0y \nabla^{-1/4}
= \sigma'_{i/4}(x) \nabla^{1/4} A_0 \nabla^{-1/4} \sigma'_{i/4}(y)
= \sigma'_{i/4}(x) \tilde\sigma_{i/4}(A_0) \sigma'_{i/4}(y) \]
is a bounded operator, and similarly, $D(\nabla^{1/4} xA_0y \nabla^{-1/4}) = D(\nabla^{-1/4})$.  So $xA_0y \in S \cap D(\tilde\sigma_{i/4})$.

Given any $\alpha \in S \cap D(\tilde\sigma_{i/4})$ we have that $\nabla^{1/4} \alpha \nabla^{-1/4} = \tilde\sigma_{i/4}(\alpha) \in D(\tilde\sigma_{-i/4})$ with $\tilde\sigma_{-i/4}(\tilde\sigma_{i/4}(\alpha)) = \nabla^{-1/4} \tilde\sigma_{i/4}(\alpha) \nabla^{1/4}$.  So the displayed equation in the proposition statement, that $\{ \nabla^{-1/4} \beta \nabla^{1/4} : \beta\in S_{i/4} \cap D(\tilde\sigma_{-i/4}) \} = S \cap D(\tilde\sigma_{i/4})$, always holds.  We've just shown that $\{ xA_0y : x,y\in D(\sigma'_{i/4}) \} \subseteq S \cap D(\tilde\sigma_{i/4})$, so by Proposition~\ref{prop:wstar_dense_nondeg}, to finish the proof, it suffices to show that $\lin\{ (xA_0y)\floatinghat \xi : \xi\in H, x,y\in D(\sigma'_{i/4}) \}$ is dense in $V$.  By \eqref{eq:action_xA_0y} above, we have $(Jc^*J A_0 y)\floatinghat J\Lambda(a) = e(J\Lambda(c^*) \otimes \overline{\Lambda(JyJa)} )$ when $\Lambda(c)$ is in the Tomita algebra, which certainly implies that $Jc^*J \in D(\sigma'_{i/4})$.  Letting $y$ run through an approximate identity in $D(\sigma'_{i/4})$, we obtain a dense subset of $V$, as required.
\end{proof}

We finally obtain a sort of generalisation of Proposition~\ref{prop:swap_e_KMS_ad_A}: there is at least a dense subset of $S_{i/4}$ such that the adjoint operation corresponds to the tensor flip operator on $e$ and/or $V$.  Example~\ref{eg:Si4} shows that we cannot hope for much more, in general.

\begin{proposition}\label{prop:core_of_Si4}
Let $e \in \mf n_{\id\otimes\varphi^\op} \cap \mf n_{\varphi\otimes\id}$ and form $V,V_\tau$ and $S,S_\tau$, as well as $S_{i/4}$ and $S_{\tau,i/4}$.  The space
\[ S^{(0)} = \lin \{ xA_0y : x,y \in M' \text{ analytic for } (\sigma'_t) \} \subseteq S\cap D(\tilde\sigma_{i/4}) \]
is weak$^*$-dense in $\mc B_{M'}(H,V)$.  Let $S^{(0)}_{i/4} = \tilde\sigma_{i/4}(S^{(0)})$, and analogously define $S^{(0)}_{\tau,i/4}$.  Then $(S^{(0)}_{i/4})^* = S^{(0)}_{\tau,i/4}$.
\end{proposition}
\begin{proof}
The proof of Proposition~\ref{prop:Si4_when_int_coint} shows already that $S^{(0)}$ is weak$^*$-dense in $\mc B_{M'}(H,V)$ and a subspace of $S\cap D(\tilde\sigma_{i/4})$.  Furthermore, for $xA_0y \in S^{(0)}$ we have that $\tilde\sigma_{i/4}(xA_0y) = \sigma'_{i/4}(x) \tilde\sigma_{i/4}(A_0) \sigma'_{i/4}(y) \in S^{(0)}_{i/4}$.
Proposition~\ref{prop:int_and_coint} and Lemma~\ref{lem:real_J_ad_KMS_ad} again show that $A_{\tau,0} = \tilde\sigma_{-i/2}(A_0^*) = \tilde\sigma_{i/2}(A_0)^*$.  Let $x,y\in M'$ be analytic for $(\sigma'_t)$ and set $\alpha = \tilde\sigma_{i/4}(xA_0y) \in S^{(0)}_{i/4}$.  Then
\[ \alpha^* = \sigma'_{i/4}(y)^* \tilde\sigma_{i/4}(A_0)^* \sigma'_{i/4}(x)^*
= \sigma'_{i/4}(y)^* \tilde\sigma_{i/4}(\tilde\sigma_{i/2}(A_0)^*) \sigma'_{i/4}(x)^*
= \tilde\sigma_{i/4}\big( \sigma'_{i/2}(y)^* A_{\tau,0} \sigma'_{i/2}(x)^* \big). \]
As $x_1 = \sigma'_{i/2}(y)^*$ is analytic for $(\sigma'_t)$, and similarly $y_1 = \sigma'_{i/2}(x)^*$, this shows that $\alpha^* = \tilde\sigma_{i/4}(x_1 A_{\tau,0} y_1) \in S^{(0)}_{\tau,i/4}$.  The exact same argument shows that $(S^{(0)}_{\tau,i/4})^* \subseteq S^{(0)}_{i/4}$, and so we have equality.
\end{proof}

\subsection{\texorpdfstring{Weak$^*$}{Weak star}-closures}\label{sec:weakstar_closures}

We finally make some comments about weak$^*$-closures.  In Proposition~\ref{prop:V_to_S} we defined $S$ from $V$, namely $S = \mc B_{M'}(H,V) \cap \mf n_{\varphi^{-1}}$, which is a (not necessarily closed) $M'$-bimodule in $\mc B(H)$.  It would then be natural to define $\mc S$ to be the weak$^*$-closure.  Example~\ref{eg:V_can_give_zero_S} shows that we can have $V$ non-zero yet $S$ (and so $\mc S$) are zero.  Proposition~\ref{prop:S_in_case_varphiinv_bdd} shows that when $\varphi^{-1}$ is bounded, then $S = \mc S$, though Example~\ref{eg:V_can_give_zero_S_varphiinv_bdd} shows that even here, we can have $\mc S=\{0\}$ with $V$ non-zero.

In Section~\ref{sec:example_BH} below we study the theory as applied to $M=\mc B(K)$.  When we give $M$ its canonical trace, Example~\ref{eg:BH_get_HS_QRs} shows how the general theory exactly captures the idea of a Hilbert--Schmidt quantum relation, as studied in Section~\ref{sec:HS}.  In Example~\ref{eg:S_to_V_not_nice_but_int} we use this to transport a counter-example from Section~\ref{sec:HS}.  This gives an example of an integrable $e$ which is symmetric ($e = \tau(e)^\op$) with $\varphi$ a trace, and yet $\mc S \cap \mf n_{\varphi^{-1}}$ is strictly larger than $S$.  So even in the situation of Section~\ref{sec:considerations_of_S}, we can lose information by passing to the weak$^*$-closure $\mc S$.

We could similarly define $\mc S_{i/4}$ to be the weak$^*$-closure of $S_{i/4}$.  Here there is a choice: we might instead form $\mc S$ first and then let $\mc S_{i/4} = \{ \tilde\sigma_{i/4}(x) : x\in \mc S \cap D(\tilde\sigma_{i/4}) \}$.  This second definition would surely lose too much information?  Similarly, given an arbitrary weak$^*$-closed bimodule $\mc T$, we would want to pass to the $V$ which ``best approximated'' this, perhaps by looking at $\tilde\sigma_{-i/4}(\mc T \cap D(\tilde\sigma_{-i/4}))$ and then intersecting with $\mf n_{\varphi^{-1}}$ to form an $S$ before looking at the self-dual completion (effectively, computing $V = \overline\lin\{ \Im(\hat\alpha) : \alpha\in S \}$)?  Proposition~\ref{prop:core_of_Si4} shows that $\mc S_{i/4}$ will at least be non-trivial for well-behaved $e$, and that, in some sense, the adjoint applied to $\mc S_{i/4}$ recovers the swap map at the level of the subspace $V$.

Nevertheless, these are tentative results, and suggest that we are missing a key detail which might lead to a more satisfactory way to pass between quantum graphs in the sense of projections $e$ (and/or subspaces $V$) and/or adjacency operators $A$, and quantum graphs in the sense of Weaver.  For example, an analogue of Proposition~\ref{thm:nice_subclass} in the non-tracial situation would be a goal.  We think that studying further classes of examples would provide some intuition here.

\subsection{Completely bounded projections}\label{sec:cb_projs}

In Section~\ref{sec:adj_proj} we found it useful to consider the map $\theta_A$, an idempotent on $\mc B(H)$.  In the literature (e.g. see after \cite[Lemma~3.13]{Wasilewski_Quantum_Cayley}) it is sometimes mentioned that $\theta_A$ is completely bounded: this is of course obvious in the finite-dimensional setting.  In this section we'll see that this viewpoint is not, as far as we see it, terribly useful in the infinite-dimensional case.

Let $\theta \in \mc{CB}^\sigma(\mc B(H))$, a completely bounded, weak$^*$-continuous map on $\mc B(H)$.  As $\mc K(H)$, the compact operators, have $\mc K(H)^{**} = \mc B(H)$, we have that $\mc{CB}^\sigma(\mc B(H)) \cong \mc{CB}(\mc K(H), \mc B(H))$.  Indeed, $T\in\mc{CB}^\sigma(\mc B(H))$ restricts to a map in $\mc{CB}(\mc K(H), \mc B(H))$.  Conversely, given $S\in \mc{CB}(\mc K(H), \mc B(H))$, let $\kappa \colon \mc B(H)_* \to \mc B(H)_*^{**} = \mc B(H)^*$ be the canonical map from a Banach (Operator) space to its bidual, and set $T = \kappa^* \circ S^{**}$.  Then $\kappa, S$ are completely bounded, so also $\kappa^*\colon\mc B(H)^{**} \to \mc B(H)$ and $S^{**}\colon\mc B(H) \to \mc B(H)^{**}$ are, and hence so is $T$.  Routine computation shows that these maps are mutual inverses.

Given $M\subseteq\mc B(H)$ a von Neumann algebra, write ${}_{M'}\mc{CB}^\sigma_{M'}(\mc B(H)) \cong {}_{M'}\mc{CB}{}_{M'}(\mc K(H), \mc B(H))$ for the subspace (in fact, subalgebra) of $M'$-bimodule maps.  By \cite[Theorem~4.2]{BS_Dual_Haagerup_TP} any $\theta \in {}_{M'}\mc{CB}^\sigma_{M'}(\mc B(H))$ is of the form
\[ \theta(x) = \sum_{i\in I} a_i^* x b_i \quad (x\in\mc B(H)), \quad\text{where}\quad
(a_i)_{i\in I}, (b_i)_{i\in I}\subseteq M \text{ with }
\sum_{i\in I} a_i^*a_i, \sum_{i\in I} b_i^*b_i < \infty. \]
This has links with the \emph{weak$^*$--Haagerup tensor product} (or \emph{extended Haagerup tensor product}, \cite[Section~5]{ER_HopfConvAlgs}).  We quickly remark why $\theta$ is completely bounded: given $(a_i)_{i\in I}$ define $\alpha \colon H \to H \otimes \ell^2(I); \xi \mapsto \sum_i a_i(\xi)\otimes\delta_i$.  Then $\alpha^*\alpha = \sum_i a_i^*a_i$, and if we also form $\beta$ from $(b_i)$ then $\theta(x) = \alpha^*(x\otimes 1)\beta$ is completely bounded.  Obviously $\theta$ is an $M'$-bimodule map.

\begin{remark}
Many results in this area were first shown by Haagerup in unpublished work.  One such result is that when every normal state on $M'$ is a vector state (as is the case here with $H=L^2(M)$) then a bounded, normal, $M'$-bimodule map $\theta$ on $\mc B(H)$ is automatically completely bounded.  See \cite[Theorem~2.5]{EK_ModuleMaps} or \cite{Smith_CB_ModMaps} for proofs.
\end{remark}

\begin{example}
From Proposition~\ref{prop:thetaA}, in the finite-dimensional case, if $A = \sum_{i=1}^n \rankone{b_i}{a_i}$ then $\theta_A(x) = \sum_i b_i x a_i^*$.  This is of the form above, with a finite index set, though it is not clear how to estimate $\sum_i b_i b_i^*$ or $\sum_i a_i a_i^*$, except to say they are finite.
\end{example}

As $\theta_A$ should be a projection onto $\mc S$, in the same way that $e$ is a projection onto $V$, we consider Proposition~\ref{prop:V_to_S} again.  So $\varphi$ is a weight on $M \subseteq \mc B(H)$ with $H=L^2(\varphi)$, and we consider how the map $\mf n_{\varphi^{-1}} = E^0_{\varphi^{-1}} \to E_{\varphi^{-1}} = \mc B_{M'}(L^2(M'), HS(H))$ interacts with the left and right actions of $M$.

\begin{lemma}\label{lem:M_action_conj}
Let $\alpha\in\mf n_{\varphi^{-1}}\subseteq \mc B(H)$ and let $x,y\in M$ with $y\in D(\sigma_{i/2})$.  Then $\beta = x \circ \alpha \circ \sigma_{i/2}(y) \in \mf n_{\varphi^{-1}}$ and $\hat\beta = (x\otimes y^\top) \circ \hat\alpha$.
\end{lemma}
\begin{proof}
Let $a\in\mf n_{\varphi'}$ so $\hat\alpha\Lambda'(a) \in HS(H)$ is the closure of $\alpha a \nabla^{-1/2}$.  As $a\in M'$, we have then that $\beta a \nabla^{-1/2} = x \alpha a \sigma_{i/2}(y) \nabla^{-1/2}$ and then Lemma~\ref{lem:sigma_nabla_commutation} shows that this is $\subseteq x \alpha a \nabla^{-1/2} y \subseteq x(\hat\alpha\Lambda'(a))y$.  Hence $\beta a \nabla^{-1/2} \in HS(H)$ (that is, is bounded with continuous extension a Hilbert--Schmidt operator) and the calculation shows that $\hat\beta = (x\otimes y^\top)\hat\alpha$, because $x\otimes y^\top$ acting on $H\otimes\overline H$ is the same as $HS(H)\ni b \mapsto xby$.
\end{proof}

Let $e\in M\vnten M^\op$ be a projection onto $V$; we regard $e \in \mc L(E_{\varphi^{-1}})$.  Write $\Lambda_{\varphi^{-1}}$ for the map $\mf n_{\varphi^{-1}} \to E_{\varphi^{-1}} = \mc B_{M'}(H, HS(H)); \alpha\mapsto\hat\alpha$.  Formally, we seek a map $\theta$ with $\Lambda_{\varphi^{-1}} \circ \theta = e \circ \Lambda_{\varphi^{-1}}$, though it is unclear what the domain of this should be.  Here $\theta$ should be a bounded normal $M'$-bimodule map, which will hence be completely bounded, so there are families $(x_i), (y_i)$ in $M$ with $\sum x_i^*x_i, \sum y_i^*y_i < \infty$, and $\theta_A(x) = \sum_i x_i^* x y_i$.  Lemma~\ref{lem:M_action_conj} shows that the modular automorphism group is important.  However, in Section~\ref{sec:ellinfty} we show that even when $M=\ell^\infty$ is commutative and atomic, $\theta_A$ may not exist (even when $e$ is very well behaved).  In Examples~\ref{eg:no_theta_A} and~\ref{eg:no_theta_A_even_varphiinv_bdd} we show similar behaviour for $M=\mc B(K)$, even when $\varphi^{-1}$ is bounded, and Remark~\ref{rem:no_theta_A_matrix} notes that we will then have examples in the case of a direct sum of matrix algebras, in the infinite-dimensional setting.

\section{Examples}\label{sec:egs}

Throughout we have motivated our treatment of the infinite-dimensional case by reference to the finite-dimensional situation, Section~\ref{sec:fd_summary}.  Let $B$ be finite-dimensional.  The weight $\tilde\varphi = \varphi'\circ \varphi^{-1}$ on $\mc B(L^2(B))$ agrees with the inner-product defined in Section~\ref{sec:ip}, see the discussion after Proposition~\ref{prop:L2BS_is_HSH}.  Example~\ref{eg:fd1} explains how we have generalised the relations between $e, V$ and $\mc S$.  Example~\ref{eg:fd3} shows that the relation between integrable $e$ and quantum adjacency operators $A$ exactly generalises Section~\ref{sec:adj_proj}; similarly Section~\ref{sec:hilb_ops}.
Example~\ref{eg:fd2} looks at the links between the quantum adjacency operator $A$ and ``orthogonal bases'' of $V$; see also Remark~\ref{rem:same_as_W} below.

\subsection{The commutative atomic case}\label{sec:ellinfty}

We start by considering $M=\ell^\infty(I)$ for an arbitrary index set $I$.  A (faithful) weight $\varphi$ is of the form $\varphi(x) = \sum_i x_i \varphi_i$ for $x\in\ell^\infty(I)^+$, where $\varphi_i \in (0,\infty)$ for each $i$.  We identify $L^2(M)$ with $\ell^2(I)$ for the GNS map $\Lambda(x) = (\varphi_i^{1/2} x_i)$.  Let $(e_i)_{i\in I}$ be the minimal idempotents in $\ell^\infty(I)$, and let $(\delta_i)$ be the standard orthonormal basis of $\ell^2(I)$.

We identify $M\vnten M^\op$ with $\ell^\infty(I\times I)$, and identify $HS(H)$ with $\ell^2(I\times I)$ by using the basis $(\rankone{\delta_i}{\delta_j})_{i,j}$.  Then projections $e\in M\vnten M^\op$ biject with subspaces $V\subseteq \ell^2(I\times I)$ of the form $V = \overline\lin\{ \delta_{i,j} : i\sim j \}$ where $\sim$ is an arbitrary relation on $I$.  We use the notation that $[i\sim j]=1$ if $i\sim j$ is true, and $0$ otherwise.

Then $e$ is integrable, Definition~\ref{defn:integrable_e_to_cp_A}, exactly when $(\id\otimes\varphi^\op)(e) = ( \sum_{i\sim j} \varphi_j )_{i\in I}$ is in $\ell^\infty(I)$, i.e. there is $K>0$ so that $\sum_{i\sim j} \varphi_j \leq K$ for each $i$.  When $\varphi$ is the uniform measure, this explains the ``bounded degree'' terminology.  When $e$ is integrable, the associated quantum adjacency operator is $A(x) = (\sum_{i\sim j} \varphi_j x_j)_{i\in I}$.  Equivalently, $A(e_j) = ( [i\sim j] \varphi_j )_{i\in I}$, so $A$ is a ``weighted'' adjacency matrix (and the usual adjacency matrix, if $\varphi$ is the uniform measure).  Notice that if $e$ is not integrable, then $A$ is not bounded.

We compute $E_{\varphi^{-1}} = \mc B_{\ell^\infty(I)}(\ell^2(I), \ell^2(I\times I))$.  This consists of maps $t$ with $t e_i = (1\otimes e_i)t$ for each $i\in I$.  Such a $t$ is hence of the form $t(\delta_i) = t_i \otimes \delta_i$ for some family $(t_i)$ in $\ell^2(I)$, and then $t$ is bounded if and only if $t^*t = (\|t_i\|_2^2)_{i\in I} \in \ell^\infty(I)$.  Then $\alpha\in\mf n_{\varphi^{-1}}$ when the map $\hat\alpha \colon \ell^2(I) \ni \Lambda(e_i) = \varphi_i^{1/2}\delta_i \mapsto \alpha e_i \in HS(H)$ is bounded.  Identifying $\alpha$ with a matrix $(\alpha_{i,j})$, we find $t_i \otimes \delta_i = \hat\alpha(\delta_i) = \varphi_i^{-1/2} \sum_j \alpha_{j,i} \delta_j\otimes\delta_i$, so $\hat\alpha^*\hat\alpha = (\varphi_i^{-1} \sum_j |\alpha_{j,i}|^2 )_{i\in I}$.  Notice the inverse which occurs, explaining the notation $\varphi^{-1}$.

With reference to Proposition~\ref{prop:V_to_S}, consider $S = \{ \alpha\in\mf n_{\varphi^{-1}} : e\hat\alpha=\hat\alpha \}$.  Then a matrix unit $e_{i,j}$ is in $S$ if and only if $i\sim j$; more generally $\alpha\in S$ implies that $\alpha_{i,j}=0$ when $i\not\sim j$.  So, much as in Example~\ref{eg_nice_vn_1}, $\mc S = S\clos^{w^*}$ is exactly $\lin\{ e_{i,j} : i\sim j \}\clos^{w^*}$ the canonical weak$^*$-closed bimodule associated with $\sim$.  We also see that $\mc S \cap \mf n_{\varphi^{-1}} = S$ will be weak$^*$-dense in $E_{\varphi^{-1}}$.  So we obtain a bijection between $\mc S, V$ and $e$.

Finally, we consider the possibility of a CB map $\theta_A$, Section~\ref{sec:cb_projs}.  We have that $\theta_A$ is of the form $x \mapsto \sum_i a_k^* x b_k$ for families $(a_k), (b_k)$ in $\ell^\infty$ with $\sum_k |a_k|^2 < \infty$ and for $(b_k)$.  Then
\[ \theta_A(e_{i,j}) = \sum_k \overline{a^{(k)}_i} b^{(k)}_j e_{i,j} = \theta_{i,j} e_{i,j} \qquad (i,j\in I), \]
say, and so $\theta_A$ is a \emph{Schur Multiplier} given by the matrix $\theta_{i,j}$, compare for example \cite[Theorem~1.1]{DD_Norms_Schur_Mults}.  Surely we require, at least, that $\Lambda_{\varphi^{-1}}(\theta_A(e_{i,j})) = A\Lambda_{\varphi^{-1}}(e_{i,j})$, which here means that $\theta_{i,j} = [i\sim j]$.  So $\theta_A$ is idempotent, projecting onto the space of matrices in $\mc B(\ell^2)$ supported on the relation $\sim$.

Such idempotent Schur multipliers are studied in \cite{Levene_norms_idem_schur} for example; it seems hard to compute the norm of a general idempotent Schur multiplier, but we make some remarks.  Suppose each row of $\theta$ contains at most one non-zero entry: so there is $I_0\subseteq I$ and $f\colon I_0\to I$ with $\theta_{i, f(i)}=1$ for $i\in I_0$, and $\theta_{i,j}=0$ otherwise.  Then for $i,j\in I$ we have
\[ \sum_{k\in I_0} e_k e_{i,j} e_{f(k)}
= \begin{cases}
  e_{i,j} &: i\in I_0, j=f(i), \\
  0 &: \text{otherwise}.
\end{cases} \]
Notice that $\sum_{k\in I_0} e_k^2$ is the projection onto the functions supported on $I_0$, but $\sum_{k\in I_0} e_{f(k)}^2 < \infty$ only when there is a bound on $|f^{-1}(\{i\})|$.  So when $\theta$ has also uniformly finitely support columns, $\theta$ is the symbol of Schur multiplier.  By linearity, if $\theta$ has finite rows and finite columns, then $\theta$ is a Schur multiplier.  Of course, if $\theta_{i,j}=1$ for all $i,j$ then we also have a Schur multiplier.

It follows from \cite[Proposition~1.2]{KP_main_tri_proj} that if $\theta_{i,j}=1$ when $j\geq i$ and $0$ otherwise, then $\theta$ is not a Schur multiplier.  Let $i\sim j$ exactly when $j\geq i$, and suppose that $\varphi$ is finite (so $I$ is necessarily countable).  Then the projection $e$ associated to $\sim$ is integrable, and so $A$ exists, but $\theta_A$ cannot exist.

\subsection{Matrix algebras}\label{sec:eg_matrix_algs}

We make more explicit links with \cite{Wasilewski_Quantum_Cayley}.  Let $M$ be the $\ell^\infty$-direct sum of matrix blocks $\mathbb M_{n(i)}$, with $\varphi(x) = \Tr(Qx)$ for $Q = (Q_i)$ where each $Q_i$ is positive and invertible.  We do not assume that $\varphi$ is a $\delta$-form; see Remark~\ref{rem:Q_norm} below.

We fix a GNS construction for $\varphi$, with $H = L^2(\varphi) = \bigoplus HS(\mathbb C^{n(i)}) = \bigoplus HS_{n(i)} = \bigoplus \mathbb C^{n(i)} \otimes \overline{\mathbb C^{n(i)}}$ and $\Lambda(x) = (x_i Q_i^{1/2})$ for $x=(x_i) \in \mf n_\varphi$.  Then $x\in M$ acts as $x_i\otimes 1$ in each block, so we could write $M = \prod_i \mathbb M_{n(i)} \otimes 1$.  In the same notation, $M' \cong \prod_i 1 \otimes \mathbb M_{n(i)}^\op$, the opposite arising as $(x_i^\op)$ acts as $(1\otimes x_i^\top)$.
The modular conjugation $J$ is $J(x_i) = (x_i^*)$ for $(x_i)\in \bigoplus HS(\mathbb C^{n(i)})$, and the modular operator is $\nabla = (\nabla_i) = (Q_i \otimes Q_i^{-1\top})$, compare Proposition~\ref{prop:mod_Trh}.  So all the operators and actions respect the direct-sum structure.  As we shall shortly consider $HS(H)$, and we wish to avoid over-using tensor products, and so we shall adopt the notation that $x\in\mathbb M_{n(i)}$ acts on the left as $x\otimes 1$, while $x^\op$ acts on the right, as $1\otimes x^\top$.  This means that $Jx^*J = x^\op$.

We have $HS(H) = \bigoplus_{i,j} HS_{n(i)} \otimes \overline{HS_{n(j)}}$, and for each ``block'' there are natural isomorphisms
\begin{align*}
HS_{n(i)} \otimes \overline{HS_{n(j)}}  \qquad&\cong&
\mc B(HS_{n(j)}, HS_{n(i)})  \qquad&\cong&
\mc B(\mathbb C^{n(j)},\mathbb C^{n(i)}) \otimes \overline{\mc B(\mathbb C^{n(j)},\mathbb C^{n(i)})},
\\
\rankone{\xi}{\eta} \otimes \overline{\rankone{\mu}{\nu}}   \qquad&\leftrightarrow&
\rankone{\xi\otimes\overline\eta}{\mu\otimes\overline\nu}   \qquad&\leftrightarrow&
\rankone{\xi}{\mu} \otimes \rankone{\overline\eta}{\overline\nu}. \qquad\qquad
\end{align*}
In the middle column, we use our usual identification of $HS(K)$ with $K\otimes\overline K$, and in the final column identify $\overline{\mc B(K)}$ with $\mc B(\overline K)$ for the isomorphism $\overline x = x^{\top *}$, so that $\rankone{\overline\eta}{\overline\nu} = (\rankone{\nu}{\eta})^\top = \overline{\rankone{\eta}{\nu}}$.  Given $x,y\in\mathbb M_{n(i)}$ it is natural to pre-compose (or post-compose) the middle operator with $x y^\op = x\otimes y^\top$, and we have the further bijections
\begin{align}
\rankone{\xi}{\eta} \otimes \overline{x^*\rankone{\mu}{\nu}y^*}   \qquad&\leftrightarrow&
\rankone{\xi\otimes\overline\eta}{\mu\otimes \overline\nu}x y^\op   \qquad&\leftrightarrow&
\rankone{\xi}{\mu}x \otimes \rankone{\overline\eta}{\overline\nu}y^\top.
\label{eq:3_HS_views_with_actions}  \\
x\rankone{\xi}{\eta}y \otimes \overline{\rankone{\mu}{\nu}}   \qquad&\leftrightarrow&
x y^\op\rankone{\xi\otimes\overline\eta}{\mu\otimes \overline\nu}   \qquad&\leftrightarrow&
x\rankone{\xi}{\mu} \otimes y^\top\rankone{\overline\eta}{\overline\nu}.
\label{eq:3_HS_views_with_left_actions}
\end{align}

We prefer to work with the final column, as it makes certain identifications easier to study.
We consider $\mc B_{M'}(H, HS(H))$, which by Remark~\ref{rem:L2M'_vs_L2M} means that $t\in \mc B_{M'}(H, HS(H))$ if and only if $t a = (1\otimes a^{\op\top})t$ for each $a\in M$.  As in the analogous calculation for $\ell^\infty(I)$ in Section~\ref{sec:ellinfty}, such a $t$ is of the form $t(x) = (t_{i,j}(x_j))_{i,j}$ for maps $t_{i,j} \colon HS_{n(j)} \to HS_{n(i)} \otimes \overline{HS_{n(j)}}$.
Further, the $t_{i,j}$ need to satisfy
\[  (t_{i,j}(a_jx_j))_{i,j} = t(ax) = (1\otimes a^{\op\top})t(x)
= ((1\otimes a_j^{\op\top})t_{i,j}(x_j))_{i,j} \qquad (a\in M, x\in H), \]
that is, each $t_{i,j}$ is of the form $s\colon HS_n \to \mc B(\mathbb C^n, \mathbb C^m) \otimes \overline{\mc B(\mathbb C^n, \mathbb C^m)}$ with $s(ax) = (1\otimes a^{\op\top})s(x)$.  That is,
\[ s \colon \mathbb C^n \otimes \overline{\mathbb C^n} \to \mathbb C^m \otimes \overline{\mathbb C^n} \otimes \overline{\mathbb C^m} \otimes \mathbb C^n; \quad
(\eta_1|\xi_2) s( \xi_1 \otimes \overline{\eta_2} ) = (1\otimes 1\otimes 1 \otimes \rankone{\xi_1}{\eta_1}) s(\xi_2 \otimes \overline{\eta_2}), \]
for all $\xi_1,\xi_2,\eta_1,\eta_2\in \mathbb C^n$.  It follows that there is $s_0 \colon \overline{\mathbb C^n} \to \mathbb C^m \otimes \overline{\mathbb C^n} \otimes \overline{\mathbb C^m}$ with $s(\xi\otimes\overline\eta) = s_0(\overline\eta) \otimes \xi$.  So we obtain maps $t^0_{i,j} \colon \overline{\mathbb C^{n(j)}} \to \mathbb C^{n(i)} \otimes \overline{\mathbb C^{n(j)}} \otimes \overline{\mathbb C^{n(i)}}$ with $t_{i,j}(\rankone{\xi}{\eta}) = t^0_{i,j}(\overline\eta) \otimes \xi$.

Given $\alpha\in\mc B(H)$, say with block-matrix form $(\alpha_{i,j})$, we consider $\hat\alpha \colon J\Lambda(a) \mapsto \alpha JaJ \nabla^{-1/2} \in HS(H)$.  For $a\in\mf n_\varphi$ we have $J\Lambda(a) = (Q_i^{1/2}a_i^*)$ and so with $x_i = Q_i^{1/2}a_i^*$ for each $i$, we have
\[ \hat\alpha (x) = \hat\alpha J\Lambda(a) = \alpha JaJ \nabla^{-1/2}
= (\alpha_{i,j} a_j^{*\op} \nabla^{-1/2}_j)
= (\alpha_{i,j} x_j^\op Q_j^{1/2 \op} Q_j^{-1/2} Q_j^{1/2\op})
= (\alpha_{i,j} Q_j^{-1/2} x_j^\op). \]
We naturally think of $\alpha_{i,j}$ as a member of the middle column of \eqref{eq:3_HS_views_with_actions}, so under the isomorphism with the final column, we obtain $\alpha_{i,j} Q_j^{-1/2} x_j^\op \leftrightarrow \alpha_{i,j}(Q_j^{-1/2} \otimes x_j^\top)$.
Given $i,j$, suppose that $\alpha_{i,j} = y \otimes \overline z \in \mc B(\mathbb C^{n(j)},\mathbb C^{n(i)}) \otimes \overline{\mc B(\mathbb C^{n(j)},\mathbb C^{n(i)})}$, and that $x_j = \rankone{\xi}{\eta}$, so that $\alpha_{i,j}(Q_j^{-1/2} \otimes x_j^\top) = yQ_j^{-1/2} \otimes z^{*\top}x_j^\top = yQ_j^{-1/2} \otimes \rankone{\overline{z\eta}}{\overline\xi}$.  Thus the associated $t^0_{i,j}$ is $\overline\eta \mapsto yQ_j^{-1/2} \otimes \overline{z\eta}$.  If we identify $t^0_{i,j} \in\mc B(\overline{\mathbb C^{n(j)}}, \mathbb C^{n(i)} \otimes \overline{\mathbb C^{n(j)}} \otimes \overline{\mathbb C^{n(i)}}) = \mathbb C^{n(i)} \otimes \overline{\mathbb C^{n(j)}} \otimes \overline{\mathbb C^{n(i)}} \otimes \mathbb C^{n(j)}$ then $t^0_{i,j} = yQ_j^{-1/2} \otimes \overline z = \alpha_{i,j}(Q_j^{-1/2}\otimes 1)$.  However, notice that the natural norms differ; compare Remark~\ref{rem:norms_t_alpha}.

We now compute $\varphi^{-1}$.
Let $x=\rankone{\xi}{\eta}$ and $y=\rankone{\mu}{\nu}$, say $x\otimes\overline y \in HS_{n(i)} \otimes \overline{HS_{n(i)}}  \cong  \mc B(\mathbb C^{n(i)}) \otimes \overline{\mc B(\mathbb C^{n(i)})}$, so it makes sense to define $z^\top = (\Tr\otimes\id)(x\otimes\overline y) \in \mc B(\overline{\mathbb C^{n(i)}})$.  Then
\begin{align*}
z^\top = \Tr(\rankone{\xi}{\mu}) \rankone{\overline\eta}{\overline\nu}
= (\mu|\xi) \rankone{\overline\eta}{\overline\nu}
\quad\implies\quad z = (\mu|\xi) \rankone{\nu}{\eta}
= \rankone{\nu}{\mu} \rankone{\xi}{\eta} = y^*x.
\end{align*}
Now let $x,y\in \mf n_\varphi$ and set $T = \rankone{ J\Lambda(x) }{ J\Lambda(y) }$, and set $z_i^\top = (\Tr\otimes\id)(Q_i^{-1} T_{i,i}) \in \mc B(\overline{\mathbb C^{n(i)}})$, for each $i$.  As $T_{i,i} = Q_i^{1/2}x_i^* \otimes \overline{Q_i^{1/2}y_i^*}$ it follows that
\[ z_i = (Q_i^{1/2}y_i^*)^* Q_i^{-1} (Q_i^{1/2}x_i^*) = y_i x_i^*, \]
which we compare to
\[ \varphi^{-1}\big( (Q_i^{1/2}x_i^*\Tr(y_j Q_j^{1/2} \,\cdot\, ))_{i,j} \big)
= \varphi^{-1}(T)
= Jxy^*J = (yx^*)^\op
= \big( (y_i x_i^*)^\op \big)_i. \]
Hence $\varphi^{-1}(T) = z^\op = ((\Tr\otimes\id)(Q_i^{-1}T_{i,i}))^{\op\top}$.

\begin{remark}\label{rem:Q_norm}
If we choose the normalisation that \cite{Wasilewski_Quantum_Cayley} does, namely $\Tr(Q_i^{-1})=1$ for each $i$, then $\Tr(Q_i^{-1}\,\cdot\,)$ is a state, and so $\varphi^{-1}(1) = 1$.  Furthermore, then $\varphi^{-1}$ is bounded, and is a conditional expectation.
\end{remark}

As in Example~\ref{eg_nice_vn_2}, a subspace $V\subseteq HS(H)$ is an $M$'-bimodule exactly when $V$ is the closed linear span of subspaces $V_{i,j} \subseteq \mc B(HS_{n(j)}, HS_{n(i)})$ each of which is of the form $V_{i,j} = V^0_{i,j} \otimes \mc B(\overline{\mathbb C^{n(j)}}, \overline{\mathbb C^{n(i)}})$ for some subspace $V^0_{i,j} \subseteq \mc B(\mathbb C^{n(j)}, \mathbb C^{n(i)})$, which may be arbitrary.  Similar remarks apply to weak$^*$-closed $M'$-bimodules in $\mc B(H)$.

Let $t\in\mc B_{M'}(H, HS(H))$.  Then $t$ maps into $V$ if and only if $\im(t_{i,j}) \subseteq V_{i,j}$ for all $i,j$.  In turn, this is equivalent to $\im(t^0_{i,j}) \subseteq V^0_{i,j} \otimes \overline{\mathbb C^{n(i)}}$, or if we think of $t^0_{i,j}$ as a tensor, to $t^0_{i,j} \in V^0_{i,j} \otimes \mc B(\overline{\mathbb C^{n(j)}}, \overline{\mathbb C^{n(i)}})$.
Consider the formation of $S=\{ \alpha : \im(\hat\alpha)\subseteq V \}$ as in Proposition~\ref{prop:V_to_S}.  Setting $t=\hat\alpha$, we identify $t^0_{i,j}$ with $\alpha_{i,j}(Q_j^{-1/2}\otimes 1)$, and so $\alpha\in S$ if and only if $\alpha_{i,j}(Q_j^{-1/2}\otimes 1) \in V^0_{i,j} \otimes \mc B(\overline{\mathbb C^{n(j)}}, \overline{\mathbb C^{n(i)}})$ for each $i,j$.

\begin{proposition}\label{prop:V_to_S_matrix}
Given $V\subseteq HS(H)$ an $M'$-bimodule, define $V^0_{i,j}$ as above.  Set $S^0_{i,j} = \{ x Q_j^{1/2} : x\in V^0_{i,j}\}$ for each $i,j$.  Then
\[ S = \{ \alpha \in \mf n_{\varphi^{-1}} : \im(\hat\alpha) \subseteq V \}
= \{ \alpha=(\alpha_{i,j}) \in \mf n_{\varphi^{-1}} : \alpha_{i,j} \in S^0_{i,j} \otimes \mc B(\overline{\mathbb C^{n(j)}}, \overline{\mathbb C^{n(i)}}) \ (i,j\in I) \}. \]
Letting $\mc S$ be the weak$^*$-closure of this space, $\mc S$ is an $M'$-bimodule in $\mc B(H)$, and this map $V \to \mc S$ is a bijection between $M'$-bimodules.
\end{proposition}
\begin{proof}
We have just shown that $\im(\hat\alpha) \subseteq V$ exactly when $\alpha_{i,j} \in S^0_{i,j} \otimes \mc B(\overline{\mathbb C^{n(j)}}, \overline{\mathbb C^{n(i)}})$ for each $i,j$.  As finitely supported $\alpha$ are in $\mf n_{\varphi^{-1}}$, we hence obtain a bijection, exactly as in Example~\ref{eg_nice_vn_2}.
\end{proof}

There is also no issue to introducing the ``twist'' as in Section~\ref{sec:diff}.  That is, given a weak$^*$-closed $\mc S \subseteq \mc B(H)$ we define
\[ \mc S_z = \lin\{ \nabla^{-iz} x \nabla^{iz}\in\mc B(H) : x\in\mc S \} \clos^{w^*}. \]
Again, as $M'$-bimodules are the weak$^*$-closed span of the matrix blocks they contain, $\mc S \mapsto \mc S_z$ is a bijection, indeed, $(\mc S_z)_{-z} = \mc S$.  As in Remark~\ref{rem:Q_or_nabla_to_twist}, and using the above, we see that if $\mc T = \mc S_z$ then the associated subspaces are $T^0_{i,j} = Q_i^{-iz} S^0_{i,j} Q_j^{iz}$.

Consider now quantum adjacency operators.  As in \cite[Section~3.3]{Wasilewski_Quantum_Cayley}, we can always densely define $A$; TODO: Say more?.  Indeed, using the relation of Proposition~\ref{prop:Abdd_implies_f_int} allows us to link any projection $e\in M \vnten M^\op$ with a map $A$ defined on $c_{00}(M)$ the algebraic direct sum of the matrix blocks.
Definition~\ref{defn:integrable_e_to_cp_A}, that of $e$ being of \emph{bounded degree}, is exactly the same, though again, as in Section~\ref{sec:ellinfty}, this strongly depends on the choice of $Q$, so we prefer to say that $e$ is \emph{integrable}.  As in Theorem~\ref{thm:Kraus_rep}, a Kraus-type representation of $A$ can be obtained from $V$ using the Hilbert module formulation.  We shall not explore \cite[Definition~3.24]{Wasilewski_Quantum_Cayley}, that of being \emph{locally finite}, as this seems to depend upon both the choice of $\varphi$, and is not obviously extendable to general von Neumann algebras.

\begin{remark}
\cite{Wasilewski_Quantum_Cayley} lets $M$ act on $K = \oplus_i \mathbb C^{n(i)}$ so that $M' = \ell^\infty(I)$.  This means that $V \subseteq \mc HS(K)$ is an $M'$-bimodule exactly when $V$ is the closed linear span of subspaces $V^0_{i,j}$ as above, and similar remarks for $M'$-bimodules in $\mc B(K)$.  We say more on this below in Section~\ref{sec:change_space}.
\end{remark}

\begin{remark}\label{rem:same_as_W}
We consider finding an orthogonal basis for $e(E_\varphi^{-1}) = \mc B_{M'}(H, V)$.  We may work one matrix block at a time, and so for each $i,j$ consider $(\hat\alpha_{i,j,k})_k$ each of which is in $HS_{n(i)} \otimes \overline{HS_{n(j)}}$.  The argument leading to Proposition~\ref{prop:V_to_S_matrix} shows that $\im(\hat\alpha_{i,j,k}) \subseteq V^0_{i,j}$ if and only if $\alpha_{i,j,k} \in S^0_{i,j} \otimes \mc B(\overline{\mathbb C^{n(j)}}, \overline{\mathbb C^{n(i)}})$ for each $k$.  So an obvious way to choose such a basis is to start with an algebraic basis $(\beta^{(i,j)}_k)$ of $S^0_{i,j}$, and then to set
\[ \alpha_{i,j,k,t} = \beta^{(i,j)}_k \otimes \rankone{\overline\delta_t}{\overline\delta_1}. \]
We then check the conditions we found at the end of Remark~\ref{eg:fd2}:
\begin{enumerate}[(1)]
\item $\hat\alpha_{i,j,k,t}^* \hat\alpha_{i,j,l,s} = \varphi^{-1}(\alpha_{i,j,k,t}^* \alpha_{i,j,l,s}) = \delta_{s,t} \varphi^{-1}(\beta^{(i,j)*}_k \beta^{(i,j)}_l \otimes e_{11})
= \delta_{s,t} \Tr(Q^{-1}_j\beta^{(i,j)*}_k \beta^{(i,j)}_l) e_{11}$.  So we need $\Tr(Q^{-1}_j\beta^{(i,j)*}_k \beta^{(i,j)}_l) = \delta_{k,l}$.
\item $\im(\hat\alpha_{i,j,k,t}) = \{ \beta^{(i,j)}_k Q_j^{-1/2} \otimes \rankone{\overline\delta_t}{\overline\delta_1}y :  y\in\mathbb M_{n(j)} \}
= \{ \beta^{(i,j)}_k Q_j^{-1/2} \otimes \rankone{\overline\delta_t}{\overline\xi} : \xi\in\mathbb C^{n(j)} \}$.  So taking the linear span over $t$ and $k$ shows that we want $\lin\{ \beta^{(i,j)}_k Q_j^{-1/2} : k \} = V^0_{i,j}$, but this already holds, as we chose an algebraic basis of $S^0_{i,j}$.
\end{enumerate}
So we just need that $(\beta^{(i,j)}_k Q_j^{-1/2})$ is an orthonormal (for the Hilbert--Schmidt inner product) basis of $V^0_{i,j}$.

Then Theorem~\ref{thm:Kraus_rep} shows that for $x \in \mathbb M_{n(j)}$, we have
\[ A(x) = \sum_{i,j,k,t} \beta^{(i,j)}_k x \beta^{(i,j)*}_k \otimes (\rankone{\overline\delta_t}{\overline\delta_1})(\rankone{\overline\delta_t}{\overline\delta_1})^*
= \sum_{i,k} \beta^{(i,j)}_k x \beta^{(i,j)*}_k \otimes \overline{1_{n(j)}}. \]
We could suppress the $\otimes 1$ factor by remembering that $\mathbb M_{n(j)}$ by definition acts on the left factor of $\mathbb C^{n(j)} \otimes \overline{\mathbb C^{n(j)}}$.

We now compare this to \cite{Wasilewski_Quantum_Cayley}.
Set $X^{(i,j)}_k = Q_i^{1/4} \beta^{(i,j)}_k Q_j^{-1/4}$ for each $i,j,k$, so that
\[ \Tr( X^{(i,j)*}_k Q_i^{-1/2} X^{(i,j)}_l Q_j^{-1/2})
= \Tr( Q_j^{-1} \beta^{(i,j)*}_k \beta^{(i,j)}_l) = \delta_{k,l}, \]
and for $x\in \mathbb M_{n(j)}$ we have
\[ A(x) = \sum_{i,k} Q_i^{-1/4} X^{(i,j)}_k Q_j^{1/4} x Q_j^{1/4} X^{(i,j)*}_k Q_i^{-1/4}. \]
We compare this to \cite[Proposition~3.30]{Wasilewski_Quantum_Cayley}.  Once we note that we write an index $(i,j)$ for maps $HS_{n(j)} \to HS_{n(i)}$ which is the opposite convention to \cite{Wasilewski_Quantum_Cayley}, and remembering that we don't normalise $Q$ (see how the weight is defined in \cite[Lemma~3.3]{Wasilewski_Quantum_Cayley}), we obtain exactly the same formula.  We have that $\lin\{ X^{(i,j)}_k : k \} = Q_i^{1/4} S^0_{i,j} Q_j^{-1/4} = T^0_{i,j}$ if $\mc T = \mc S_{i/4}$, see the discussion above.  This confirms our claim in Section~\ref{sec:op_sys_A} that \cite{Wasilewski_Quantum_Cayley} builds a bijection between $e, V$ and $\mc T$ (not $\mc S$) (and that we got the sign $i/4$, not $-i/4$, correct).
\end{remark}

\begin{remark}\label{rem:norms_t_alpha}
We did not compute the norms of $t$ or $\hat\alpha$ above as they are a little tedious.
As $\|t(x)\|^2 = \sum_{j,i} \|t_{ij}(x_j)\|^2$, if we take the supremum over $\sum_j \|x_j\|^2\leq 1$ we get
\[ \|t\|^2 = \sup_j \sup\Big\{ \sum_i \|t_{ij}(x)\|^2  : x\in HS_{n(j)}, \|x\|_{HS}\leq1\Big\}. \]
For $x\in HS_{n(j)}$ let $x = \sum_k \rankone{\delta_k}{\eta_k}$ so $t_{ij}(x) = \sum_k t^0_{ij}(\overline\eta_k) \otimes \delta_k$ and so the supremum above becomes
\[ \|t\|^2 = \sup_j \sup\Big\{ \sum_{i,k} \|t_{ij}^0(\overline\eta_k)\|^2 : \sum_k \|\eta_k\|^2\leq 1 \Big\}
= \sup_j \sup\Big\{ \sum_i \|t_{ij}^0(\overline\eta)\|^2 : \|\eta\|\leq 1 \Big\}. \]
When $t = \hat\alpha$ we identify $t^0_{i,j} = \alpha_{i,j} (Q_j^{-1/2}\otimes 1)$ both in $\mathbb C^{n(i)} \otimes \overline{\mathbb C^{n(j)}} \otimes \overline{\mathbb C^{n(i)}} \otimes \mathbb C^{n(j)}$.  However, the norms are different, and furthermore, we then need to compute $\|\alpha\|$ for the block matrix entries $(\alpha_{i,j})$, which has no simple expression, in general.
\end{remark}

\begin{example}\label{eg:all_diag}
Things become easier if we suppose that $t$ and/or $\alpha$ are ``diagonal'', that is, $t_{ij}=0$ for $i\not=j$.  Then $\|t\| = \sup_i \|t^0_{ii}\|$ and $\|\alpha\| = \sup_i \|\alpha_{ii}\|$.  For each $i$ let $(u^{(i)}_k)$ be an orthonormal basis of $V^0_{i,i} \subseteq HS_{n(i)}$.  For $t\in\mc B_{M'}(H,V)$ we can write each $t^0_{i,i}$ as
\[ t^0_{i,i} = \sum_k u^{(i)}_k \otimes s^{(i)}_k \in HS_{n(i)} \otimes \mc B(\overline{\mathbb C^{n(i)}}), \]
for some (arbitrary) $s^{(i)}_k$.  We then see that $\|t^0_{i,i}(\overline\eta)\|^2 = \sum_k \|s_k^{(i)}(\overline\eta)\|^2 = (\overline\eta | \sum_k {s^{(i)}_k}^* s^{(i)}_k \overline\eta)$ and so
\begin{equation}
\|t\| = \sup_i \|t^0_{i,i}\|
= \sup_i \Big\| \sum_k {s^{(i)}_k}^* s^{(i)}_k \Big\|^{1/2}.
\label{eq:norm_t_sks}
\end{equation}
Then $t=\hat\alpha$ if and only if, defining $\alpha_{i,i} = t^0_{i,i}(Q_i^{1/2}\otimes 1)$, we have that $\sup_i \|\alpha_{i,i}\| < \infty$.  In general, $\|\alpha_{i,i}\|$ has no simple form.
\end{example}

\begin{example}\label{ex:diag_whole_space}
We continue with the previous example.
Take $Q=1$ and $V^0_{i,i} = HS_{n(i)}$ for each $i$.  We first perform a calculation: set $x = n^{-1/2} \sum_{i,j} e_{i,j} \otimes \overline{e_{i,j}} \in \mathbb M_n \otimes \overline{\mathbb M_n}$.  Then $n^{-1} \sum_{i,j} \overline{e_{i,j}}^* \overline{e_{i,j}} = \sum_j \overline{e_{j,j}} = 1$.  Also note that $x=x^*=x^2$ is a projection, so is positive.  Let $\xi = \sum_{i,j} \xi_{i,j} \delta_i \otimes \overline{\delta_j} \in \mathbb C^n \otimes \overline{\mathbb C^n}$ so
\begin{align*}
(\xi|x\xi)
&= n^{-1/2} \sum_{s,t,r,u} \overline{\xi_{s,t}} \xi_{r,u} \sum_{i,j} (\delta_s \otimes \overline{\delta_t} | e_{i,j}\delta_r \otimes \overline{e_{i,j}\delta_u})  \\
&= n^{-1/2} \sum_{s,t,r} \overline{\xi_{s,t}} \xi_{r,r} \sum_{i} (\delta_s \otimes \overline{\delta_t} | \delta_i \otimes \overline{\delta_i})
= n^{-1/2} \sum_{s,r} \overline{\xi_{s,s}} \xi_{r,r}.
\end{align*}
Thus $|(\xi|x\xi)| \leq n^{-1/2} |\sum_r \xi_{r,r}|^2 \leq n^{1/2} \sum_r |\xi_{r,r}|^2 \leq n^{1/2} \|\xi\|^2$ with equality for the case $\xi_{s,r} = \delta_{s,r} n^{-1/2}$, and so $\|x\| = n^{1/2}$.  Returning to $t$ and $\alpha$, take as an orthonormal basis the matrix units $(e^{(i)}_{j,k})_{j,k}$ and set $s^{(i)}_{j,k} = n(i)^{-1/2} \overline{e^{(i)}_{j,k}}$ for each $i,j,k$.  By the calculation, $\|t^0_{i,i}\| = 1$ while $\|\alpha_{i,i}\| = n(i)^{1/2}$.  If the dimensions $(n(i))$ are unbounded, we conclude that there are $t$ which cannot arise as $\hat\alpha$ for some $\alpha$, and so this choice does not give a self-dual $S$.  Compare with Example~\ref{eg:self-dual_changes} below.
\end{example}

We now consider the projection $e\in M\vnten M^\op$ onto $V$.  Using $M\vnten M^\op \cong \bigoplus_{i,j} \mathbb M_{n(i)} \otimes \mathbb M_{n(j)}^\op$, we naturally decompose $e$ as a matrix $e=(e_{i,j})$, noting that this matrix decomposition is slightly different to that considered above for operators.  Then $e_{i,j} \in \mathbb M_{n(i)} \otimes \mathbb M_{n(j)}^\op \cong \mc B(\mathbb C^{n(i)} \otimes \overline{\mathbb C^{n(j)}})$ is the projection of $\mathbb C^{n(i)} \otimes \overline{\mathbb C^{n(j)}} \cong \mc B(\mathbb C^{n(j)}, \mathbb C^{n(i)})$ onto $V^0_{i,j}$.

\begin{example}
We continue with Example~\ref{ex:diag_whole_space}.  Here $e_{i,i}=1 = 1\otimes\overline 1$ for each $i$, and $e_{i,j}=0$ for $i\not=j$.  Thus $(\id\otimes\Tr^\op)e_{i,i} = n(i) 1_i$ for each $i$, and hence if the dimensions $(n(i))$ are unbounded, $e$ is not integrable (as we expect!)
\end{example}

\subsection{Changing the space we act on}\label{sec:change_space}

Throughout the second part of this paper, we have always let $M$ act on $L^2(M)$.  Given Weaver's work, compare Section~\ref{sec:inv_reps}, it would be natural to consider $M \subseteq \mc B(H)$ for more general $H$.  We now make some brief comments on this, and look at some examples, but leave a fuller exploration to further work.

We can still form $\varphi^{-1}$ from $\mc B(H)$ to $M'$, where now $M' \subseteq \mc B(H)$ need not be isomorphic to $M^\op$, of course.  The ideas of Section~\ref{sec:inv_weight} still hold, allowing us to form $\mf n_{\varphi^{-1}}$ and its self-dual completion $E_{\varphi^{-1}} = \mc B_{M'}(L^2(\phi), HS(H))$ where we make some (now essentially arbitrary) choice of a weight $\phi$ on $M'$.  We carry out this procedure for matrix algebras below.  Of course, we really form $\tilde\phi = \phi\circ\varphi^{-1}$ a weight on $\mc B(H)$, and then use Section~\ref{sec:weights_bh} to find an explicit isomorphism $L^2(\tilde\phi) \cong HS(H)$.  The bijection between $e$ and $V$ remains unchanged, but now the space $S$ from Proposition~\ref{prop:V_to_S}, $S = \{ \alpha\in\mf n_{\varphi^{-1}} : e\hat\alpha = \hat\alpha \}$, is different.  Presumably the interesting question is how $S$ changes as we vary $H$.  A step in this direction are the examples below.

We shall first look at the example of $M=\mc B(K)$ acting on $H=K$, so that $M'=\mathbb C$.  Set $\varphi = \Tr_h$, as in Section~\ref{sec:weights_bh}.  Then one can check that $\varphi^{-1} = \Tr_{h^{-1}}$, that $E_{\varphi^{-1}} = \mc B(\mathbb C, HS(H)) = HS(H)$ and that $\hat\alpha = \alpha h^{-1/2}$ for $\alpha\in\mf n_{\varphi^{-1}}$.  We shall see similar formula in the case when $M$ acts on $L^2(M)$ in Section~\ref{sec:example_BH} below.

An $M'$-bimodule in $HS(K)$ is then simply a closed subspace $V\subseteq HS(K)$.  We shall be interested in when $S = \{ \alpha\in\mf n_{\varphi^{-1}} : \im \hat\alpha \subseteq V \} = \{ \alpha\in\mc B(K) : \alpha h^{-1/2} \in V \}$ is self-dual, Section~\ref{sec:self-dual_S}.  In this situation, this is equivalent to $t\in V$ implying there is $\alpha\in S$ with $\alpha h^{-1/2} = t$, that is, $th^{1/2} \in \mc B(K)$.

\begin{lemma}\label{lem:BH_eg_selfdual}
With this notation, $S$ is self-dual if and only if there is $K>0$ with $\|th^{1/2}\| \leq K\|t\|_{HS}$ for each $t\in V$.
\end{lemma}
\begin{proof}
If this condition holds, then obviously $S$ is self-dual.  For the converse, we apply the Closed Graph Theorem to the map $V \to \mc B(H); t \mapsto \alpha = t h^{1/2}$, noting that by assumption $V$ is closed.  Let $t_n\to t$ in $V$, with $t_n h^{1/2} \to \beta \in \mc B(H)$.  For $\xi\in D(h^{1/2})$ we have $\beta\xi = \lim_n t_n h^{1/2}\xi = t h^{1/2} \xi$ as also $t_n\to t$ in the operator norm.  As $D(h^{1/2})$ is dense this shows that $\beta = th^{1/2}$ as required.
\end{proof}

We shall also suppose that $\varphi^{-1}$ is bounded, equivalently, $h^{-1}$ is a trace-class operator.  So we can find an orthonormal basis $(e_n)$ of $H$ (which must be separable) with $h^{-1}e_n = h_n^{-1}e_n$ for each $n$, where $\sum h_n^{-1} < \infty$ and $h_n\in (0,\infty)$.  We shall give an example of a $V$ associated to a self-dual $S$ where the projection $e$ onto $V$ is not integrable, compare Theorem~\ref{thm:S_bdd_below_e_int}.

\begin{example}\label{eg:BH_selfdual_notint}
Let $(\delta_i)$ be some orthogonal sequence in $H$ (which could just be $(e_n)$) and set $v_1 = \rankone{\delta_1}{e_1}$, and then for $n>1$ choose $0<\lambda_n<1$ and set $v_n = \lambda_n \rankone{\delta_n}{e_1} + (1-\lambda_n^2)^{1/2} \rankone{\delta_1}{e_n}$.  This gives an orthonormal family in $HS(\ell^2)$; let $V$ be the closed linear span.  With $t = \sum_{n\leq N} a_n v_n \in V$ we have
\begin{align*}
th^{1/2} = a_1 h_1^{1/2} \rankone{\delta_1}{e_1} + \sum_{n=2}^N h_1^{1/2} a_n \lambda_n \rankone{\delta_n}{e_1} + h_n^{1/2} a_n (1-\lambda_n^2)^{1/2} \rankone{\delta_1}{e_n}.
\end{align*}
Let this act on $b = \sum_k b_k e_k$ to get
\begin{align*}
\|th^{1/2}(b)\|^2 &= \Big\| a_1 h_1^{1/2} b_1 \delta_1 + \sum_{n=2}^N h_1^{1/2} a_n \lambda_n b_1 \delta_n + h_n^{1/2} a_n (1-\lambda_n^2)^{1/2} b_n \delta_1 \Big\|^2   \\
&= \big| a_1 h_1^{1/2} b_1 + \sum_{n=2}^{N} h_n^{1/2} a_n (1-\lambda_n^2)^{1/2} b_n \big|^2
   + \sum_{n=2}^N |h_1^{1/2} a_n \lambda_n b_1|^2 \\
&\leq \Big( \max\big( h_1, \max_{n\geq 2} (1-\lambda_n^2) h_n \big)
   + h_1\Big) \|b\|^2 \sum_n |a_n|^2.
\end{align*}
This shows that $t\in V$ implies that $\|th^{1/2}\| \leq K \|t\|$ for some $K$, so long as $(1-\lambda_n^2) h_n$ is bounded in $n$.  By Lemma~\ref{lem:BH_eg_selfdual}, $S$ is self-dual in this case.

Let $e$ be the projection onto $V$, so $e = \sum_n \rankone{v_n}{v_n}$.  For $\xi\in K$ and $n>1$ we have
\begin{align*}
(\id\otimes\varphi^\op)(\rankone{v_n}{v_n})
&= \lambda_n^2 h_1 \rankone{\delta_n}{\delta_n}
+ (1-\lambda_n^2) h_n \rankone{\delta_1}{\delta_1},
\end{align*}
and similarly $(\id\otimes\varphi^\op)(\rankone{v_1}{v_1}) = h_1 \rankone{\delta_1}{\delta_1}$.  So
\[ (\id\otimes\varphi^\op)(e) \geq \sum_{n\geq 2} (1-\lambda_n^2) h_n \rankone{\delta_1}{\delta_1}. \]

Set $h_n = 2^n$, so $\sum_n h_n^{-1} = \sum_n 2^{-n} < \infty$ and $h^{-1}$ is trace-class, as we assumed.  Set $1-\lambda_n^2 = 2^{-n}$ for $n\geq 2$, that is, $\lambda_n = (1-2^{-n})^{1/2}$ so $0 < \lambda_n < 1$.  Then $(1-\lambda_n^2) h_n = 1$ for all $n\geq 2$, showing that $S$ is self-dual, but $\sum_{n\geq 2} (1-\lambda_n^2) h_n = \infty$, showing that $e$ is not integrable.
\end{example}

We now concentrate upon the case of $M = \prod_i \mathbb M_{n(i)}$, now acting on $H = \bigoplus \mathbb C^{n(i)}$.  Then $M' \cong \ell^\infty(I)$ where the minimal projection $e_i \in \ell^\infty(I)$ acts as the identity on the block $\mathbb C^{n(i)}$.  We omit the calculations (which proceed exactly as in Section~\ref{sec:inv_weight}), but one may carefully check that $\varphi^{-1} \colon \mc B(H)_+ \to \wh{\ell^\infty(I)}$ has $\mf m_{\varphi^{-1}}^+ = \{ T = (T_{ij}) \in \mc B(H) : (\Tr(Q_i^{-1}T_{i,i}))_{i\in I} \in \ell^\infty(I) \}$ and $\varphi^{-1}(T) = (\Tr(Q_i^{-1}T_{i,i}))$ for each $T \in \mf m_{\varphi^{-1}}$.  Let $\phi$ be the weight on $\ell^\infty(I)$ given by the counting measure, and set $\tilde\phi = \phi\circ\varphi^{-1}$.  Then $\tilde\phi$ has density $Q^{-1}$, where $Q^{-1}$ is considered with its natural domain, $D(Q^{-1}) = \{ \xi = (\xi_i) \in H : \sum_i \|Q_i^{-1} \xi_i \|^2 < \infty \}$.

We identify $L^2(\phi)$ with $\ell^2(I)$ (denoted $\ell^2$ for ease of notation).  Then $t \in \mc B_{\ell^\infty}(\ell^2, HS(H))$ when $t(\delta_i) = (1\otimes e_i)t(\delta_i)$ for all $i$.  This is equivalent to having $t_{i,j} \in \mathbb C^{n(i)} \otimes \overline{\mathbb C^{n(j)}} \cong \mc B(\mathbb C^{n(j)}, \mathbb C^{n(i)})$ with $t(\xi) = (t_{i,j} \xi_j)$ for each $\xi = (\xi_j)\in\ell^2$.  It follows that $\|t\|^2 = \sup_j \sum_i \|t_{i,j}\|_{HS}^2$.  For $\alpha\in \mf n_{\varphi^{-1}}$ we see that the resulting $t=\hat\alpha$ has $t_{j,i} = \alpha_{j,i} Q_i^{-1/2}$.  So $t$ arises as $\hat\alpha$ for some $\alpha$ exactly when the block matrix $(t_{i,j} Q_j^{1/2})$ is bounded as an operator on $H$.

Let $V\subseteq HS(H)$ be an $M'$-bimodule.  Then $V$ is the closed span of subspaces $V^0_{i,j} \subseteq \mc B(\mathbb C^{n(j)}, \mathbb C^{n(i)})$ as before (but now we have no ``auxiliary space'' to tensor with).  Then
\[ S = \{ \alpha\in\mf n_{\varphi^{-1}} : \hat\alpha(\ell^2) \subseteq V \}
= \{ \alpha\in\mf n_{\varphi^{-1}} : \alpha_{i,j} Q_j^{-1/2} \in V^0_{i,j} \ (i,j\in I) \}. \]
We might want to also consider $\mc S = S\clos^{w^*}$.  As $S$ and $\mc S$ are determined by the block matrix entries, we see that $\mc S = \{ \alpha\in\mc B(H) : \alpha_{i,j} Q_j^{-1/2} \in V^0_{i,j} \}$ because any $\alpha\in\mc B(H)$ can be approximated by finitely-supported matrices in the weak$^*$-topology.  Again, considering $S_{i/4}$ or $\mc S_{i/4}$ does not present problems in this context.

\begin{example}\label{eg:self-dual_changes}
Let $Q=1$ and suppose that $V$ is diagonal as in Example~\ref{eg:all_diag}.  So $t\in S$ means that $t_{i,j}=0$ for $i\not=j$, and hence $\sup_i \|t_{i,i}\|_{HS} < \infty$.  As $\|s\|_{HS} \geq \|s\|$ for any $s\in\mathbb M_n$, we immediately see that setting $\alpha_{i,j} = t_{i,j}$ gives a bounded operator, so $\alpha\in S$ and $t=\hat\alpha$.  That is, $S$ is self-dual.  This should be compared with Example~\ref{ex:diag_whole_space} where, with the same $V$, we obtained a non-self-dual $S$ when representing $M$ on $L^2(M)$.
\end{example}

We now adapt the ideas of Example~\ref{eg:BH_selfdual_notint} to the matrix setting.

\begin{example}\label{eg:matrix_sd_notint}
We define each $Q_i$ to be diagonal, with entries $(h_k)_{k=1}^{n(i)}$, for some $(h_k) \subseteq (0,\infty)$.  Then $\Tr(Q_i^{-1}) = \sum_{k=1}^{n(i)} h_k^{-1}$, so it is natural to suppose that $\sum_k h_k^{-1} < \infty$, which ensures that $\varphi^{-1}$ is bounded.  Note that $\varphi^{-1}(1)\not=1$, but with some more book-keeping, we could normalise each $Q_i$ to get $\Tr(Q_i^{-1})=1$, and much the same construction would work.

We define $V$ by specifying the subspaces $V^0_{i,j} \subseteq \mathbb C^{(n(i))} \otimes \overline{ \mathbb C^{(n(j)} }$.  Set $V^0_{i,j} = \{0\}$ when $i\not=j$, and define $V^0_{i,i}$ by giving it the orthonormal basis
\[ v^{(i)}_1 = \rankone{\delta^{(i)}_1}{\delta^{(i)}_1}, \quad
v^{(i)}_k = \lambda_k \rankone{\delta^{(i)}_k}{\delta^{(i)}_1} +  (1-\lambda_k^2)^{1/2} \rankone{\delta^{(i)}_1}{\delta^{(i)}_k}
\quad  (2\leq k\leq n(i)). \]
Again here $\lambda_k \in (0,1)$ are to chosen later.

The same calculation as in Example~\ref{eg:BH_selfdual_notint} shows that if $(1-\lambda_n^2)h_n$ is bounded in $n$, there is a constant $K$ so that $\|t_{i,i} Q_i^{1/2}\| \leq K \|t_{i,i}\|_{HS}$ for any $i$ and any $t_{i,i} \in V^0_{i,i}$.  As $V$ is diagonal, it follows that any $t\in\mc B_{\ell^\infty}(\ell^2, V)$ is of the form $t = \hat\alpha$ for some $\alpha$.  So the associated $S$ is self-dual.

The projection $e=(e_{i,j})$ is also diagonal, so $(\id\otimes\varphi^\op)(e) = ((\id\otimes\Tr(Q_i\cdot))(e_{i,i}))$.  We again find that $(\id\otimes\Tr(Q_i\cdot))(e_{i,i})$ has norm at least $\sum_{k=2}^{n(i)} h_k (1-\lambda_k^2)$, and this can be large, if $n(i)$ is large.  Thus, if $(n(i))$ is unbounded, we have found an example of a self-dual $S$ with non-integrable $e$, and with $\varphi^{-1}$ bounded.  This provides a (counter-)example to \cite[Remark~3.28]{Wasilewski_Quantum_Cayley}.
\end{example}

\subsection{Bounded operators}\label{sec:example_BH}

Now set $M = \mc B(K)$ for some (infinite-dimensional) $K$, with $\varphi = \Tr_h$, as in Section~\ref{sec:weights_bh}.  We continue with some of the notation from Section~\ref{sec:eg_matrix_algs}.  We have $H = K \otimes \overline K$ with GNS map $\Lambda(x) = x h^{1/2} \in HS(K)$, and $M' = \mc B(K)^\op$ acting as $1 \otimes \mc B(K)^\top$.  We compute $\mc B_{M'}(H, HS(H))$, this being maps $t\colon H\to H\otimes\overline H$, where we identify $H\otimes\overline H$ with $K \otimes \overline K \otimes \overline K \otimes K$, with $ta = (1\otimes a^{\op\top})t$ for each $a\in\mc B(H)$.  Taking $a$ to be rank-one shows that
\[ (\nu|\mu) t(\xi\otimes\overline\eta)
= t \rankone{\xi}{\nu} (\mu\otimes\overline\eta)
= (1\otimes(1\otimes\rankone{\xi}{\nu})) t(\mu\otimes\overline\eta)
\in K\otimes\overline{K}\otimes\overline{K} \otimes \xi, \]
and so there is $t_0 \colon \overline{K} \to K\otimes\overline{K}\otimes\overline{K}$ with $t(\xi\otimes\overline\eta) = t_0(\overline\eta) \otimes \xi$.  It is now easy to see that if $t$ is of this form, then $ta = (1\otimes a^{\op\top})t$ for any $a$, and $\|t\| = \|t_0\|$.
The inner-product is $(t|s) = t^*s = 1 \otimes t_0^*s_0 = (t_0^*s_0)^{\top \op} \in \mc B(K)^\op$.

In the following, let $\tau_{23} = 1\otimes\tau\otimes 1$ be the tensor swap map applied to the 2nd and 3rd ``legs'' of $K \otimes \overline K \otimes \overline K \otimes K$.  We also regard $\mc B(H)$ as $\mc B(K) \vnten \mc B(\overline K)$ when taking slice maps.

\begin{lemma}\label{lem:BH_alpha_nvarphiinv}
We have that $\alpha\in\mf n_{\varphi^{-1}}$ if and only if, for each $\xi,\eta\in K$ we have that $(\id\otimes\omega_{\overline\xi,\overline\eta})(\alpha) h^{-1/2} \in HS(K)$, and furthermore, there is a constant $C$ so that for any orthonormal basis $(\delta_j)$ of $K$ and any $\eta$ we have $\sum_j \| (\id\otimes\omega_{\overline\delta_j,\overline\eta})(\alpha) h^{-1/2} \|^2_{HS} \leq C^2\|\eta\|^2$.  Then $\tau_{23} \hat\alpha \colon \xi\otimes\overline\eta \mapsto \sum_j (\id\otimes\omega_{\overline\delta_j,\overline\eta})(\alpha) h^{-1/2} \otimes \overline\delta_j \otimes \xi$.
\end{lemma}
\begin{proof}
With reference to Proposition~\ref{prop:Trh}, we have $\hat\alpha(h^{1/2}a^*) = \hat\alpha J\Lambda(a) = \alpha (1\otimes a^{*\top}) (h^{-1/2} \otimes h^{\top 1/2}) \in HS(H)$ for each $a\in\mc B(H)$ with $a h^{1/2} \in HS(K)$.  Equivalently, $\hat\alpha(b) = \alpha (h^{-1/2} \otimes b^\top)$ for $b\in HS(K)$. Then $\hat\alpha = t_0$ has the form above, so for $\mu\in D(h^{-1/2})$ we have
\begin{align*}
& \big( \delta_i \otimes \overline\delta_j \otimes \overline\mu \otimes \delta_l \big| t_0(\overline\eta)\otimes \xi \big)
= \big( \delta_i \otimes \overline\delta_j \big| \alpha(h^{-1/2} \otimes \rankone{\overline\eta}{\overline\xi}) \mu \otimes \overline\delta_l \big)   \\
\iff &
\big( \delta_i \otimes \overline\delta_j \otimes \overline\mu \big| t_0(\overline\eta) \big)
= \big( \delta_i \otimes \overline\delta_j \big| \alpha(h^{-1/2}\mu \otimes \overline\eta) \big)
= \big( \delta_i \big| (\id\otimes\omega_{\overline\delta_j, \overline\eta})(\alpha) h^{-1/2} \mu \big).
\end{align*}
Let $s_0 = \tau_{23}t_0$ so this condition is equivalent to $(\id\otimes\omega_{\overline\xi,\overline\eta})(\alpha) h^{-1/2} = (\id\otimes\id\otimes\overline\xi)s_0(\overline\eta) \in K \otimes \overline K \cong HS(K)$ for each $\xi,\eta$, and with $\sum_j \| (\id\otimes\omega_{\overline\delta_j,\overline\eta})(\alpha) h^{-1/2} \|^2_{HS} = \|s_0(\overline\eta)\|^2 \leq \|s_0\|^2 \|\eta\|^2$.
\end{proof}

$M'$-bimodules $V\subseteq HS(H)$ biject with closed subspaces $V_0 \subseteq K\otimes\overline K$ with
\[ V = (V_0)_{13} = \tau_{23}(V_0\otimes \overline K\otimes K), \]
and so $\alpha\in S = \{ \alpha\in\mf n_{\varphi^{-1}} : \im \hat\alpha \subseteq V \}$ is those $\alpha$ as in Lemma~\ref{lem:BH_alpha_nvarphiinv} with $(\id\otimes\omega_{\overline\xi,\overline\eta})(\alpha) h^{-1/2} \in V_0$ for each $\xi,\eta$.
The norm estimate in the lemma shows that $\|(\id\otimes\omega_{\overline\xi,\overline\eta})(\alpha) h^{-1/2}\|_{HS} \leq C \|\xi\| \|\eta\|$ and so as every $\omega\in\mc B(\overline K)$ is the absolutely convergent sum of such functionals, we see that $(\id\otimes\omega)(\alpha) h^{-1/2} \in V_0$ for any $\omega$.  
In particular, given $a\in\mc B(K), 0\not=x\in\mc B(\overline K)$, consider $\alpha=a\otimes x$, which by the lemma will be in $\mf n_{\varphi^{-1}}$ if and only if $ah^{-1/2}\in HS(K)$.  In this case,  $\tau_{23} \hat\alpha \colon \xi\otimes\overline\eta \mapsto ah^{-1/2} \otimes x\overline\eta \otimes \xi$.
Furthermore, $ah^{-1/2} \in V_0$ if and only if $\alpha\in S$.

\begin{example}\label{eg:V_can_give_zero_S}
Let $v_0 \in HS(K)$ and set $V_0 = \mathbb C v_0$.  Let $\alpha\in\mf n_{\varphi^{-1}}$ so by Lemma~\ref{lem:BH_alpha_nvarphiinv} for each $\xi,\eta$ we have $(\id\otimes\omega_{\overline\xi,\overline\eta})(\alpha) h^{-1/2} \in V_0$.  If $\alpha$ is non-zero we can find $\xi,\eta$ with $(\id\otimes\omega_{\overline\xi,\overline\eta})(\alpha) h^{-1/2} = v_0$, and so for $\mu \in D(h^{1/2})$ we have $v_0 h^{1/2} \mu = (\id\otimes\omega_{\overline\xi,\overline\eta})(\alpha) \mu$ showing that $v_0 h^{1/2}$ is bounded.  As we can choose $v_0$ non-zero (say, rank-one) with $v_0 h^{1/2}$ unbounded, this shows that we can have $V\not=\{0\}$ and yet $S=\{0\}$.

We note for later that if $v_0h^{1/2}$ is bounded, then $(\id\otimes\omega_{\overline\xi,\overline\eta})(\alpha) \in \mathbb C v_0h^{1/2}$ for all $\xi,\eta$, so $\alpha = v_0h^{1/2} \otimes x$ for some $x\in\mc B(\overline K)$, and $\alpha\in S$.  The converse argument is given above, and so we conclude that $S = v_0h^{1/2} \otimes \mc B(\overline K)$.
\end{example}

\begin{example}\label{eg:V_can_give_zero_S_varphiinv_bdd}
Suppose that $h^{-1/2}$ is a Hilbert-Schmidt operator (equivalently, $h^{-1}$ is trace-class).  Then for any $\alpha\in\mc B(H)$ we have that $(\id\otimes\omega_{\overline\xi,\overline\eta})(\alpha) h^{-1/2} \in HS(K)$, and indeed,
\begin{align*}
\sum_j \| (\id\otimes\omega_{\overline\delta_j,\overline\eta})(\alpha) h^{-1/2} \|^2_{HS}
&= \sum_j \Tr\big( h^{-1/2} (\id\otimes\omega_{\overline\eta,\overline\eta})(\alpha^*(1\otimes\rankone{\overline\delta_j}{\overline\delta_j}) \alpha ) h^{-1/2} \big)  \\
&= \Tr\big( h^{-1/2} (\id\otimes\omega_{\overline\eta,\overline\eta})(\alpha^*\alpha) h^{-1/2} \big)
\leq \|h^{-1/2}\|_{HS}^2 \|\alpha^*\alpha\| \|\eta\|^2.
\end{align*}
Hence $\alpha\in\mf n_{\varphi^{-1}}$, and this calculation also essentially shows that $\varphi^{-1}(\alpha^*\alpha) = (\Tr(h^{-1}\cdot) \otimes \id)(\alpha^*\alpha)$.  So $\varphi^{-1}$ is bounded.
As this can easily occur with $h$ still unbounded, combined with Example~\ref{eg:V_can_give_zero_S} this gives an example where $\varphi^{-1}$ is bounded, but we can have $V$ non-zero and yet the associated $S$ is zero.
\end{example}

\begin{example}\label{eg:BH_int_coint}
We continue with Example~\ref{eg:V_can_give_zero_S}.  The projection $e$ onto $V_0$ is $e = \rankone{v_0}{v_0}$ and so for $\xi,\eta,\nu,\mu \in K$ we have
\begin{align*}
\big( \overline\mu \big| (\omega_{\xi,\eta}\otimes\id)(e) \overline\nu \big)
&= ( \xi\otimes\overline\mu | e(\eta\otimes\overline\nu) )
= ( \rankone{\xi}{\mu} | v_0 )_{HS} ( v_0 | \rankone{\eta}{\nu} )_{HS}
= (\xi|v_0(\mu))(v_0(\nu)|\eta)
= (\overline\mu|\overline{v_0^*\xi}) (\overline{v_0^*\eta}|\overline\nu),
\end{align*}
so $(\omega_{\xi,\eta}\otimes\id)(e) = \rankone{\overline{v_0^*\xi}}{\overline{v_0^*\eta}} = v_0^\top \rankone{\overline\xi}{\overline\eta}v_0^{\top*}$.  It follows from Proposition~\ref{prop:Trh} that $\rankone{\overline{v_0^*\xi}}{\overline{v_0^*\xi}} \in \mf m^+_{\varphi^\op}$ if and only if $v_0^*\xi \in D(h^{1/2})$, and then $\varphi^\op(\rankone{\overline{v_0^*\xi}}{\overline{v_0^*\xi}}) = \|h^{1/2} v_0^*\xi\|^2$.  So $e$ is integrable, $(\id\otimes\varphi^\op)(e) < \infty$, is equivalent to there being a constant $C$ with $\|h^{1/2} v_0^* \xi\| \leq C \|\xi\|$ for each $\xi$.  By Lemma~\ref{lem:domain_S_Tr_h}, this is the same as $v_0 h^{1/2}$ being bounded.  Similarly, $(\varphi\otimes\id)(e) < \infty$ if and only if $v_0^* h^{1/2}$ is bounded.

Choose $v_0\in HS(K)$ with $v_0 h^{1/2}$ bounded, so $e$ is integrable, and as in Example~\ref{eg:V_can_give_zero_S}, $S = v_0 h^{1/2} \otimes \mc B(\overline K)$.  With reference to Section~\ref{sec:hilb_ops} consider $\tau(e)^\op$ and the associated $V_\tau = J_0(V) = \mathbb C v_0^*$.  If $v_0^* h^{1/2}$ is not bounded then $S_\tau=\{0\}$, as in Example~\ref{eg:V_can_give_zero_S}, and so we can have $S_\tau=\{0\}$ even when $e$ is integrable.

Suppose now both $e$ and $\tau(e)^\op$ are integrable.  Recall from Proposition~\ref{prop:mod_Trh} that $\nabla = h \otimes h^{-1\top}$, and so given $\alpha = v_0 h^{1/2} \otimes x \in S$ we have that $\nabla^{-1/2} \alpha^* \nabla^{1/2} = h^{-1/2} h^{1/2} v_0^* h^{1/2} \otimes h^{1/2 \top} x^* h^{-1/2 \top}$.  There is no reason why $h^{1/2 \top} x^* h^{-1/2 \top}$ should be bounded, but when it is, we obtain $v_0^*h^{1/2} \otimes h^{1/2 \top} x^* h^{-1/2 \top} \in S_\tau$.  Compare Proposition~\ref{prop:flip_to_twisted_adj}, so here $ S_{i/2} = \{ \nabla^{-1/2} \alpha^* \nabla^{1/2} : \alpha \in S \cap D(\tilde\sigma_{i/2})\}$ and we don't need to intersect with $\mf n_{\varphi^{-1}}$.  However, notice that $S_{i/2}$ will not be norm-dense in $S_\tau$ as we cannot norm approximate an arbitrary $v_0^* h^{1/2} \otimes y \in S_\tau$.

We now compute the CB operator $A$, using Theorem~\ref{thm:Kraus_rep}.  Set $\alpha = v_0h^{1/2} \otimes 1$, so $\hat\alpha(\xi\otimes\overline\eta) = v_0 \otimes \overline\eta \otimes \xi$.  We may assume that $\|v_0\|_{HS}=1$, so $\hat\alpha$ is a (partial) isometry, with image all of $V$, and hence $(\hat\alpha)$ is an orthogonal basis of $\mc B_{M'}(H,V)$.  Thus $A(x) = \alpha x \alpha^*$ for each $x\in\mc B(K)$, remembering that this formula is interpretted in $\mc B(H)$.  More directly, $A(x) = (v_0h^{1/2}) x (v_0h^{1/2})^*$ for $x\in\mc B(K)$.  Then, given $a\in\mc B(K)$ with $ah^{1/2} \in HS(K)$, we have $A_0(ah^{1/2}) = A_0 \Lambda(a) = \Lambda(A(a)) = (v_0h^{1/2}) a (v_0h^{1/2})^* h^{1/2}$.  Equivalently, for $b\in HS(K)$ we have $A_0(b) = (v_0h^{1/2}) bh^{-1/2} (v_0h^{1/2})^* h^{1/2} = (v_0h^{1/2}) b (v_0^*h^{1/2})$.  So $A_0 = (v_0h^{1/2}) \otimes (v_0^*h^{1/2})^\top \in \mc B(H)$, and we immediately see that $A_0 \in S$.
\end{example}

\begin{example}\label{eg:Si4}
We continue Example~\ref{eg:BH_int_coint} and consider $S_{i/4}$ (see before Proposition~\ref{prop:Si4_when_int_coint}).  For the moment, suppose that $v_0 \in HS(K)$ with $v_0h^{1/2}$ bounded (so $e$ is integrable).  We give an example where $h^{1/4} v_0 h^{1/4}$ is unbounded.  Set $K=\ell^2$ with $h(\delta_n) = n^2\delta_n$ for each $n$.  Let $v_0 = \sum_k k^{-1} \rankone{\delta_k}{\delta_1}$, so $v_0\in HS(K)$ and $v_0h^{1/2} = v_0$ is certainly bounded.  However, $h^{1/4} v_0 h^{1/4} = \sum_k k^{-1} k^{1/2} \rankone{\delta_k}{\delta_1}$ is not bounded, as $(k^{-1/2})\not\in\ell^2$.  As $S = v_0h^{1/2} \otimes \mc B(\overline K)$, if $\alpha = v_0h^{1/2} \otimes y \in D(\tilde\sigma_{i/4})$ then $\nabla^{1/4} \alpha \nabla^{-1/4} = h^{1/4} v_0h^{1/2} h^{-1/4} h^{\top -1/4}y h^{\top 1/4}$ is bounded, contradiction, unless $\alpha=0$.  So $S_{i/4}=\{0\}$ in this case.

Now suppose additionally that $v_0^*h^{1/2}$ is bounded.  From Lemma~\ref{lem:domain_S_Tr_h} we have that $h^{1/2}v_0$ is bounded and everywhere defined, because $v_0^*h^{1/2}$ is bounded, and so $v_0(K) \subseteq D(h^{1/2}) \subseteq D(h^{1/4})$.  So $D(h^{1/4}v_0 h^{1/4}) = D(h^{1/4})$.  We claim that $x = v_0h^{1/2} \in D(\sigma_{-i/2})$, which is equivalent to $D(h^{1/2} x h^{-1/2}) = D(h^{-1/2})$ and $h^{1/2} x h^{-1/2}$ being bounded.  As $h^{1/2} x h^{-1/2} = h^{1/2} v_0 |_{D(h^{-1/2})}$, this is clear.  So also $x\in D(\sigma_{-i/4})$, and hence $h^{1/4} x h^{-1/4} = h^{1/4} v_0 h^{1/4}$ is bounded.

We now see that $S \cap D(\tilde\sigma_{i/4}) = \{ v_0h^{1/2} \otimes y : y \in D(\sigma^\op_{i/4}) \}$ where $\sigma^\op_t(y) = h^{\top it} y h^{\top -it}$ for $t\in\mathbb R$, this being compatible with Lemma~\ref{lem:L2Mop_L2M}, and using that $\nabla = h \otimes h^{\top -1}$.  Then $\tilde\sigma_{i/4}(v_0h^{1/2}\otimes y) = h^{1/4} v_0 h^{1/4} \otimes h^{\top -1/4} y h^{\top 1/4}$, and $S_{i/4}$ is exactly the set of such operators.  Similarly, $S_{\tau,i/4} = \{ h^{1/4} v_0^* h^{1/4} \otimes h^{\top -1/4} y h^{\top 1/4} : y \in D(\sigma^\op_{i/4}) \}$, and so it is not the case that $S_{i/4} ^* = S_{\tau,i/4}$, in general.

We now show that $S_{i/4}$ need not be a subset of $\mf n_{\varphi^{-1}}$.  We would need $h^{1/4} v_0 h^{1/4} h^{-1/2} = h^{1/4} v_0 h^{-1/4} \in HS(K)$, to which we give a counter-example.
Let $K = \ell^2(\mathbb Z)$ with $h$ defined by $h(\delta_n) = e^{2n} \delta_n$ for each $n$.  Define $v_0 = \sum_{k>0} k^{-1} \rankone{\delta_0}{\delta_{-k}} \in HS(K)$.  Then $v_0 h^{1/2} = \sum_{k>0} k^{-1} e^{-k} \rankone{\delta_0}{\delta_{-k}}$ is bounded, and $v_0^*h^{1/2} = v_0^* = \sum_{k>0} k^{-1} \rankone{\delta_{-k}}{\delta_0}$ is bounded.  However, $h^{1/4} v_0 h^{-1/4} = \sum_{k>0} k^{-1} e^{k/2} \rankone{\delta_0}{\delta_{-k}}$ is not even bounded, so certainly not Hilbert--Schmidt.
\end{example}

\begin{example}\label{eg:BH_get_HS_QRs}
Suppose we are given a general $V_0$, and consider forming $S$.  Set $\mc S = S\clos^{w^*}$ and $\mc S_0 = \{ a\in\mc B(K) : ah^{-1/2} \in V_0 \} \clos^{w^*}$.  From Lemma~\ref{lem:BH_alpha_nvarphiinv} and the remarks after, we know that $\{ a\otimes x : a\in\mc B(K), x\in\mc B(\overline K), ah^{-1/2} \in V_0 \subseteq HS(K) \} \subseteq S$, and so taking weak$^*$-closures shows that $\mc S_0 \vnten \mc B(\overline K) \subseteq \mc S$.  However, for $\alpha\in S$ and any $\omega\in\mc B(\overline K)_*$, we have that $(\id\otimes\omega)(\alpha) h^{-1/2} \in V_0$ and so $(\id\otimes\omega)(\alpha)\in\mc S_0$.  By weak$^*$-continuity this also holds for $\alpha\in\mc S$, and so using \eqref{eq:Slice_against_BK}, we conclude that $\mc S_0 \vnten \mc B(\overline K) = \mc S$.

Going the other way, given some $\mc S_0$ a weak$^*$-closed subspace of $\mc B(K)$, we have that $\mc S = \mc S_0 \vnten \mc B(\overline K)$ is a weak$^*$-closed $M'$-bimodule.  Set $S = \mc S \cap \mf n_{\varphi^{-1}}$ and $V = S \otimes_{M'} L^2(M') \cong \overline\lin\{ \im\hat\alpha : \alpha\in S \} \subseteq HS(H)$ (compare Section~\ref{sec:self-dual_S}).  As $\mf n_{\varphi^{-1}}$ is an $M'$-bimodule, also $S$ is an $M'$-bimodule, and so $V$ will be an $M'$-bimodule, Lemma~\ref{lem:alpha_hat_bimod}.  So $V = (V_0)_{13}$ for some $V_0\subseteq HS(K)$.  We now see that $V_0 = \lin\{ (\id\otimes\omega)(\alpha)h^{-1/2} : \alpha\in S, \omega\in\mc B(\overline K)_* \}\closhs$.  Given $a\in\mc S_0$ with $ah^{-1/2}\in HS(K)$, and given $x\in\mc B(\overline K)$, we have that $a\otimes x \in \mc S \cap \mf n_{\varphi^{-1}} = S$, and hence $ah^{-1/2} \in V_0$.  So $V_0 = \{ ah^{-1/2} : a\in\mc S_0, ah^{-1/2}\in HS(K)\}\closhs$.

In the special case when $h=1$, starting with $V_0$ we get that $\mc S_0 = V_0\clos^{w^*}$, while starting from $\mc S_0$ we get $V_0 = (\mc S_0 \cap HS(K))\closhs$, hence recovering the theory of Section~\ref{sec:HS}.
\end{example}

\begin{example}\label{eg:S_to_V_not_nice_but_int}
Set $K=\ell^2$ and $h=1$.  By using Example~\ref{eg:1}, this gives an example of a $V$ such that the related $S$ given by Proposition~\ref{prop:V_to_S} has that $S$ is a proper subspace of $\mc S \cap \mf n_{\varphi^{-1}}$ (where $\mc S = \overline{S}^{w^*}$).  Indeed, let $\kappa = (\sum_n n^{-2})^{-1/2}$ and $\kappa_n = \kappa n^{-1}$ so $\sum_n \kappa_n^2 = \kappa^2 \sum_n n^{-2} = 1$.  Then set $V_0 = \{ \sum_n x_n e_{nn} :  \sum_n x_n \kappa_n = 0, x=(x_n)\in\ell^2 \}$, and following Example~\ref{eg:1}, we find that $\mc S = \ell^\infty \vnten\mc B(\overline K)$, and that $\mc S \cap \mf n_{\varphi^{-1}}$ is strictly larger than $S$, because the subspace of $HS(K)$ associated to this is $\ell^2$, thought of as the diagonal of $HS(K)$.

Let $e\in\mc B(K)\vnten\mc B(\overline K)$ be the projection onto $V_0$.  We shall show that $e$ is integrable.  Given $\xi\in K$ consider
\[ \Tr((\omega_\xi\otimes\id)(e)) = \sum_n (\overline\delta_n|(\omega_\xi\otimes\id)(e)\overline\delta_n)
= \sum_n \| e(\xi\otimes\delta_n) \|^2. \]
For a given $n$, we compute $e(\xi\otimes\overline\delta_n)$ which is the orthogonal projection of $\xi\otimes\overline\delta_n$ onto $V_0$, say $e(\xi\otimes\delta_n) = \sum_k x_k e_{kk} \in V_0$.  This is determined by the property that $(\sum y_ke_{kk}|\xi\otimes\overline\delta_n - \sum_k x_ke_{kk})=0$ for each $\sum_k y_k e_{kk} \in V_0$.  As $y_n = -\kappa_n^{-1}\sum_{k\not=n} y_k\kappa_k$, we have
\begin{align*}
\overline{y_n} \xi_n = \sum_k \overline{y_k}x_k  &\implies
- (\xi_n-x_n) \kappa_n^{-1} \sum_{k\not=n} \overline{y_k} \kappa_k = \sum_{k\not=n} \overline{y_k}x_k,
\end{align*}
this now being true for any choice of $(y_k)_{k\not=n}\in\ell^2$.  So $x_k = -\kappa_n^{-1}(\xi_n-x_n)\kappa_k$ for each $k\not=n$, and
\[ x_n = -\kappa_n^{-1} \sum_{k\not=n} x_k\kappa_k
= \kappa_n^{-1} \sum_{k\not=n} \kappa_k^2 \kappa_n^{-1}(\xi_n-x_n)
= \kappa_n^{-2}(\xi_n-x_n) (1-\kappa_n^2)
= (x_n - \xi_n) (1 - \kappa_n^{-2}). \]
So $x_n = \xi_n(1-\kappa_n^2)$ and hence $x_k = -\xi_n\kappa_n\kappa_k$ for $k\not=n$.  Thus
\[ \|e(\xi\otimes\overline\delta_n)\|^2 = \sum_k |x_k|^2 = |\xi_n|^2(1-\kappa_n^2)^2 + \sum_{k\not=n} |\xi_n|^2 \kappa_n^2 \kappa_k^2
= |\xi_n|^2 \big( (1-\kappa_n^2)^2 + \kappa_n^2(1-\kappa_n^2) \big)
= |\xi_n|^2 (1-\kappa_n^2), \]
and so $\Tr((\omega_\xi\otimes\id)(e)) = \sum_n |\xi_n|^2 (1-\kappa_n^2) < \sum_n |\xi_n|^2= \|\xi\|^2$.  It follows that $(\id\otimes\Tr)(e)$ is finite, so $e$ is integrable.

As $V_0$ is diagonal and self-adjoint, clearly $\tau(e)=e$, and so $e$ is not only integrable, but induces a Hilbert space operator, see Section~\ref{sec:hilb_ops}.  Let $x\in\mc B(\ell^2) \setminus HS(\ell^2)$ with $x^*(\kappa) = 0$ (for example, $x = 1-\rankone{\kappa}{\kappa}$).  Then define $\alpha\in\mc B(H)$ by
\[ (\delta_s \otimes \overline\delta_i|\alpha(\delta_t\otimes\overline\delta_j))
= \delta_{i,1} \delta_{s,t} (\delta_s|x\delta_j). \]
Then $(\id\otimes\omega_{\overline\delta_i,\overline\eta})(\alpha) = \delta_{i,1} \sum_{j,s} \overline{\eta_j} (\delta_s|x\delta_j) e_{s,s} = \delta_{i,1} \sum_s (\delta_s|x(\overline\eta)) e_{s,s}$ where we identify $\overline\eta\in\ell^2$ with the pointwise conjugation.  Then $\sum_s \kappa_s (\delta_s|x(\overline\eta)) = (\kappa|x(\overline\eta)) = 0$ showing that $(\id\otimes\omega_{\overline\delta_i,\overline\eta})(\alpha) \in V_0$.  Also
\[ \sum_i \| (\id\otimes\omega_{\overline\delta_i,\overline\eta})(\alpha) \|_{HS}^2
= \sum_{i,s} \delta_{i,1}  |(\delta_s|x(\overline\eta))|^2
= \|x(\overline\eta)\|^2 \leq \|x\|^2 \|\overline\eta\|^2, \]
so by Lemma~\ref{lem:BH_alpha_nvarphiinv} we have that $\alpha\in\mf n_{\varphi^{-1}}$, so $\alpha\in S$.  Now consider $\alpha^*$, which has $(\id\otimes\omega_{\overline\delta_j, \overline\eta})(\alpha^*) = (\id\otimes\omega_{\overline\eta, \overline\delta_j})(\alpha)^* = \overline{\eta_1} \sum_s (\delta_s|x\delta_j) e_{ss}$, and so
\[ \sum_j \| (\id\otimes\omega_{\overline\delta_j, \overline\eta})(\alpha^*) \|_{HS}^2
= |\eta_1|^2 \sum_{j,s} |(\delta_s|x\delta_j)|^2 
= |\eta_1|^2 \|x\|_{HS}^2, \]
which by assumption is not finite, when $\eta_1\not=0$.  This shows that the intersection with $\mf n_{\varphi^{-1}}$ in Proposition~\ref{prop:flip_to_twisted_adj} is necessary.
\end{example}

We now consider the possibility of a map $\theta \colon \mc B(H) \to \mc B(H)$ as in Section~\ref{sec:cb_projs}.  At minimum, $\theta$ should be bounded, normal and an $M'$-bimodule map, so $\theta((1\otimes a)x(1\otimes b)) = (1\otimes a) \theta(x) (1\otimes b)$ for $a,b\in\mc B(\overline K)$.  Taking $a,b$ to be rank-one readily leads to the conclusion that there is a bounded normal map $\theta_0 \colon \mc B(K) \to \mc B(K)$ with $\theta(x\otimes y) = \theta_0(x) \otimes y$ for each $x\in\mc B(K), y\in\mc B(\overline K)$.   We would further like a condition of the form $\theta(\alpha)\floatinghat = e \hat\alpha$ for ``suitable'' $\alpha$.  With Lemma~\ref{lem:BH_alpha_nvarphiinv} in mind, this is equivalent to $e (ah^{-1/2}) = \theta_0(a)h^{-1/2} \in HS(K)$ for each $a\in B(K)$ with $ah^{-1/2}\in HS(K) = K\otimes\overline K$.  As $\theta_0$ is normal, this will then determine $\theta_0$ on all of $\mc B(K)$.

\begin{example}\label{eg:no_theta_A}
We consider Example~\ref{eg:BH_int_coint} again, so $V_0 = \mathbb C v_0$ and hence $e = \rankone{v_0}{v_0}$.  Then $(v_0|ah^{-1/2})_{HS} v_0 = \theta_0(a)h^{-1/2}$ for each $a\in\mc B(K)$ with $ah^{-1/2}\in HS(K)$.  So for $\xi\in K, \eta\in D(h^{-1/2})$, set $a=\rankone{\xi}{\eta}$ to see that $(h^{-1/2}\eta|v_0^*(\xi)) v_0 = \theta_0(\rankone{\xi}{\eta}) h^{-1/2}$.  This implies that $v_0 h^{1/2}\in\mc B(K)$, that is, $e$ is integrable.  Then $\theta_0$ being bounded implies that $\eta \mapsto v_0h^{-1/2}\eta$ is bounded, so $v_0 h^{-1/2} \in \mc B(K)$, equivalently, $h^{-1/2}v_0^*\in\mc B(K)$, Lemma~\ref{lem:domain_S_Tr_h}.  In this case, $\theta_0(a) = \Tr(h^{-1/2}v_0^* a) v_0h^{1/2}$ for each finite-rank $a$, and so $\theta_0$ being bounded is equivalent to $h^{-1/2}v_0^*$ being trace-class (or $v_0h^{-1/2}$ trace-class).  When this holds, $\theta_0$ obviously extends to $\mc B(K)$ with the same formula.  Thus, even for integrable $e$, that $\theta_0$ exists is a much stronger condition.  Notice that this remains true even if $h=1$.  However, notice that if $\varphi^{-1}$ is bounded, that is, $h^{-1/2} \in HS(K)$, then when $e$ is integrable, $\theta_0$ will be bounded.
\end{example}

\begin{example}\label{eg:no_theta_A_even_varphiinv_bdd}
Suppose that $\varphi^{-1}$ is bounded, so $h^{-1}$ is trace-class, and we can find an orthonormal basis $(e_n)$ of $K$ with $h^{-1}(e_n) = h_n^{-1} e_n$ for some summable sequence $(h_n^{-1})$ in $(0,1)$.  Let $(f_{i,j})$ be some $0,1$-valued matrix, so there is a projection $e\in\mc B(HS(K))$ with $e(e_i \otimes \overline{e_j}) = f_{i,j} e_i \otimes \overline{e_j}$.  Suppose there is a bounded $\theta_0$ with $e(ah^{-1/2}) = \theta_0(a)h^{-1/2}$ for each $a\in\mc B(K)$ (recalling that $h^{-1/2}\in HS(K)$ here).  Then
\[ f_{i,j} \rankone{e_i}{e_j} = e(\rankone{e_i}{e_j}h^{-1/2}) h_j^{1/2} = \theta_0(\rankone{e_i}{e_j}) h^{-1/2} h_j^{1/2}
\implies \theta_0(\rankone{e_i}{e_j}) = f_{i,j} \rankone{e_i}{e_j}. \]
Thus $\theta_0$ is Schur multiplier with symbol $f$, compare Section~\ref{sec:ellinfty}, and notice that $h$ plays no role here.  In particular, the examples given in Section~\ref{sec:ellinfty} now provide examples of $e$ such that $\theta_0$ cannot exist, independent of the choice of $h$.
\end{example}

\begin{remark}\label{rem:no_theta_A_matrix}
Using the idea of Example~\ref{eg:matrix_sd_notint}, we could readily adapt this example to the case of $M = \prod_i M_{n(i)}$ when the dimensions $(n(i))$ are unbounded.
\end{remark}

\appendix

\section{Weights}\label{sec:weights}

We prove some technical results about (always assumed nfs) weights on von Neumann algebras.  These results are surely known to experts, but we have not found exact references.  We also state a few results which are quoted from the literature, for ease of reference.  Let $M$ be a von Neumann algebra and let $\varphi$ be a nfs weight on $M$.

We recall the meaning of \emph{Tomita algebra} from \cite[Section~VI.2]{TakesakiII} or \cite[Sections~10.20, 10.21]{StratilaZsido2nd}.  An element $x\in M$ is \emph{analytic} when the one-parameter map $\mathbb R \ni t \mapsto \sigma_t(x)$ can be extended to an analytic function $\mathbb C \to M$.  Recall the \emph{smearing} technique: for $x\in M$ define
\[ x_n = \frac{n}{\sqrt\pi} \int_{\mathbb R} \exp(-n^2t^2) \sigma_t(x) \ dt \qquad (n>0). \]
Then $x_n$ is analytic, and $x_n\to x$ weak$^*$ as $n\to \infty$.  This technique is used in, for example, \cite[Section~10.21]{StratilaZsido2nd} and \cite[Section~2.16]{Stratila_ModTheoryBook2}.

Let us make a few comments about analytic extensions of $(\sigma_t)$ (or more general one-parameter $*$-automorphism groups).  For $z\in\mathbb C$ define the horizontal strip $S(z) = \{ w\in\mathbb C : |\im w| \leq |\im z| \}$.  Then $x\in D(\sigma_z)$ when there is a $\sigma$-weakly continuous function $F\colon S(z) \to M$ which is analytic on the interior of $S(z)$, and with $F(t) = \sigma_t(x)$ for each $t\in\mathbb R$.  Such an $F$ is necessarily unique if it exists; we set $\sigma_z(x) = F(z)$.  As $\sigma_t(x) = \nabla^{it} x \nabla^{-it}$ for $t\in\mathbb R$, by \cite[Proposition~9.24]{StratilaZsido2nd}, see also \cite[Theorem~6.2]{CZ_Analytic_Generators}, we have the alternative description of $D(\sigma_z)$ as those $x$ for which $D(\nabla^{iz} x \nabla^{-iz}) = D(\nabla^{-iz})$ (it suffices for $D(\nabla^{iz} x \nabla^{-iz})$ to be a core for $\nabla^{-iz}$) and $\nabla^{iz} x \nabla^{-iz}$ is bounded, the continuous extension being $\sigma_z(x)$.

\begin{lemma}\label{lem:sigma_nabla_commutation}
Let $x \in D(\sigma_z)$.  Then $\nabla^{-iz}\sigma_z(x) \supseteq x \nabla^{-iz}$.
\end{lemma}
\begin{proof}
As just stated, $D(\nabla^{iz} x \nabla^{-iz}) = D(\nabla^{-iz})$ and $\nabla^{iz} x \nabla^{-iz}$ is bounded.  So for $\xi \in D(\nabla^{-iz})$ we have that $\sigma_z(x)\xi = \nabla^{iz} x \nabla^{-iz}\xi$ and hence $\nabla^{-iz}\sigma_z(x)\xi = x \nabla^{-iz}\xi$.  Thus $\nabla^{-iz}\sigma_z(x) \supseteq x \nabla^{-iz}$.
\end{proof}

\begin{lemma}[{\cite[Lemma~VIII.3.18]{TakesakiII}}]
\label{lem:action_J}
We have that $a\in D(\sigma_{i/2})$ if and only if $\varphi(a^*xa) \leq \|\sigma_{i/2}(a)\| \varphi(x)$ for each $x\in M_+$, and when this holds, $\Lambda(xa) = J\sigma_{i/2}(a)^*J \Lambda(x)$ for $x\in \mf n_\varphi$.
\end{lemma}

\begin{lemma}[{\cite[Proposition~2.17]{Stratila_ModTheoryBook2}}]
\label{lem:mod_aut_usage}
Let $z\in\mathbb C$.  For $x\in\mf n_\varphi^* \cap D(\sigma_{z-i})$ with $\sigma_{z-i}(x)\in\mf n_\varphi$, and $y\in\mf n_\varphi\cap D(\sigma_z)$ with $\sigma_z(y)\in \mf n_\varphi^*$, we have $\varphi(xy) = \varphi(\sigma_z(y) \sigma_{z-i}(x))$.  Taking $z=i/2$ shows that $\varphi(xy) = \varphi(\sigma_{i/2}(y) \sigma_{-i/2}(x))$ for suitable $x,y$.  Taking $z=i$ shows that $\varphi(xy) = \sigma(\sigma_i(y) x)$ for suitable $x,y$.
\end{lemma}

\begin{lemma}\label{lem:form_J}
Let $x\in \mf n_\varphi \cap D(\sigma_{i/2})$ with $\sigma_{i/2}(x) \in \mf n_\varphi^*$.  Then $J\Lambda(x) = \Lambda(\sigma_{i/2}(x)^*)$.
\end{lemma}
\begin{proof}
We argue using the ideas of the proof of \cite[Proposition~2.17]{Stratila_ModTheoryBook2}.  As $x^* \in D(\sigma_{-i/2})$ we have $\Lambda(yx) = J \sigma_{-i/2}(x^*) J \Lambda(y)$ for each $y\in\mf n_\varphi$.  By \cite[Proposition~2.16]{StratilaZsido2nd} we can find a net of analytic elements $(a_\alpha)$ with $a_\alpha \in \mf n_\varphi$ (in fact, $\Lambda(a_\alpha)$ in the Tomita algebra) and with $\sigma_z(a_\alpha) \to 1$ $\sigma$-strong$^*$ for each $z\in\mathbb C$.

As $\Lambda(a_\alpha)$ is in the Tomita algebra, $J\Lambda(a_\alpha) = \nabla^{1/2} S \Lambda(a_\alpha) = \nabla^{1/2}\Lambda(a_\alpha^*) = \Lambda(\sigma_{-i/2}(a_\alpha^*))$; for this last point compare \cite[Section~2.15]{Stratila_ModTheoryBook2}.  So
\begin{align*}
a_\alpha \Lambda(x)
&= \Lambda(a_\alpha x)
= J \sigma_{-i/2}(x^*) J \Lambda(a_\alpha)
= J \sigma_{-i/2}(x^*) \Lambda(\sigma_{-i/2}(a_\alpha^*))
= J \Lambda( \sigma_{-i/2}(x^*) \sigma_{-i/2}(a_\alpha^*)) \\
&= J J \sigma_{-i/2}(\sigma_{i/2}(a_\alpha)) J \Lambda( \sigma_{-i/2}(x^*) )
= a_\alpha J \Lambda( \sigma_{-i/2}(x^*) ).
\end{align*}
Here we use that $x\in\mf n_\varphi$ and that $\sigma_{-i/2}(x^*) \in \mf n_\varphi$.  Taking the limit shows that $\Lambda(x) = J \Lambda( \sigma_{-i/2}(x^*) )$ as claimed.
\end{proof}

\begin{lemma}\label{lem:sigmaiswap}
Let $x\in\mf m_\varphi$ and $a\in D(\sigma_{-i})$.  Then $ax, x\sigma_{-i}(a) \in \mf m_{\varphi}$ and $\varphi(ax) = \varphi(x\sigma_{-i}(a))$.
\end{lemma}
\begin{proof}
By linearity we may suppose that $x = b^*c$ for some $b,c\in\mf n_\varphi$.
As $a\in D(\sigma_{-i/2})$, from Lemma~\ref{lem:action_J}, $ba^* \in \mf n_\varphi$ with $\Lambda(ba^*) = J\sigma_{-i/2}(a)J\Lambda(b)$.  Also $d = \sigma_{-i}(a)^* = \sigma_i(a^*) \in D(\sigma_{-i/2})$ with $\sigma_{-i/2}(d) = \sigma_{i/2}(a^*) = \sigma_{-i/2}(a)^*$, so $cd^* \in \mf n_\varphi$ with $\Lambda(cd^*) = J\sigma_{-i/2}(d)J\Lambda(c)$.  That is, $\Lambda(c\sigma_{-i}(a)) = J\sigma_{-i/2}(a)^*J\Lambda(c)$.  So
\begin{align*}
\varphi(ax) &= \varphi(xb^*c) = (\Lambda(ba^*)|\Lambda(c))
= (J\sigma_{-i/2}(a)J\Lambda(b)|\Lambda(c))
= (\Lambda(b)|J\sigma_{-i/2}(a)^*J\Lambda(c)) \\
&= (\Lambda(b)|\Lambda(c\sigma_{-i}(a)))
= \varphi(b^*c\sigma_{-i}(a)) = \varphi(x\sigma_{-i}(a)),
\end{align*}
as claimed.
\end{proof}

We make some remarks about the various locally convex topologies on a von Neumann algebra.  Given a von Neumann algebra $N \subseteq \mc B(H)$ we have the strong and weak topologies, but these depend upon the choice of $H$.  The $\sigma$-strong and $\sigma$-weak topologies are independent of $H$.  However, given any $H$, if we dilate and replace $H$ by $H\otimes\ell^2$ and replace $N$ by $N\otimes 1$, then the strong and $\sigma$-strong topologies agree, and the same for the weak, and indeeed strong$^*$, topologies.  Furthermore, for $M$ acting on $L^2(\varphi)$, as every $\omega\in M_*$ is of the form $\omega_{\xi,\eta}$, the strong and $\sigma$-strong topologies agree, and so forth.

We now seek a way to find $\varphi$ using $L^2(\varphi)$, in the precise sense of the following.

\begin{proposition}\label{prop:motivation_approx_weight}
Let $(a_i)$ be a bounded net in $D(\sigma_{i/2}) \cap \mf n_\varphi$ with $a_i \to 1$ and $\sigma_{i/2}(a_i)^* \to 1$ $\sigma$-strongly.  Then
\[ \lim_i \| x^{1/2} \Lambda(a_i) \|^2 = \lim_i \varphi(a_i^* x a_i) = \varphi(x) \qquad (x\in M_+). \]
\end{proposition}
\begin{proof}
For $x\in\mf n_\varphi$, we see that
\begin{align*}
\| x \Lambda(a_i) \|^2 &= \| \Lambda(xa_i) \|^2 =
\| J\sigma_{i/2}(a_i)^*J \Lambda(x) \|^2,
\end{align*}
using Lemma~\ref{lem:action_J}.
By hypothesis, $\lim_i \sigma_{i/2}(a_i)^*J \Lambda(x) = J \Lambda(x)$ and so 
\[ \lim_i \| x \Lambda(a_i) \|^2 = \|J\Lambda(x)\|^2 = \|\Lambda(x)\|^2 = \varphi(x^*x). \]
We conclude that for $x\in\mf m_\varphi^+$ we have $\lim_i \| x^{1/2} \Lambda(a_i) \|^2 = \varphi(x)$.

Now let $x\in M_+$ with $\varphi(x)=\infty$.  As $a_i\to 1$ in the $\sigma$-strong topology, we have that $a_i^* x a_i \to x$ $\sigma$-weakly, and as $\varphi$ is $\sigma$-weakly lower semicontinuous, \cite[Theorem~VII.1.11]{TakesakiII}, we have
\[ \infty = \varphi(x) \leq \liminf_i \varphi(a_i^*xa_i), \]
and so $\lim_i \varphi(a_i^*xa_i) = \infty$ as desired.
\end{proof}

\begin{proposition}\label{prop:tomita_nice_bai}
There is a net $(a_i)$ of unit vectors, consisting of analytic elements, such that for each $z\in\mathbb C$, we have $\sigma_z(a_i)\in \mf n_\varphi\cap \mf n_\varphi^*$, and $\sigma_z(a_i) \to 1$ $\sigma$-strong$^*$.  Then $\lim_i \varphi(\sigma_z(a_i)^*x\sigma_z(a_i)) = \varphi(x)$ for each $x\in M_+$ or $x\in\mf m_\varphi$.
\end{proposition}
\begin{proof}
By \cite[Proposition~2.16]{Stratila_ModTheoryBook2}, there is a net $(a_i)$ of analytic elements such that $\Lambda(a_i)$ is in the Tomita algebra, for each $i$, so certainly $a_i \in \mf n_\varphi \cap \mf n_\varphi^*$, and such that $\sigma_z(a_i) \to 1$ $\sigma$-strong$^*$, for each $z$.  By the construction of that result, we may further suppose that each $a_i$ is a unit vector.  Then $(\sigma_z(a_i))$ verifies the hypotheses of Proposition~\ref{prop:motivation_approx_weight}.  When $x\in\mf m_\varphi$ we have that $x$ is in the linear span of $\mf m_\varphi^+$, and so $\lim_i \varphi(\sigma_z(a_i)^*x\sigma_z(a_i)) = \varphi(x)$ by linearity.
\end{proof}

We remark that the net $(a_i)$ is constructed by starting with an approximate identity in $\mf n_\varphi$ and smearing.

\section{Self-dual modules}\label{sec:selfdual_mods}

The results in this section go back to Rieffel's work in \cite{Rieffel_mortia_cstar_wstar} and Paschke's in \cite{paschke_inner_prod_mods}.  See also \cite{BDH_IndexCondExp}.
We follow the notation of Section~\ref{sec:hilb_mods}.

\begin{proposition}[{\cite[Theorem~6.5]{Rieffel_mortia_cstar_wstar}}]
\label{prop:std_mod_self_dual}
$\mc B_{M}(L^2(M), H)$ is self-dual.
\end{proposition}

There is a tight link between Hilbert $C^*$-modules and operator space theory, which is explored in \cite{BlecherLeMerdy_Book}.  In particular, \cite[Lemma~8.5.4(1)]{BlecherLeMerdy_Book} tells us that a Hilbert $C^*$-module $E$ over $M$ is self-dual if and only if $E$ has a predual making the inner-product separately weak$^*$-continuous.  As $\mc B_{M}(L^2(M), H)$ is weak$^*$-closed in $\mc B(L^2(M), H)$ and the inner-product is easily seen to be separately weak$^*$-continuous, this gives another way to show the result.  We remark that then $\mc L(E)$ becomes a von Neumann algebra: see also \cite[Proposition~3.10]{paschke_inner_prod_mods}.

The interior tensor product $E \otimes_M L^2(M)$ is the Hilbert space formed as the separation completion of $E\odot L^2(M)$ for the pre-inner-product
\[ (\alpha\otimes \xi | \beta\otimes \eta) = (\xi|(\alpha|\beta) \eta) \qquad (\alpha,\beta\in E, \xi,\eta\in L^2(M)). \]
See \cite[Chapter~4]{Lance_HilbModsBook}.  We have that $E\odot L^2(M)$ is a right $M$-module for the right action on $L^2(M)$ (we do not see that this is obvious, but it can be shown using the techniques of \cite[Lemma~4.2 and Page~42]{Lance_HilbModsBook}).

\begin{proposition}[{\cite[Theorem~6.12]{Rieffel_mortia_cstar_wstar}}]
\label{prop:all_self_dual_mods}
Any self-dual $E$ is of the form $\mc B_{M}(L^2(M), H)$ with $H = E\otimes_M L^2(M)$.
\end{proposition}

If $\pi \colon N \to \mc L(E)$ is a normal unital $*$-homomorphism (for some von Neumann algebra $N$) then $\pi$ induces a normal unital $*$-homomorphism $\hat\pi \colon N \to \mc B(H) =\mc B(E\otimes_M L^2(M))$, again see \cite[Page~42]{Lance_HilbModsBook}, for example.  Thus $\mc B(H)$ is an $N$-$M$-bimodule, or a \emph{correspondence}, \cite[Chapter~IX.3]{TakesakiII} and \cite[Chapter~5, Appendix~B]{Connes_NCG_Book}.  This point of view is explored in \cite{BDH_IndexCondExp}.

For a projection $p\in M$, we have that $E = pM$ is a self-dual module, for the inner-product $(px|py) = x^*py \in M$: indeed, then $H = p L^2(M)$.  Given projections $(p_i)_{i \in I}$ in $M$, we define $\bigoplus^w_i p_i M$ to be the space of families $(p_i x_i)$ with $x_i\in M$, such that $\sum_i (p_i x_i | p_i x_i)$ converges in $M$ (notice that this is an increasing sum of positives, so it converges weak$^*$ if and only if it is bounded above).

\begin{theorem}[{\cite[Theorem~3.12]{paschke_inner_prod_mods}}]
\label{thm:selfdual_is_weak_direct_sum}
For a self-dual $E$ there is a collection $(e_i)$ in $E$ such that $(e_i|e_i) = p_i$ is a projection in $M$ for each $i$, we have $(e_i|e_j)=0$ for $i\not=j$, and $E$ is isomorphic to $\bigoplus_i^w p_i M$.
\end{theorem}

It follows that for such $(e_i)$ we have $x = \sum x_i \cdot (x_i|x)$ for $x\in E$, convergence in the weak$^*$-sense.  We hence term $(e_i)$ an \emph{orthogonal basis} for $E$.

With $E = \mc B_M(L^2(M), H)$, note that each $e_i$ from the theorem will be a partial isomerty $L^2(M) \to H$, having pairwise disjoint ranges, and such that the linear span of the ranges is dense in $H$.  One can find such a family by a maximality argument: the key step is the following.  Let $H_0 \subseteq H$ be a (right) $M$-submodule, and let $\eta\in H_0$ be a unit vector.  Then $M^\op \ni x \mapsto (\eta|\eta\cdot x)$ is a normal state, and as $M^\op$ is isomorphic to $M'$ via $x^\op \mapsto Jx^*J$, and $M'$ is in standard position on $L^2(M)$, we find a unit vector $\xi\in L^2(M)$ with $(\eta|\eta\cdot x) = (\xi |Jx^*J\xi)$ for each $x\in M$.  Then $u \colon Jx^* J\xi \mapsto \eta\cdot x$ is a partial isometry (defined to be $0$ on $(JMJ\xi)^\perp$) which is in $\mc B_M(L^2(M), H)$.

Given a Hilbert $C^*$-module $E$, Paschke gives a procedure for embedding $E$ in a larger self-dual module, the \emph{self-dual completion} $E'$, see \cite[Theorem~3.2]{paschke_inner_prod_mods}.

\begin{theorem}[{\cite[Proposition~6.10]{Rieffel_mortia_cstar_wstar}}]
\label{thm:self_dual_completion}
For any $E$, we have that $E' \cong \mc B_M(L^2(M), E\otimes_M L^2(M))$, or equivalently, $E' \otimes_M L^2(M) \cong E \otimes_M L^2(M)$.
\end{theorem}

We can embed $E$ into $\mc B_M(L^2(M), E\otimes_M L^2(M)) \cong E'$ by $E\in \alpha \mapsto \hat\alpha$ where $\hat\alpha(\xi) = \alpha\otimes\xi$ for $\xi\in L^2(M)$.

\begin{proposition}\label{prop:adjointables_extend_sdc}
For any $E$, let $H=E\otimes_ML^2(M)$ and $\rho\colon M^\op\to\mc B(H)$ the right action of $M$ on $H$.  Given $T\in\mc L(E)$ there is $t\in\rho(M^\op)'$ with $t \circ \hat\alpha = \wh{T(\alpha)}$ for $x\in E$.  The map $T\mapsto t$ is a $*$-homomorphism.  When $E$ is self-dual, this gives an isomorphism $\mc L(E) \cong \rho(M^\op)'$.
\end{proposition}
\begin{proof}
As discussed above, any $T\in\mc L(E)$ induces a map on $E\otimes_M L^2(M)$ given by $x\otimes\xi \mapsto T(x)\otimes\xi$.  Let this map be $t$, which obviously commutes with the right action of $M$, so $t\in\rho(M^\op)'$.  Then $T\mapsto t$ is a $*$-homomorphism.

When $E$ is self-dual, we have that $E\cong\mc B_M(L^2(M), H)$.  For $t\in \rho(M^\op)'$ and $\alpha\in\mc B_M(L^2(M), H)$ we have that $(t\circ \alpha)(\xi\cdot a) = t(\alpha(\xi\cdot a)) = t(\alpha(\xi)\cdot a) = (t\circ \alpha)(\xi) \cdot a$ for each $a\in M, \xi\in L^2(M)$, so that $t\circ \alpha$ is also in $\mc B_M(L^2(M), H)$.  We have for $\alpha,\beta\in \mc B_M(L^2(M), H)$ that $(\alpha|t\circ \beta) = \alpha^*t\beta = (t^*\circ \alpha|\beta)$ so the map $\alpha\mapsto t\circ \alpha$ is adjointable, with adjoint $\beta\mapsto t^*\circ \beta$.  The claim that $\mc L(E) \cong \rho(M^\op)'$ is hence equivalent to showing that if $t\not=0$ then the map $\alpha\mapsto t\circ \alpha$ is non-zero.  This follows, for example, by the argument after Theorem~\ref{thm:selfdual_is_weak_direct_sum}, which shows in particular that $\lin\{ \alpha(\xi) : \xi\in L^2(M), \alpha\in \mc B_M(L^2(M), H) \}$ is dense in $H$.
\end{proof}

In particular, this shows that any $\mc L(E)$ embeds into $\mc L(E')$, see also \cite[Corollary~3.7]{paschke_inner_prod_mods}.

It is easy to see that $\mc B_M(L^2(M), H)$ is weak$^*$-closed in $\mc B(L^2(M), H)$ and so carries a natural weak$^*$-topology, which agrees with that constructed by Paschke on self-dual modules, \cite[Proposition~3.8]{paschke_inner_prod_mods}.  Rieffel claims without proof, see before \cite[Proposition~6.10]{Rieffel_mortia_cstar_wstar}, that $E$ is weak$^*$-dense in $\mc B_M(L^2(M), E\otimes L^2(M))$, but we do not immediately see this.  Absent a suitable reference, we give a sketch proof of the next result, using ideas from \cite{Skeide_vnmods_ints}, for example.

\begin{proposition}\label{prop:wstar_dense_nondeg}
Let $E \subseteq \mc B_M(L^2(M), H)$ be a sub-module.  Then $E$ is weak$^*$-dense if and only if $\lin\{ \alpha(\xi) : \xi\in L^2(M), \alpha\in E \}$ is dense in $H$.
\end{proposition}
\begin{proof}
Set $F = \mc B_M(L^2(M), H)$.  We first suppose that $\lin\{ \alpha(\xi) : \xi\in L^2(M), \alpha\in E \}$ in dense in $H$.  Consider the ``linking algebra'',
\[ L = \Big\{ \begin{pmatrix} x & \beta^* \\ \alpha & t \end{pmatrix} : x\in M, \alpha,\beta\in E, t \in \mc K(E) \Big\}
\subseteq \mc B(L^2(M) \oplus H). \]
Here $\mc K(E)$ is the ``compact'' operators on $E$, the closed linear span of $\{ \alpha\beta^* : \alpha,\beta\in E \}$ in $\mc B(H)$.
It is easy to see that $L$ is a $*$-algebra.
Let $\rho \colon M^\op \to \mc B(H)$ be the $*$-homomorphism which turns $H$ into a right $M$-module.  A calculation shows that be commutant and bicommutant are
\[ L' = \Big\{ \begin{pmatrix} Jy^*J & 0 \\ 0 & \rho(y^\op) \end{pmatrix} : y^\op \in M^\op \Big\},
\qquad
L'' = \Big\{ \begin{pmatrix} x & \beta^* \\ \alpha & t \end{pmatrix} : x\in M, \alpha,\beta\in F, t \in \mc L(F) \cong \rho(M^\op)' \Big\}, \]
see Proposition~\ref{prop:adjointables_extend_sdc}.  An application of the Kaplansky Density Theorem shows that the unit ball of $L$ is weak$^*$-dense in the unit ball of $L''$, and it follows that $E$ is weak$^*$-dense in $F$.

Converserly, suppose that $E$ is weak$^*$-dense in $F$.  As $F$ is self-dual, Proposition~\ref{thm:self_dual_completion} implies that $\lin\{ \alpha(\xi) : \alpha\in F, \xi\in L^2(M) \}$ is dense in $H$.  Let $\eta\in \{ \alpha(\xi) : \alpha\in E, \xi\in L^2(M) \}^\perp$, so by weak$^*$-continuity, $(\eta|\alpha(\xi)) = 0$ for each $\alpha\in F$ as well, and hence $\eta=0$.  The claim follows.
\end{proof}

\section{Operator-valued weights}\label{sec:op_valued_weights}

We use the notation introduced in Definition~\ref{defn:module_from_opvalweight} and the discussion after, so $T$ is an operator-valued weight from $M$ to $N\subseteq M$.  While the module $E^0_T$ was defined in \cite{BDH_IndexCondExp}, our main inspiration here comes from Kustermans's treatment of the $C^*$-algebra case in \cite{Kustermans_RegCstarWeights}; we ask some related questions below.

Given $\alpha, \beta\in\mf n_T, x,y\in\mf n_\phi$ we see that
\begin{align*}
(\alpha\otimes\Lambda_\phi(x)|\beta\otimes\Lambda_\phi(y))
&= (\Lambda_\phi(x)|T(\alpha^*\beta)\Lambda_\phi(y))
= \phi(x^* T(\alpha^* \beta) y)
= \phi(T(x^*\alpha^*\beta y)) \\
&= \tilde\phi(x^*\alpha^*\beta y)
= (\Lambda_{\tilde\phi}(\alpha x)|\Lambda_{\tilde\phi}(\beta y)).
\end{align*}
Here we used the $N$-bimodule property of $T$.  In the last step, we should really first note that the calculation shows that $\tilde\phi(x^*\alpha^*\alpha x) = \|\alpha\otimes\Lambda_\phi(x)\|^2 < \infty$, so indeed $\alpha x\in \mf n_{\tilde\phi}$, and the same for $\beta y$.
Thus $\mf n_T \odot \Lambda_\phi(\mf n_\phi) \to L^2(M); x\otimes\Lambda_\phi(a) \mapsto \Lambda_{\tilde\phi}(xa)$ is an isometry, and so extends to an isometry $E_T^0 \otimes L^2(N) \to L^2(M)$.

\begin{proposition}\label{prop:HT_is_L2_tildephi}
This map is surjective, thus establishing a unitary between $E^0_T\otimes L^2(N)$ and $L^2(M)$.
\end{proposition}
\begin{proof}
We follow the idea of \cite[Lemma~IX.4.21]{TakesakiII}.  Let $y\in\mf n_{\tilde\phi}$ so $\phi(T(y^*y))<\infty$, where here $\phi$ has been extended to $\widehat{M}_+$, see Section~\ref{sec:ext_weights} below.  As in the proof of \cite[Lemma~IX.4.21]{TakesakiII}, $T(y^*y)$ has spectral decomposition
\[ T(y^*y) = \int_0^\infty t \rd E, \]
where $E$ is a spectral measure with values in $N$.  For each $n$ let $e_n = E([0,\infty])$, so $e_n T(y^*y) = T(y^*y) e_n = \int_0^n t \rd E \in N$ because for each Borel $A\subseteq [0,\infty)$ we have $e_n E(A) = E(A) e_n = E(A \cap [0,n])$.  Set $y_n = y e_n$ so $T(y_n^*y_n) = e_n T(y^*y) e_n$ by the bimodule property.  As the operators involve commute, we have that $T(y_n^*y_n) \leq T(y_m^*y_m)$ for $n\leq m$, and so
\[ \lim_n \phi(T(y_n^*y_n)) = \sup_n \phi(T(y_n^*y_n)) = \phi(T(y^*y)) = \tilde\phi(y^*y) < \infty. \]
Hence each $y_n \in \mf n_{\tilde\phi} \cap \mf n_T$.  Further, $\tilde\phi(y_n^*y) = \phi(T(e_n y^*y)) = \phi(e_n T(y^*y)) = \phi(e_n T(y^*y)e_n) = \tilde\phi(y_n^*y_n)$ which also equals $\tilde\phi(y^*y_n)$, and so
\[ \|\Lambda_{\tilde\phi}(y) - \Lambda_{\tilde\phi}(y_n)\|^2
= \tilde\phi(y^*y) - \tilde\phi(y_n^*y) - \tilde\phi(y^*y_n) + \tilde\phi(y_n^*y_n)
= \tilde\phi(y^*y) - \tilde\phi(y_n^*y_n) \to 0, \]
as $n\to\infty$.

Given $a\in D(\sigma^\phi_{i/2}) \subseteq N$, as $\sigma^{\tilde\phi}_t = \sigma^\phi_t$ on $N$ (see \cite[Lemma~IX.4.21]{TakesakiII}), also $a\in D(\sigma^{\tilde\phi}_{i/2}) \subseteq M$, and so
\[ \Lambda_{\tilde\phi}(xa) = J \sigma_{i/2}(a)^* J \Lambda_{\tilde\phi}(x)
\qquad (x\in\mf n_{\tilde\phi}). \]
Now let $(a_i)$ be the net from Proposition~\ref{prop:tomita_nice_bai}, so each $a_i \in D(\sigma^\phi_{i/2}) \cap \mf n_\phi$ and $\sigma_{i/2}(a_i)^* \to 1$ strongly.  Hence $\lim_i \Lambda_{\tilde\phi}(xa_i) = \Lambda_{\tilde\phi}(x)$ for each $x\in\mf n_{\tilde\phi}$.

It follows that $y_n \otimes \Lambda_\phi(a_i) \mapsto \Lambda_{\tilde\phi}(y_na_i)$ converges to $\Lambda_{\tilde\phi}(y)$ as $i\to\infty$ and then $n\to\infty$.  The result follows.
\end{proof}

For $\alpha\in\mf n_T$ we see that $\hat\alpha \Lambda_\phi(x) = \alpha \otimes \Lambda_\phi(x) \cong \Lambda_{\tilde\phi}(\alpha x)$ for $x\in\mf n_\phi$.  As discussed in Section~\ref{sec:selfdual_mods}, essentially Propositions~\ref{prop:all_self_dual_mods} and~\ref{prop:adjointables_extend_sdc}, that $E_T^0 = \mf n_T$ is an $M$-$N$-bimodule means that $H_T = L^2(M,\tilde\phi)$ becomes an  $M$-$N$-bimodule.  These actions are the natural ones:

\begin{lemma}
The left $M$ action on $E_T^0$ becomes the natural left action of $M$ on $L^2(M,\tilde\phi)$, and the right $N$ action on $H_T$ is the restriction, to $N$, of the right $M$ action on $L^2(M,\tilde\phi)$.
\end{lemma}
\begin{proof}
For $x\in M$ we have $\pi(x)\alpha = x\alpha$ and so $\pi(x) \hat\alpha = \wh{x\alpha} \colon \Lambda_\phi(a) \mapsto \Lambda_{\tilde\phi}(x\alpha a) = x \Lambda_{\tilde\phi}(\alpha a)$.  That is, $\pi(x)$ is the natural left action of $x\in M$ on $L^2(M, \tilde\phi)$.

We now consider the right $N$ action.  By Lemma~\ref{lem:action_J}, $\Lambda_\phi(x) \cdot y = Jy^*J \Lambda_\phi(x) = \Lambda_\phi(x \sigma_{-i/2}(y))$ for $x\in\mf n_\phi, y\in D(\sigma_{-i/2}) \subseteq N$.
Again by \cite[Lemma~IX.4.21]{TakesakiII}, we know that $\sigma_t^{\tilde\phi} = \sigma_t^\phi$ when restricted to $N$, and so
\[ \hat\alpha \Lambda_\phi(x \sigma_{-i/2}(y))
= \Lambda_{\tilde\phi}(\alpha x \sigma_{-i/2}(y))
= Jb^*J \Lambda_{\tilde\phi}(\alpha x). \]
Hence the right action of $N$ on $H_T = L^2(M,\tilde\phi)$ is simply the restriction of the natural right action of $M$ on $L^2(M, \tilde\phi)$.
\end{proof}

We can recover $T$ from this construction.  We have $(\hat\alpha\Lambda_\phi(x)|\pi(z) \hat \beta\Lambda_\phi(y)) = (\Lambda_{\tilde\phi}(\alpha x)|\Lambda_{\tilde\phi}(z\beta y))
= (\Lambda_\phi(x)|T(\alpha^* z \beta)\Lambda_\phi(y))$ and so $\hat\alpha^* \pi(z) \hat\beta = T(\alpha^* z \beta)$.

\begin{question}\label{ques:op_valued_weight_from_mod}
Given nfs weights $\phi$ and $\tilde\phi$, the existence theorem for operator-valued weights, \cite[Theorem~IX.4.18]{TakesakiII}, says that there is (a unique) $T\colon M_+ \to \wh N_+$ with $\tilde\phi = T\circ \phi$ if and only if the modular automorphism groups for $\phi$ and $\tilde\phi$ agree on $N$.

Can we use the Hilbert module construction to shed light on this?  Let $\mf n$ be the collection of $\alpha\in M$ with $\hat\alpha\colon\Lambda_\phi(x) \mapsto \Lambda_{\tilde\phi}(\alpha x)$ bounded, for $x\in\mf n_\phi$.  This is a left ideal, containing $\mf n_T$, and by the result below, actually equals $\mf n_T$.

We could \emph{define} $S(\alpha^* \beta) = \hat\alpha^* \hat\beta \in M$.  Do we have $S = T$?  Starting from an abstract self-dual module $E$ and some map $\mf n \to E$, when do we get an operator-valued weight from this sort of procedure?  Very related is \cite[Section~6]{Kustermans_RegCstarWeights}.
\end{question}

\begin{theorem}\label{thm:nT_vs_module}
For $\alpha\in M$ the following are equivalent:
\begin{enumerate}[(a)]
  \item\label{thm:nT_vs_module:1}
    $\alpha\in\mf n_T$;
  \item\label{thm:nT_vs_module:2}
    $\alpha x \in \mf n_{\tilde\phi}$ for each $x\in \mf n_{\phi}$, and the resulting map $\Lambda_\phi(x) \mapsto \Lambda_{\tilde\phi}(\alpha x)$ is bounded.
\end{enumerate}
The map $\mf n_T \to E_T; \alpha \mapsto \hat\alpha$ is closed for the $\sigma$-strong topology on $\mf n_T$ and the SOT on $E_T$.  It is also closed for the weak$^*$-topologies on $\mf n_T$ and $E_T$.
\end{theorem}
\begin{proof}
We have already seen that \ref{thm:nT_vs_module:1}$\implies$\ref{thm:nT_vs_module:2}.  Let $\alpha$ be as in \ref{thm:nT_vs_module:2}, and let $\hat\alpha$ be the resulting bounded map.
For $x\in\mf n_\phi$ we have $\| \Lambda_{\tilde\phi}(\alpha x) \|^2 = \tilde\phi(x^*\alpha^*\alpha x) = \phi( x^* T(\alpha^*\alpha) x) = T(\alpha^*\alpha)(\omega)$ where $\omega = \omega_{\Lambda_\phi(x)}$, by Lemma~\ref{lem:conj_weight_extension}.  Thus $T(\alpha^*\alpha)(\omega_{\Lambda_\phi(x)}) \leq \|\hat\alpha\|^2 \|\Lambda_\phi(x)\|^2 = \|\hat\alpha\|^2\|\omega_{\Lambda_\phi(x)}\|$.  

Now let $\omega\in N_*^+$ be arbitrary, so $\omega = \omega_\xi$ for some $\xi\in L^2(\phi)$, and we can find $(x_n)\subseteq \mf n_\phi$ with $\Lambda_\phi(x_n) \to \xi$, so $\omega_{\Lambda_\phi(x_n)} \to \omega_\xi$ in norm.  As $T(\alpha^*\alpha) \in \widehat{N}_+$ is by definition lower semi-continuous, we have that $T(\alpha^*\alpha)(\omega) \leq \liminf_n T(\alpha^*\alpha)(\omega_{\Lambda_\phi(x_n)}) \leq \|\hat\alpha\|^2 \|\omega\|$.
So $T(\alpha^*\alpha) < \infty$ and $\alpha\in\mf n_T$; in fact, it now follows that $T(\alpha^*\alpha) = \hat\alpha^*\hat\alpha$.

As in \cite[Theorem~VII.1.11]{TakesakiII}, let $\Phi_{\tilde\phi} = \{ \tilde\omega \in M^+_* : \tilde\omega \leq \tilde\phi \}$, so that $\tilde\phi(x) = \sup\{ \tilde\omega(x) : \tilde\omega\in \Phi_{\tilde\phi} \}$ for each $x\in M_+$.  Let $(\alpha_i)\subseteq\mf n_T$ be a net with $\alpha_i\to \alpha$ $\sigma$-strongly in $M$, and with $\hat\alpha_i \to x_0$ strongly in $\mc B_N(L^2(N), L^2(M))$.  For $\tilde\omega \in \Phi_{\tilde\phi}$, the map $t_{\tilde\omega} \colon L^2(\tilde\phi) \to L^2(\tilde\omega); \Lambda_{\tilde\phi}(b) \mapsto \Lambda_{\tilde\omega}(b)$ is contractive and so extends to all of $L^2(\tilde\phi)$; it is clearly onto.  Then for $x\in\mf n_\phi$,
\[ t_{\tilde\omega}\Lambda_{\tilde\phi}(\alpha_ix) = \Lambda_{\tilde\omega}(\alpha_ix) = \alpha_i \Lambda_{\tilde\omega}(x) \to \alpha \Lambda_{\tilde\omega}(x) = \Lambda_{\tilde\omega}(\alpha x), \]
as $\alpha_i \to \alpha$ strongly on $L^2(\tilde\omega)$.  However, we also have
\[ t_{\tilde\omega}\Lambda_{\tilde\phi}(\alpha_i x)
= t_{\tilde\omega} \hat\alpha_i \Lambda_{\phi}(x)
\to t_{\tilde\omega} x_0 \Lambda_{\phi}(x), \]
as $\hat\alpha_i \to x_0$ strongly.  So $\Lambda_{\tilde\omega}(\alpha x ) = t_{\tilde\omega} x_0 \Lambda_{\phi}(x)$ and hence $\tilde\omega((\alpha x)^*\alpha x) \leq \|x_0\Lambda_{\phi}(x)\|^2 \leq \|x_0\|^2 \|\Lambda_{\phi}(x)\|^2$.  Taking the supremum over $\tilde\omega \in \Phi_{\tilde\phi}$ shows that $\tilde\phi((\alpha x)^*\alpha x) \leq \|x_0\|^2 \|\Lambda_{\phi}(x)\|^2$ so $\alpha x \in \mf n_{\tilde\phi}$ with $\| \Lambda_{\tilde\phi}(\alpha x)\| \leq \|x_0\| \|\Lambda_{\phi}{x}\|$.  Furthermore, $t_{\tilde\omega} \Lambda_{\tilde\phi}(\alpha x) = t_{\tilde\omega} x_0\Lambda_{\phi}(x)$ for each $\tilde\omega$, and so $\Lambda_{\tilde\phi}(\alpha x) = x_0 \Lambda_{\phi}(x)$.  This holds for all $x \in \mf n_\phi$, so by the first part, we conclude that $\alpha\in\mf n_T$, and further, $\hat\alpha \Lambda_{\phi}(x) = \Lambda_{\tilde\phi}(\alpha x) = x_0 \Lambda_{\phi}(x)$ for all $x$, so $x_0 = \hat\alpha$ as desired.

Finally, we modify the argument for the weak$^*$-topologies.  For $\xi\in L^2(\tilde\omega)$,
\[ (\xi|\Lambda_{\tilde\omega}(\alpha x))
= \lim_i (\xi | t_{\tilde\omega}\Lambda_{\tilde\phi}(\alpha_ix))
= \lim_i (t_{\tilde\omega}^*(\xi)|\hat\alpha_i\Lambda_\phi(x))
= (t_{\tilde\omega}^*(\xi)|x_0\Lambda_\phi(x)), \]
firstly using that $\alpha_i\to\alpha$ weak$^*$ in any representation, and secondly that $\hat\alpha_i \to x_0$ weak$^*$ in $E_T \subseteq \mc B(L^2(\phi), L^2(\tilde\phi))$.  This shows that $\Lambda_{\tilde\omega}(\alpha x ) = t_{\tilde\omega} x_0 \Lambda_{\phi}(x)$ and then the argument proceeds as before.
\end{proof}

\subsection{Extending weights}\label{sec:ext_weights}

We shall need some technical results about extending weights on $M$ to the extended positive part of $M$.  This is stated without proof in \cite[Corollary~IX.4.9]{TakesakiII}.  In both \cite[Proposition~11.4]{Stratila_ModTheoryBook2} and \cite[Proposition~1.10]{Haagerup_OpValuedWeightsI} the result is proved by using that any weight is a sum of normal positive functionals.  

Instead, we give an alternative proof using increasing cones.  Let $\varphi$ be an nfs weight on $M$.  Define
\[ \Phi^0_\varphi = \{ t\omega\in M_*^+ : 0<t<1, \omega(x) \leq \varphi(x) \ (x\in M_+) \}, \]
so that $\Phi^0_\varphi$ is a directed set, \cite[Theorem~2.6]{Stratila_ModTheoryBook2}.  This result is attributed to Coombes; for an accessible proof, compare \cite[Proposition~3.5]{kustermans1997kmsweightscalgebras}.  By the definition of normality for a weight, \cite[Theorem~VII.1.11]{TakesakiII}, we have that $\varphi(x) = \sup \{ \omega(x) : \omega \in \Phi^0_\varphi \}$ for $x\in M_+$.  We now give a different proof of the extension result.

\begin{proposition}\label{prop:weight_ext}
Let $\varphi$ be a normal weight on $M$, with associated $\Phi_\varphi^0$.  Define
\[ \hat\varphi(m) = \sup\{ m(\omega) : \omega\in\Phi^0_\varphi \} \qquad (m\in \widehat{M}_+). \]
Then $\hat\varphi$ agrees with $\varphi$ on $M_+$, is positive homogeneous, additive, and when $(m_i)$ is an increasing net in $\widehat{M_+}$ we have $\hat\varphi(\sup_i m_i) = \sup_i \hat\varphi(m_i)$.  Such an extension $\hat\varphi$ is unique.
\end{proposition}
\begin{proof}
Obviously $\hat\varphi(x) = \varphi(x)$ for $x\in M_+$ and $\hat\varphi$ is positive homogeneous.  Given $m_1,m_2 \in \widehat{M_+}$ we have $\hat\varphi(m_1+m_2) \leq \hat\varphi(m_1) + \hat\varphi(m_2)$.  If $\hat\varphi(m_1)=\infty$, or $\hat\varphi(m_2)=\infty$, then also $\hat\varphi(m_1+m_2)=\infty$.  If both $\hat\varphi(m_1)<\infty, \hat\varphi(m_2)<\infty$, then given $\epsilon>0$ we can find $\omega_i \in \Phi_\varphi^0$ with $m_i(\omega_i) > \hat\varphi(m_i) - \epsilon$.  As $\Phi_\varphi^0$ is directed, there is $\omega\in \Phi_\varphi^0$ with $\omega_i\leq\omega$ for $i=1,2$, and so $\hat\varphi(m_1 + m_2) \geq m_1(\omega) + m_2(\omega) \geq m_1(\omega_1) + m_2(\omega_2) > \hat\varphi(m_1) + \hat\varphi(m_2) - 2\epsilon$.  As $\epsilon>0$ was arbitrary, $\hat\varphi(m_1 + m_2) \geq\hat\varphi(m_1) + \hat\varphi(m_2)$, and so we have equality, and conclude that $\hat\varphi$ is additive.

Given an increasing net $(m_i)$ we have
\[ \hat\varphi(\sup m_i) = \sup\{ \sup_i m_i(\omega) : \omega \in \Phi_\varphi^0 \}
= \sup_i \sup\{ m_i(\omega) : \omega \in \Phi_\varphi^0 \}
= \sup_i \hat\varphi(m_i). \]
Uniqueness follows as any $m\in\widehat{M}_+$ is the limit of an increasing family in $M_+$.
\end{proof}

We can now give our technical result.

\begin{lemma}\label{lem:conj_weight_extension}
Let $\varphi$ be a normal weight on $M$, let $a\in \mf n_\varphi$, and define $\omega_0 = \omega_{\Lambda(a)} \in M_*^+$, so $\omega_0(x) = \varphi(a^*xa)$ for $x\in M$.  Then $\hat\varphi(a^*ma) = m(\omega_0)$ for each $m\in \widehat{M}_+$.
\end{lemma}
\begin{proof}
Recall that $a^*ma \in \widehat{M}_+$ is defined by $(a^*ma)(\omega) = m(a\omega a^*)$ for $\omega\in M_*^+$.
For $x \in M_+$, as $\omega$ increases in $\Phi_\varphi^0$ (which is a directed set) we have that $(a\omega a^*)(x) = \omega(a^*xa) \uparrow \varphi(a^*xa) = \omega_0(x)$.  By linearity, $(a\omega a^*)(x) \to \omega_0(x)$ for each $x\in M$, and so the weak closure of $C = \{ a\omega a^* : \omega\in \Phi_\varphi^0 \}$ contains $\omega_0$.  As $C$ is a convex set (because $\Phi_\varphi^0$ is) it follows that the norm closure of $C$ contains $\omega_0$.  As $(a \omega a^*)$ is an increasing net, and as any $m\in \widehat{M}_+$ is lower semi-continuous, $m(\omega_0) \leq \liminf m(a\omega a^*) = \sup (a^*ma)(\omega) = \hat\varphi(a^*ma)$.

As any $\omega\in \Phi_\varphi^0$ satisfies $\omega(x)\leq\varphi(x)$ for $x\geq 0$, also $a\omega a^*(x) = \omega(a^*xa) \leq \varphi(a^*xa) = \omega_0(x)$, so that $a\omega a^* \leq \omega_0$.  Thus $m(a\omega a^*) \leq m(\omega_0)$ and taking the supremum shows that $\hat\varphi(a^*ma) \leq m(\omega_0)$.  We hence have equality, as claimed.
\end{proof}

\subsection{Slicing with weights}\label{sec:slicing_with_weights}

Let $N,M$ be von Neumann algebras, let $\psi$ be a weight on $N$, and define an operator-valued weight $\psi\otimes\id \colon (N\vnten M)_+ \to \wh{M}_+$ (where we identify $M\cong 1\vnten M$ as a subalgebra of $N\vnten M$) by
\[ \ip{(\psi\otimes\id)(x)}{\omega} = \psi\big( (\id\otimes\omega)(x) \big)
\qquad (\omega \in M_*^+, x\in (N\vnten M)_+). \]
It is a routine check that this is an operator-valued weight.  We see that $x \in \mf n_{\varphi\otimes\id}$ exactly when there is $z\in M^+$ with
\[ \ip{z}{\omega} = \varphi( (\id\otimes\omega)(x^*x) ) \qquad (\omega\in M_*^+). \]

Given a weight $\varphi$ on $M$, we extend $\varphi$ to $\wh M_+$ as in Section~\ref{sec:ext_weights}, and then we can make sense of $\varphi \circ (\psi\otimes\id)$.  This is one way to define the tensor product of weights.  Alternatively, it is natural to take the tensor products of the underlying Hilbert algebras: this is the approach of \cite[Section~VIII.4]{TakesakiII} and \cite[Section~8]{Stratila_ModTheoryBook2}.  It is perhaps surprising that it is not obvious that these two constructions give the same result.  Form $\Phi^0_\varphi$ and $\Phi^0_\psi$ as in Section~\ref{sec:ext_weights}.  Then, for $x\in (N\vnten M)_+$,
\begin{align*}
\varphi(\psi\otimes\id)(x)
&= \sup_{\omega\in\Phi^0_\varphi} \psi((\id\otimes\omega)(x))
= \sup_{\omega\in\Phi^0_\varphi} \sup_{\tau\in\Phi^0_\psi} (\tau\otimes\omega)(x)
= \sup_{\tau\in\Phi^0_\psi} \varphi( (\tau\otimes\id)(x) )
= \psi(\id\otimes\varphi)(x),
\end{align*}
where we define $\id\otimes\varphi$ in the obvious way.  Notice here we implicitly used the extension procedure of Proposition~\ref{prop:weight_ext}.  Thus we have a natural symmetry, and furthermore, from \cite[Proposition~8.3]{Stratila_ModTheoryBook2}, it immediately follows from this calculation that $\varphi(\psi\otimes\id) = \psi\otimes\varphi$.

We apply the construction of Section~\ref{sec:op_valued_weights} to $\psi\otimes\id$.  So we consider $L^2(N\vnten M)$ constructed from the weight $\varphi\circ(\psi\otimes\id) = \psi\otimes\varphi$, which by construction means that $L^2(N\vnten M) = L^2(N,\psi) \otimes L^2(M, \varphi)$.  The self-dual completion of $E^0_{\psi\otimes\id}$ is hence $\mc B_{M}(L^2(M), L^2(N) \otimes L^2(M))$.  As the tensor product of weights gives rise to the tensor product of the modular conjugations, the right action of $M$ on $L^2(N) \otimes L^2(M)$ is just the action on the $L^2(M)$ factor.

For $x\in \mf n_{\psi\otimes\id}$ we hence obtain $\hat x \colon L^2(M) \to L^2(N) \otimes L^2(M)$.  We'll write $\Lambda_{\psi\otimes\id}(x)$ for this map, by analogy with the GNS construction.  Then 
\[ \Lambda_{\psi\otimes\id}(x) \colon \Lambda(a) \mapsto \Lambda(x(1\otimes a))
\qquad (x\in \mf n_{\psi\otimes\id}\subseteq N\vnten M, a\in \mf n_{\varphi}\subseteq M). \]
It is easy to see that $\mf n_\psi \odot M \subseteq \mf n_{\psi\otimes\id}$.  So for $a\in\mf n_\varphi$ and $x\in\mf n_\psi$, we have $x\otimes 1 \in \mf n_{\psi\otimes\id}$ and $\Lambda_{\psi\otimes\id}(x\otimes 1)\Lambda(a) = \Lambda(x\otimes a)$, and so we clearly obtain all of $L^2(N) \otimes L^2(M)$, compare Proposition~\ref{prop:HT_is_L2_tildephi}.

This construction is an analogue of the $C^*$-algebraic construction in \cite[Section~3]{kustermans1999weighttheorycalgebraicquantum}, see in particular Proposition~3.18 in that paper.  Analogously, in our setting, we see that $\Lambda_{\varphi\otimes\id}(yx) = \pi(y) \Lambda_{\varphi\otimes\id}(x)$ for $x\in \mf n_{\varphi\otimes\id}, y\in N\vnten M$, and where $\pi$ is the natural representation of $N\vnten M$ on $L^2(N) \otimes L^2(M)$.

\begin{lemma}\label{lem:slice_alt}
Let $x\in \mf n_{\psi\otimes\id}$.  For $\xi,\eta\in L^2(M)$ we have
$y = (\id\otimes\omega_{\xi,\eta})(x^*) \in \mf n_{\psi}^*$, and
\begin{equation}
\big( \Lambda_{\psi\otimes\id}(x)\xi \big| \Lambda(c)\otimes \eta \big)
= \psi\big( (\id\otimes\omega_{\xi,\eta})(x^*) c  \big)
\qquad (c\in\mf n_{\psi}).
\label{eq:opvaluedslice}  
\end{equation}
\end{lemma}
\begin{proof}
Firstly,
\[ yy^* = (\id\otimes\omega_{\xi,\eta})(x^*) (\id\otimes\omega_{\eta,\xi})(x)
= (\id\otimes\omega_{\xi})\big( x^*(1\otimes\rankone{\eta}{\eta})x \big)
\leq \|\eta\|^2 (\id\otimes\omega_{\xi})(x^*x) \in \mf m_\psi^+, \]
so indeed $y^* \in \mf n_\psi$.

Now let $c\in\mf n_{\psi}$, let $a,b\in\mf n_\varphi$ and set $\xi=\Lambda(a), \eta=\Lambda(b)$, so
\begin{align*}
\big( \Lambda_{\varphi\otimes\id}(x)\xi \big| \Lambda(c)\otimes \eta \big)
&= \big( \Lambda(x(1\otimes a)) \big| \Lambda(c\otimes b) \big)
= (\psi\otimes\varphi)\big( (1\otimes a^*)x^*(c\otimes b) \big) \notag\\
&= \varphi\big( a^* (\psi\otimes\id)(x^*(c\otimes 1)) b \big)
= \ip{(\psi\otimes\id)(x^*(c\otimes 1))}{\omega_{\xi,\eta}} \notag\\
&= \psi\big( (\id\otimes\omega_{\xi,\eta})(x^*(c\otimes 1))  \big)
= \psi\big( (\id\otimes\omega_{\xi,\eta})(x^*) c  \big).
\end{align*}
By continuity, this holds for all $\xi,\eta$.
\end{proof}

\section{Weights on bounded operators}\label{sec:weights_bh}

We consider weights on $\mc B(H)$.  As $\mc B(H)$ is semi-finite, every weight is of the form $\Tr_h$ where $h$ is a positive operator on $H$, \cite[Theorem~VIII.3.14]{TakesakiII}.  This is defined by
\[ \Tr_h(x) = \lim_{\epsilon\to0} \Tr(h_\epsilon^{1/2}xh_\epsilon^{1/2})
\qquad (x\in\mc B(H)_+), \]
where $h_\epsilon = h(1+\epsilon h)^{-1}$.
As we'll shortly see, $\Tr(h_\epsilon^{1/2}xh_\epsilon^{1/2})$ increases as $\epsilon\downarrow0$, so the limit exists in $[0,\infty]$.  Furthermore, $\Tr_h$ is faithful if and only if $h(\xi)\not=0$ for each $0\not=\xi\in D(h)$, see \cite[Lemma~VII.2.8]{TakesakiII}.

In this section, we wish to explictly identify $L^2(\Tr_h)$ with $HS(H)$, the Hilbert--Schmidt operators on $H$, and to explicitly describe $\mf n_{\Tr_h}$.  We're not aware of a reference for such calculations, and as they are central to our arguments, here we give the details.

We recall the spectral theory for self-adjoint, unbounded operators, see \cite[Chapter~13]{Rudin_FunctionalAnalysis} and \cite[Chapter~X]{Conway_FunctionalAnalysisBook}; we follow the latter for notation.

By \cite[Theorem~X.4.7]{Conway_FunctionalAnalysisBook} given a spectral measure $E$ on a measure space $(X,\Omega)$, and given $\phi\colon X\to\mathbb R$ a measurable function, setting
\[ D_\phi = \big\{ \xi\in H : \int |\phi|^2 \rd E_{\xi,\xi} < \infty \big\}, \]
there is a self-adjoint operator $N_\phi \colon D_\phi \to H$ satisfying
\[ (\xi|N_\phi\eta) = \int \phi \rd E_{\xi,\eta} \qquad (\eta\in D_\phi, \xi\in H), \]
noting that $\phi\in D_\phi \implies \phi \in L^1(|E_{\xi,\eta}|)$.  Furthermore, $\| N_\phi \eta \|^2 = \int |\phi|^2 \rd E_{\eta,\eta}$ for $\eta\in D_\phi$.
We have that $N_\phi^* = N_{\overline\phi}$, that $N_{\phi\psi} \supseteq N_\phi N_\psi$ with $D(N_\phi N_\psi) = D_\psi \cap D_{\phi\psi}$, and that $N_\phi^*N_\phi = N_{|\phi|^2}$.  If $\psi$ is bounded, then $N_{\phi\psi} = N_\phi N_\psi = N_\psi N_\phi$.

The Spectral Theorem states that if $T$ is self-adjoint, then there is a spectral measure on $\mathbb R$ such that $T = N_t$ where $t\colon \mathbb R\to\mathbb R$ is the identity, and that $E$ is concentrated on $\sigma(T)$.  In particular, letting $T$ be positive, $E$ is a measure on $[0,\infty)$.  Then $N_{t^{1/2}} N_{t^{1/2}} = N_t = T$ which means in particular that for $\xi\in D(T)$ we have $\xi\in D(T^{1/2})$ and $T^{1/2}\xi \in D(T^{1/2})$.

Returning to defining $\Tr_h$, we define $h_\epsilon = h(1+\epsilon h)^{-1}$ by the Spectral Theorem.  Notice that the function $\phi_\epsilon \colon \mathbb R \to \mathbb R; t \mapsto t / (1+\epsilon t)$ has $\phi_\epsilon(0)=0$, sends $(0,\infty)$ to $(0,\infty)$, satisfies $\phi_\epsilon(t) = 1 / (1/t + \epsilon) \leq 1/\epsilon$ and $\lim_{t\to\infty} \phi_\epsilon(t) = 1/\epsilon$, while for $t$ fixed, $\phi_\epsilon(t) \uparrow t$ as $\epsilon\downarrow 0$.

As $\phi_\epsilon$ is bounded, $h_\epsilon = h(1+\epsilon h)^{-1} = N_{\phi_\epsilon} \in \mc B(H)$.  As $(\phi_\epsilon)$ is (pointwise) increasing as $\epsilon\downarrow 0$, we see that $(\phi_\epsilon)$ increases, and so for $x\geq 0$ we have $\Tr(h_\epsilon^{1/2} x h_\epsilon^{1/2}) = \Tr(x^{1/2} h_\epsilon x^{1/2})$ is increasing.  In particular, the limit used to define $\Tr_h$ exists.

Let $\xi\in D(h^{1/2}) = D_{t^{1/2}}$ so $\int t \rd E_{\xi,\xi} < \infty$.  Then
\[ \Tr(h_\epsilon^{1/2} \rankone{\xi}{\xi} h_\epsilon^{1/2})
= \|h_\epsilon^{1/2}\xi\|^2 = (\xi|h_\epsilon\xi)
= \int t(1+\epsilon t)^{-1} \rd E_{\xi,\xi}. \]
As $t(1+\epsilon t)^{-1} \uparrow t$ as $\epsilon \downarrow 0$, by the Dominated Convergence Theorem, $\lim_{\epsilon\downarrow 0} \Tr(h_\epsilon^{1/2} \rankone{\xi}{\xi} h_\epsilon^{1/2}) = (\xi|h\xi) = \|h^{1/2}\xi\|^2$.
By polarisation, for $\xi,\eta\in D(h^{1/2})$ we have $\Tr_h(\rankone{\eta}{\xi}) = (h^{1/2}\xi|h^{1/2}\eta)$.

For the following, recall that $D(h^{1/2}x) = \{ \xi : x\xi \in D(h^{1/2})\}$ and by saying that $h^{1/2}x$ is bounded, we implicitly mean that $D(h^{1/2}x)$ is dense, with the closure of the operator being bounded.  

\begin{lemma}\label{lem:domain_S_Tr_h}
For any $x\in\mc B(H)$ we have that $(x^*h^{1/2})^*=h^{1/2}x$.
We have that $x^*h^{1/2}$ is bounded if and only if $h^{1/2} x$ is bounded, if and only if $D(h^{1/2}x) = H$.  As such, $x^*h^{1/2} \in HS(H)$ if and only if $h^{1/2}x\in HS(H)$.
\end{lemma}
\begin{proof}
For any $x$, we have that
\begin{align*}
G((x^*h^{1/2})^*)
&= \{ (\alpha,\beta) : (\alpha|x^*h^{1/2}\xi) = (\beta|\xi) \ (\xi\in D(h^{1/2})) \} \\
&= \{ (\alpha,\beta) : (x\alpha,\beta) \in G((h^{1/2})^*) = G(h^{1/2}) \} \\
&= \{ (\xi, h^{1/2}x\xi) : x\xi\in D(h^{1/2}) \} = G(h^{1/2}x),
\end{align*}
in the middle step using that $h^{1/2}$ is self-adjoint.  So always $(x^*h^{1/2})^* = h^{1/2}x$, we see that $h^{1/2}x$ is always closed, and $x^*h^{1/2}$ is closable if and only if $h^{1/2}x$ is densely-defined.

If $x^*h^{1/2}$ is bounded, its adjoint $h^{1/2}x$ is bounded, and if $h^{1/2}x$ is bounded, it is everywhere defined (as it's always closed).  If $D(h^{1/2}x) = H$ then $h^{1/2}x$ is closed and everywhere defined, so bounded by the Closed Graph Theorem.  Finally, if $h^{1/2}x$ is bounded, then $x^*h^{1/2}$ is closable, with closure $(x^*h^{1/2})^{**} = (h^{1/2}x)^*$ which is bounded because $h^{1/2}x$ is.
\end{proof}

\begin{proposition}\label{prop:Trh}
For $x\in\mc B(H)$ the following are equivalent:
\begin{enumerate}[(1)]
  \item\label{prop:Trh:1}
  $x\in\mf n_{\Tr_h}$;
  \item\label{prop:Trh:2}
  for some (equivalently, any) orthonormal basis $(e_n)$ we have that $x^*e_n \in D(h^{1/2})$ for each $n$, and $\sum_n \| h^{1/2} x^* e_n \|^2 < \infty$;
  \item\label{prop:Trh:3}
  $x h^{1/2}$ is bounded, with extension to all of $H$ giving an operator in $HS(H)$.
\end{enumerate}
In this case, $\Tr_h(x^*x) = \|xh^{1/2}\|_{HS}^2$ and so there is a unitary $u \colon L^2(\Tr_h) \to HS(H) = H\otimes\overline H$ satisfying $u\Lambda_{\Tr_h}(x) = xh^{1/2}$.
\end{proposition}
\begin{proof}
Let $x\in \mf n_{\Tr_h}$ so $\Tr(h_\epsilon^{1/2}x^*xh_\epsilon^{1/2}) \uparrow K<\infty$ say.  Hence $(xh_\epsilon^{1/2})$ is a bounded family in $HS(H)$, so for any orthonormal basis $(e_n)$ (this could be uncountable) we have
\[ K \geq \|xh_\epsilon^{1/2}\|_{HS}^2 = \sum_n \| xh_\epsilon^{1/2} e_n \|^2
= \sum_{n,m} |(e_m | xh_\epsilon^{1/2} e_n)|^2
= \sum_{n,m} |(e_m | h_\epsilon^{1/2} x^* e_n)|^2
= \sum_n \| h_\epsilon^{1/2} x^* e_n\|^2. \]
Set $\xi_n = x^* e_n$ so that by the Monotone Convergence Theorem,
\[ K \geq \sum_n (\xi_n | h_\epsilon \xi_n)
= \sum_n \int t (1+\epsilon t)^{-1} \rd E_{\xi_n,\xi_n}
\xrightarrow{\epsilon\to0}
\sum_n \int t \rd E_{\xi_n,\xi_n}
= \sum_n \| h^{1/2} \xi_n \|^2. \]
Hence $x^* e_n \in D(h^{1/2})$ for each $n$ and $\sum_n \| h^{1/2} x^* e_n \|^2 \leq K$, showing \ref{prop:Trh:2}.

Now let \ref{prop:Trh:2} hold for $x$.  Then $e_n \in D(h^{1/2}x^*)$ for each $n$, so $D(h^{1/2}x^*)$ is dense, and we have $h^{1/2}x^* \in HS(H)$, so by the lemma, $xh^{1/2} \in HS(H)$ as well, showing \ref{prop:Trh:3}.

Let $x$ satisfy \ref{prop:Trh:3} so by the lemma again, also $h^{1/2}x^* \in HS(H)$, and $D(h^{1/2}x^*) = H$.  
Let $(e_n)$ be an orthonormal basis, and let $\epsilon>0$.  As $x^* e_n \in D(h^{1/2}) = \{ \xi : \int t \rd E_{\xi,\xi} < \infty \}$, we have
\[ (e_n | xh_\epsilon x^* e_n)
= \int t(1+t\epsilon)^{-1} \rd E_{x^*e_n, x^*e_n}
\xrightarrow{\epsilon\to0}
\int t \rd E_{x^*e_n, x^*e_n}
= \| h^{1/2} x^* e_n \|^2. \]
Summing over $n$ gives $\lim_{\epsilon\to 0}\Tr(xh_\epsilon x^*) = \|h^{1/2}x^*\|_{HS} < \infty$.  As $\Tr(xh_\epsilon x^*) = \Tr(h_\epsilon^{1/2}x^*xh_\epsilon^{1/2})$, this shows that $x\in\mf n_{\Tr_h}$.  Notice that this also shows that $\Tr_h(x^*x) = \|h^{1/2}x^*\|_{HS} = \|xh^{1/2}\|_{HS}$.  The existence of $u$ now follows.
\end{proof}

We henceforth identify the GNS space for $\Tr_h$ with $HS(H)$.  It is easy to see that the left action of $\mc B(H)$ is as $\mc B(H) \otimes 1$ on $H\otimes\overline{H} \cong HS(H)$; that is, left multiplication of operators on $HS(H)$.

To determine the modular operator for $\Tr_h$, we wish to consider tensor products of densely-defined operators, the theory of which is as expected, though proofs are perhaps unexpected technical, see \cite[Section~7.5]{Schmudgen_UnboundedBook}.  Consider $h^{1/2} \otimes h^{-1/2\,\top}$ acting on $H\otimes\overline H$, or equivalently $h^{1/2} \cdot h^{-1/2}$ acting on $HS(H)$.  By definition, this is the closure of $h^{1/2} \odot h^{-1/2\,\top}$ acting on $D(h^{1/2}) \otimes \overline{ D(h^{-1/2}) }$, or as $h^{1/2}$ and $h^{-1/2}$ are self-adjoint, this is the same as the adjoint of $h^{1/2} \odot h^{-1/2\,\top}$.

\begin{lemma}\label{lem:pos_tensor_unbdd}
Let $T,S$ be densely defined and closable operators on $H_1, H_2$ respectively.  Then $(T\otimes S)^*(T\otimes S) = T^*T \otimes S^*S$.
\end{lemma}
\begin{proof}
Recall that $T^*T$ is densely defined and self-adjoint.
For $\xi\in D(T^*T), \eta\in D(S^*S)$, we have that $\xi\otimes\eta \in D(T)\otimes D(S)$ and $(T\otimes S)(\xi\otimes\eta) = T(\xi) \otimes S(\eta)$.  Then $T(\xi) \in D(T^*), S(\eta)\in D(S^*)$ and as $(T\otimes S)^* = T^* \otimes S^*$, \cite[Proposition~7.26]{Schmudgen_UnboundedBook}, we conclude that $(T\otimes S)^*(T\otimes S)(\xi\otimes\eta) = T^*T(\xi) \otimes S^*S(\eta)$.  So $T^*T \odot S^*S \subseteq (T\otimes S)^*(T\otimes S)$ and hence $T^*T \otimes S^*S = (T^*T \odot S^*S)^* \supseteq (T\otimes S)^*(T\otimes S)$.  Taking the adjoint again shows $T^*T \otimes S^*S \subseteq (T\otimes S)^*(T\otimes S)$, and so we have equality.
\end{proof}

\begin{proposition}\label{prop:mod_Trh}
For $\Tr_h$ with GNS construction $HS(H)$, the modular operator is $\nabla = h \otimes h^{-1\,\top}$ and the modular conjugation is $J \colon H\otimes\overline H \to H\otimes\overline H; \xi\otimes\overline\eta \mapsto \eta\otimes\overline\xi$.  As such, $\sigma^{\Tr_h}(x) = h^{it} x h^{-it}$ for all $t\in\mathbb R, x\in\mc B(H)$.
\end{proposition}
\begin{proof}
Define $\nabla^{1/2} = h^{1/2} \otimes h^{-1/2\,\top}$, so by Lemma~\ref{lem:pos_tensor_unbdd}, $(\nabla^{1/2})^2 = (\nabla^{1/2})^* \nabla^{1/2} = \nabla$, as we might hope.  Define $J$ as in the statement, and notice that on $HS(H)$, we have $J(x) = x^*$.

Let $S$ be the adjoint operator, $S \Lambda_{\Tr_h}(x) = \Lambda_{\Tr_h}(x^*)$ for $x\in \mf n_{\Tr_h} \cap \mf n_{\Tr_h}^*$, that is, $xh^{1/2}, h^{1/2}x \in HS(H)$, by Lemma~\ref{lem:domain_S_Tr_h}.  As a densely defined operator on $HS(H)$, we see that $S$ is the closure of $xh^{1/2} \mapsto x^*h^{1/2}$.

Let $x\in\mc B(H)$ with $xh^{1/2}$ and $h^{1/2}x\in HS(H)$.  We will show that $(xh^{1/2}, h^{1/2}x) \in G(\nabla^{1/2})$.  Let $\xi\in D(h^{1/2}), \eta\in D(h^{-1/2})$ so $\rankone{\xi}{\eta} \cong \xi\otimes\overline\eta \in D(h^{1/2} \odot h^{-1/2\,\top})$ with $(h^{1/2} \odot h^{-1/2\,\top})\rankone{\xi}{\eta} = \rankone{h^{1/2}\xi}{h^{-1/2}\eta}$.  Then
\begin{align*}
\big( \rankone{h^{1/2}\xi}{h^{-1/2}\eta} \big| xh^{1/2} \big)_{HS}
    - \big( \rankone{\xi}{\eta} \big| h^{1/2}x \big)_{HS}
&= ( h^{1/2}\xi | xh^{1/2} h^{-1/2}\eta ) - ( \xi | h^{1/2} x \eta ) \\
&= ( x^*h^{1/2}\xi | \eta ) - ( \xi | h^{1/2} x \eta ) = 0,
\end{align*}
using that $(x^*h^{1/2})^* = h^{1/2}x$.  As $D(h^{1/2} \odot h^{-1/2\,\top})$ is the span of such $\rankone{\xi}{\eta}$, this calculation shows that $(xh^{1/2}, h^{1/2}x) \in G((h^{1/2} \odot h^{-1/2\,\top})^*) = G(\nabla^{1/2})$.  Further, as $y = \rankone{\xi}{h^{-1/2}\eta} \in HS(H)$ with $yh^{1/2} = \rankone{\xi}{\eta}, h^{1/2}x = \rankone{h^{1/2}\xi}{h^{-1/2}\eta}$, we see that $\{ x : xh^{1/2}, h^{1/2}x\in HS(H) \}$ is a core for $\nabla^{1/2}$.

Let $y\in D(S)$ with $S(y) = z$.  Then there is a sequence $(x_n)$ in $\mf n_{\Tr_h} \cap \mf n_{\Tr_h}^*$ with $x_n h^{1/2} \to y, x_n^*h^{1/2} \to z$, equivalently, $h^{1/2} x_n \to z^*$, convergence in $HS(H)$.  Then $\nabla^{1/2}( x_nh^{1/2} ) = h^{1/2}x_n \to z^*$ so $y \in D(\nabla^{1/2})$ with $J\nabla^{1/2}(y) = J(z^*) = z$.  Hence $S \subseteq J\nabla^{1/2}$.  We showed above that $\mf n_{\Tr_h} \cap \mf n_{\Tr_h}^*$ is a core for $\nabla^{1/2}$, so we have equality $S = J\nabla^{1/2}$, and so $D(S) = D(\nabla^{1/2})$.  Obviously $J \nabla^{1/2} J = \nabla^{-1/2}$.  The uniqueness claim of \cite[Lemma~VI.1.5(vi)]{TakesakiII} shows that these properties are enough to characterise the modular operator and conjugation.

As $x\in\mc B(H)$ acts on $H\otimes\overline H$ as $x\otimes 1$, the final claim now follows.
\end{proof}

We now see that $\mc B(H)' = J\mc B(H)J = 1 \otimes \mc B(\overline H)$ as $J(x^*\otimes 1)J = 1\otimes x^\top$.

\section{Inverses of weights}\label{sec:inv_weight}

In this section we study the operator-valued weight $\varphi^{-1} \colon \mc B(H) \to M'$ for a weight $\varphi$ on $M$.  The construction is due to Connes \cite[Corollary~16]{Connes_SpatialVN}, but see also \cite[Corollary~12.12]{Stratila_ModTheoryBook2} and \cite[Corollary~IX.4.25]{TakesakiII}.

We recall some of the notation, following \cite{Connes_SpatialVN} and \cite[Section~12]{Stratila_ModTheoryBook2} (see also \cite[Section~IX.3]{TakesakiII} but be aware of differing conventions).  Let $M\subseteq\mc B(H)$ be a von Neumann algebra (for now, we assume nothing about $H$) and let $\varphi$ be a weight on $M$.  Denote by $D(H,\varphi)$ the set of $\xi\in H$ such that there is $C>0$ with $\|x\xi\| \leq C\|\Lambda(x)\|$ for each $x\in\mf n_\varphi$.  Then $D(H,\varphi)$ is dense in $H$, and for each $\xi\in D(H,\varphi)$ there is an operator $R^\varphi(\xi) \colon L^2(\varphi) \to H; \Lambda(x) \mapsto x\xi$.  For $a\in M'\subseteq\mc B(H)$ and $\xi\in D(H,\varphi)$ we have that $a\xi\in D(H,\varphi)$ with $R^\varphi(a\xi) = a R^\varphi(\xi)$.  Notice also that $y R^\varphi(\xi) = R^\varphi(\xi) y$ for $y\in M, \xi\in D(H,\varphi)$.

\begin{example}\label{eg:left_bdd}
Suppose $H=L^2(\varphi)$, and let $\mf A = \Lambda(\mf n_\varphi\cap\mf n_\varphi^*)$ be the full left Hilbert algebra associated with $\varphi$; here we use the notation of \cite[Section~VI.1]{TakesakiII}.  For $\xi\in D(H,\varphi)$ we see that $\|x\xi\| \leq C\|\Lambda(x)\|$ for each $x\in \mf n_\varphi\cap\mf n_\varphi^*$, and so $\xi$ is right bounded in the sense of \cite[Definition~VI.1.6]{TakesakiII}.  Conversely, let $\xi$ be right bounded, and let $x\in\mf n_\varphi$.  By the right Hilbert algerba version of \cite[Theorem~VI.1.26]{TakesakiII} there is a sequence $(x_n)$ in $\mf n_\varphi\cap\mf n_\varphi^*$ with $\Lambda(x_n) \to\Lambda(x)$ and $x_n\to x$ strongly.  Hence $\|x\xi\| = \lim_n \|x_n\xi\| \leq \limsup_n C\|\Lambda(x_n)\| = C\|\Lambda(x)\|$ and so $\xi\in D(H,\varphi)$.

So $D(H,\varphi) = \mf B'$, and $R^\varphi(\xi) = \pi_r(\xi)$ for each $\xi\in D(H,\varphi)$.  As $J\mf B' = \mf B = \Lambda(n_\varphi)$ we see that $D(H,\varphi) = \{ J\Lambda(a) : a\in \mf n_\varphi \}$ and $R^\varphi(J\Lambda(a)) = JaJ$.
\end{example}

Now let $\phi$ be a (normal, semifinite, maybe not faithful) weight on $M' \subseteq \mc B(H)$.  We define the \emph{spatial derivative} $d\phi / d\varphi$ to be the positive self-adjoint operator associated with the form $q$ defined by $D(H,\varphi)$ by $q(\xi) = \phi(R^\varphi(\xi)R^\varphi(\xi)^*)$.  For more about unbounded quadratic forms see \cite{Connes_SpatialVN}, \cite[Chapter~10]{Schmudgen_UnboundedBook}; we like the presentation in \cite[Section~7]{Stratila_ModTheoryBook2} and the appendix of that book.  In particular, \cite[Section~7.3]{Stratila_ModTheoryBook2} shows that $D(q) = \{ \xi \in D(H,\varphi) : q(\xi) < \infty \}$ is dense in $H$ and is a core for $(d\phi/d\varphi)^{1/2}$; see also \cite[Theorem~IX.3.8]{TakesakiII}.

The equality of positive self-adjoint operators is very tricky, depending crucially on definition domains: see for example the counter-example in \cite[Section~6.6]{Stratila_ModTheoryBook2} which shows the existence of positive self-adjoint operators $S,T$ with $D(S)\subseteq D(T)$ and $\|S\xi\| = \|T\xi\|$ for each $\xi\in D(S)$, and yet with $S\not=T$.

\begin{lemma}\label{lem:SSTT}
Let $S,T$ be closed densely defined operators on a Hilbert space $H$, and suppose that $D\subseteq H$ is a core for both $S$ and $T$, and that $\|S\xi\| = \|T\xi\|$ for $\xi\in D$.  Then $S^*S = T^*T$.
\end{lemma}
\begin{proof}
We first show that $D(S)=D(T)$.  Let $\xi\in D(T)$ so as $D$ is a core, there is a sequence $(\xi_n)$ in $D$ with $\xi_n\to\xi$ and $T\xi_n \to T\xi$.  Then $(T\xi_n)$ is bounded, so by hypothesis, also $(S\xi_n)$ is bounded.  By passing to a subsequence if necessary, we may suppose that $S\xi_n \to \eta$ weakly, for some $\eta$.  Hence $(\xi,\eta)$ is in the weak closure of $G(S)$, but $G(S)$ is convex, so $(\xi,\eta)$ is in the norm closure of $G(S)$ which is $G(S)$ itself.  In particular, $\xi\in D(S)$, and so $D(T) \subseteq D(S)$.  By symmetry, also $D(S) \subseteq D(T)$ and we have equality.  Let $D_0 = D(S) = D(T)$.

By definition, we have that
\[ (\xi,\eta) \in G(T^*) \quad\iff\quad
(\eta, -\xi) \in \{ (\alpha, T\alpha) :  \alpha\in D_0 \}^\perp. \]
As $(S\xi|S\xi) = (T\xi|T\xi)$ for $\xi\in D_0$, by polarisation, also $(S\xi | S\eta) = (T\xi | T\eta)$ for $\xi,\eta\in D_0$.  Thus $(\xi,\eta) \in G(T^*T)$ if and only if $\xi\in D(T)$ and $(T\xi, \eta) \in G(T^*)$, if and only if $\xi\in D_0$ and $(\alpha|\eta) = (T\alpha|T\xi) = (S\alpha|S\xi)$ for each $\alpha\in D_0$.  We now see that this is the same as $(\xi,\eta) \in G(S^*S)$ and so $S^*S  = T^*T$ as claimed.
\end{proof}

\begin{theorem}\label{thm:varphiinv}
Let $\varphi$ be a weight on $M$, and let $H = L^2(\varphi)$.  We have that $\varphi^{-1}(\rankone{J\Lambda(a)}{J\Lambda(b)}) = Jab^*J$ for $a,b\in\mf n_\varphi$.
Let $\varphi'$ be the canonical weight on $M'\subseteq\mc B(L^2(\varphi))$ and set $\tilde\varphi = \varphi' \circ \varphi^{-1}$ a weight on $\mc B(H)$.  Then $\tilde\varphi$ has density given by $\nabla^{-1}$.
\end{theorem}
\begin{proof}
By \cite[Corollary~16]{Connes_SpatialVN}, we have that $\varphi^{-1}(\rankone{\xi}{\xi}) = R^\varphi(\xi)R^\varphi(\xi)^*$ for $\xi\in D(H,\varphi)$.  As observed in Example~\ref{eg:left_bdd}, $D(H,\varphi) = J\Lambda(\mf n_\varphi)$.  Thus for $a\in\mf n_\varphi$ we have $\varphi^{-1}(\rankone{J\Lambda(a)}{J\Lambda(a)}) = (JaJ) (JaJ)^* = Jaa^*J$, and the first claim follows by polarisation.

In fact, $\varphi^{-1}$ is constructed so that $\tilde\varphi = \varphi' \circ \varphi^{-1}$ has density given by $h = d\varphi' / d\varphi$.  We construct the form $q$, noting that $q(J\Lambda(a)) = \varphi'(Jaa^*J) = \varphi(aa^*)$ for $a\in\mf n_\varphi$.  So $D(q) = \{ J\Lambda(a) : a \in \mf n_\varphi \cap \mf n_\varphi^* \} = J\mf A = \mf A'$ is a core for $h^{1/2}$ and by definition, $\| h^{1/2} J\Lambda(a) \|^2 = q(J\Lambda(a)) = \varphi(aa^*) = \|\Lambda(a^*)\|^2 = \|\nabla^{-1/2} J\Lambda(a)\|$ for each $a \in \mf n_\varphi \cap \mf n_\varphi^*$.  Now, $D(\nabla^{-1/2}) = D^\flat$ in the notation of \cite[Lemma~VI.1.5]{TakesakiII}, and \cite[Lemma~VI.1.13]{TakesakiII} shows that $\mf A' = D(q)$ is a core for $F = J\nabla^{-1/2}$ and hence is a core for $\nabla^{-1/2}$.  Lemma~\ref{lem:SSTT} now shows that $\nabla^{-1} = h$.
\end{proof}

\bibliographystyle{plain}
\bibliography{qg.bib}

\end{document}